\pdfoutput=1
\documentclass[10pt,a4paper,final]{memoir}
\usepackage[english]{babel}
\makeatletter\AtBeginDocument{\let\@elt\relax}\makeatother 
\usepackage{amsmath,amssymb}
\usepackage[only,ovee,owedge]{stmaryrd}
\usepackage{mathrsfs}
\usepackage{amsthm,thmtools}
\usepackage{enumitem}
\usepackage[shortcuts]{extdash}
\usepackage{lmodern}
\usepackage{microtype}
\usepackage{tikz}
\usetikzlibrary{cd,arrows,calc}
\usepackage[unicode]{hyperref}
\usepackage[toc=true,numberedsection=nameref,nomain,symbols]{glossaries} 

\newcommand{\myaddress}[1]{\def\theaddress{#1}}
\newcommand{\mythanks}[1]{\def\thethanks{#1}}
\newcommand{\mymathclass}[1]{\def\themathclass{#1}}
\newcommand{\mykeywords}[1]{\def\thekeywords{#1}}
\newcommand{\myabstract}[1]{\long\def\theabstract{#1}}

\title{Tensor Products of Convex Cones}
\author{Josse van Dobben de Bruyn}
\myaddress{Delft Institute of Applied Mathematics\\ Delft University of Technology\\ Mekelweg 4\\ 2628 CD Delft\\ The Netherlands\\ E-mail: J.vanDobbendeBruyn@tudelft.nl}
\mythanks{The author was partially supported by the Dutch Research Council (NWO), project number 613.009.127.}
\mymathclass{Primary: 46A40. Secondary: 06F20, 46A32, 52A05, 47L07, 52A20}
\mykeywords{Convex cone, partially ordered vector space, ordered tensor product, face, order ideal}

\myabstract{
	Tensor products of convex cones have recently come up in several different areas of mathematics (and beyond), ranging from functional analysis and operator theory to approximation theory and theoretical physics.
	However, most of the existing literature focuses on one of two particular cases, namely Archimedean lattice cones (in the functional analysis literature) or closed, proper and generating cones in finite-dimensional spaces (in linear algebra and most applications to other fields).
	This excludes most cones from being considered, including even standard cones such as infinite-dimensional positive semidefinite cones and lexicographical cones.
	For general cones, results are few and far between, and many basic questions remain unanswered.
	
	In this memoir, we aim to develop the theory of tensor products of convex cones in full generality, without any restrictions on the cones or the ambient spaces.
	We generalize a few known results to the general case, and we prove many results which are altogether new.
	Our main contributions are the following:
	\begin{enumerate}[label=(\roman*)]
		\item We show that the projective/injective cone has mapping properties analogous to those of the projective/injective norm;
		\item We establish direct formulas for the lineality space of the projective/injective cone, in particular providing necessary and sufficient conditions for the cone to be proper;
		\item We prove that the projective/injective tensor product of two closed proper cones is contained in a closed proper cone;
		\item We show how to construct faces of the projective/injective cone from faces of the base cones, in particular providing a complete characterization of the extremal rays of the projective cone.
		As an application, we also show that the tensor product of two symmetric convex sets preserves proper faces;
		\item For closed cones in finite-dimensional spaces, we show that the projective cone is closed, and almost always strictly contained in the injective cone, thereby confirming a conjecture of Barker for nearly all convex cones.
	\end{enumerate}
	As this manuscript was being written, this last result was superseded by simultaneous discovery by Aubrun, Lami, Palazuelos and Pl\'avala \cite{Aubrun-et-al-ii}, who independently proved Barker's conjecture in full generality (for all closed, proper and generating cones in finite-dimensional spaces).
	We recover their result for a large class of cones, using completely different techniques.
}

\newcommand{\hair}{\ifmmode\mskip1mu\else\kern0.08em\fi}

\let\mycheckmark\checkmark

\newcommand{\footnotehack}[1]{\footnote{#1}\addtocounter{footnote}{-1}\addtocounter{Hfootnote}{-1}} 

\settypeblocksize{610pt}{410pt}{*}
\setlrmargins{*}{*}{1}
\setulmargins{*}{*}{1}
\checkandfixthelayout[nearest]

\makechapterstyle{madsen-nomargin}{%
	\chapterstyle{default}

}
\chapterstyle{madsen-nomargin}

\setsecnumdepth{section}

\copypagestyle{mypagestyle}{ruled}
\makeatletter
\makepsmarks{mypagestyle}{%
	\createmark{chapter}{left}{shownumber}{}{.~}
	\createmark{section}{right}{shownumber}{}{.~}
}
\makeatother
\nouppercaseheads
\makeevenhead{mypagestyle}{\thepage}{}{\leftmark}
\makeoddhead{mypagestyle}{\leftmark}{}{\thepage}   
\makeevenfoot{mypagestyle}{}{}{}
\makeoddfoot{mypagestyle}{}{}{}
\newcommand*{\minwidthbox}[2]{%
	\makebox[{\ifdim#1<\width\width\else#1\fi}][l]{#2}%
}

\def\linkcolor{blue!60!black}
\def\citecolor{green!40!black}
\def\urlcolor{blue!70!black}
\hypersetup{unicode,pdfauthor={\theauthor},pdftitle={\thetitle},colorlinks=true,urlcolor=\urlcolor,citecolor=\citecolor,linkcolor=\linkcolor}

\addto\extrasenglish{%
}

\makeatletter
\def\myautoref{\@ifstar\@myautoref\@@myautoref}
\def\@@myautoref#1#2{\hyperref[#2]{\autoref*{#1}\ref*{#2}}}  
\def\@myautoref#1#2{\hyperref[#1]{\autoref*{#1}(#2)}}        
\makeatother

\newcommand{\myref}[1]{\textit{\ref{#1}}}

\makeatletter 
\def\@mysecref#1\stophere{\hyperref[#1]{\S\ref*{#1}}}
\def\@myappref#1\stophere{\hyperref[#1]{Appendix \ref*{#1}}}
\newcommand*{\mysecref}[1]{\@ifnextchar a\@myappref\@mysecref#1\stophere}
\makeatother
\newcommand{\mychpref}[1]{\autoref{#1}} 

\newcommand{\exampleqed}{\unskip\kern5pt\ensuremath{\scriptstyle\triangle}}

\numberwithin{table}{chapter}
\numberwithin{equation}{chapter}
\makeatletter
\let\c@table=\c@equation 
\makeatother

\declaretheorem[style=definition,sibling=equation]{definition}
\declaretheorem[style=definition,qed=\exampleqed,sibling=definition]{example}
\declaretheorem[style=definition,sibling=definition]{remark}

\declaretheorem[style=plain,sibling=definition]{theorem}
\declaretheorem[style=plain,sibling=definition]{lemma}
\declaretheorem[style=plain,sibling=definition]{proposition}
\declaretheorem[style=plain,sibling=definition]{corollary}
\declaretheorem[style=plain,sibling=definition]{question}

\declaretheorem[style=plain,sibling=definition]{conjecture}

\declaretheorem[style=plain,title={Theorem}]{AlphTheorem}

\declaretheorem[style=plain,sibling=AlphTheorem,title={Corollary}]{AlphCorollary}
\declaretheorem[style=definition,sibling=AlphTheorem,title={Remark}]{AlphRemark}

\newcommand{\N}{\ensuremath{\mathbb{N}}}
\newcommand{\R}{\ensuremath{\mathbb{R}}}
\providecommand{\C}{}
\renewcommand{\C}{\ensuremath{\mathbb{C}}}

\def\tensor{\mathbin{\otimes}}
\DeclareMathOperator{\rank}{rank}
\DeclareMathOperator{\spn}{span}
\DeclareMathOperator{\aff}{aff}
\DeclareMathOperator{\conv}{conv}
\DeclareMathOperator{\ran}{ran}
\DeclareMathOperator{\id}{id}
\DeclareMathOperator{\tr}{tr}
\DeclareMathOperator{\vol}{vol}
\newcommand{\boundary}{\partial}
\DeclareMathOperator{\interior}{int}

\newcommand{\algdual}{^*}
\newcommand{\algdualdual}{^{**}}
\newcommand{\topdual}{'}
\newcommand{\topdualdual}{''}
\newcommand{\polar}{^\circ}

\newcommand{\dualface}{{\hspace*{.3pt}\diamond}}
\newcommand{\predualface}{{\diamond\hspace*{-.6pt}}}

\newcommand{\weak}{{w}}
\newcommand{\weakstar}{{w*}}

\newcommand{\hattensor}{\mathbin{\tilde\otimes}}
\newcommand{\settensor}{\mathbin{\otimes_s}}
\newcommand{\Biltensor}{\mathbin{\boxtimes}}
\newcommand{\SCBiltensor}{\mathbin{\circledast}}
\newcommand{\ind}{{i}}
\newcommand{\epsprod}{\mathbin{\varepsilon}}

\newcommand{\minwedge}[2]{{{#1} \mathbin{\otimes^\pi} \, {#2}}}
\newcommand{\maxwedge}[2]{{{#1} \mathbin{\otimes^\varepsilon} \, {#2}}}
\newcommand{\topminwedge}[3][]{{{#2} \mathbin{\otimes_{#1}^\pi} \, {#3}}}

\newcommand{\hatminwedge}[3][]{{{#2} \mathbin{\tilde\otimes{}_{#1}^\pi} \, {#3}}}
\newcommand{\hatmaxwedge}[3][]{{{#2} \mathbin{\tilde\otimes{}_{#1}^\varepsilon} \, {#3}}}
\let\minwedgeFIN=\minwedge
\let\maxwedgeFIN=\maxwedge
\newcommand{\tensorwedge}[3][\alpha]{{{#2} \mathbin{\otimes^{#1}} \, {#3}}}

\newcommand{\minopsys}[2][]{#2^{\hspace*{.2pt}\textup{min}}_{#1}}
\newcommand{\maxopsys}[2][]{#2^{\hspace*{.2pt}\textup{max}}_{#1}}

\newcommand{\convtensor}[2]{\minwedge{#1}{#2}}
\newcommand{\homogen}[1]{\mathscr C(#1)}

\newcommand{\symm}[1]{\mathcal{S}^{{#1}}}
\newcommand{\herm}[1]{\mathcal{H}^{{#1}}}
\newcommand{\Lorentz}[1]{\mathcal{L}^{{#1}}}

\newcommand{\orface}[2]{{#1} \mathbin{\ovee^\pi} {#2}}
\newcommand{\andface}[2]{{#1} \mathbin{\owedge^\pi} {#2}}
\newcommand{\SCorface}[2]{{#1} \mathbin{\ovee^\varepsilon} {#2}}
\newcommand{\SCandface}[2]{{#1} \mathbin{\owedge^\varepsilon} {#2}}
\newcommand{\SCorideal}[2]{{#1} \mathbin{\ovee} {#2}}
\newcommand{\SCandideal}[2]{{#1} \mathbin{\owedge} {#2}}

\newcommand{\MsetNface}[2]{{#1} \mathbin{\ltimes} {#2}}
\newcommand{\MfaceNset}[2]{{#1} \mathbin{\rtimes} {#2}}

\newcommand{\dualfacelattice}[2]{\mathscr F_{\langle #1 , #2 \rangle}}

\newcommand{\lintop}{\mathfrak{T}}
\newcommand{\mywedge}{\mathcal{K}}

\DeclareMathOperator{\lineal}{lin}
\DeclareMathOperator{\rext}{rext}
\DeclareMathOperator{\ext}{ext}
\DeclareMathOperator{\LL}{L} 
\DeclareMathOperator{\CL}{\mathfrak{L}}
\DeclareMathOperator{\CK}{\mathfrak{K}}
\DeclareMathOperator{\CN}{\mathfrak{N}}
\DeclareMathOperator{\Bil}{Bil}
\DeclareMathOperator{\CBil}{\mathscr{B}\hspace*{-1.2pt}\textit{i}\ell} 
\DeclareMathOperator{\SCBil}{\mathfrak{Bil}}

\DeclareMathOperator{\myGL}{GL}
\newcommand{\GL}[1]{\myGL_{#1}(\R)}
\DeclareMathOperator{\pred}{prop}

\newcommand{\thetitlepage}{%
	\clearpage
	\pagestyle{plain}
	\begin{center}
		\Large
		\textsc{\LARGE\thetitle}
		
		\bigskip
		\medskip
		\theauthor
		
		\bigskip
		\Large
		6 December 2022
	\end{center}
	
	\bigskip
	\bigskip
	\begin{abstract}
		\theabstract
	\end{abstract}
	
	\bigskip
	\bigskip
	\emph{Financial support.} \thethanks
	
	\bigskip
	\emph{2020 Mathematics Subject Classification:} \themathclass
	
	\emph{Key words and phrases:} \thekeywords
	
	\pagebreak
	\noindent
	\theauthor\\
	\theaddress
	\cleardoublepage
	\pagestyle{mypagestyle}
}

\addto\captionsenglish{
}

\makeindex

\makenoidxglossaries
\def\myglossaryname{Glossary of notation}

\newlength{\mytmpdimenA}
\newlength{\mytmpdimenB}
\newlength{\myglsWidestName}
\newlength{\myglsWidestPagenum}
\newlength{\myglsdescwidth}

\makeatletter
\newcommand{\myglsFindWidestNameAndPagenum}[1][\@glo@types]{%
	\mytmpdimenA=0pt\relax%
	\mytmpdimenB=0pt\relax%
	\myglsWidestName=0pt\relax%
	\myglsWidestPagenum=0pt\relax%
	\forallglossaries[#1]{\@gls@type}{%
		\forglsentries[\@gls@type]{\@glo@label}{%
			\ifglshasparent{\@glo@label}{}{%
				\settowidth{\mytmpdimenA}{\glossentryname{\@glo@label}}%
				\ifdim\mytmpdimenA>\myglsWidestName%
					\myglsWidestName=\mytmpdimenA%
				\fi%
				\letcs{\@gls@loclist}{glo@\glsdetoklabel{\@glo@label}@loclist}%
				\ifdef\@gls@loclist{%
					\settowidth{\mytmpdimenB}{\glossaryentrynumbers{\glsnoidxloclist{\@gls@loclist}}}%
					\ifdim\mytmpdimenB>\myglsWidestPagenum%
						\myglsWidestPagenum=\mytmpdimenB%
					\fi%
				}{%
					\@latex@warning{\noexpand\myglsFindWidestNameAndPagenum: missing location list for entry `\@glo@label'}
				}%
			}%
		}%
	}%
	\setlength{\myglsdescwidth}{\dimexpr\textwidth - \myglsWidestName - \myglsWidestPagenum - 6\tabcolsep} 
}
\makeatother

\newglossaryentry{CBil}{
	name={\ensuremath{\CBil(E \times F)}},
	description={the space of continuous bilinear forms $E \times F \to \R$},
	sort={Bil10EF}
}
\newglossaryentry{SCBil}{
	name={\ensuremath{\SCBil(E \times F)}},
	description={the space of separately continuous bilinear forms $E \times F \to \R$},
	sort={Bil20EF}
}
\newglossaryentry{Bil}{
	name={\ensuremath{\Bil(E \times F)}},
	description={the space of all bilinear forms $E \times F \to \R$},
	sort={Bil00EF}
}
\newglossaryentry{bMcdot}{
	name={\ensuremath{b(M , \,\cdot\,)}},
	description={the set $\{b(x, \,\cdot\,) \, : \, x \in M\}$},
	sort={bMcdot}
}
\newglossaryentry{bcdotN}{
	name={\ensuremath{b(\,\cdot\, , N)}},
	description={the set $\{b(\,\cdot\, , y) \, : \, y \in N\}$},
	sort={bNcdot}
}
\newglossaryentry{bxcdot}{
	name={\ensuremath{b(x_0 , \,\cdot\,)}},
	description={the linear functional $y \mapsto b(x_0,y)$},
	sort={bxcdot}
}
\newglossaryentry{bcdoty}{
	name={\ensuremath{b(\,\cdot\, , y_0)}},
	description={the linear functional $x \mapsto b(x,y_0)$},
	sort={bycdot}
}
\newglossaryentry{homogenization}{
	name={\ensuremath{\homogen{C}}},
	description={the homogenization of $C$},
	sort={CC}
}
\newglossaryentry{polar}{
	name={\ensuremath{C\polar}},
	description={the one-sided polar of $C$},
	sort={Cpolar}
}
\newglossaryentry{algebraic dual space}{
	name={\ensuremath{E\algdual}},
	description={the algebraic dual space of $E$},
	sort={E10}
}
\newglossaryentry{topological dual space}{
	name={\ensuremath{E\topdual}},
	description={the topological dual space of $E$},
	sort={E20}
}
\newglossaryentry{dual pair}{
	name={\ensuremath{\langle E,E'\rangle}},
	description={dual pair},
	sort={E30}
}
\newglossaryentry{tensor product}{
	name={\ensuremath{E \tensor F}},
	description={the tensor product of $E$ and $F$},
	sort={E40}
}
\newglossaryentry{completed tensor product}{
	name={\ensuremath{E \hattensor_\alpha F}},
	description={the completion of $E \tensor F$ with respect to a compatible locally convex topology $\alpha$ on $E \tensor F$},
	sort={E50}
}
\newglossaryentry{SCBiltensor}{
	name={\ensuremath{E \SCBiltensor F}},
	description={the space of separately weak-$*$ continuous bilinear forms ${E' \times F' \to \R}$},
	sort={E60}
}
\newglossaryentry{primal cone}{
	name={\ensuremath{E_+}},
	description={the positive cone of a preordered vector space $E$},
	sort={Eplus00}
}
\newglossaryentry{bipolar cone}{
	name={\ensuremath{\overline{E_+}^{\,\weak}}},
	description={the weak closure of $E_+$ ($=$ the bipolar cone with respect to $\langle E,E'\rangle$)},
	sort={Eplus05}
}
\newglossaryentry{algebraic dual cone}{
	name={\ensuremath{E_+\algdual}},
	description={the algebraic dual cone of $E_+$},
	sort={Eplus10}
}
\newglossaryentry{topological dual cone}{
	name={\ensuremath{E_+\topdual}},
	description={the topological dual cone of $E_+$},
	sort={Eplus20}
}
\newglossaryentry{maxwedge}{
	name={\ensuremath{\maxwedge{E_+}{F_+}}},
	description={the injective cone in the algebraic tensor product $E \tensor F$},
	sort={Eplustensor10Fplus}
}
\newglossaryentry{hatmaxwedge}{
	name={\ensuremath{\hatmaxwedge[\alpha]{E_+}{F_+}}},
	description={the injective cone in the completed locally convex tensor product $E \hattensor_\alpha F$},
	sort={Eplustensor11Fplus}
}
\newglossaryentry{minwedge}{
	name={\ensuremath{\minwedge{E_+}{F_+}}},
	description={the projective cone in the algebraic tensor product $E \tensor F$},
	sort={Eplustensor20Fplus}
}
\newglossaryentry{hatminwedge}{
	name={\ensuremath{\hatminwedge[\alpha]{E_+}{F_+}}},
	description={the projective cone in the completed locally convex tensor product $E \hattensor_\alpha F$},
	sort={Eplustensor21Fplus}
}
\newglossaryentry{E_weak}{
	name={\ensuremath{E_\weak}},
	description={the space $E$ equipped with the $\sigma(E,E')$-topology},
	sort={Eweak00}
}
\newglossaryentry{E_weakstar}{
	name={\ensuremath{E_\weakstar'}},
	description={the space $E'$ equipped with the $\sigma(E',E)$-topology},
	sort={Eweak10star}
}
\newglossaryentry{herm}{
	name={\ensuremath{\herm{n}}},
	description={the real vector space of complex $n \times n$ hermitian matrices},
	sort={Hn10}
}
\newglossaryentry{herm+}{
	name={\ensuremath{\herm{n}_+}},
	description={the cone of complex $n \times n$ positive semidefinite matrices},
	sort={Hn20}
}
\newglossaryentry{SCorideal}{
	name={\ensuremath{\SCorideal{I}{J}}},
	description={the upper ideal of $E \SCBiltensor F$ defined by the ideals $I$ and $J$},
	sort={I50tensorJ10}
}
\newglossaryentry{SCandideal}{
	name={\ensuremath{\SCandideal{I}{J}}},
	description={the lower ideal of $E \SCBiltensor F$ defined by the ideals $I$ and $J$},
	sort={I50tensorJ20}
}
\newglossaryentry{L}{
	name={\ensuremath{\LL(E,F)}},
	description={the space of linear maps $E \to F$},
	sort={LEF10}
}
\newglossaryentry{CL}{
	name={\ensuremath{\CL(E,F)}},
	description={the space of continuous linear maps $E \to F$},
	sort={LEF20}
}
\newglossaryentry{lineality space}{
	name={\ensuremath{\lineal(E_+)}},
	description={the lineality space $\lineal(E_+) = E_+ \cap -E_+$ of the convex cone $E_+$},
	sort={linealEplus}
}
\newglossaryentry{Lorentz}{
	name={\ensuremath{\Lorentz{n}}},
	description={the second-order ($=$ Lorentz) cone $\Lorentz{n} \subseteq \R^n$},
	sort={Ln}
}
\newglossaryentry{weak closure}{
	name={\ensuremath{\overline{M}^{\,\weak}}},
	description={the weak closure of the set $M$},
	sort={M10w}
}
\newglossaryentry{Mperp}{
	name={\ensuremath{M^\perp}},
	description={the set $\{\varphi \in E' \, : \, \langle x , \varphi\rangle = 0\ \text{for all $x \in M$}\}$},
	sort={M14perp}
}
\newglossaryentry{dualface}{
	name={\ensuremath{M^\dualface}},
	description={the dual face of $M$},
	sort={M16dualface}
}
\newglossaryentry{settensor}{
	name={\ensuremath{M \settensor N}},
	description={the entry-wise tensor product of the sets $M$ and $N$},
	sort={M20tensorsN}
}
\newglossaryentry{SCorface}{
	name={\ensuremath{\SCorface{M}{N}}},
	description={the upper face of $(E \SCBiltensor F)_+$ defined by the faces $M$ and $N$},
	sort={M30tensorN10}
}
\newglossaryentry{orface}{
	name={\ensuremath{\orface{M}{N}}},
	description={the upper face of $\minwedge{E_+}{F_+}$ defined by the faces $M$ and $N$},
	sort={M30tensorN11}
}
\newglossaryentry{SCandface}{
	name={\ensuremath{\SCandface{M}{N}}},
	description={the lower face of $(E \SCBiltensor F)_+$ defined by the faces $M$ and $N$},
	sort={M30tensorN20}
}
\newglossaryentry{andface}{
	name={\ensuremath{\andface{M}{N}}},
	description={the lower face of $\minwedge{E_+}{F_+}$ defined by the faces $M$ and $N$},
	sort={M30tensorN21}
}
\newglossaryentry{MsetNface}{
	name={\ensuremath{\MsetNface{M'}{N}}},
	description={the set $\{ b \in E \SCBiltensor F \, : \, b(M' , \,\cdot\,) \subseteq N\}$},
	sort={M40facesetN10}
}
\newglossaryentry{MfaceNset}{
	name={\ensuremath{\MfaceNset{M}{N'}}},
	description={the set $\{ b \in E \SCBiltensor F \, : \, b(\,\cdot\, , N') \subseteq M\}$},
	sort={M40facesetN20}
}
\newglossaryentry{weak-* closure}{
	name={\ensuremath{\overline{N}^{\,\weakstar}}},
	description={the weak\nobreakdash-$*$ closure of the set $N$},
	sort={N10wstar}
}
\newglossaryentry{perpN}{
	name={\ensuremath{{}^\perp N}},
	description={the set $\{x \in E \, : \, \langle x , \varphi\rangle = 0\ \text{for all $\varphi \in N$}\}$},
	sort={N14perp}
}
\newglossaryentry{predualface}{
	name={\ensuremath{{}^\predualface N}},
	description={the predual face of $N$},
	sort={N16perp}
}
\newglossaryentry{rext}{
	name={\ensuremath{\rext(E_+)}},
	description={the set of extremal directions of $E_+$},
	sort={rextEplus}
}
\newglossaryentry{weak topology}{
	name={\ensuremath{\sigma(E,E')}},
	description={the weak topology on $E$},
	sort={SigmaE1E2}
}
\newglossaryentry{weak-* topology}{
	name={\ensuremath{\sigma(E',E)}},
	description={the weak\nobreakdash-$*$ topology on $E'$},
	sort={SigmaE2E1}
}
\newglossaryentry{symm}{
	name={\ensuremath{\symm{n}}},
	description={the space of real $n \times n$ symmetric matrices},
	sort={Sn10}
}
\newglossaryentry{symm+}{
	name={\ensuremath{\symm{n}_+}},
	description={the cone of real $n \times n$ positive semidefinite matrices},
	sort={Sn20}
}
\newglossaryentry{algebraic adjoint}{
	name={\ensuremath{T\algdual}},
	description={the algebraic adjoint $F\algdual \to E\algdual$ of a linear map $T : E \to F$},
	sort={T10}
}
\newglossaryentry{topological adjoint}{
	name={\ensuremath{T\topdual}},
	description={the topological adjoint $F\topdual \to E\topdual$ of a continuous linear map $T : E \to F$},
	sort={T20}
}
\newglossaryentry{T tensor S}{
	name={\ensuremath{T \tensor S}},
	description={the linear map $E \tensor F \to G \tensor H$ determined by the linear maps $T : E \to G$ and $S : F \to H$},
	sort={TtensorS10}
}
\newglossaryentry{T SCBiltensor S}{
	name={\ensuremath{T \SCBiltensor S}},
	description={the linear map $E \SCBiltensor F \to G \SCBiltensor H$ determined by the linear maps $T \in \CL(E_\weak , G_\weak)$ and $S \in \CL(F_\weak , H_\weak)$},
	sort={TtensorS20}
}
\newglossaryentry{T Biltensor S}{
	name={\ensuremath{T \Biltensor S}},
	description={the linear map $\Bil(E' \times F') \to \Bil(G' \times H')$ determined by the linear maps $T \in \CL(E_\weak , G_\weak)$ and $S \in \CL(F_\weak , H_\weak)$},
	sort={TtensorS30}
}

\begin{document}

\frontmatter

\thetitlepage

\microtypesetup{protrusion=false} 
\tableofcontents
\microtypesetup{protrusion=true}

\mainmatter

\makeoddhead{mypagestyle}{\rightmark}{}{\thepage}

\chapter{Introduction}
\label{chp:introduction}

\section{Background}
\label{sec:intro:background}

Convex cones have applications in almost all branches of mathematics, from algebra and geometry to analysis and optimization.
Consequently, convex cones have been studied extensively in their own right, and there is a vast body of work on all kinds of geometrical, analytical, and combinatorial properties of convex cones.

In the study of convex cones, just as in any other area of mathematics, it is important to have good ways of creating new objects from old.
One such problem which has attracted a lot of attention is the following.
Suppose that we are given convex cones $E_+$ and $F_+$ in the vector spaces $E$ and $F$, respectively.
Can we then use the data of $E_+$ and $F_+$ to somehow construct a natural cone in the tensor product $E \tensor F$?
As it turns out, there are multiple ways to do so \cite{Merklen,Peressini-Sherbert}, just as there are multiple ways to define a norm on the tensor product of two normed spaces \cite{Ryan}.

Among all ``reasonable'' cones in the tensor product $E \tensor F$, there is a smallest and a largest one, which we denote by $\minwedge{E_+}{F_+}$ and $\maxwedge{E_+}{F_+}$, respectively.
These have come up many times in the literature, motivated by problems in a variety of different fields.
We outline a few of these applications:
\begin{itemize}
	\item In functional analysis, one is often interested in tensor products of various types of spaces (e.g.{} Banach spaces, $C^*$-algebras, operator spaces, etc.).
	Often the two factors come with natural order structure, in which case it is desirable to find a compatible order structure in the tensor product.
	This is equivalent to finding a tensor product of the positive cones, and so tensor products of convex cones are closely linked to tensor products of ordered (topological) vector spaces.
	
	\item In operator theory, the minimal and maximal tensor product of a positive semidefinite cone with an arbitrary cone $C$ correspond to the smallest and largest operator system with $C$ at its ground level.
	Question surrounding this minimal and maximal operator system have been studied by several authors; for instance, \cite{Paulsen-Todorov-Tomforde,Fritz-Netzer-Thom,Huber-Netzer}.
	Furthermore, these questions turn out to be closely related to questions about matrix convex sets \cite[\S 7]{Passer-Shalit-Solel}, \cite[Thm.{} 9.11]{Shalit-dilation} and free spectrahedra \cite{Fritz-Netzer-Thom}, topics which have been studied by authors in geometry, optimization, and quantum information theory.
	
	\item In theoretical physics, the theory of ``general probabilistic theories'' (GPTs) forms a new framework which generalizes both classical and quantum probability \cite{Lami-dissertation,Muller-lecture-notes,Plavala-lecture-notes,Aubrun-et-al-ii}.
	A GPT derives probability from an arbitrary finite-dimensional Archimedean cone with an order unit.
	Classical (resp.{} quantum) probability can then be recovered as a special case by taking a simplex (resp.{} positive semidefinite) cone.
	
	Given GPTs $E$ and $F$, the tensor product $E \tensor F$ corresponds to the composite system $(E,F)$.
	In this setting, the elements of the smallest cone $\minwedge{E_+}{F_+}$ correspond to the separable states, whereas the elements of $\maxwedge{E_+}{F_+} \setminus \minwedge{E_+}{F_+}$ correspond to the entangled states \cite[Def.{} 5.8]{Plavala-lecture-notes}.
	Thus, understanding tensor products of convex cones is crucial to understanding entanglement in GPTs.
	
	\item In polyhedral geometry, the minimal and maximal tensor product of two polyhedral cones are closely related to the tensor product and Hom-polytope of the underlying polytopes.
	Since Hom and tensor are fundamental constructions in the category of polytopes, their properties have been studied in detail in the literature; see for instance \cite{Bogart-Contois-Gubeladze}.
	
	\item In approximation theory, tensor products of convex cones come up naturally in the context of multivariate shape preserving interpolation with cone constraints.
	For a precise description of this problem and its relation to tensor products of convex cones, see \cite{Mulansky}.
\end{itemize}

In each of these settings, the underlying construction is just a tensor product of convex cones.
This underpins the importance of a systematic study of tensor products of convex cones, and indeed many papers have already been written about this.
However, most of the existing literature only focuses on one of two particular cases: lattice cones and finite-dimensional cones.
As a result, the literature is divided into two separate lines of investigation, neither of which addresses the problem in full generality.

The first line of investigation comes from functional analysis.
In this setting, the focus has mostly been on Riesz spaces and Banach lattices.
Although most classical Banach spaces are lattice-ordered, many other interesting classes of ordered vector spaces are not.
For example, the self-adjoint part of a $C^*$-algebra $\mathcal A$ is an ordered vector space with a closed, proper and generating cone $\mathcal A_+$, but by Sherman's theorem it is lattice-ordered if and only if $\mathcal A$ is commutative.
This shows that, in a way, restricting one's attention to lattice-ordered spaces is akin to restricting one's attention to commutative $C^*$-algebras.

The second line of investigation comes from linear algebra, and encompasses the remaining applications from the preceding list.
In this setting, research has dealt exclusively with closed, proper and generating cones in finite-dimensional spaces.
This is once again a severe limitation, at least from the perspective of analysis, as finite-dimensional spaces are often of limited use there.
Furthermore, even in the finite-dimensional case one occasionally encounters cones which are not closed or not proper.
For example, lexicographical cones in a space of dimension at least $2$ are never closed, and quotients/projections of closed, proper cones are not guaranteed to be closed either (see e.g.{} \cite[Example 6.3]{Dobben-semisimplicity}).

There has been very little cross-pollination between these two lines of investigation, and very little has been done beyond these two specific cases.
In particular, almost nothing is known about tensor products of infinite-dimensional ordered vector spaces which are not lattice-ordered, or about tensor products of cones which are not closed and/or not proper.
This disqualifies many cones from consideration, including even standard cones such as the positive semidefinite cone over an infinite-dimensional Hilbert space.

Furthermore, even in the cases that have been studied, many basic properties have not been noted or proved in the existing literature.
For instance, whereas mapping properties play an important role in the similar theory of \emph{normed} tensor products, we are not aware of prior papers which establish the mapping properties of the minimal/maximal tensor product of convex cones.
Likewise, only partial results are known about properness of the minimal/maximal tensor product of convex cones, or whether the minimal/maximal tensor product preserves faces of the base cones.

This memoir aims to develop a general theory of tensor products of convex cones, without any restrictions on the cones or the ambient spaces.
By using ideas from both lines of investigation and borrowing additional techniques from the similar theory of \emph{normed} tensor products, we are able to extend known results to the general setting and prove many completely new results.

In the next section, we give a very brief overview of the existing literature.
After that, the remainder of this chapter gives a comprehensive overview of the main results of this memoir.

\section{Brief literature overview}

The study of tensor products of ordered topological vector spaces was initiated in the 1960s by Merklen \cite{Merklen},%
	\hair\footnote{It appears that Merklen was the first to study tensor products of ordered vector spaces, but his article is very hard to find, and contains several errors. For instance, \cite[Teorema~5]{Merklen} states that the weak closure of the projective cone $\minwedge{E_+}{F_+}$ is a proper cone if at least one of $E_+$ and $F_+$ is proper, provided that $E_+$ and $F_+$ are weakly closed. Likewise, \cite[Teorema~9]{Merklen} states the same for the injective cone. Both of these statements are incorrect, as can be seen by taking $E_+ = \R_{\geq 0}$ and $F_+ = \R$. The correct statement is that both $E_+$ and $F_+$ should be proper; see \autoref{thm:intro:semisimplicity} and \autoref{thm:intro:proper-cone} below.}
Hulanicki and Phelps \cite{Hulanicki-Phelps}, Popa \cite{Popa-I,Popa-II}, and Peressini and Sherbert \cite{Peressini-Sherbert}.
From the 1970s onwards, the focus has mostly been on Riesz spaces (\cite{Schaefer-tensor,Fremlin,Fremlin-Banach,Wittstock,Schaefer-pos-ops-book,Birnbaum,Fremlin-decomposition,Nielsen-operators,Grobler-Labuschagne,Nielsen-approximation,Grobler-Labuschagne-f-algebra,Blanco,Azouzi-BenAmor-Jaber,Buskes-Thorn,BenAmor-Gok-Yaman}) and, in a separate line of investigation, on closed cones in finite\-/dimensional spaces (\cite{Barker-Loewy,Barker-monotone,Haynsworth-Fiedler-Ptak,Barker-duality,Barker-perfect,Barker-survey,projectionally-exposed-faces,projectionally-exposed-faces-ii,Tam-faces,Tam-survey,Mulansky,Hildebrand-descriptions,Huber-Netzer,Aubrun-et-al-ii}).
For general ordered vector spaces, some of the basic questions remain unanswered (and, on one occasion, escaped from collective memory, as we point out below).

The most comprehensive paper on tensor products of general ordered vector spaces is the article of Peressini and Sherbert \cite{Peressini-Sherbert}.
It contains an in-depth study of the properties of the \emph{projective} (minimal) and \emph{injective} (maximal) cone in the tensor product.%
	\hair\footnote{A note on terminology: several authors refer to the maximal cone as the \emph{biprojective cone}. We aim to show that it is in many ways analogous to the injective norm, and as such deserves the name \emph{injective cone}. This term has also occasionally been used before, for instance by Wittstock \cite{Wittstock} and Mulansky \cite{Mulansky}.}
It answers various topological and order-theoretic questions about these cones, for instance relating to normality and order units.
Furthermore, it establishes a few sufficient conditions for the projective/injective cone to be proper, but it does not provide precise necessary and sufficient conditions.

Conditions for the projective cone to be proper were quickly provided by Dermenjian and Saint-Raymond \cite{Dermenjian-Saint-Raymond}, but their result seems to have been unknown to later generations of mathematicians.
Only recently was this question answered (again) by Wortel \cite{Wortel}; until then only special cases were assumed to be known in the literature.
For the injective cone, no precise necessary and sufficient conditions for properness are known in the literature.

The situation is much better in the setting of lattice-ordered or finite-dimensional spaces.
For lattice-ordered spaces, a lot has been said about the problem of turning the tensor product (or its completion) into a lattice-ordered space as well \cite{Schaefer-tensor,Fremlin,Fremlin-Banach,Wittstock}, and connections between such lattice tensor products and lattices of operators are well-known \cite[\S IV.7]{Schaefer-pos-ops-book}.
However, such results are rather specific to lattice-ordered spaces, and have little hope of being generalized to general (non-lattice-ordered) spaces.

Likewise, in the finite-dimensional setting, much more is known.
Here research has focused on cones that are closed, proper and generating.
This is sufficient to guarantee that the projective and injective tensor product are closed, proper and generating as well \cite{Tam-projective-closed}, so finding criteria for properness is not an issue here.
More advanced results have been obtained as well; see for instance \cite{Barker-Loewy,Barker-monotone,Barker-survey,Tam-faces}.
In particular, in the context of cones of positive operators, Tam gave a construction which can be used to obtain faces in the injective cone from faces of the base cones \cite[\S 4]{Tam-faces}.
We will extend this result; see \mysecref{sec:intro:faces}.

There is very little overlap between the lattice-ordered and the finite-dimensional theory, because they deal with very different questions.
After all, the only finite-dimensional closed lattice cones are the simplex cones (i.e.{} the ones isomorphic to $\R_{\geq 0}^n$), which are not very interesting from either perspective.
However, one problem that has been studied in both settings is the question whether or not the projective cone is dense in the injective cone.
Birnbaum \cite{Birnbaum} showed that this is true whenever $E$ and $F$ are locally convex lattices and gave an example which shows that it is not true in general.
However, in the infinite-dimensional setting, this problem does not appear to be well understood.

More can be said in the finite-dimensional case, where results in this direction have been proved by various authors over the years \cite{Barker-Loewy,Poole-dissertation,Tam-projective-closed,Fritz-Netzer-Thom,Huber-Netzer,Aubrun-et-al-i,Aubrun-et-al-ii}.
We will extend previous results by showing that the projective cone is closed and properly contained in the injective cone for nearly all closed, proper, and generating base cones (in finite-dimensional spaces), including most standard cones.
However, as this manuscript was being written, our results were superseded by simultaneous work of Aubrun, Lami, Palazuelos and Pl\'avala \cite{Aubrun-et-al-ii}, who independently managed to solve this problem completely in the finite-dimensional case.
We recover their result for a large class of cones, using completely different techniques.
Our contribution will be discussed in more detail in \mysecref{sec:intro:many-examples}.

\section{Scope and notation}

We now outline the scope of this memoir, and we cover the basic notation needed to state our main results in the upcoming sections.

In this memoir, we study the projective and injective tensor product of two convex cones $E_+ \subseteq E$ and $F_+ \subseteq F$, where $E$ and $F$ are either real vector spaces or real \emph{topological} vector spaces.

Topological considerations will not matter too much for our investigation, but to build a satisfactory duality theory we need to at least keep track of the duals of all spaces involved.
Hence, instead of remembering the topology of a vector space $E$ (or the fact that $E$ has no topology), we only remember the dual pair $\langle E, E'\rangle$ to which $E$ belongs.
The advantage of this approach is twofold: it allows us to treat the topological and non-topological cases simultaneously (if $E$ has no topology, let $E' := E\algdual$ be the algebraic dual), and it allows us to completely ignore any topological issues in the tensor product, thereby sidestepping the notoriously difficult theory of topological tensor products.
One downside of this approach is the following: since we have no topology on $E$, we must occasionally refer to the weak closure $\overline{E_+}^{\,\weak}$ of $E_+$, instead of the ordinary closure.
However, we remind the reader that in every locally convex space, the weak closure of a convex set coincides with its original closure.

We now recall some basic notation.
A \index{convex cone}\emph{convex cone} (otherwise known as a \emph{wedge}) in a real vector space $E$ is a non-empty subset $\mywedge \subseteq E$ satisfying $\mywedge + \mywedge \subseteq \mywedge$ and $\lambda \mywedge \subseteq \mywedge$ for all $\lambda \in \R_{\geq 0}$.
The \index{lineality space}\emph{lineality space} of a convex cone $\mywedge$ is the linear subspace \glsadd{lineality space}$\lineal(\mywedge) := \mywedge \cap -\mywedge$.
We say that a convex cone $\mywedge$ is \index{convex cone!proper}\emph{proper} if $\lineal(\mywedge) = \{0\}$ and \index{convex cone!semisimple}\emph{semisimple} if its weak closure $\overline{\mywedge}^{\,\weak}$ is proper.

Let $E$ and $F$ be vector spaces, and let $E_+ \subseteq E$, $F_+ \subseteq F$ be convex cones.
The \index{projective cone}\emph{projective cone} in $E \tensor F$ is given by
\[ \glsadd{minwedge}\minwedge{E_+}{F_+} := \left\{\sum_{i=1}^k x_i \tensor y_i \ : \ k\in\N,\ x_1,\ldots,x_k\in E_+,\ y_1,\ldots,y_k\in F_+\right\}. \]

Furthermore, if $E$ and $F$ belong to the dual pairs $\langle E, E' \rangle$ and $\langle F, F' \rangle$, then the \index{injective cone}\emph{injective cone} in $E \tensor F$ is given by
\[ \glsadd{maxwedge}\maxwedge{E_+}{F_+} := \big\{u \in E \tensor F \, : \, \langle u , \varphi \tensor \psi \rangle \geq 0\ \text{for all $\varphi \in E_+'$, $\psi \in F_+'$}\big\}. \]
For additional notation, see \mychpref{chp:preliminaries}, or refer to the glossary of notation on \autopageref{symbols}.

\subsection{A note about cones in the completed tensor product}
So far, the study of tensor products of convex cones has mostly been limited to cones in the algebraic tensor product, with the exception of some results on tensor products of Banach lattices.
However, the algebraic tensor product is often of limited use in analysis; instead, one is usually interested in its completion with respect to some suitable topology.
For this reason, we also aim to initiate a study of the projective and injective cones in completed locally convex tensor products.

When dealing with \emph{topological} tensor products, one has to define the topology \emph{before} taking the completion, for obviously the completion depends on the chosen topology.
On the other hand, the cone is unrelated to the topology, and can therefore be defined directly on the completion.
This gives rise to a natural extension of the injective cone to the completed tensor product, which we will also study in this memoir.
On the other hand, the projective cone in the completed tensor product $E \hattensor_\alpha F$ is merely the same cone embedded in a larger ambient space%
	\hair\footnote{We define the projective cone algebraically, without taking its closure. This is the prevalent definition in the literature, but might not be appropriate for all applications. We do prove a few results about its closure; see \autoref{cor:intro:reasonable-extremal-rays} and \autoref{thm:intro:semisimplicity}.},
so there is little reason to study this cone separately.

An overview of the cones under consideration, their notation, and their domains of definition, is given in \autoref{tbl:tensor-cones-context}.
(In all cases, $E_+ \subseteq E$ and $F_+ \subseteq F$ are convex cones in the primal spaces.)

\begin{table}[h!t]
	\centering
	\renewcommand{\arraystretch}{1.25}
	\caption{The domain of definition of the projective/injective cones studied in this memoir.}
	\begin{tabular}{lllp{55mm}}
		\toprule
		Cone & Ambient space & Notation & Domain of definition \\
		\midrule
		Projective & $E \tensor F$            & \glsadd{minwedge}$\minwedge{E_+}{F_+}$                 & $E$ and $F$ vector spaces \\
		Injective  & $E \tensor F$            & \glsadd{maxwedge}$\maxwedge{E_+}{F_+}$                 & $\langle E, E'\rangle$ and $\langle F, F'\rangle$ dual pairs\\
		Injective  & $E \hattensor_\alpha F$  & \glsadd{hatmaxwedge}$\hatmaxwedge[\alpha]{E_+}{F_+}$   & $E$ and $F$ complete lcs;\par $\alpha$ a compatible lc topology on $E \tensor F$ \\
		\bottomrule
	\end{tabular}
	\label{tbl:tensor-cones-context}
\end{table}

In the remainder of this chapter, we state our main results only for cones in the algebraic tensor product.
Similar results hold for the injective cone in completed locally convex tensor products, but these are harder to state, as they often require additional (topological) assumptions.
Precise statements can be found in \mychpref{chp:injective} on the injective cone.

\section{Mapping properties}
\label{sec:intro:mapping-properties}
In the theory of normed tensor products, it is well-known that the projective norm preserves metric surjections (quotients) and the injective norm preserves metric injections (isometries), and these simple mapping properties play an important role in the theory.
By looking at the corresponding types of positive linear maps, we show that the projective and injective cones have analogous mapping properties.

Let $E$ and $F$ be vector spaces, and let $E_+ \subseteq E$, $F_+ \subseteq F$ be convex cones.
We say that a linear map $T \in \LL(E,F)$ is \index{positive linear map}\emph{positive} if $T[E_+] \subseteq F_+$, a \index{pullback}\emph{pullback} (or \emph{bipositive operator}) if $E_+ = T^{-1}[F_+]$, and a \index{pushforward}\emph{pushforward} if $T[E_+] = F_+$.
Furthermore, if $E$ and $F$ belong to dual pairs $\langle E,E'\rangle$ and $\langle F,F'\rangle$, then we say that an operator $T \in \CL(E_\weak, F_\weak)$ is an \index{approximate pullback}\emph{approximate pullback} (or \emph{approximately bipositive}) if $\overline{E_+}^{\,\weak} = T^{-1}[\overline{F_+}^{\,\weak}]$, and an \index{approximate pushforward}\emph{approximate pushforward} if $\overline{T[E_+]}^{\,\weak} = \overline{F_+}^{\,\weak}$.
(Recall that in a locally convex space, the weak closure of a convex set coincides with its original closure.)
Every pushforward is also an approximate pushforward, but a pullback is not necessarily an approximate pullback (see \mysecref{sec:convex-cones}).

A typical example of a pullback is an order embedding: if $(E,E_+)$ is order isomorphic to a subspace of $(F,F_+)$, then the embedding $E \hookrightarrow F$ is a pullback.
A typical example of a pushforward is a quotient or projection.

In the normed theory, the projective norm preserves metric surjections (quotients) and the injective norm preserves metric injections (isometries).
We prove a similar result for cones:

\begin{AlphTheorem}
	\label{thm:intro:mapping-properties}
	The projective cone preserves positive linear maps, \textup(approximate\textup) pushforwards, and order retracts, but not \textup(approximate\textup) pullbacks.
	
	The injective cone preserves weakly continuous positive linear maps, approximate pullbacks, and topological order retracts, but not pullbacks or \textup(approximate\textup) pushforwards.
\end{AlphTheorem}

In particular, the injective cone preserves order embeddings when the cones are weakly closed.
We believe this result to be new, even in the finite-dimensional setting.

The proof of \autoref{thm:intro:mapping-properties} will be given in \mysecref{sec:projective-positive-linear-maps} (projective cone) and \mysecref{sec:injective-positive-linear-maps} (injective cone).
An overview of these mapping properties is given in \autoref{tbl:mapping-properties}.

\begin{table}[!htbp]
	\centering
	\renewcommand{\arraystretch}{1.25}
	\caption{Types of maps preserved by the projective/injective cone.}
	\begin{tabular}{lcc}
		\toprule
		Type of map                   & \multicolumn{2}{c}{Preserved by} \\
		\cmidrule{2-3}
		                              & Projective cone  & Injective cone \\
		\midrule
		Positive map                  & \mycheckmark     & \mycheckmark \\
		Pushforward                   & \mycheckmark     & \\
		Approximate pushforward       & \mycheckmark     & \\
		Pullback                      &                  & \\
		Approximate pullback          &                  & \mycheckmark \\
		Retract (positive projection) & \mycheckmark     & \mycheckmark \\
		\bottomrule
	\end{tabular}
	\label{tbl:mapping-properties}
\end{table}

Note that the injective cone only preserves \emph{approximate} pullbacks.
It is not so strange that it does not preserve all pullbacks: the injective cone does not see the difference between $E_+$ and $\overline{E_+}^{\,\weak}$, and a pullback for $E_+$ is not necessarily a pullback for $\overline{E_+}^{\,\weak}$ (for details, see \mysecref{sec:injective-positive-linear-maps}).
In general, the properties of the injective cone depend on those of $\overline{E_+}^{\,\weak}$ and $\overline{F_+}^{\,\weak}$ rather than $E_+$ and $F_+$.
By contrast, the \emph{projective} cone does see the difference between $E_+$ and $\overline{E_+}^{\,\weak}$, so it preserves both pushforwards and approximate pushforwards.

\section{Criteria for properness, the lineality space, and semisimplicity}
\label{sec:intro:properness-lineality-space}
Another basic question about tensor products of convex cones is to determine when the projective or injective cone is proper.
Peressini and Sherbert \cite{Peressini-Sherbert} found a few sufficient conditions, but their paper does not specify precise necessary and sufficient criteria.
For the projective cone, precise conditions were found by Dermenjian and Saint-Raymond \cite{Dermenjian-Saint-Raymond}, and rediscovered in recent years by Wortel \cite{Wortel}.%
	\hair\footnote{The result of Dermenjian and Saint-Raymond seems to have been unknown to later generations of mathematicians, and until recently only special cases were assumed to be known in the literature.}
For the injective cone, no such result is known, except in the finite-dimensional case.
In this memoir, we give a simpler proof for the projective cone, and we also settle the problem for the injective cone.

\begin{AlphTheorem}
	\label{thm:intro:proper-cone}
	The projective cone $\minwedge{E_+}{F_+}$ is proper if and only if $E_+ = \{0\}$, or $F_+ = \{0\}$, or both $E_+$ and $F_+$ are proper cones \textup(cf.~\cite[Th\'eor\`eme 2]{Dermenjian-Saint-Raymond}\textup).
	
	The injective cone $\maxwedge{E_+}{F_+}$ is proper if and only if $E = \{0\}$, or $F = \{0\}$, or both $\overline{E_+}^{\,\weak}$ and $\overline{F_+}^{\,\weak}$ are proper cones.
\end{AlphTheorem}

The proof of \autoref{thm:intro:proper-cone} will be given in \mysecref{sec:projective-proper-cone} (projective cone) and \mysecref{sec:injective-proper-cone} (injective cone).
Note that there is a subtle difference between the corner cases in \autoref{thm:intro:proper-cone}: the corner case for the projective cone is when one of the \emph{cones} is trivial, whereas the corner case for the injective cone is when one of the \emph{spaces} is trivial.
In partial explanation of this discrepancy, we establish direct formulas for the lineality spaces, from which the criteria of \autoref{thm:intro:proper-cone} can easily be recovered.

\begin{AlphTheorem}
	\label{thm:intro:lineality-space}
	The lineality space of the projective/injective cone is
	\begin{align*}
		\lineal(\minwedge{E_+}{F_+}) &= (\lineal(E_+) \tensor \spn(F_+)) + (\spn(E_+) \tensor \lineal(F_+));\\\noalign{\smallskip}
		\lineal(\maxwedge{E_+}{F_+}) &= (\lineal(\overline{E_+}^{\,\weak}) \tensor F) + (E \tensor \lineal(\overline{F_+}^{\,\weak})).
	\end{align*}
\end{AlphTheorem}

The proof of \autoref{thm:intro:lineality-space} will be given in \autoref{cor:projective-lineality-space} (projective cone) and \autoref{cor:tensor-subspace} (injective cone).

We also address the related question of finding precise necessary and sufficient conditions for the \emph{closure} of the projective cone to be proper.
(For the injective cone, this is already addressed by \autoref{thm:intro:proper-cone}, because the injective cone is always closed.)
Recall that we say that $E_+$ is \emph{semisimple} if its weak closure is proper.
If $E$ is locally convex, then this is the same as saying that its ordinary closure is proper, because the weak and original closure of a convex set in a locally convex space coincide.
We prove the following semisimplicity version of \autoref{thm:intro:proper-cone}.

\begin{AlphTheorem}
	\label{thm:intro:semisimplicity}
	The projective cone $\minwedge{E_+}{F_+}$ is semisimple if and only if $E_+ = \{0\}$, or $F_+ = \{0\}$, or both $E_+$ and $F_+$ are semisimple.
\end{AlphTheorem}

A parallel result was proved by van Gaans and Kalauch \cite{vanGaans-Kalauch}: if $E_+$ and $F_+$ are Archimedean, then their projective tensor product is contained in an Archimedean proper cone.
Neither of these two results implies the other.

The proof of \autoref{thm:intro:semisimplicity} will be given in \mysecref{sec:semisimple-algebraic}.

We also address the question of semisimplicity in completed locally convex tensor products.
In \mysecref{sec:semisimple-completed}, we prove that the injective cone remains semisimple in the completed injective tensor product, and more generally, in every completion $E \hattensor_\alpha F$ for which the natural map $E \hattensor_\alpha F \to E \hattensor_\varepsilon F$ is injective.
However, we do not know whether the projective cone remains semisimple in the completed projective tensor product $E \hattensor_\pi F$; see \autoref{q:completed-semisimple}.
This question is related to the approximation property, as we will explain in \mysecref{sec:semisimple-completed}.

\section{Faces and extremal rays}
\label{sec:intro:faces}

Next, we turn our attention to another permanence property: preservation of faces.
In the normed theory, it follows from a result of Tseitlin \cite{Tseitlin} that the projective norm sometimes preserves extreme points of the closed unit ball, provided that certain topological requirements are met.
Further results in this direction are known for stronger notions of extreme points, such as denting points \cite{Ruess-Stegall-weak*-denting,Werner}.
This leads us to ask to which extent the projective and injective cones preserve extremal rays, or more generally, faces.

In the infinite-dimensional setting, we are not aware of prior results in this direction.
For closed cones in finite-dimensional spaces, some constructions are known in the literature.
The injective cone $\maxwedge{E_+}{F_+}$ can be interpreted as a cone of positive operators $E\algdual \to F$, whose faces have already been studied by many authors.
However, some information is lost in passing from $E$ to $E\algdual$, so instead we give a different construction which we believe to be more natural, and we extend this to the general setting.
For the projective cone, Tam \cite[\S{} 4]{Tam-faces} pointed out one way to construct faces (without proof).
We extend this construction to the general setting, give a full proof, and show that a pair of faces of the base cones give rise to not one but four natural faces of the projective tensor product.

\subsection{Faces of the projective cone}

By combining the mapping properties with the properness criteria, we can show that the projective cone preserves faces.

\begin{AlphTheorem}
	\label{thm:intro:projective-faces}
	If $M \subseteq E_+$ and $N \subseteq F_+$ are faces, then $(\minwedge{M}{F_+}) + (\minwedge{E_+}{N})$ and $(\minwedge{M}{N}) + \lineal(\minwedge{E_+}{F_+})$ are faces of the projective cone $\minwedge{E_+}{F_+}$.
	
	In particular, if $E_+$ and $F_+$ are proper cones, then $\minwedge{M}{N}$ is a face of $\minwedge{E_+}{F_+}$.
\end{AlphTheorem}

For closed, proper and generating cones in finite-dimensional spaces, this last property was already noted (without proof) by Tam in \cite[p.{} 53]{Tam-dissertation} and \cite[p.{} 71]{Tam-faces}.
We have recovered his simple proof for this special case (see \autoref{rmk:Tams-proof}), but a different technique is needed to prove the general case.
In the full generality stated here, \autoref{thm:intro:projective-faces} is a non-trivial result which contains \autoref{thm:intro:proper-cone} as a special case (by setting $M = N = \{0\}$).
Remarkably, it is true without any niceness assumptions on the cones $E_+$ and $F_+$ or the faces $M$ and $N$.

The proof of \autoref{thm:intro:projective-faces} will be given in \mysecref{sec:projective-faces}.
There we also mention two other faces induced by $M$ and $N$, showing that a pair of faces $M \subseteq E_+$ and $N \subseteq F_+$ gives rise to not one but four natural faces of the projective cone.
This is a new result, even in the finite-dimensional case.

As an application of \autoref{thm:intro:projective-faces}, we prove that the tensor product of symmetric convex sets preserves proper faces.

\begin{AlphTheorem}
	\label{thm:intro:absolutely-convex-faces}
	Let $E$ and $F$ be real vector spaces, let $C \subseteq E$, $D \subseteq F$ be absolutely convex, and let $M \subset C$, $N \subset D$ be proper faces. Then $\conv(M \settensor N)$ is a face of $\conv(C \settensor D)$.
\end{AlphTheorem}

For extreme points, results in this direction were known in the setting of normed tensor products (where $C$ and $D$ are the closed unit balls of the norms of $E$ and $F$).
However, in that setting, stronger assumptions are needed ($E$ and $F$ must be Banach spaces such that at least one of $E$ and $F$ has the approximation property and at least one of $E$ and $F$ has the Radon--Nikodym property), and a stronger conclusion is obtained ($x \tensor y$ is an extreme point of the \emph{closure} of $\conv(C \settensor D)$ in the \emph{completed} projective Banach space tensor product); see our remarks following \autoref{cor:projective-norm-faces}.
To our knowledge, no such results are known for higher faces, and we are not aware of a general statement like \autoref{thm:intro:absolutely-convex-faces} in the literature.

Note that \autoref{thm:intro:absolutely-convex-faces} is a purely algebraic statement, as we do not take closures.
We do not know whether it remains true after taking closures, but we suspect it does not (see \autoref{rmk:projective-norm-extreme-points}).
However, if $C$ and $D$ are compact, then $\conv(C \settensor D)$ is compact as well, so in particular it follows from \autoref{thm:intro:absolutely-convex-faces} that the projective norm preserves proper faces of the closed unit ball of finite-dimensional spaces (\autoref{cor:projective-norm-faces}).
As far as we know, this had only been known for extreme points.

The proof of \autoref{thm:intro:absolutely-convex-faces} will be given in \mysecref{sec:application-to-convex-sets}.

\subsection{Ideals for the injective cone}

The injective cone $\maxwedge{E_+}{F_+}$ can be interpreted as a subcone of the cone of positive operators $E\topdual \to F$.
Since there has been a lot of research into the properties of such cones, a lot has already been said about their faces.
For every set $M\topdual \subseteq E_+\topdual$ and every face $N \subseteq F_+$, it is trivially easy to show that the set of positive operators $T \in \CL(E\topdual, F)$ satisfying $T[M\topdual] \subseteq N$ forms a face (see \autoref{lem:injective-faces}), so this gives us a plethora of faces in the injective cone.
On the other hand, finding all extremal rays of the cone of positive operators is a notoriously difficult problem, so the face structure of the injective cone is still far from fully understood.

Although it is not so hard to construct faces of the injective cone from faces of the base cones, it is unclear what the ``right'' way of doing so is.
Only interpreting $\maxwedge{E_+}{F_+}$ as a cone of positive operators $E\topdual \to F$ is a bit unsatisfactory; we might just as well have interpreted it as a cone of positive operators $F\topdual \to E$.
Apart from the fact that this is not a symmetric formulation, this poses a bigger problem: in both interpretations, we have one primal and one dual space, but faces are not well-behaved under duality (not every face of $E_+\topdual$ is the dual of a face of $E_+$ and vice versa).
As far as we know, this problem has not been addressed in the literature, where the focus has been on cones of positive operators $E \to F$ instead of injective tensor products.

In this memoir, we set out to prove a satisfactory injective counterpart of \autoref{thm:intro:projective-faces}.
To do so, we believe we should change perspective from faces to ideals.
An \emph{\textup(order\textup) ideal} in a preordered vector space $(E,E_+)$ is a subspace $I \subseteq E$ for which the quotient cone $(E/I)_+$ is proper (for other equivalent definitions, see \autoref{prop:ideal-definitions}).
There is a close relationship between faces and ideals: the map $I \mapsto I_+$ ($= I \cap E_+$) defines a surjective many-to-one correspondence between the order ideals of the preordered vector space $(E,E_+)$ and the faces of $E_+$ (see \mysecref{app:faces-ideals}).

The benefit of working with ideals instead of faces is twofold.
First, in the infinite-dimensional (topological) setting, it is often important to work with \emph{closed} ideals so as to have a useable quotient, but it is not always easy (or even possible) to tell from a face whether or not it occurs as the positive part of a closed ideal.
Second, whereas faces of the injective cone can only be given by implicit formulas (all positive operators mapping certain sets into certain faces), for ideals we get the following very simple explicit formulas.

\begin{AlphTheorem}
	\label{thm:intro:injective-ideals}
	If $I \subseteq E$ and $J \subseteq F$ are ideals with respect to $\overline{E_+}^{\,\weak}$ and $\overline{F_+}^{\,\weak}$, then $(I \tensor J) + \lineal(\maxwedge{E_+}{F_+})$ is an ideal with respect to the injective cone $\maxwedge{E_+}{F_+}$.
	
	Additionally, if $I$ is weakly closed and $(E/I)_+$ is semisimple, or if $J$ is weakly closed and $(F/J)_+$ is semisimple, then $(I \tensor F) + (E \tensor J)$ is also an ideal with respect to the injective cone.
\end{AlphTheorem}

We believe \autoref{thm:intro:injective-ideals} to be new, even in the finite-dimensional case.
Note that the first formula simplifies to $I \tensor J$ whenever $\overline{E_+}^{\,\weak}$ and $\overline{F_+}^{\,\weak}$ are proper (by \autoref{thm:intro:proper-cone} or \autoref{thm:intro:lineality-space}).

The proof of \autoref{thm:intro:injective-ideals} will be given in \mysecref{sec:injective-ideals}.
There we will also show that the extra assumption that $(E/I)_+$ or $(F/J)_+$ is semisimple cannot be omitted from the second part of \autoref{thm:intro:injective-ideals} (see \autoref{xmpl:SCorface-ideal-fail}).
\hyphenation{Further-more}
Furthermore, we will extend \autoref{thm:intro:injective-ideals} to completed locally convex tensor products, though this requires additional topological assumptions (see \autoref{thm:injective-ideals-completed-i} and \autoref{thm:injective-ideals-completed-ii}).

Contrary to the projective case, the preceding results do not have an application to tensor products of symmetric convex sets.
The injective analogue of \autoref{thm:intro:absolutely-convex-faces} is simply not true, because the injective norm does not preserve extreme points of the unit balls (see \autoref{rmk:injective-norm-extreme-points}).
This makes it all the more remarkable that the injective cone preserves faces and extremal rays.

\subsection{Extremal rays}

As a special case of \autoref{thm:intro:projective-faces} and \autoref{thm:intro:injective-ideals}, we show that the projective and injective cones preserve extremal rays.

\begin{AlphTheorem}
	\label{thm:intro:extremal-rays}
	A vector $u \in E \tensor F$ is an extremal direction of the projective cone $\minwedge{E_+}{F_+}$ if and only if $u$ can be written as $u = x \tensor y$, where $x$ and $y$ are extremal directions of $E_+$ and $F_+$.
	
	If $x$ and $y$ are extremal directions of $\overline{E_+}^{\,\weak}$ and $\overline{F_+}^{\,\weak}$, then $x \tensor y$ is an extremal direction of the injective cone $\maxwedge{E_+}{F_+}$. All extremal directions of \textup(tensor\textup) rank one are of this form, but there may also be extremal directions of larger rank.
\end{AlphTheorem}

For closed, proper and generating cones in finite-dimensional spaces, this was already known; see for instance \cite[Thm.{} 3.4(2)]{Haynsworth-Fiedler-Ptak}%
	\hair\footnote{Their result is formulated only for polyhedral cones, but the proof also works for other closed, proper and generating cones.}
or \cite[Thm.{} 3.1]{Tam-survey}.
We extend it to arbitrary cones, and also to cones in completed locally convex tensor products (\autoref{cor:injective-completed-extremal-rays}).

The proof of \autoref{thm:intro:extremal-rays} will be given in \mysecref{sec:projective-extremal-rays} (projective cone) and \mysecref{sec:injective-extremal-rays}/\mysecref{sec:rank-one} (injective cone).
An immediate consequence is that every ``reasonable crosscone'' (see \mychpref{chp:reasonable}) preserves extremal rays whenever $E_+$ and $F_+$ are weakly closed.

\begin{AlphCorollary}
	\label{cor:intro:reasonable-extremal-rays}
	If $E_+$ and $F_+$ are weakly closed, and if $x$ and $y$ are extremal directions of $E_+$ and $F_+$, then $x \tensor y$ is an extremal direction of every convex cone $\mywedge \subseteq E \tensor F$ with $\minwedge{E_+}{F_+} \subseteq \mywedge \subseteq \maxwedge{E_+}{F_+}$.
\end{AlphCorollary}

In particular, the closure of the projective cone also preserves extremal rays.
We do not know if it also preserves higher faces (as in \autoref{thm:intro:projective-faces}).
For further discussion of this problem, see \mychpref{chp:open-problems}.
For more on reasonable crosscones, see \mychpref{chp:reasonable}.

\section{Special properties in the finite-dimensional case}
\label{sec:intro:finite-dimensional}

In the linear algebra literature, all papers on tensor products of convex cones have focused on closed, proper and generating cones in finite\-/dimensional spaces.
On the other hand, in the functional analysis literature, most of the focus has been on tensor products of Archimedean lattice cones.
Therefore the only cones which are covered by both regimes are the ones isomorphic to the standard cone $\R_{\geq 0}^n$ (all Archimedean lattice cones in $\R^n$ are isomorphic to $\R_{\geq 0}^n$), which are not very interesting from either perspective.
Consequently, these two lines of investigation have focused on completely different problems.

Even if one is primarily interested in tensor products of infinite\-/dimensional convex cones, it is good to be aware of the finite\-/dimensional theory, as various fundamental phenomena can already be observed here.
For this reason, in \mychpref{chp:finite-dimensional}, we give an overview of the most important additional properties in the finite-dimensional setting (with closed cones).

The main results of \mychpref{chp:finite-dimensional} are threefold.
First, for closed cones $E_+$ and $F_+$ in finite-dimensional spaces, we show that the projective cone $\minwedge{E_+}{F_+}$ can be interpreted as the cone of positive operators $E\algdual \to F$ that factor positively through some finite-dimensional Archimedean Riesz space (i.e.{} though some $\R^n$ with the standard cone $\R_{\geq 0}^n$).
Second, we show that the closure of the projective cone $\minwedge{E_+}{F_+}$ is equal to the projective cone $\minwedge{\overline{E_+}}{\overline{F_+}}$, thereby extending a result of Tam \cite{Tam-projective-closed}, who proved this in the case that $E_+$ and $F_+$ are closed, proper and generating.
Third, we study the basic properties of order retracts of finite-dimensional cones, and give many examples of retracts occurring in standard cones.

\section{Many examples where the projective and injective cone differ}
\label{sec:intro:many-examples}

Another question which has attracted a lot of attention is to determine under which circumstances the projective cone $\minwedge{E_+}{F_+}$ is dense in the injective cone $\maxwedge{E_+}{F_+}$.
For locally convex lattices $E$ and $F$, Birnbaum \cite[Prop.{} 3]{Birnbaum} proved that $\minwedge{E_+}{F_+}$ is dense in $\maxwedge{E_+}{F_+}$ in the projective topology (and therefore in every coarser topology), and followed this by an example showing that this is not true for all ordered locally convex spaces (not necessarily lattice ordered).
In general, however, the infinite-dimensional version of this problem does not appear to be well understood.

A lot more is known in the finite-dimensional setting (with closed, proper and generating cones), where various results in this direction have been obtained since the 1970s.
Here $\minwedge{E_+}{F_+}$ is automatically closed whenever $E_+$ and $F_+$ are closed (by the results from \mysecref{sec:intro:finite-dimensional}), so the question is whether or not the projective and injective cones are equal.

Let $E$ and $F$ be finite-dimensional spaces, and let $E_+ \subseteq E$, $F_+ \subseteq F$ be closed, proper and generating convex cones.
We say that $E_+$ is a \emph{simplex cone} (or \emph{Yudin cone}) if it is generated by a basis (or equivalently, if there is a linear isomorphism $E \cong \R^n$ that identifies $E_+$ with the standard cone $\R_{\geq 0}^n$).

In the 1970s, Barker showed that $\minwedge{E_+}{F_+} = \maxwedge{E_+}{F_+}$ whenever $E_+$ or $F_+$ is a simplex cone \cite{Barker-monotone}, and conversely conjectured that $E_+$ or $F_+$ must be a simplex cone whenever $\minwedge{E_+}{F_+} = \maxwedge{E_+}{F_+}$ \cite[p.{} 277]{Barker-survey}.
His conjecture remained open for a very long time, but partial results were obtained by Barker and Loewy \cite[Prop.{} 3.1]{Barker-Loewy}, who proved the conjecture when $F_+ = E_+\algdual$, and by Poole \cite[Thm.{} 5.15]{Poole-dissertation}, who proved it when $E_+$ and $F_+$ are polyhedral.
More recently, Huber and Netzer \cite{Huber-Netzer} proved the conjecture when $E_+$ is a positive semidefinite cone and $F_+$ is a polyhedral cone (or vice versa).

In \mychpref{chp:many-examples}, we prove Barker's conjecture for nearly all%
	\hair\footnotehack{The term `nearly all' has a precise meaning (namely, up to a $\sigma$-porous set); see for instance \cite{Zamfirescu-smooth-strictly-convex}.}
pairs $(E_+,F_+)$ of closed, proper and generating cones in finite-dimensional spaces.
Recall that a closed, proper and generating convex cone is called \emph{strictly convex} if every non-zero boundary point is an extremal direction, and \emph{smooth} if every non-zero boundary point has exactly one supporting hyperplane. It is well-known that $E_+$ is strictly convex if and only if $E_+\algdual$ is smooth, and vice versa.
We prove Barker's conjecture in the case that $\dim(E) \geq \dim(F)$ and $E_+$ is smooth or strictly convex.

\begin{AlphTheorem}
	\label{thm:intro:many-examples-i}
	Let $E$, $F$ be finite\-/dimensional real vector spaces, and let $E_+ \subseteq E$, $F_+ \subseteq F$ be closed, proper, and generating convex cones. If $\dim(E) \geq \dim(F)$, and if $E_+$ is strictly convex or smooth, then one has $\minwedgeFIN{E_+}{F_+} = \maxwedgeFIN{E_+}{F_+}$ if and only if $F_+$ is a simplex cone.
\end{AlphTheorem}

The set of convex bodies in $\R^n$ which are not smooth or strictly convex is meagre in the Hausdorff metric \cite{Klee-smoothness-rotundity}, and even satisfies the stronger notion of ``$\sigma$-porosity'' \cite{Zamfirescu-smooth-strictly-convex}.
As such, \autoref{thm:intro:many-examples-i} shows that the projective and injective cone differ for nearly all\hair\footnotemark{} pairs of closed, proper, and generating cones $(E_+,F_+)$.

The proof of \autoref{thm:intro:many-examples-i} will be given in \mysecref{sec:smooth-strictly-convex}.

Although \autoref{thm:intro:many-examples-i} covers nearly all cones, it does not cover most standard cones.
For instance, polyhedral cones and positive semidefinite cones (and their duals) have many non-trivial faces, so they are not smooth or strictly convex.
We complement \autoref{thm:intro:many-examples-i} with a similar result for combinations of standard cones.

\begin{AlphTheorem}
	\label{thm:intro:many-examples-ii}
	Let $E$, $F$ be finite\-/dimensional real vector spaces, and let $E_+ \subseteq E$, $F_+ \subseteq F$ be closed, proper, and generating convex cones. Assume that each of $E_+$ and $F_+$ is one of the following \textup(all combinations allowed\textup):
	\begin{enumerate}[label=(\roman*)]
		\item a polyhedral cone;
		\item a second-order cone;
		\item a \textup(real or complex\textup) positive semidefinite cone.
	\end{enumerate}
	Then one has $\minwedgeFIN{E_+}{F_+} = \maxwedgeFIN{E_+}{F_+}$ if and only if at least one of $E_+$ and $F_+$ is a simplex cone.
\end{AlphTheorem}

Using retracts, \autoref{thm:intro:many-examples-ii} could already be deduced from the aforementioned known results of Poole \cite[Thm.{} 5.15]{Poole-dissertation} and Huber and Netzer \cite{Huber-Netzer}.
We give a new proof of \autoref{thm:intro:many-examples-ii}, thereby also providing new proofs of the results of Poole and of Huber and Netzer.

The proof of \autoref{thm:intro:many-examples-ii} will be given in \mysecref{sec:standard-cones}.

\begin{AlphRemark}
	\label{rmk:intro:Aubrun-et-al}
	As this manuscript was being written, the preceding results were superseded by independent work of Aubrun, Lami, Palazuelos and Pl\'avala \cite{Aubrun-et-al-ii}.
	Motivated by questions in theoretical physics, they proved that $\minwedgeFIN{E_+}{F_+} = \maxwedgeFIN{E_+}{F_+}$ if and only if at least one of $E_+$ and $F_+$ is a simplex cone (provided that $E_+$ and $F_+$ are closed, proper, and generating).
	Both \autoref{thm:intro:many-examples-i} and \autoref{thm:intro:many-examples-ii} are special cases of this result.
	We had not been aware of these results until shortly before publication of the first version of this manuscript on arXiv, and our proofs differ significantly from the proof in \cite{Aubrun-et-al-ii}.
	Although we recover their result for nearly all cones, we have not been able to recover it in full generality.
\end{AlphRemark}

\subsection{Applications to operator systems}

The recent resurgence of interest in the question of whether or not $\minwedge{E_+}{F_+} = \maxwedge{E_+}{F_+}$ is due in part to recent developments in the study of operator systems \cite{Fritz-Netzer-Thom,Huber-Netzer}.
Reformulated in terms of operator systems (using notation from \cite{Fritz-Netzer-Thom}), our results prove the following.

\begin{AlphCorollary}
	\label{cor:intro:operator-systems}
	Let $C \subseteq \R^d$ be a closed, proper, and generating convex cone. If $d \leq 4$, or if $C$ is strictly convex, or smooth, or polyhedral, or \textup(real or complex\textup) positive semidefinite, then the following are equivalent:
	\begin{enumerate}[label=(\roman*)]
		\item $C$ is a simplex cone;
		\item the minimal and maximal operator systems $\minopsys{C}$ and $\maxopsys{C}$ are equal;
		\item there exists $n \geq 2$ for which $\minopsys[n]{C} = \maxopsys[n]{C}$;
		\item one has $\minopsys[2]{C} = \maxopsys[2]{C}$.
	\end{enumerate}
\end{AlphCorollary}

Again, \autoref{cor:intro:operator-systems} was superseded by the work of Aubrun, Lami, Palazuelos and Pl\'avala \cite{Aubrun-et-al-ii}, who removed the additional assumptions on the cone $C$ (see \autoref{rmk:intro:Aubrun-et-al}).

The proof of \autoref{cor:intro:operator-systems} will be given in \mysecref{sec:standard-cones}.

\section{Appendix: faces and ideals}
The main body of this memoir is complemented by an appendix on faces and ideals of convex cones in infinite\-/dimensional spaces.
This material is not directly related to tensor products, but will be used extensively in the proofs.

Although faces and ideals have each received a lot of attention in the literature, the link between these concepts does not appear to be well-known.
The relationship is very simple: the map $I \mapsto I_+$ defines a surjective many-to-one correspondence between ideals and faces (\mysecref{app:faces-ideals}).
Going back and forth between faces and ideals is crucial in our study of the faces of the projective/injective cone.

In \mychpref{app:faces}, we study the basic properties of faces and ideals and the connection between the two (\mysecref{app:faces-ideals}), we outline to which extent the homomorphism and isomorphism theorems hold (\mysecref{app:isomorphism-theorems}), and we study dual and exposed faces in infinite-dimensional cones (\mysecref{app:dual-and-exposed}).
A more detailed outline will be given at the beginning of \mychpref{app:faces}.

\section{Organization}

In \mychpref{chp:preliminaries}, we recall all required notation and terminology.
This is complemented by a glossary notation and an index, both of which can be found at the end of this memoir.

In \mychpref{chp:projective}, we study the properties of the projective cone.
Here we prove all of the main results for the projective cone (see \mysecref{sec:intro:mapping-properties}--\ref{sec:intro:faces}), with the exception of \autoref{thm:intro:semisimplicity}, whose proof is deferred until \mychpref{chp:reasonable}.

Likewise, in \mychpref{chp:injective}, we study the properties of the injective cone.
Here we prove all of the main results for the injective cone (see \mysecref{sec:intro:mapping-properties}--\ref{sec:intro:faces}).

In \mychpref{chp:reasonable}, we study the basic properties of the so-called `reasonable crosscones' (that is, arbitrary cones which lie somewhere between the projective and injective cone).
We show that all reasonable crosscones have the same rank $1$ tensors whenever $E_+$ and $F_+$ are weakly closed and proper, and we briefly look at ideals and extremal rays of reasonable crosscones.
Furthermore, we study semisimplicity of reasonable crosscones, and we give a proof of \autoref{thm:intro:semisimplicity} (on the semisimplicity of the projective cone).
We also look at questions surrounding semisimplicity of cones in \emph{completed} locally convex tensor products, and we discuss how these questions are related to topological issues and the approximation property.

In \mychpref{chp:finite-dimensional}, we give an overview of the most important additional properties in the finite-dimensional setting (see \mysecref{sec:intro:finite-dimensional}).
Building on this, in \mychpref{chp:many-examples}, we give many (finite-dimensional) examples where the projective and injective cone are different (see \mysecref{sec:intro:many-examples}).

In \mychpref{chp:open-problems}, we discuss a few open problems.

Finally, in \mychpref{app:faces}, we discuss the relation between faces and order ideals.
These results are not very well known, and will be used extensively in the main body of this memoir, so we have included them for completeness.

\chapter{Preliminaries}
\label{chp:preliminaries}

In this memoir, we study tensor products of convex cones $E_+,F_+$ in real vector spaces $E,F$.
Occasionally, $E$ and $F$ will be \emph{topological} vector spaces, but usually they are only assumed to be the primal spaces of the dual pairs $\langle E, E' \rangle$, $\langle F , F' \rangle$.
In this chapter, we cover the necessary terminology and notation.
This is complemented by a glossary of notation and an index, both of which can be found at the end of this manuscript.

\section{Topological vector spaces}
\label{sec:tvs}
Throughout this memoir, all vector spaces are over $\R$.

If $E$ is a vector space, then a linear map $E \to \R$ is called a \emph{linear functional}.
The \index{algebraic dual space}\emph{algebraic dual space} \glsadd{algebraic dual space}$E\algdual$ of $E$ is the space of all linear functionals $E \to \R$.

A \index{topological vector space}\emph{topological vector space} is a vector space $E$ equipped with a topology $\lintop$ such that the map $E \times E \to E$, $(x,y) \mapsto x + y$ and the map $\R \times E \to E$, $(\lambda,x) \mapsto \lambda x$ are (jointly) continuous with respect to $\lintop$.
If $E$ is a topological vector space, then its \index{topological dual space}\emph{topological dual space} \glsadd{topological dual space}$E\topdual$ is the space of all continuous linear functionals $E \to \R$.
It is a subspace of the algebraic dual, but usually the two are different.

A topological vector space is \index{topological vector space!locally convex}\emph{locally convex} if it has a neighbourhood base at $0$ consisting of convex sets.
Locally convex spaces are more well-behaved than general topological vector spaces, and almost all important spaces in functional analysis are locally convex.
For more on locally convex spaces, the reader is referred to a graduate level textbook on functional analysis, for instance \cite{Schaefer,Rudin,Conway}.

\subsection{Dual pairs and weak topologies}

Let $E$ and $F$ be vector spaces, and let $b : E \times F \to \R$ be a bilinear form.
We say that a subset $N \subseteq F$ \emph{separates points on $E$ \textup(via $b$\textup)} if for every $x \in E$ there is some $y' \in N$ such that $b(x,y') \neq 0$.
Likewise, a subset $M \subseteq E$ \emph{separates points on $F$ \textup(via $b$\textup)} if for every $y \in F$ there is some $x' \in M$ such that $b(x',y) \neq 0$.
If $F$ separates points on $E$ and $E$ separates points on $F$, then $b$ is called a \emph{dual pairing} (or \emph{non-degenerate bilinear form}), and we denote it by the shorthand notation $\langle x,y \rangle := b(x,y)$.
A \index{dual pair}\emph{dual pair} is a tuple $(E,F,\langle\:\cdot\:,\:\cdot\:\rangle)$, where $E$ and $F$ are vector spaces and $\langle\:\cdot\:,\:\cdot\:\rangle : E \times F \to \R$ is a dual pairing.
We usually use the shorthand notation \glsadd{dual pair}$\langle E, F\rangle$ to denote the dual pair $(E,F,\langle\:\cdot\:,\:\cdot\:\rangle)$.

Let $\langle E,F\rangle$ be a dual pair.
The \emph{$\sigma(E,F)$-topology} on $E$ is the initial topology induced by the family of linear functionals $\{ x \mapsto \langle x , y \rangle \, : \, y \in F\}$, and the \emph{$\sigma(F,E)$-topology} on $F$ is the initial topology induced by the family of linear functionals $\{ y \mapsto \langle x , y \rangle \, : \, x \in E\}$.
If $F = E\topdual$ is the topological dual of $E$, then the $\sigma(E,E')$-topology on $E$ is called the \index{topology!weak}\glsadd{weak topology}\emph{weak topology}, and the $\sigma(E',E)$-topology on $E'$ is called the \index{topology!weak-*@weak\nobreakdash-$*$}\glsadd{weak-* topology}\emph{weak\nobreakdash-$*$ topology}.
In this case, we denote the resulting topological vector spaces by \glsadd{E_weak}$E_\weak$ and \glsadd{E_weakstar}$E_\weakstar'$, respectively.
Likewise, the weak closure of a subset $M \subseteq E$ will be denoted \glsadd{weak closure}$\overline{M}^{\,\weak}$, and the weak\nobreakdash-$*$ closure of a subset $N \subseteq E'$ will be denoted \glsadd{weak-* closure}$\overline{N}^{\,\weakstar}$.

Throughout this memoir, we tacitly assume that all dual pairs consist of a topological vector space and its topological dual space (or algebraic dual space, if the primal space has no topology).
Consequently, by a slight abuse of notation, we denote our dual pairs as $\langle E,E'\rangle$, $\langle F,F'\rangle$, etc., and we say that $E$ \emph{belongs to the dual pair} $\langle E,E'\rangle$.
To keep the topological prerequisites to a minimum, we will forget about the original topology of $E$, and only remember the dual pair $\langle E,E'\rangle$ to which $E$ belongs.
When $E$ has no topology, we tacitly assume that $E' := E\algdual$ is the algebraic dual.

\subsection{Linear maps}
If $E$ and $F$ are vector spaces, then the space of linear maps $E \to F$ is denoted by \glsadd{L}$\LL(E,F)$.
If $E$ and $F$ are topological vector spaces, then the space of \emph{continuous} linear maps $E \to F$ is denoted by \glsadd{CL}$\CL(E,F)$.
If the (topological) duals separate points, then every continuous map $E \to F$ is also weakly continuous (see e.g.~\cite[\S 20.4.(5)]{Kothe-I}), so we have
\[ \CL(E,F) \subseteq \CL(E_\weak, F_\weak) \subseteq \LL(E,F). \]
If $E$ and $F$ are vector spaces without topologies, then every linear map $T : E \to F$ is $\sigma(E,E\algdual)$-$\sigma(F,F\algdual)$-continuous (since $\psi \circ T$ is $\sigma(E,E\algdual)$-continuous for every $\psi \in F\algdual$), so we have
\[ \qquad \CL(E_\weak,F_\weak) = \LL(E,F) \qquad\qquad \text{(if $E' = E\algdual$, $F' = F\algdual$)}. \]
The adjoint of a (continuous) linear map $T : E \to F$ is denoted \glsadd{algebraic adjoint}$T\algdual : F\algdual \to E\algdual$ (algebraic adjoint) or \glsadd{topological adjoint}$T\topdual : F\topdual \to E\topdual$ (topological adjoint).

\subsection{Bilinear maps}
Let $E,F,G$ be topological vector spaces.
A bilinear map $b : E \times F \to G$ is \emph{\textup(jointly\textup) continuous} if it is continuous with respect to the product topology on $E \times F$, and \emph{separately continuous} if for all fixed $x_0 \in E$ and $y_0 \in F$ the maps $y \mapsto b(x_0,y)$ and $x \mapsto b(x,y_0)$ are continuous.

From left to right, let
\[ \glsadd{CBil}\CBil(E \times F) \subseteq \glsadd{SCBil}\SCBil(E \times F) \subseteq \glsadd{Bil}\Bil(E \times F) \]
denote the spaces of continuous, separately continuous, and all bilinear forms $E \times F \to \R$.%
	\hair\footnote{Note: with this notation it is possible to confuse $\CBil(E \times F)$ with $\SCBil(E \times F)$, but notation like this appears to be at least moderately common (e.g.~\cite[p.~91]{Schaefer}, \cite[p.~154]{Kothe-II}).}

For our purposes, the most important of these is $\SCBil(E \times F)$, the space of separately continuous bilinear forms. It will be used extensively in the study of the injective cone in \mychpref{chp:injective}.

Given a bilinear form $b \in \Bil(E \times F)$ and fixed vectors $x_0 \in E$, $y_0\in F$, we let \glsadd{bxcdot}$b(x_0 , \,\cdot\, ) \in F\algdual$ and \glsadd{bcdoty}$b(\,\cdot\, , y_0) \in E\algdual$ denote the linear functionals
\begin{align*}
	b(x_0 , \,\cdot\, ) := \big(y \mapsto b(x_0,y)\big);\\\noalign{\smallskip}
	b(\,\cdot\, , y_0 ) := \big(x \mapsto b(x,y_0)\big).
\end{align*}
Using this notation, we see that $b$ is separately continuous if and only if one has $b(x_0 , \,\cdot\, ) \in F\topdual$ for all $x_0 \in E$ and $b(\,\cdot\, , y_0) \in E\topdual$ for all $y_0 \in F$. In particular, it follows that $\SCBil(E \times F)$ does not depend on the topologies of $E$ and $F$, but only on the dual pairs $\langle E,E\topdual\rangle$, $\langle F,F\topdual\rangle$. Likewise, it follows that $\SCBil(E \times F) = \Bil(E \times F)$ whenever $E\topdual = E\algdual$ and $F\topdual = F\algdual$.

It follows from \cite[\S 40.1.(2')]{Kothe-II} and the preceding remarks that the maps $b \mapsto (x \mapsto b(x , \,\cdot\, ))$ and $b \mapsto (y \mapsto b(\,\cdot\, , y))$ define linear isomorphisms
\[ \SCBil(E \times F) = \SCBil(E_\weak \times F_\weak) \cong \CL(E_\weak , F_\weakstar\topdual) \cong \CL(F_\weak , E_\weakstar\topdual ). \]
(The isomorphism $\CL(E_\weak , F_\weakstar\topdual) \cong \CL(F_\weak , E_\weakstar\topdual )$ is simply $T \mapsto T'$.)

\subsection{Tensor products}
We assume the reader to be familiar with the basics of the (algebraic) theory of tensor products.
We will need very little on the side of topological tensor products (but many results are inspired by the theory of normed tensor products).

For clarity, we shall occasionally use the following notation: if $E$ and $F$ are vector spaces and $M \subseteq E$ and $N \subseteq F$ are subsets, then we define the ``set-wise'' tensor product
\[ M \glsadd{settensor}\settensor N := \{x \tensor y \, : \, x\in M,\ y\in N\} \subseteq E \tensor F. \]

\section{Subspaces, quotients, and tensor products of dual pairs}
\label{sec:dual-pairs}
Many of the properties of a convex cone $E_+$ in topological vector space $E$ depend only on the geometry of $E_+$ and on the dual pair $\langle E,E\topdual\rangle$, not on the precise topology of $E$.
In particular, we don't need to know the exact topology of $E \tensor F$, because for our purposes it suffices to know what its dual space is.
This enables us to ignore topological issues in the tensor product, thereby circumventing the notoriously complicated theory of locally convex tensor products.
Instead, we formulate our results for a wide range of \emph{reasonable} duals of $E \tensor F$ (see below).

Throughout this memoir, we encode the ``input spaces'' $E$ and $F$ and the ``output space'' $E \tensor F$ by the dual pairs to which they belong; that is, by only remembering what the appropriate (algebraic or topological) dual space is, without remembering the exact topology.
In this section, we briefly discuss how to handle subspaces, quotients, and tensor products of dual pairs.

Questions about the projective/injective cone that depend not only on the dual pair, but also on a specific topology on $E \tensor F$, will not be treated in this memoir.
In particular, for questions about normality of the projective/injective cone, we refer the reader to \cite{Peressini-Sherbert}.

\begin{remark}
	Because we choose to forget about the topology of $E$ and only formulate results in terms of the dual pair $\langle E , E' \rangle$, we occasionally have to make use of the weak topology.
	In particular, we often refer to the \emph{weak closure} of a convex cone and to \emph{weakly closed} subspaces.
	We should point out that the adjective ``weak'' can be omitted here if $E$ is a locally convex space, because in this setting the weak and original closure of a convex set (in particular, a convex cone or a subspace) coincide, by \cite[Theorem 3.12]{Rudin}.
	
	If $E$ is a topological vector space which is not locally convex, then the adjective ``weak'' cannot be omitted.
\end{remark}

\subsection{Subspaces}
If $\langle E,E'\rangle$ is a dual pair and if $I \subseteq E$ is a subspace, then we will understand $I$ to belong to the dual pair $\langle I , E'/I^\perp \rangle$.

We show that this is usually, but not always, the natural dual pair for $I$.
To that end, assume that $E$ a topological vector space, $E'$ is its (topological) dual, and $I$ carries the subspace topology. Let $T : I \hookrightarrow E$ denote the inclusion and $T' : E' \to I'$ its adjoint.

If $E$ is locally convex, then every continuous linear functional on $I$ can be extended to $E$, so $T'$ is surjective. Clearly $\ker(T') = I^\perp$, so $T'$ restricts to a linear isomorphism $E\topdual/I^\perp \to I\topdual$.
Furthermore, the relative $\sigma(E,E\topdual)$-topology on $I$ coincides with the $\sigma(I,E\topdual/I^\perp)$-topology (even if $I$ is not closed), so we may unambiguously refer to this as the weak topology on $I$. On the other hand, the $\sigma(E\topdual/I^\perp,I)$-topology on $E\topdual/I^\perp = I\topdual$ coincides with the quotient topology $E_\weakstar\topdual/I^\perp$ if and only if $I$ is closed (see e.g.~\cite[\S IV.4.1, Corollary 1]{Schaefer}).

If $E$ is not locally convex, then $I$ may have continuous linear functionals that cannot be extended. In this case one still has $\ker(T') = I^\perp$, but $T'$ is not surjective, so $I' \neq E'/I^\perp$. Nevertheless, $E'/I^\perp$ is the dual of $I$ with respect to the relative $\sigma(E,E')$-topology on $I$.

\subsection{Quotients}
If $E$ is a topological vector space and if $I \subseteq E$ is a closed subspace, then $E/I$ is a Hausdorff topological vector space. Every continuous linear functional $E/I \to \R$ extends to a continuous linear functional $E \to \R$ that vanishes on $I$. Conversely, if $\varphi : E \to \R$ is a continuous linear functional that vanishes on $I$, then $\varphi$ factors through $E/I$, by the universal property of quotients. Therefore: $(E/I)\topdual \cong I^\perp$ as vector spaces.

Thus, if $\langle E,E'\rangle$ is a dual pair and if $I \subseteq E$ is a weakly closed subspace, then we can understand $E/I$ to belong to the dual pair $\langle E/I , I^\perp \rangle$. The quotient topology on $E_\weak/I$ coincides with the $\sigma(E/I,I^\perp)$-topology, and the subspace topology on $I^\perp \subseteq E_\weakstar\topdual$ coincides with the $\sigma(I^\perp , E/I)$-topology (see e.g.~\cite[\S IV.4.1, Corollary 1]{Schaefer}), so we may unambiguously refer to these as the weak topology on $E/I$ and the weak\nobreakdash-$*$ topology on $I^\perp$, respectively.

The only downside to this approach is that we cannot ``see'' all quotients of $E$.
If $E$ is locally convex, then every closed subspace is also weakly closed, but this is not true for general topological vector spaces (see e.g.~\cite{Kalton-quotients}).
However, if $I$ is closed but not weakly closed, then the quotient $E/I$ is Hausdorff, but its topological dual $(E/I)' = I^\perp$ does not separate points.
Throughout this memoir, we assume that all duals separate points, so we only consider quotients $E/I$ where $I$ is weakly closed.

\subsection{Tensor products}
\label{subsec:reasonable-dual}
Let $\langle E,E'\rangle$ and $\langle F,F'\rangle$ be dual pairs.
Recall from \mysecref{sec:tvs} that the space $\SCBil(E \times F)$ of separately continuous bilinear forms $E \times F \to \R$ can be defined without specifying topologies on $E$ and $F$, since this space depends only on the dual pairs $\langle E,E'\rangle$ and $\langle F,F'\rangle$.
Since the algebraic dual of $E \tensor F$ is isomorphic with $\Bil(E \times F)$, we can identify $\SCBil(E \times F)$ with a subspace of $(E \tensor F)\algdual$.
We say that a subspace $G \subseteq (E \tensor F)\algdual$ is a \index{reasonable dual space}\emph{reasonable dual of $E \tensor F$} (with respect to the dual pairs $\langle E,E'\rangle$, $\langle F,F'\rangle$) if
\[ E' \tensor F' \subseteq G \subseteq \SCBil(E \times F). \]
This definition will allow us to treat (duality of) convex cones in topological tensor products without having to deal with the specifics of topological tensor products.

We show that this definition covers all important cases.
First, if $E,F$ are locally convex and ${E \tensor F}$ carries a \emph{compatible} topology $\alpha$ (in the sense of Grothendieck \cite[p.~89]{Grothendieck-top}; see also \cite[\S 44.1]{Kothe-II}), then we claim that the topological dual \glsadd{completed tensor product}$(E \tensor_\alpha F)\topdual$ is a reasonable dual of $E \tensor F$.
Indeed, one of the requirements for $\alpha$ to be compatible is $E\topdual \tensor F\topdual \subseteq (E \tensor_\alpha F)\topdual$.
Moreover, every compatible topology is coarser than the \index{topology!inductive}inductive topology, whose dual is $\SCBil({E \times F})$ (see e.g.~\cite[\S 44.1.(5)]{Kothe-II}), so one has $(E \tensor_\alpha F)\topdual \subseteq \SCBil(E \times F)$. This shows that $(E \tensor_\alpha F)\topdual$ is a reasonable dual.

Second, if $E$ and $F$ originate from spaces without topologies, then we understand these to belong to the dual pairs $\langle E,E\algdual\rangle$, $\langle F,F\algdual\rangle$. In this case we have $\SCBil(E \times F) = \Bil(E \times F)$ (see \mysecref{sec:tvs}), so we find that $(E \tensor F)\algdual = \Bil(E \times F)$ is a reasonable dual of $E \tensor F$. This is useful when applying topological results in the non-topological setting (for instance, see \autoref{cor:tensor-products-of-algebraic-approximate-pushforwards}).

\section{Ordered vector spaces}
\label{sec:convex-cones}
\subsection{Convex cones and their duals}
Let $E$ be a (real) vector space. A \index{convex cone}\emph{convex cone}%
	\hair\footnote{A note about terminology: some authors call this a \index{wedge|see {convex cone}}\emph{wedge}, and reserve the term \emph{cone} for what we call a \emph{proper cone} (e.g.~\cite{Day}, \cite{Peressini}, \cite{Aliprantis-Tourky}).}
is a non-empty subset $\mywedge \subseteq E$ satisfying $\mywedge + \mywedge \subseteq \mywedge$ and $\lambda \mywedge \subseteq \mywedge$ for all $\lambda \in \R_{\geq 0}$. If $\mywedge$ is a convex cone, then \glsadd{lineality space}$\lineal(\mywedge) := \mywedge \cap -\mywedge$ is a linear subspace of $E$, called the \index{lineality space}\emph{lineality space} of $\mywedge$. We say that $\mywedge$ is \index{convex cone!proper}\emph{proper}%
	\hair\footnote{Some authors call this \emph{pointed} or \emph{salient}.}
if $\lineal(\mywedge) = \{0\}$, and \index{convex cone!generating}\emph{generating} if $\mywedge - \mywedge = E$.

If $\mywedge \subseteq E$ is a convex cone, then its \index{dual cone}\index{dual cone!algebraic}\glsadd{algebraic dual cone}\emph{algebraic dual cone} $\mywedge\algdual$ is the set of all positive linear functionals:
\[ \mywedge\algdual := \big\{\varphi \in E\algdual \, : \, \varphi(x) \geq 0 \ \text{for all $x \in \mywedge$}\big\}. \]
If $\langle E,E'\rangle$ is a dual pair, then we define $\mywedge' := \mywedge\algdual \cap E'$ (the \index{dual cone!with respect to a dual pair}\glsadd{topological dual cone}\emph{dual cone} for the dual pair $\langle E,E'\rangle$).
The dual cone of $\mywedge' \subseteq E'$ with respect to the dual pair $\langle E',E\rangle$ is the \index{bipolar cone}\emph{bipolar cone}
\[ \mywedge'' := \big\{x \in E \, : \, \langle x,\varphi\rangle \geq 0 \ \text{for all $\varphi \in \mywedge'$}\big\} = (\mywedge')'.\hair\footnote{There is some chance of confusion here, because $\mywedge''$ could also refer to the positive cone of the second dual $E''$ of $E$. To avoid this confusion, and in light of the bipolar theorem, we will henceforth refer to the bipolar cone as $\overline{\mywedge}^{\,\weak}$ instead of $\mywedge''$.} \]
Using the (one-sided) \index{bipolar theorem}bipolar theorem, one easily shows that \glsadd{bipolar cone}$\mywedge'' = \overline{\mywedge}^{\,\weak}$.
It follows that ${}^\perp (\mywedge') = \mywedge'' \cap -\mywedge'' = \lineal(\overline{\mywedge}^{\,\weak})$.
In particular, $\overline{\mywedge}^{\,\weak}$ is a proper cone if and only if $\mywedge'$ separates the points of $E$.
If this is the case, then we say that $\mywedge$ is \index{convex cone!semisimple}\emph{semisimple}.
(For an equivalent definition of semisimplicity in terms of representations, see \cite{Dobben-semisimplicity}.)

\subsection{Ordered vector spaces}
Let $E$ be a vector space.
A \index{vector preorder}\emph{vector preorder} is a preorder $\leq$ on $E$ such that for all $x,y,z \in E$ and $\lambda \in \R_{> 0}$ one has $x \leq y$ if and only if $x + z \leq y + z$ if and only if $\lambda x \leq \lambda y$.

There is a natural bijective correspondence between vector preorders on $E$ and convex cones in $E$, which identifies the preorder $\leq$ with the convex cone \glsadd{primal cone}$E_+ := \{x \in E \, : \, x \geq 0\}$ of positive elements of $E$.
In the reverse direction, a convex cone $\mywedge \subseteq E$ is identified with the vector preorder $\leq_{\mywedge}$ given by $x \leq_{\mywedge} y$ if and only if $y - x \in \mywedge$ (for all $x,y\in E$).

A \index{preordered (topological) vector space}\emph{preordered vector space} is a tuple $(E,E_+)$ where $E$ is a vector space and $E_+ \subseteq E$ is a convex cone.
We understand $E$ to be preordered by the vector preorder associated with $E_+$.
Likewise, a \emph{preordered topological vector space} is a tuple $(E,E_+)$, where $E$ is a topological vector space and $E_+ \subseteq E$ is a convex cone.
Note that we do not assume any kind of compatibility between the topology and the cone $E_+$.

The positive cone $E_+$ of a preordered (topological) vector space $E$ is proper if and only if the associated vector preorder is antisymmetric (so it is a partial order).
If this is the case, then $(E,E_+)$ is called an \index{ordered (topological) vector space}\emph{ordered \textup(topological\textup) vector space}.

Whenever we have a (topological) vector space $E$ and a convex cone $E_+ \subseteq E$, we will implicitly assume that $E$ is a preordered (topological) vector space with positive cone $E_+$.
Furthermore, the preorder of $E$ will be denoted by $\leq$.

\subsection{Positive linear maps}
Let $(E,E_+)$ and $(F,F_+)$ be preordered vector spaces.
We say that a linear map $T \in \LL(E,F)$ is \index{positive linear map}\emph{positive} if $T[E_+] \subseteq F_+$, a \index{pullback}\emph{pullback} (or \index{bipositive linear map|see {pullback}}\emph{bipositive operator}) if $E_+ = T^{-1}[F_+]$, and a \index{pushforward}\emph{pushforward} if $T[E_+] = F_+$.

Furthermore, if $E$ and $F$ belong to dual pairs $\langle E,E'\rangle$ and $\langle F,F'\rangle$, then we say that an operator $T \in \CL(E_\weak, F_\weak)$ is \index{approximately positive linear map}\emph{approximately positive} if $T[\overline{E_+}^{\,\weak}] \subseteq \overline{F_+}^{\,\weak}$, an \index{approximate pullback}\emph{approximate pullback} (or \index{approximately bipositive linear map|see {approximate pullback}}\emph{approximately bipositive}) if $\overline{E_+}^{\,\weak} = T^{-1}[\overline{F_+}^{\,\weak}]$, and an \index{approximate pushforward}\emph{approximate pushforward} if $\overline{T[E_+]}^{\,\weak} = \overline{F_+}^{\,\weak}$.
A continuous positive map (resp.~pushforward) is also approximately positive (resp.~an approximate pushforward), but a pullback is not necessarily an approximate pullback.
(Concrete example: let $F = \R^2$ with $F_+ = \{(x,y) \, : \, x > 0\} \cup \{(0,0)\}$, let $E := \spn\{(0,1)\} \subseteq F$ with $F_+ := F_+ \cap E$, and let $T$ be the inclusion $E \hookrightarrow F$.)

These approximate type operators are not particularly natural from the perspective of ordered vector spaces, but they come into play as soon as one starts to make use of duality.
Given $T \in \CL(E_\weak , F_\weak)$, it is not hard to show that the adjoint $T' \in \CL(F_\weakstar' , E_\weakstar')$ is positive if and only if $T$ is approximately positive.
\label{p:approximate-duality}
In addition, using that $(T[C])\polar = (T')^{-1}[C\polar]$ (e.g.{} \cite[Proposition IV.2.3(a)]{Schaefer}), it is easy to show that $T$ is an approximate pullback if and only if $T'$ is a weak\nobreakdash-$*$ approximate pushforward, and vice versa.
This is no longer true if the adjective ``approximate'' is omitted.

We shall treat pullbacks and pushforwards as the natural ordered analogues of metric injections (isometries) and metric surjections (quotients); see \autoref{tbl:dictionary}.
As soon as duality comes into play, it will be helpful to pass to the corresponding approximate versions.
In particular, we show that the injective cone preserves approximate pullbacks, but not all pullbacks.

Note that every linear map $E \to F$ can be made a pullback/pushforward by choosing appropriate cones. In particular, a pullback is not necessarily injective, and a pushforward is not necessarily surjective. However, if $E_+$ is a proper cone, then every pullback $T : E \to F$ is injective (since $\ker(T) \subseteq T^{-1}[F_+] = E_+$), and if $F_+$ is generating then every pushforward $E \to F$ is surjective.

\begin{table}[h!t]
	\centering
	\renewcommand{\arraystretch}{1.25}
	\caption{Ordered analogues of common concepts in the normed theory.}
	\begin{tabular}{ll}
		\toprule
		Normed theory & Ordered theory \\
		\midrule
		Continuous operator & Positive operator \\
		Metric injection (isometry) & Pullback (bipositive operator) \\
		Metric surjection (quotient) & Pushforward (quotient) \\
		Projection (complemented subspace) & Positive projection (order retract) \\
		\bottomrule
	\end{tabular}
	\label{tbl:dictionary}
\end{table}

\subsection{Retracts}
Let $(E,E_+)$ be a preordered vector space.
A subspace $F \subseteq E$ is an \index{order retract}\emph{order retract} if there exists a positive projection $E \to F$.
If $E$ is furthermore a topological vector space, then we say that $F$ is a \index{topological order retract}\emph{topological order retract} if there exists a \emph{continuous} positive projection $E \to F$.

For simplicity, we shall speak of \index{retract|see {order retract}}\emph{retracts} and \index{top-retract|see {topological order retract}}\emph{top-retracts}, as there is minimal chance of confusion with other types of retracts (e.g.~from topology).

Note that a retract provides at the same time an injective pullback (i.e.~bipositive map) $F \hookrightarrow E$ and a surjective pushforward (``quotient'') $E \twoheadrightarrow F$. We will show that, although the projective tensor product does not preserve bipositive maps and the injective tensor product does not preserve quotients, retracts are sufficiently rigid to be preserved by both.

To illustrate their place in the theory, note that every top-retract is a complemented subspace (after all, it admits a continuous projection%
	\hair\footnote{Some authors require a complemented subspace to be closed, but this is automatic: if $P : E \to E$ is a continuous projection with range $F$, then $F = \ker(\id_E - P)$, so $F$ is closed.}).
If $E_+ = \{0\}$, then the top-retracts are precisely the complemented subspaces.

As far as we know, order retracts are not a very common notion, and have not received much attention. However, some special cases already play a role in the theory, such as \emph{projection bands} in Riesz spaces (see e.g.~\cite[\S 11]{Zaanen}) and \emph{projectionally exposed faces} in finite\-/dimensional cones (see e.g.~\cite{projectionally-exposed-faces}, \cite{projectionally-exposed-faces-ii}).

\subsection{Positive bilinear maps}
If $(E,E_+)$, $(F,F_+)$, $(G,G_+)$ are preordered vector spaces, then a bilinear map $b : E \times F \to G$ is called \index{positive bilinear map}\emph{positive} if $b(E_+,F_+) \subseteq G_+$.

In terms of the isomorphism $\SCBil(E_\weak \times F_\weak) \cong \CL(E_\weak , F_\weakstar')$ (see \mysecref{sec:tvs}), we note that a bilinear form $b \in \SCBil(E_\weak \times F_\weak)$ is positive if and only if $b(x, \,\cdot\,)$ defines a positive linear functional on $F$ for every $x \in E_+$, or equivalently, if and only if the corresponding map $E_\weak \to F_\weakstar'$ is positive. Thus, contrary to the topological setting, there is no difference between positive and ``separately positive'' bilinear forms.

\subsection{Faces and extremal rays}
Let $E$ be a vector space and let $E_+ \subseteq E$ be a convex cone.
A \index{face}\emph{face} (or \index{extremal set|see {face}}\emph{extremal set}) of $E_+$ is a (possibly empty) convex subset $M \subseteq E_+$ such that, if $M$ contains a point in the relative interior of a line segment in $E_+$, then $M$ also contains the endpoints of that segment. If $\varphi$ is a continuous positive linear functional, then $\ker(\varphi) \cap E_+$ is a face. Faces of this type are called \index{face!exposed}\emph{exposed}.

Every convex cone has a unique \index{face!minimal}minimal non-empty face (the lineality space $\lineal(E_+)$, contained in every face) and a unique \index{face!maximal}maximal face (the cone itself, containing every face). Note that $E_+$ is a proper cone if and only if $\{0\}$ is a face.

Let $x_0 \in E_+ \setminus \{0\}$. If $M := \{\lambda x_0 \, : \, \lambda \in \R_{\geq 0}\}$ is a face, then we say that $x_0$ is an \index{extremal direction}\emph{extremal direction}, and $M$ is an \index{extremal ray}\emph{extremal ray}. If $x_0$ is an extremal direction, then so is $\mu x_0$ for every $\mu > 0$. We let \glsadd{rext}$\rext(E_+) \subseteq E_+ \setminus \{0\}$ denote the set of all extremal directions of $E_+$.

If $M \subseteq E_+$ is a non-empty subset, then $E_+' \cap M^\perp$ defines a face of $E_+'$. Faces of this type are called \index{face!dual}\emph{dual faces}. In the finite\-/dimensional case, dual faces are precisely the exposed faces, but this is not true in locally convex spaces. For more on dual and exposed faces, see \mysecref{app:dual-and-exposed}.

\subsection{Order ideals}
Let $(E,E_+)$ be a preordered vector space.
A subspace $I \subseteq E$ is called an \index{order ideal}\emph{order ideal} if the pushforward of $E_+$ along the quotient map $E \to E/I$ is a proper cone. If no ambiguity can arise (i.e.~if the space does not carry a multiplication), then we call $I$ simply an \index{ideal|see {order ideal}}\emph{ideal}.

A subspace $I \subseteq E$ is an ideal if and only if $I \cap E_+$ is a face of $E_+$ (see \autoref{prop:ideal-definitions}).
Conversely, if $M \subseteq E_+$ is a face, then $\spn(M)$ is an ideal satisfying $\spn(M) \cap E_+ = M$ (see \autoref{prop:ideals-obtained}).
Thus, $I \mapsto I_+$ defines a many-to-one correspondence between ideals and faces. We shall draw heavily upon this correspondence.

If $\mywedge \subseteq E_+$ is a subcone, then every ideal $I \subseteq E$ with respect to $E_+$ is also an ideal with respect to $\mywedge$.
More generally, if $T : E \to F$ is a positive linear map and if $J \subseteq F$ is an ideal, then $T^{-1}[J] \subseteq E$ is also an ideal (see \autoref{prop:ideals-obtained}).
In particular, if $F_+$ is a proper cone, then $\{0\} \subseteq F_+$ is a face, so $\ker(T) \cap E_+$ is a face of $E_+$. It can be shown that all faces can be written in this form (see \myautoref{prop:ideal-face-kernel}{itm:face-kernel}).

We will show in \autoref{cor:maximal-ideals} that the maximal order ideals are precisely the kernels of non-zero positive linear functionals, or in other words, the supporting hyperplanes of $E_+$.
In particular, not every maximal ideal in a preordered topological vector space is closed. (Example: the kernel of a discontinuous positive linear functional.)

For more about ideals and faces, see \mysecref{app:faces} and \cite{Bonsall}.

\chapter{The projective cone}
\label{chp:projective}

In this chapter, we carry out an in-depth study of the properties of the projective cone.
This cone does not depend on any topological data, so we will mostly ignore topological issues in this chapter.
A few questions about the \emph{closure} of the projective cone will be discussed in \mychpref{chp:reasonable}.

Let $E,F$ be (real) vector spaces and let $E_+ \subseteq E$, $F_+ \subseteq F$ be convex cones. The simplest way to define a cone in ${E \tensor F}$ is to consider the \index{projective cone}\emph{projective cone}
\[ \glsadd{minwedge}\minwedge{E_+}{F_+} := \left\{\sum_{i=1}^k x_i \tensor y_i \ : \ k\in\N,\ x_1,\ldots,x_k\in E_+,\ y_1,\ldots,y_k\in F_+\right\}. \]
If $E,F$ are locally convex and if $\alpha$ is a compatible locally convex topology on $E \tensor F$, then we denote by $\topminwedge[\alpha]{E_+}{F_+}$ and \glsadd{hatminwedge}$\hatminwedge[\alpha]{E_+}{F_+}$ the same cone, but embedded in the topological vector spaces $E \tensor_\alpha F$ and $E \hattensor_\alpha F$, respectively. (The topology is denoted in the subscript; the cone in the superscript.)

It is easy to see that $\minwedge{E_+}{F_+}$ is indeed a convex cone. This cone has received a lot of attention in the literature; see for instance \cite{Merklen,Peressini-Sherbert,Grobler-Labuschagne,vanGaans-Kalauch,Wortel}.

In the subsequent sections, we will study the basic properties of the projective cone. We point out a characteristic property of the projective cone (\mysecref{sec:projective-characteristic-property}), study its mapping properties (\mysecref{sec:projective-positive-linear-maps}), prove precise necessary and sufficient conditions for $\minwedge{E_+}{F_+}$ to be proper (\mysecref{sec:projective-proper-cone}), and show that the projective tensor product of two faces is again a face (\mysecref{sec:projective-faces}, \mysecref{sec:projective-extremal-rays}).
Finally, as an application of the results from this section, we prove that the tensor product of absolutely convex sets also preserves faces (\mysecref{sec:application-to-convex-sets}).

\section{The characteristic property of the projective cone}
\label{sec:projective-characteristic-property}
Let $E,F,G$ be vector spaces equipped with convex cones $E_+ \subseteq E$, $F_+ \subseteq F$, $G_+ \subseteq G$. There is a natural isomorphism $\Bil(E \times F , G) \cong L(E \tensor F,G)$, which identifies a bilinear map $\Phi : E \times F \to G$ with its linearization $\Phi^L : E \tensor F \to G$, $\Phi^L(\sum_{i=1}^k x_i \tensor y_i) = \sum_{i=1}^k \Phi(x_i,y_i)$.
\begin{proposition}
	\label{prop:projective-property}
	If $E \tensor F$ is equipped with the projective cone $\minwedge{E_+}{F_+}$, then a linear map $\Phi^L : E \tensor F \to G$ is positive if and only if its corresponding bilinear map $\Phi : E \times F \to G$ is positive.
\end{proposition}
\begin{proof}
	A bilinear map $\Phi : E \times F \to G$ is positive if and only if $\Phi^L(x \tensor y) = \Phi(x,y) \geq 0$ for all $x \in E_+$, $y \in F_+$. On the other hand, since $\minwedge{E_+}{F_+}$ is generated by $E_+ \tensor F_+$, we also find that a linear map $\Phi^L : E \tensor F \to G$ is positive if and only if $\Phi^L(x \tensor y) \geq 0$ for all $x \in E_+$, $y \in F_+$.
\end{proof}

This is the ordered analogue of the characteristic property of the projective topology. It follows that
\begin{align*}
	(\minwedge{E_+}{F_+})\algdual &= \Bil(E \times F)_+  &&\text{($E,F$ vector spaces)};\\\noalign{\smallskip}
	(\topminwedge[\pi]{E_+}{F_+})\topdual &= (\hatminwedge[\pi]{E_+}{F_+})\topdual = \CBil(E \times F)_+  &&\text{($E,F$ locally convex)}.
\end{align*}

\section{Mapping properties of the projective cone}
\label{sec:projective-positive-linear-maps}
The projective norm preserves continuous linear maps, quotients, and complemented subspaces (see e.g.~\cite[Propositions 3.2, 3.8, and 3.9(1)]{Defant-Floret}, or \cite[\S 41.5]{Kothe-II} for the more general locally convex setting). The projective cone has analogous mapping properties.

\begin{proposition}
	\label{prop:projective-positive-linear-maps}
	Let $T \in \LL(E,G)$ and $S \in \LL(F,H)$.
	\begin{enumerate}[label=(\alph*)]
		\item\label{itm:projective-positive-maps} If $T[E_+] \subseteq G_+$ and $S[F_+] \subseteq H_+$, then \glsadd{T tensor S}$({T \tensor S})[\minwedge{E_+}{F_+}] \subseteq \minwedge{G_+}{H_+}$.
		\item\label{itm:projective-pushforwards} If $T[E_+] = G_+$ and $S[F_+] = H_+$, then $({T \tensor S})[\minwedge{E_+}{F_+}] = \minwedge{G_+}{H_+}$.
		\item\label{itm:projective-retracts} If $(E,E_+)$ and $(F,F_+)$ are retracts of $(G,G_+)$ and $(H,H_+)$, then $(E \tensor F,\minwedge{E_+}{F_+})$ is a retract of $(G \tensor H,\minwedge{E_+}{F_+})$.
	\end{enumerate}
	In summary: the projective cone preserves positive linear maps, pushforwards, and retracts.
\end{proposition}
It follows immediately that the same statements hold for maps between the completions (in the locally convex case), for the projective cone is contained in the algebraic tensor product.
\begin{proof}\ \par
	\begin{enumerate}[label=(\alph*)]
		\item Let $z \in \minwedge{E_+}{F_+}$ be given. Then we may write $z = \sum_{i=1}^k x_i \tensor y_i$ with $x_1,\ldots,x_k \in E_+$, $y_1,\ldots,y_k \in F_+$. Consequently, we have $({T \tensor S})(z) = \sum_{i=1}^k T(x_i) \tensor S(y_i)$, which lies in $\minwedge{G_+}{H_+}$ since $T(x_1),\ldots,T(x_k) \in G_+$, $S(y_1),\ldots,S(y_k) \in H_+$.

		\item By \ref{itm:projective-positive-maps}, ${T \tensor S}$ is positive. Now let $u \in \minwedge{G_+}{H_+}$ be given, and write $u = \sum_{i=1}^k v_i \tensor w_i$ with $v_1,\ldots,v_k \in G_+$ and $w_1,\ldots,w_k \in H_+$. By assumption there are $x_1,\ldots,x_k \in E_+$, $y_1,\ldots,y_k \in F_+$ such that $v_i = T(x_i)$ and $w_i = S(y_i)$, for all $i$. Consequently, we have $z := \sum_{i=1}^k x_i \tensor y_i \in \minwedge{E_+}{F_+}$, and $u = ({T \tensor S})(z)$.
		
		\item There are positive linear maps $T_1,T_2,S_1,S_2$ so that the following two diagrams commute:
		\begin{center}
			\begin{tikzcd}
				& G\arrow[dr, twoheadrightarrow, "T_2"] & && & H\arrow[dr, twoheadrightarrow, "S_2"] & \\
				E\arrow[ur, hook, "T_1"]\arrow[rr, "\id_E"] &  & E, && F\arrow[ur, hook, "S_1"]\arrow[rr, "\id_F"] &  & F.
			\end{tikzcd}
		\end{center}
		Consequently, the following diagram commutes:
		\begin{center}
			\begin{tikzcd}[row sep=large]
				& G \tensor H \arrow[dr, twoheadrightarrow, "T_2 \tensor S_2"] & \\
				E \tensor F \arrow[ur, hook, "T_1 \tensor S_1"]\arrow[rr, "\id_E \tensor \id_F"] & & E \tensor F.
			\end{tikzcd}
		\end{center}
		By \ref{itm:projective-positive-maps}, the maps in the preceding diagram are all positive for the projective cone, so it follows that $(E \tensor F, \minwedge{E_+}{F_+})$ is a retract of $(G \tensor H, \minwedge{G_+}{H_+})$. \qedhere
	\end{enumerate}
\end{proof}

Next, we prove that the projective tensor product also preserves approximate pushforwards: if $T$ and $S$ are maps whose adjoints are bipositive, then the same is true of $T \tensor S$.

\begin{lemma}
	\label{lem:tensor-product-of-approximate-pushforwards}
	Let $\langle E,E'\rangle$, $\langle F,F'\rangle$, $\langle G,G'\rangle$, $\langle H,H'\rangle$ be dual pairs, and let $E_+,F_+$, $G_+,H_+$ be convex cones in the primal spaces.
	If $T \in \CL(E_\weak , G_\weak)$ and $S \in \CL(F_\weak , H_\weak)$ are approximate pushforwards, then the map $(T \tensor S)' : \SCBil(G_\weak \times H_\weak) \to \SCBil(E_\weak \times F_\weak)$, $((T \tensor S)'b)(x,y) = b(Tx,Sy)$ is bipositive.
\end{lemma}
Here $(T \tensor S)'$ denotes the adjoint of $T \tensor S : E \tensor F \to G \tensor H$, assuming that $E \tensor F$ and $G \tensor H$ are equipped with the largest reasonable duals (see \mysecref{sec:dual-pairs}).
\begin{proof}
	Note that $(T \tensor S)'b$ is a positive bilinear functional on $E \times F$ if and only if $b$ is positive on $T[E_+] \times S[F_+]$, so if $b$ is separately weakly continuous, then this is the case if and only if $b$ is positive on $\overline{T[E_+]}^{\,\weak} \times \overline{S[F_+]}^{\,\weak}$. (First use weak continuity in the second variable to pass from $T[E_+] \times S[F_+]$ to $T[E_+] \times \overline{S[F_+]}^{\,\weak}$, then use weak continuity in the first variable to proceed to $\overline{T[E_+]}^{\,\weak} \times \overline{S[F_+]}^{\,\weak}$.) Analogously, $b$ itself is a positive bilinear functional on $G \times H$ if and only if $b$ is positive on $\overline{G_+}^{\,\weak} \times \overline{H_+}^{\,\weak}$.
	By assumption, we have $\overline{T[E_+]}^{\,\weak} = \overline{G_+}^{\,\weak}$ and $\overline{S[F_+]}^{\,\weak} = \overline{H_+}^{\,\weak}$, so it follows that $b$ is positive if and only if $(T \tensor S)'b$ is positive.
\end{proof}

The preceding lemma has immediate applications to algebraic tensor products (of vector spaces without topologies) and to completed locally convex topologies. It will also be used to prove one of the fundamental properties of the injective cone (see \myautoref{lem:injective-positive-linear-maps}{itm:injective-bipositive-maps}).

\begin{corollary}
	\label{cor:tensor-products-of-algebraic-approximate-pushforwards}
	Let $E,F,G,H$ be preordered vector spaces, and let $T \in \LL(E,G)$, $S \in \LL(F,H)$ be linear maps such that $T\algdual$ and $S\algdual$ are bipositive. Then $(T \tensor S)\algdual$ is bipositive with respect to the dual cones $(\minwedge{E_+}{F_+})\algdual \subseteq (E \tensor F)\algdual,(\minwedge{G_+}{H_+})\algdual \subseteq (G \tensor H)\algdual$.
\end{corollary}
\begin{proof}
	If we understand the primal spaces to belong to the dual pairs $\langle E, E\algdual\rangle$, \ldots, $\langle H,H\algdual\rangle$, then every linear map is weakly continuous. Furthermore, $(E \tensor F)\algdual = \Bil(E \times F) = \SCBil(E_\weak \times F_\weak)$, and the positive cone of $\SCBil(E_\weak \times F_\weak)_+$ coincides with the dual cone $(\minwedge{E_+}{F_+})\algdual \subseteq (E \tensor F)\algdual$, by \autoref{prop:projective-property}. Hence the result is a special case of \autoref{lem:tensor-product-of-approximate-pushforwards}.
\end{proof}

\begin{corollary}
	Let $E,F,G,H$ be locally convex preordered topological vector spaces and let $T \in \CL(E,G)$ and $S \in \CL(F,H)$ be approximate pushforwards. If $T \tensor_{\alpha\to\beta} S : E \tensor_\alpha F \to G \tensor_\beta H$ is continuous \textup($\alpha$ and $\beta$ compatible locally convex topologies\textup), then $T \tensor_{\alpha\to\beta} S$ and $T \hattensor_{\alpha\to\beta} S$ are approximate pushforwards.
\end{corollary}
\begin{proof}
	Every continuous linear map is also weakly continuous (see \cite[\S 20.4.(5)]{Kothe-I}), so we have $T \in \CL(E_\weak , G_\weak)$ and $S \in \CL(F_\weak , H_\weak)$. Furthermore, since $\alpha$ and $\beta$ are compatible topologies, we have $(E \tensor_\alpha F)\topdual \subseteq \SCBil(E \times F) = \SCBil(E_\weak \times F_\weak)$ and $(G \tensor_\beta H)\topdual \subseteq \SCBil(F \times H) = \SCBil(F_\weak \times H_\weak)$. It follows that $(T \tensor_{\alpha\to\beta} S)\topdual$ is a restriction of the map $(T \tensor S)'$ from \autoref{lem:tensor-product-of-approximate-pushforwards}, and therefore it is also bipositive. For the completion, note that $(T \hattensor_{\alpha\to\beta} S)\topdual = (T \tensor_{\alpha\to\beta} S)\topdual$.
\end{proof}

Interestingly, \autoref{cor:tensor-products-of-algebraic-approximate-pushforwards} uses topological techniques to prove a purely algebraic result.
We don't know a purely algebraic proof of \autoref{cor:tensor-products-of-algebraic-approximate-pushforwards}.

Finally, we show that the projective tensor product does not preserve bipositive maps, even if all spaces are finite\-/dimensional and all cones are closed and generating (\autoref{xmpl:projective-bipositive-fail-1}), or even closed, generating and proper (\autoref{xmpl:projective-bipositive-fail-2}).

\begin{example}
	\label{xmpl:projective-bipositive-fail-1}
	As a very simple example, let $F = G = \R^2$ with $F_+ = \R^2$ and $G_+ = \R_{\geq 0}^2$. Furthermore, let $E = \spn\{(1,-1)\} \subseteq G$, and write $E_+ := E \cap G_+ = \{0\}$. Then the inclusion $T : E \hookrightarrow G$ is bipositive, but $\minwedge{E_+}{F_+} = \{0\}$ whereas $\minwedge{G_+}{F_+} = G \tensor F$. Since $E \tensor F \neq \{0\}$, we have $(\minwedge{G_+}{F_+}) \cap (E \tensor F) \neq \minwedge{E_+}{F_+}$, which shows that $T \tensor \id_F$ is not bipositive.
\end{example}

\begin{example}[{Compare \cite[Situation 4]{Dobben-extensions}}]
	\label{xmpl:projective-bipositive-fail-2}
	For a more advanced example, let $E$ be a finite\-/dimensional space equipped with a proper, generating, polyhedral cone $E_+$ which is \emph{not} a simplex cone. Choose $\varphi_1,\ldots,\varphi_m \in E\algdual$ such that $E_+ = \bigcap_{i=1}^m \{x \in E \, : \, \varphi_i(x)\geq 0\}$, and let $\R^m$ be equipped with the standard cone $\R_{\geq 0}^m$. Then the map $T : E \to \R^m$, $x \mapsto (\varphi_1(x),\ldots,\varphi_m(x))$ is bipositive.
	
	Since $E_+$ is not a simplex cone, it follows from \cite[Proposition 3.1]{Barker-Loewy} (see also \autoref{thm:min-equals-max} below) that $\minwedge{E_+}{E_+\algdual} \neq \maxwedge{E_+}{E_+\algdual}$. On the other hand, it is well-known that $\minwedge{\R_{\geq 0}^m}{E_+\algdual} = \maxwedge{\R_{\geq 0}^m}{E_+\algdual}$, and it follows from \myautoref{thm:injective-positive-linear-maps:algebraic}{itm:algebraic-injective-bipositive-maps} below that $T \tensor \id_{E\algdual}$ is bipositive for the injective cone. Therefore:
	\[ (T \tensor \id_{E\algdual})^{-1}[\minwedge{\R_{\geq 0}^m}{E_+\algdual}] = (T \tensor \id_{E\algdual})^{-1}[\maxwedge{\R_{\geq 0}^m}{E_+\algdual}] = \maxwedge{E_+}{E_+\algdual} \neq \minwedge{E_+}{E_+\algdual}. \]
	This shows that $T \tensor \id_{E\algdual}$ is not bipositive for the projective cone.
	
	Note that all cones in this example are polyhedral, and therefore closed. In particular, the situation is not resolved by taking closures.
\end{example}

The finite\-/dimensional techniques used in \autoref{xmpl:projective-bipositive-fail-2} will be discussed in more detail in \mychpref{chp:finite-dimensional} and \mychpref{chp:many-examples}.

Despite the preceding counterexamples, bipositivity can be preserved under certain additional conditions. First, if $E \subseteq G$ and $F \subseteq H$ are retracts, then $E \tensor F \subseteq G \tensor H$ is also a retract (by \myautoref{prop:projective-positive-linear-maps}{itm:projective-retracts}), so in particular the inclusion $E \tensor F \hookrightarrow G \tensor H$ is bipositive. Furthermore, we prove in \autoref{prop:projective-bipositive-maps} that the projective cone also preserves ideals of proper cones bipositively.

\section{When is the projective cone proper?}
\label{sec:projective-proper-cone}

There is a simple necessary and sufficient condition for $\minwedge{E_+}{F_+}$ to be proper, which we prove in \autoref{thm:projective-proper-cone} below. This result was first proved (in three different ways) by Dermenjian and Saint-Raymond \cite{Dermenjian-Saint-Raymond}, and recently rediscovered by Wortel \cite{Wortel}. (The original proof seems to have been forgotten, and before Wortel only special cases were known in the literature.)
The proof given here is different from each of the existing proofs. Further methods of proof will be discussed in \autoref{rmk:alternative-proofs}.

We proceed via reduction to the finite\-/dimensional case, using the following lemmas.

\begin{lemma}
	\label{lem:generating-proper}
	A convex cone $E_+ \subseteq E$ is generating if and only if its algebraic dual cone $E_+\algdual$ is proper.
\end{lemma}
\begin{proof}
	Note that $E_+\algdual$ is \emph{not} proper if and only if there is some $\varphi \in E\algdual \setminus \{0\}$ such that both $\varphi$ and $-\varphi$ are positive linear functionals, or equivalently, $\varphi(x) = 0$ for all $x \in E_+$. This is in turn equivalent to $E_+$ being contained in a (linear) hyperplane, which happens if and only if $E_+$ is \emph{not} generating.
\end{proof}

\begin{corollary}
	\label{cor:proper-generating}
	If $E$ is finite\-/dimensional, then a closed convex cone $E_+ \subseteq E$ is proper if and only if its dual cone $E_+\algdual$ is generating.
\end{corollary}
\begin{proof}
	Set $F := E\algdual$ and $F_+ := E_+\algdual$. Under the canonical isomorphism $E \cong E\algdualdual$, we have $F_+\algdual = E_+$, by the bipolar theorem (here we use that $E_+$ is closed). The result follows from \autoref{lem:generating-proper}, applied to the cone $F_+ \subseteq F$.
\end{proof}

We are now ready to state and prove the main result of this section.
\begin{theorem}[{cf.~\cite{Dermenjian-Saint-Raymond}}]
	\label{thm:projective-proper-cone}
	Let $E$ and $F$ be vector spaces with convex cones $E_+ \subseteq E$, $F_+ \subseteq F$. Then the projective cone $\minwedge{E_+}{F_+}$ is proper if and only if $E_+ = \{0\}$, or $F_+ = \{0\}$, or both $E_+$ and $F_+$ are proper.
\end{theorem}
\begin{proof}
	Suppose first that $E_+,F_+ \neq \{0\}$ and $E_+$ is \emph{not} proper. Then we may choose $x \in E \setminus \{0\}$ such that $x,-x \in E_+$, and $y \in F_+ \setminus \{0\}$. Both $x \tensor y$ and $-x \tensor y$ belong to $\minwedge{E_+}{F_+}$, but we have $x \tensor y \neq 0$, so we see that $\minwedge{E_+}{F_+}$ is not a proper cone.
	
	For the converse, if $E_+ = \{0\}$, then $\minwedge{E_+}{F_+} = \{0\}$ regardless of any properties of $F_+$ (and similarly if $F_+ = \{0\}$). So assume now that both $E_+$ and $F_+$ are proper (not necessarily $\neq \{0\}$). Let $z \in E \tensor F$ be given such that $z,-z \in \minwedge{E_+}{F_+}$. Then we may choose integers $n \geq k \geq 0$ and vectors $x_1,\ldots,x_n \in E_+$, $y_1,\ldots,y_n \in F_+$ such that $z = \sum_{i=1}^k x_i \tensor y_i$ and $-z = \sum_{i=k+1}^n x_i \tensor y_i$. Consequently, we have $\sum_{i=1}^n x_i \tensor y_i = 0$.
	
	Now set $X := \spn(x_1,\ldots,x_n) \subseteq E$ and $Y := \spn(y_1,\ldots,y_n) \subseteq F$, and let $X_+ \subseteq X$ and $Y_+ \subseteq Y$ be the convex cones generated by $x_1,\ldots,x_n$ and $y_1,\ldots,y_n$, respectively. Note that $X_+$ is a closed proper cone in the finite\-/dimensional vector space $X$, since it is finitely generated (hence closed; see \cite[Lemma 3.19]{Aliprantis-Tourky}) and contained in the proper cone $X \cap E_+$ (hence also proper). Similarly, $Y_+$ is a closed proper cone in $Y$.
	
	It follows from \autoref{cor:proper-generating} that $X_+\algdual$ and $Y_+\algdual$ are generating cones in $X\algdual$ and $Y\algdual$, respectively. Therefore clearly $\minwedge{X_+\algdual}{Y_+\algdual}$ is generating in $X\algdual \tensor Y\algdual$.
	Since $\langle x \tensor y , \varphi \tensor \psi \rangle = \langle x , \varphi\rangle \langle y , \psi \rangle \geq 0$ for all $x\in X_+$, $y \in Y_+$, $\varphi \in X_+\algdual$, $\psi \in Y_+\algdual$, we have $\minwedge{X_+\algdual}{Y_+\algdual} \subseteq (\minwedge{X_+}{Y_+})\algdual$.
	It follows that $(\minwedge{X_+}{Y_+})\algdual$ is also generating, and therefore $(\minwedge{X_+}{Y_+})\algdualdual = \overline{\minwedge{X_+}{Y_+}}$ is a proper cone, by \autoref{lem:generating-proper}. Since $z,-z \in \minwedge{X_+}{Y_+} \subseteq (\minwedge{X_+}{Y_+})\algdualdual$, it follows that $z = 0$.
\end{proof}

\begin{remark}
	The final steps in the proof of \autoref{thm:projective-proper-cone} can be simplified with well-known results from the finite\-/dimensional theory, but we didn't need that. The dual of the projective cone $\minwedge{X_+}{Y_+}$ is the injective cone $\maxwedge{X_+\algdual}{Y_+\algdual}$, and $\minwedge{X_+}{Y_+}$ is automatically closed, by \cite{Tam-projective-closed} (see also \autoref{thm:projective-cone-closed} below).
\end{remark}

\begin{remark}
	\label{rmk:alternative-proofs}
	In the proof of \autoref{thm:projective-proper-cone}, we reduced the problem to finitely generated proper cones. There are many ways to prove this special case. Apart from the method used here and the proofs given in \cite{Dermenjian-Saint-Raymond} and \cite{Wortel}, we could also have applied either one of the sufficient criteria from \cite[Proposition 2.4]{Peressini-Sherbert}. Yet another method is mentioned in \autoref{rmk:another-way}.
\end{remark}

\autoref{thm:projective-proper-cone} will be extended in \autoref{cor:projective-lineality-space} below, where we determine the lineality space of $\minwedge{E_+}{F_+}$ for arbitrary convex cones $E_+$, $F_+$. Furthermore, a result very similar to \autoref{thm:projective-proper-cone}, giving criteria for $\minwedge{E_+}{F_+}$ to be semisimple (i.e.~contained in a weakly closed proper cone), will be given in \autoref{cor:projective-semisimple}.

\section{Faces of the projective cone}
\label{sec:projective-faces}
As a simple application of the theory developed so far, we develop two ways to combine faces of $E_+$ and $F_+$ to form a face of $\minwedge{E_+}{F_+}$.
For closed, proper and generating cones in finite-dimensional spaces, one of these constructions was already pointed out (without proof) by Tam in \cite[p.{} 53]{Tam-dissertation} and \cite[p.{} 71]{Tam-faces}.
He likely had a different proof in mind which does not work in general; see \autoref{rmk:Tams-proof}.

First we carry out the following very general construction; more convenient formulas and special cases will be studied afterwards.

\begin{theorem}
	\label{thm:projective-faces}
	Let $E,F$ be vector spaces, let $E_+ \subseteq E$, $F_+ \subseteq F$ be convex cones, and let $M \subseteq E_+$, $N \subseteq F_+$ be non-empty faces. Define
	\begin{align*}
		\glsadd{orface}\orface{M}{N} &:= (\minwedge{M}{F_+}) + (\minwedge{E_+}{N});\\\noalign{\smallskip}
		\glsadd{andface}\andface{M}{N} &:= (\minwedge{M}{N}) + (\minwedge{\lineal(E_+)}{F_+}) + (\minwedge{E_+}{\lineal(F_+)}).
	\end{align*}
	Then:
	\begin{enumerate}[label=(\alph*),series=projective-faces]
		\item\label{itm:projective-faces} $\orface{M}{N}$ and $\andface{M}{N}$ are faces of $\minwedge{E_+}{F_+}$.
		
		\item\label{itm:projective-sublattice} The face lattice of $\minwedge{E_+}{F_+}$ contains the following sublattice:
		\begin{center}
			\begin{tikzcd}
				& \orface{M}{N} \arrow[dl, no head] \arrow[dr, no head] & \\
				\orface{M}{\lineal(F_+)} = \andface{M}{F_+} \arrow[dr, no head] & & \orface{\lineal(E_+)}{N} = \andface{E_+}{N} \arrow[dl, no head] \\
				& \andface{M}{N} &
			\end{tikzcd}
		\end{center}
		Furthermore, $\orface{M}{N}$ is not just the face generated by $\andface{M}{F_+}$ and $\andface{E_+}{N}$, but even the sum of these faces, so we have
		\begin{align*}
			\orface{M}{N} &= (\orface{M}{\lineal(F_+)}) + (\orface{\lineal(E_+)}{N}) = (\andface{M}{F_+}) + (\andface{E_+}{N});\\\noalign{\smallskip}
			\andface{M}{N} &= (\orface{M}{\lineal(F_+)}) \cap (\orface{\lineal(E_+)}{N}) = (\andface{M}{F_+}) \cap (\andface{E_+}{N}).
		\end{align*}
	\end{enumerate}
	Assume furthermore that $E$, $F$, and $E \tensor F$ belong to the dual pairs $\langle E,E'\rangle$, $\langle F,F'\rangle$, and $\langle E \tensor F , G \rangle$, where $G$ is a reasonable dual \textup(i.e.~$E' \tensor F' \subseteq G \subseteq \SCBil(E \times F)$\textup). Then:
	\begin{enumerate}[resume*=projective-faces]
		\item\label{itm:projective-exposed-or} If $M$ and $N$ are dual \textup(resp.~exposed\textup) faces, then $\orface{M}{N}$ is a dual \textup(resp.~exposed\textup) face of $\minwedge{E_+}{F_+}$.
		\item\label{itm:projective-exposed-and} If $M$ and $N$ as well as $\lineal(E_+)$ and $\lineal(F_+)$ are dual \textup(resp.~exposed\textup) faces, then $\andface{M}{N}$ is a dual \textup(resp.~exposed\textup) face of $\minwedge{E_+}{F_+}$.
	\end{enumerate}
\end{theorem}

\noindent
A mnemonic for the chosen notation: $\orface{M}{N}$ is generated by the elements $x \tensor y \in E_+ \settensor F_+$ with $x \in M$ \underline{or} $y \in N$, whereas $\andface{M}{N}$ is generated by the elements $x \tensor y \in E_+ \settensor F_+$ with $x \in M$ \underline{and} $y \in N$, together with what turns out to be the lineality space of $\minwedge{E_+}{F_+}$ (see \autoref{cor:projective-faces} below).

\begin{proof}[Proof of \autoref{thm:projective-faces}]\ \par
	\begin{enumerate}[label=(\alph*)]
		\item Let $I \subseteq E$ be an order ideal such that $M = I \cap E_+$ (e.g.~$I = \spn(M)$; see \myautoref{prop:ideals-obtained}{itm:ideal-face}). Then the quotient cone $(E/I)_+ \subseteq E/I$ is proper, the natural map $\pi_I : E \to E/I$ is positive, and $M = \ker(\pi_I) \cap E_+$. Similarly, let $J \subseteq F$ be an ideal such that $N = J \cap F_+$; then $\pi_J : F \to F/J$ is a positive map to a space with a proper cone, and $N = \ker(\pi_J) \cap F_+$.
		
		Now consider the linear map $\pi_I \tensor \pi_J : {E \tensor F} \to {E/I \tensor F/J}$. It follows from \autoref{prop:projective-positive-linear-maps} that $\pi_I \tensor \pi_J$ is positive, and it follows from \autoref{thm:projective-proper-cone} that $\minwedge{(E/I)_+}{(F/J)_+}$ is a proper cone in ${E/I \tensor F/J}$, so $\ker({\pi_I \tensor \pi_J}) \cap (\minwedge{E_+}{F_+})$ is a face of $\minwedge{E_+}{F_+}$ (see \myautoref{prop:ideal-face-kernel}{itm:face-kernel}). We claim that
		\begin{equation}
			\ker({\pi_I \tensor \pi_J}) \cap (\minwedge{E_+}{F_+}) = \orface{M}{N}.\label{eqn:projective-face-claim}
		\end{equation}
		Indeed, if $z = \sum_{i=1}^k x_i \tensor y_i$ with $x_1,\ldots,x_k \in E_+$, $y_1,\ldots,y_k \in F_+$ is such that $(\pi_I \tensor \pi_J)(z) = 0$, then we must have $(\pi_I \tensor \pi_J)(x_i \tensor y_i) = 0$ for all $i$ (since $\minwedge{(E/I)_+}{(F/J)_+}$ is proper). As such, for each $i$ we must have $x_i \in \ker(\pi_I) = I$ or $y_i \in \ker(\pi_J) = J$, or possibly both. Equivalently: $x_i \in I \cap E_+ = M$ or $y_i \in J \cap F_+ = N$. This proves our claim \eqref{eqn:projective-face-claim}, and we conclude that $\orface{M}{N}$ is a face of $\minwedge{E_+}{F_+}$.
		
		To see that $\andface{M}{N}$ is a face, we proceed analogously, where the linear map $\pi_I \tensor \pi_J$ is replaced by the linear map
		\begin{align*}
			Q_{I,J} : E \tensor F &\to (E/I \tensor F/\lineal(F_+)) \oplus (E/\lineal(E_+) \tensor F/J),\\\noalign{\smallskip}
			x \tensor y &\mapsto (\pi_I(x) \tensor \pi_{\lineal(F_+)}(y)) \oplus (\pi_{\lineal(E_+)}(x) \tensor \pi_J(y)).
		\end{align*}
		If $z = \sum_{i=1}^k x_i \tensor y_i$ with $x_1,\ldots,x_k \in E_+$, $y_1,\ldots,y_k \in F_+$ and $Q_{I,J}(z) = 0$, then again we must have $Q_{I,J}(x_i \tensor y_i) = 0$ for all $i$ (since $Q_{I,J}$ is positive and the cone in the codomain is proper). Then either $x_i \in \ell(E_+) \subseteq M$, or $y_i \in \ell(F_+) \subseteq N$, or $x_i \notin \ell(E_+)$ \underline{and} $y_i \notin \ell(F_+)$. In the latter case, we must have $x_i \in M$ \underline{and} $y_i \in N$. This way we find
		\[ \ker(Q_{I,J}) \cap (\minwedge{E_+}{F_+}) = \andface{M}{N}. \]
		It follows that $\andface{M}{N}$ is also a face of $\minwedge{E_+}{F_+}$.

		\item Using the notation from the proof of \ref{itm:projective-faces}, note that
		\[ \ker(Q_{I,J}) = \ker(\pi_I \tensor \pi_{\lineal(F_+)}) \cap \ker(\pi_{\lineal(E_+)} \tensor \pi_J). \]
		It follows that
		\[ \andface{M}{N} = (\orface{M}{\lineal(F_+)}) \cap (\orface{\lineal(E_+)}{N}). \]
		The other formulas follow straight from the definitions: we have
		\begin{align*}
			(\orface{M}{\lineal(F_+)}) + (\orface{\lineal(E_+)}{N}) =& \, (\minwedge{M}{F_+}) + (\minwedge{E_+}{\lineal(F_+)}) \\
			&+ (\minwedge{\lineal(E_+)}{F_+}) + (\minwedge{E_+}{N})\\\noalign{\smallskip}
			=& \, (\minwedge{M}{F_+}) + (\minwedge{E_+}{N})\\\noalign{\smallskip}
			=& \, \orface{M}{N},
		\end{align*}
		since $\lineal(E_+) \subseteq M$ and $\lineal(F_+) \subseteq N$.
		Likewise,
		\begin{align*}
			\andface{M}{F_+} &= (\minwedge{M}{F_+}) + (\minwedge{\lineal(E_+)}{F_+}) + (\minwedge{E_+}{\lineal(F_+)})\\\noalign{\smallskip}
			&= (\minwedge{M}{F_+}) + (\minwedge{E_+}{\lineal(F_+)})\\\noalign{\smallskip}
			&= \orface{M}{\lineal(F_+)},
		\end{align*}
		and the formula $\andface{E_+}{N} = \orface{\lineal(E_+)}{N}$ follows analogously.

		\item If $M = {}^\predualface M_1$ and $N = {}^\predualface N_1$, then it is routinely verified that $\orface{M}{N} = {}^\predualface (M_1 \settensor N_1)$. If $M$ and $N$ are furthermore exposed, then we may take $M_1$ and $N_1$ to be singletons; consequently, $M_1 \settensor N_1$ is also a singleton.

		\item This follows from \ref{itm:projective-exposed-or} and the intersection formula from \ref{itm:projective-sublattice}. \qedhere
	\end{enumerate}
\end{proof}

\begin{remark}
	\label{rmk:lineality-space-dual-exposed-i}
	In \myautoref{thm:projective-faces}{itm:projective-exposed-and}, it is required that $\lineal(E_+)$ and $\lineal(F_+)$ are exposed/dual faces. Sometimes this is automatically the case. If $E_+$ is weakly closed, then $\lineal(E_+) = \lineal(\overline{E_+}^{\,\weak}) = {}^\perp (E_+') = {}^\predualface (E_+')$, so in this case $\lineal(E_+)$ is always a dual face. Likewise, if $E$ is a separable normed space and $E_+$ is closed, then $\lineal(E_+)$ is automatically exposed; see \autoref{cor:closed-lineality-space-exposed}.
	
	To see that this assumption cannot be omitted, let $E := \R^2$ with the lexicographical cone, and let $F := \R$ with the standard cone. Then the unique one-dimensional face $M \subseteq E_+$ and the trivial face $N := \{0\} \subseteq \R$ are both exposed (hence dual), but $\andface{M}{N} = \{0\}$ is neither exposed nor dual in $\minwedge{E_+}{F_+} \cong E_+$.
\end{remark}

\begin{remark}
	\label{rmk:not-every-facet}
	By dualizing the example from \autoref{xmpl:injective-extremal-rays} below, one can show that not every facet of $\minwedge{E_+}{F_+}$ is necessarily of the form $\orface{M}{N}$ or $\andface{M}{N}$. In follows that, in general, not every face of $\minwedge{E_+}{F_+}$ can be written as an intersection of faces of the type $\orface{M}{N}$ or $\andface{M}{N}$.
\end{remark}

We proceed to point out the consequences of \autoref{thm:projective-faces}. First of all, it allows us to extend \autoref{thm:projective-proper-cone}, giving a direct formula for the lineality space of $\minwedge{E_+}{F_+}$.

\begin{corollary}[The lineality space of the projective cone]
	\label{cor:projective-lineality-space}
	Let $E$ and $F$ be vector spaces, and let $E_+ \subseteq E$ and $F_+ \subseteq F$ be convex cones. Then one has
	\begin{align*}
		\lineal(\minwedge{E_+}{F_+}) &= (\minwedge{\lineal(E_+)}{F_+}) + (\minwedge{E_+}{\lineal(F_+)})\\\noalign{\smallskip}
		&= (\lineal(E_+) \tensor \spn(F_+)) + (\spn(E_+) \tensor \lineal(F_+)).
	\end{align*}
\end{corollary}
\begin{proof}
	If $x \in \lineal(E_+)$ and $y \in F_+$, then $\pm x \tensor y \in \minwedge{E_+}{F_+}$, so $x \tensor y \in \lineal(\minwedge{E_+}{F_+})$. Similarly, if $x \in E_+$ and $y \in \lineal(F_+)$, then $x \tensor y \in \lineal(\minwedge{E_+}{F_+})$, so we have
	\[ (\minwedge{\lineal(E_+)}{F_+}) + (\minwedge{E_+}{\lineal(F_+)}) \subseteq \lineal(\minwedge{E_+}{F_+}). \]
	Conversely, it follows from \myautoref{thm:projective-faces}{itm:projective-faces} that $\orface{\lineal(E_+)}{\lineal(F_+)} = (\minwedge{\lineal(E_+)}{F_+}) + (\minwedge{E_+}{\lineal(F_+)})$ is a face of $\minwedge{E_+}{F_+}$, so it must contain the minimal face $\lineal(\minwedge{E_+}{F_+})$. The first equality follows.
	
	For the second equality, we claim that $\minwedge{\lineal(E_+)}{F_+} = {\lineal(E_+) \tensor \spn(F_+)}$. Indeed, for $x \in \lineal(E_+)$ and $y \in \spn(F_+)$ we may write $y = u - v$ (for some $u,v\in F_+$), so we have $x \tensor y = (x \tensor u) + ((-x) \tensor v) \in \minwedge{E_+}{F_+}$. Taking positive linear combinations proves our claim. Analogously, we have $\minwedge{E_+}{\lineal(F_+)} = \spn(E_+) \tensor \lineal(F_+)$, and the second equality follows.
\end{proof}

This direct formula for the lineality space also simplifies the formula for the lower face $\andface{M}{N}$.

\begin{corollary}
	\label{cor:projective-faces}
	Let $E,F$ be vector spaces, let $E_+ \subseteq E$, $F_+ \subseteq F$ be convex cones, and let $M \subseteq E_+$, $N \subseteq F_+$ be non-empty faces. Then one has
	\[ \andface{M}{N} = (\minwedge{M}{N}) + \lineal(\minwedge{E_+}{F_+}), \]
	and this defines a face of $\minwedge{E_+}{F_+}$.
	
	In particular, if $E_+$ and $F_+$ are proper cones, then $\minwedge{M}{N}$ is a face of $\minwedge{E_+}{F_+}$, and the sublattice from \myautoref{thm:projective-faces}{itm:projective-sublattice} reduces to
	\begin{center}
		\begin{tikzcd}
			& (\minwedge{M}{F_+}) + (\minwedge{E_+}{N}) \arrow[dl, no head] \arrow[dr, no head] & \\
			\minwedge{M}{F_+} \arrow[dr, no head] & & \minwedge{E_+}{N} \arrow[dl, no head] \\
			& \minwedge{M}{N} &
		\end{tikzcd}
	\end{center}
\end{corollary}

For closed, proper and generating cones in finite-dimensional spaces, the fact that $\minwedge{M}{N}$ is a face of $\minwedge{E_+}{F_+}$ was already pointed out (without proof) by Tam in \cite[p.{} 53]{Tam-dissertation} and \cite[p.{} 71]{Tam-faces}.
He likely had a different proof in mind, which we outline in \autoref{rmk:Tams-proof} below.

\begin{remark}
	In general, $\minwedge{M}{N}$ is not a face of $\minwedge{E_+}{F_+}$.
	If $E_+$ or $F_+$ is not proper, then the term $+ \lineal(\minwedge{E_+}{F_+})$ cannot be omitted in \autoref{cor:projective-faces}.
	Indeed, suppose that $E_+$ is not proper.
	Choose $x \in \lineal(E_+) \setminus \{0\}$ and $y \in \spn(F_+) \setminus \spn(N)$.
	Then $x \tensor y \in \lineal(\minwedge{E_+}{F_+})$, by \autoref{cor:projective-lineality-space}.
	However, $x \tensor y \notin \minwedge{M}{N}$, so $\minwedge{M}{N}$ is not a face, because every face must contain the lineality space.
\end{remark}

\begin{remark}
	\label{rmk:Tams-proof}
	If $E_+\algdual$ and $F_+\algdual$ separate points on $E$ and $F$,%
		\hair\footnote{In other words, $E_+$ and $F_+$ are semisimple with respect to the dual pairs $\langle E, E\algdual \rangle$ and $\langle F, F\algdual \rangle$; see \cite{Dobben-semisimplicity}. Schaefer \cite{Schaefer-I} called such cones \emph{regular}.}
	then there is a simpler way to show that $\minwedge{M}{N}$ is a face of $\minwedge{E_+}{F_+}$.
	Indeed, let $z,z' \in \minwedge{E_+}{F_+}$ be such that $z'' := z + z' \in \minwedge{M}{N}$, and write $z = \sum_{i=1}^k x_i \tensor y_i$, where $x_1,\ldots,x_k \in E_+$ and $y_1,\ldots,y_k \in F_+$ are all non-zero.
	For $i \in \{1,\ldots,k\}$, choose $\varphi_i \in E_+\algdual$ and $\psi_i \in F_+\algdual$ such that $\varphi_i(x_i),\psi_i(y_i) > 0$.
	Then we have $0 < \varphi_i(x_i)y_i \leq \sum_{j=1}^k \varphi_i(x_j)y_j = (\varphi_i \tensor \id_F)(z) \leq (\varphi_i \tensor \id_F)(z'') \in N$, hence $y_i \in N$.
	Likewise, $0 < \psi_i(y_i)x_i \leq (\id_E \tensor \psi_i)(z'') \in M$, hence $x_i \in M$.
	It follows that $z \in \minwedge{M}{N}$, which shows that $\minwedge{M}{N}$ is a face.
	
	In particular, this simple proof settles the case when $E$ and $F$ are finite-dimensional and $E_+$ and $F_+$ are closed, proper and generating.
	This special case was already pointed out (without proof) by Tam in \cite[p.{} 53]{Tam-dissertation} and \cite[p.{} 71]{Tam-faces}.
	The proof he had in mind is probably similar to short proof given here.
\end{remark}

As a final application, we note that \autoref{thm:projective-faces} is also a statement about preservation of bipositive maps.

\begin{proposition}
	\label{prop:projective-bipositive-maps}
	Let $E$ and $F$ be vector spaces, and let $E_+ \subseteq E$, $F_+ \subseteq F$ be convex cones. If $E_+$ and $F_+$ are proper and if $I \subseteq E$, $J \subseteq F$ are ideals, then the inclusion $I \tensor J \hookrightarrow E \tensor F$ is bipositive \textup(with respect to the projective cone\textup).
\end{proposition}
\begin{proof}
	Let $Q_{I,J} : E \tensor F \to (E/I \tensor F) \oplus (E \tensor F/J)$ be the map from the proof of \myautoref{thm:projective-faces}{itm:projective-faces}.
	It follows from said proof (and \autoref{cor:projective-faces}) that $\minwedge{I_+}{J_+} = \ker(Q_{I,J}) \cap (\minwedge{E_+}{F_+})$.
	To complete the proof, note that $\ker(Q_{I,J}) = I \tensor J$.
\end{proof}

\autoref{xmpl:projective-bipositive-fail-1} shows that this is not true if one of the cones is not proper.

\section{Extremal rays of the projective cone}
\label{sec:projective-extremal-rays}
The results from \mysecref{sec:projective-faces} show us how to construct faces in the projective tensor cone, even though not all faces are reached this way (see \autoref{rmk:not-every-facet}). Nevertheless, it turns out that all extremal rays of $\minwedge{E_+}{F_+}$ are obtained in this way.

Recall that $\rext(E_+) \subseteq E_+ \setminus \{0\}$ denotes the set of extremal directions, and $M \settensor N$ denotes the entry-wise tensor product $\{x \tensor y \, : \, x\in M,\ y\in N\}$.

\begin{theorem}[The extremal rays of the projective cone]
	\label{thm:projective-extremal-rays}
	Let $E$, $F$ be vector spaces equipped with convex cones $E_+ \subseteq E$, $F_+ \subseteq F$. Then
	\[ \rext(\minwedge{E_+}{F_+}) = \rext(E_+) \settensor \rext(F_+). \]
\end{theorem}
\begin{proof}
	``$\subseteq$''. Suppose that $z \in (\minwedge{E_+}{F_+}) \setminus \{0\}$ defines an extremal ray. Write $z = \sum_{i=1}^k x_i \tensor y_i$ with $x_1,\ldots,x_k \in E_+$, $y_1,\ldots,y_k \in F_+$, and $x_i \tensor y_i \neq 0$ for all $i \in [k]$. By extremality of $z$ there are $\lambda_1,\ldots,\lambda_k \in \R_{>0}$ such that $\lambda_i x_i \tensor y_i = z$ ($i \in [k]$). In particular, $z = \lambda_1 x_1 \tensor y_1$. Now suppose that $0 \leq v \leq x_1$, then $0 \leq \lambda_1 v \tensor y_1 \leq z$, so by extremality of $z$ we must have $\mu\lambda_1 v \tensor y_1 = z$ for some $\mu \in \R_{\geq 0}$. Since $y_1 \neq 0$ and $\lambda_1 \neq 0$, it follows that $\mu v = x_1$, so we see that $x_1$ defines an extremal ray of $E_+$. Analogously, $y_1$ defines an extremal ray of $F_+$. This proves the inclusion $\rext(\minwedge{E_+}{F_+}) \subseteq \rext(E_+) \settensor \rext(F_+)$.
	
	``$\supseteq$''. Let $x_0 \in E_+\setminus \{0\}$ and $y_0 \in F_+ \setminus \{0\}$ define extremal rays in $E_+$ and $F_+$, respectively. Then $M := \{\lambda x_0 \, : \, \lambda \geq 0\}$ defines a face of $E_+$. Every face contains the lineality space, but $M$ does not contain a non-zero subspace, so it follows that $E_+$ is a proper cone. Analogously, $N := \{\mu y_0 \, : \, \mu \geq 0\}$ defines a face of $F_+$, so $F_+$ is proper. Now it follows from \autoref{cor:projective-faces} that $\minwedge{M}{N}$ is a face of $\minwedge{E_+}{F_+}$. In other words: $x_0 \tensor y_0$ defines an extremal ray of $\minwedge{E_+}{F_+}$.
\end{proof}

\begin{remark}
	Remarkably, \autoref{thm:projective-extremal-rays} has no corner cases: it is true for every pair of convex cones. In particular, if $\rext(E_+)$ or $\rext(F_+)$ is empty, then $\rext(\minwedge{E_+}{F_+})$ is empty as well. Conversely, if each of $E_+$ and $F_+$ has an extremal ray, then so does $\minwedge{E_+}{F_+}$.%
		\hair\footnote{It should be noted that many standard cones in infinite\-/dimensional spaces do not have sufficiently many extremal rays to generate the cone. For instance, the positive cone of $C[0,1]$ has no extremal rays at all.}
	
	Again, in the case where $E$ and $F$ are finite-dimensional and $E_+$ and $F_+$ are closed, proper and generating, this was already pointed out (without proof) by Tam in \cite[p.{} 53]{Tam-dissertation} and \cite[p.{} 71]{Tam-faces}.
	See also \autoref{rmk:Tams-proof}.
\end{remark}

\section{An application to tensor products of absolutely convex sets}
\label{sec:application-to-convex-sets}
We conclude our study of the projective cone with an application in convex geometry. Using a slight modification of the construction from \mysecref{sec:projective-faces}, we show that faces of absolutely convex sets $M$ and $N$ determine faces of their tensor product $\convtensor{M}{N} := \conv\{x \tensor y \, : \, x \in M,\ y\in N\}$.%
	\hair\footnote{Some authors define the projective tensor product of convex sets to be the \emph{closed} convex hull of $M \settensor N$ (e.g.~\cite[\S 4.1.4]{Aubrun-Szarek}), but our methods are not equipped to deal with closures. See also \autoref{rmk:absolutely-convex-closed}.}

This application is based on the following general principle, giving sufficient conditions for the sum of faces $\andface{M_1}{N_1}$ and $\andface{M_2}{N_2}$ (see \mysecref{sec:projective-faces}) to be another face in the projective cone $\minwedge{E_+}{F_+}$.
(This is a vast generalization of the method of \cite[Example 3.7]{Bogart-Contois-Gubeladze}.)

\begin{proposition}
	\label{prop:projective-faces-combined}
	Let $E$, $F$ be vector spaces, let $E_+ \subseteq E$, $F_+ \subseteq F$ be convex cones, and let $M_1,M_2 \subseteq E_+$ and $N_1,N_2 \subseteq F_+$ be faces. If $M_1 \cap M_2 = \lineal(E_+)$ and $N_1 \cap N_2 = \lineal(F_+)$, then
	\[ (\andface{M_1}{N_1}) + (\andface{M_2}{N_2}) \, = \, (\orface{M_1}{N_2}) \cap (\orface{M_2}{N_1}). \]
	In particular, in this case $(\andface{M_1}{N_1}) + (\andface{M_2}{N_2})$ is a face of $\minwedge{E_+}{F_+}$.
\end{proposition}
\begin{proof}
	``$\subseteq$''. It follows from \myautoref{thm:projective-faces}{itm:projective-sublattice} that $\andface{M_1}{N_1} \subseteq \andface{M_1}{F_+} = \orface{M_1}{\lineal(F_+)} \subseteq \orface{M_1}{N_2}$. Three analogous inclusions prove the forward inclusion.
	
	``$\supseteq$''. Let $z \in (\orface{M_1}{N_2}) \cap (\orface{M_2}{N_1})$, and write $z = \sum_{i=1}^k x_i \tensor y_i$ with $x_1,\ldots,x_k \in E_+$ and $y_1,\ldots,y_k \in F_+$.
	Since $z \in \orface{M_1}{N_2}$, it follows from the proof of \myautoref{thm:projective-faces}{itm:projective-faces} that for all $i$ we have $x_i \in M_1$ or $y_i \in N_2$, or possibly both. Likewise, for all $i$ we have $x_i \in M_2$ or $y_i \in N_1$, or possibly both.
	
	If $x_i \in \lineal(E_+)$ or $y_i \in \lineal(F_+)$, then $x_i \tensor y_i \in \lineal(\minwedge{E_+}{F_+}) \subseteq (\andface{M_1}{N_1}) \cap (\andface{M_2}{N_2})$, since every face contains the lineality space. So assume $x_i \notin \lineal(E_+)$ and $y_i \notin \lineal(F_+)$.
	Then, by assumption, $x_i$ (resp.~$y_i$) is contained in at most one of $M_1$ and $M_2$ (resp.~$N_1$ and $N_2$).
	Combined with earlier constraints, this show that we must either have $x_i \in M_1 \setminus M_2$ and $y_i \in N_1 \setminus N_2$, or otherwise $x_i \in M_2 \setminus M_1$ and $y_i \in N_2 \setminus N_1$. Either way, $x_i \tensor y_i \in (\andface{M_1}{N_1}) + (\andface{M_2}{N_2})$.
\end{proof}

If $E$ is a vector space and $C \subseteq E$ is a convex subset, then the \index{homogenization}\glsadd{homogenization}\emph{homogenization} $\homogen{C}$ of $C$ is the convex cone generated by $C \oplus \{1\} \subseteq E \oplus \R$. Note that $\homogen{C}$ is always a proper cone, and that the faces of $C$ are in bijective correspondence with the faces of $\homogen{C}$.

Since we are working over the real numbers, a convex set $C \subseteq E$ is absolutely convex if and only if $C = -C$. For sets of this kind, there is a simple way to identify the projective tensor product of the homogenizations $\homogen{C}$ and $\homogen{D}$ with the homogenization of $\conv(C \settensor D)$:

\begin{proposition}
	\label{prop:absolutely-convex-tensor-product}
	Let $E$ and $F$ be \textup(real\textup) vector spaces and let $C \subseteq E$, $D \subseteq F$ be absolutely convex sets. Under the natural isomorphism $(E \oplus \R) \tensor (F \oplus \R) = (E \tensor F) \oplus E \oplus F \oplus \R$, one has
	\[ (\minwedge{\homogen{C}}{\homogen{D}}) \cap \big((E \tensor F) \oplus \{0\} \oplus \{0\} \oplus \{1\}\big) \, = \, \{(z,0,0,1) \, : z \in \conv(C \settensor D)\}. \]
\end{proposition}
\begin{proof}
	Under the aforementioned natural isomorphism, we have $(x,\lambda) \tensor (y,\mu) = (x \tensor y , \mu x , \lambda y , \lambda\mu)$.
	
	``$\subseteq$''. Let $(z,0,0,1) \in \minwedge{\homogen{C}}{\homogen{D}}$ be given, and write $(z,0,0,1) = \sum_{i=1}^k \lambda_i\cdot (x_i,1) \tensor (y_i,1)$ with $\lambda_1,\ldots,\lambda_k \geq 0$, $x_1,\ldots,x_k\in C$ and $y_1,\ldots,y_k \in D$. Then $(z,0,0,1) = \sum_{i=1}^k \lambda_i \cdot (x_i \tensor y_i,$ $x_i , y_i , 1)$, so we have $\sum_{i=1}^k \lambda_i = 1$ and $z = \sum_{i=1}^k \lambda_i x_i \tensor y_i \in \conv(C \settensor D)$.
	
	``$\supseteq$''. Let $z \in \conv(C \settensor D)$ be given, and write $z = \sum_{i=1}^k \lambda_i x_i \tensor y_i$ with $x_1,\ldots,x_k \in C$, $y_1,\ldots,y_k \in D$, $\lambda_1,\ldots,\lambda_k \geq 0$, and $\sum_{i=1}^k \lambda_i = 1$.
	Since $(x_i,1) \tensor (y_i,1) \, + \, (-x_i,1) \tensor (-y_i,1) \, = \, 2(x_i \tensor y_i , 0 , 0 , 1)$, we may write
	\begin{equation}
		(z,0,0,1) \, = \, \sum_{i=1}^k \tfrac{1}{2}\lambda_i \cdot \big( (x_i,1) \tensor (y_i,1) \, + \, (-x_i,1) \tensor (-y_i,1) \big). \label{eqn:absolutely-convex-tensor-product}
	\end{equation}
	Since $C$ and $D$ are absolutely convex, we have $(\pm x_i,1) \in \homogen{C}$ and $(\pm y_i,1) \in \homogen{D}$ for all $i \in \{1,\ldots,k\}$, hence $(z,0,0,1) \in \minwedge{\homogen{C}}{\homogen{D}}$.
\end{proof}

\begin{theorem}
	\label{thm:absolutely-convex-faces}
	Let $E$ and $F$ be \textup(real\textup) vector spaces, let $C \subseteq E$, $D \subseteq F$ be absolutely convex, and let $M \subset C$, $N \subset D$ be proper faces. Then $\conv(M \settensor N)$ is a face of $\conv(C \settensor D)$.
\end{theorem}
\begin{proof}
	By symmetry, $-M \subseteq C$ and $-N \subseteq D$ also define faces of $C$ and $D$. First we prove that $M \cap -M = \varnothing$. Suppose that $x \in M \cap -M$. Then also $-x \in M \cap -M$, so by convexity $0 \in M \cap -M$. But then for every $y \in C$ we must have $y,-y \in M$, since $0$ belongs to the relative interior of the line segment joining $y$ and $-y$. This contradicts our assumption that $M$ is a proper face, so we conclude that $M \cap -M = \varnothing$. Analogously, $N \cap -N = \varnothing$.
	
	Let $M_1 \subseteq \homogen{C}$ be the face of $\homogen{C}$ associated with $M$, and let $M_2 \subseteq \homogen{C}$ be the face associated with $-M$. Since $M \cap -M = \varnothing$, it follows that $M_1 \cap M_2 = \{0\}$. Similarly, let $N_1$ and $N_2$ be the faces of $\homogen{D}$ associated with $N$ and $-N$, respectively; then $N_1 \cap N_2 = \{0\}$.
	Hence it follows from \autoref{prop:projective-faces-combined} that $(\andface{M_1}{N_1}) + (\andface{M_2}{N_2})$ is a face of $\minwedge{\homogen{C}}{\homogen{D}}$. To complete the proof, we show that
	\begin{align*}
		\big((\andface{M_1}{N_1}) + (\andface{M_2}{N_2})\big) &\cap \big((E \tensor F) \oplus \{0\} \oplus \{0\} \oplus \{1\}\big) \\
		&\qquad = \, \{(z,0,0,1) \, : z \in \conv(M \settensor N)\}.
	\end{align*}
	We proceed analogously to the proof of \autoref{prop:absolutely-convex-tensor-product}.
	
	``$\subseteq$''. Let $(z,0,0,1) \in (\andface{M_1}{N_1}) + (\andface{M_2}{N_2})$ be given. Then we may choose integers $n \geq k \geq 0$, scalars $\lambda_1,\ldots,\lambda_n \geq 0$ and vectors $x_1,\ldots,x_n \in M$, $y_1,\ldots,y_n \in N$ such that $(z,0,0,1) = \sum_{i=1}^k \lambda_i \cdot (x_i,1) \tensor (y_i,1) + \sum_{i=k+1}^n \lambda_i \cdot (-x_i,1) \tensor (-y_i,1)$. Therefore $\sum_{i=1}^n \lambda_i = 1$ and $z = \sum_{i=1}^n \lambda_i x_i \tensor y_i$, which shows that $z \in \conv(M \settensor N)$.
	
	``$\supseteq$''. Let $z \in \conv(M \settensor N)$ be given, and write $z = \sum_{i=1}^k \lambda_i x_i \tensor y_i$ with $x_1,\ldots,x_k \in M$, $y_1,\ldots,y_k \in N$, $\lambda_1,\ldots,\lambda_k \geq 0$, and $\sum_{i=1}^k \lambda_i = 1$.
	Then it follows from \eqref{eqn:absolutely-convex-tensor-product} that $(z,0,0,1) \in (\andface{M_1}{N_1}) + (\andface{M_2}{N_2})$.
\end{proof}

\begin{corollary}
	\label{cor:absolutely-convex-extreme-points}
	Let $E$ and $F$ be \textup(real\textup) vector spaces, let $C \subseteq E$, $D \subseteq F$ be absolutely convex, and let $x_0 \in C$, $y_0 \in D$ be extreme points. Then $x_0 \tensor y_0$ is an extreme point of $\conv(C \settensor D)$.
\end{corollary}

\begin{remark}
	\autoref{thm:absolutely-convex-faces} fails if one of the faces is not proper. Indeed, if $M = C$, then $0 \in M \settensor N$, so now $\conv(M \settensor N)$ is a face only if $\conv(M \settensor N) = \conv(C \settensor D)$.
	
	Furthermore, \autoref{thm:absolutely-convex-faces} and \autoref{cor:absolutely-convex-extreme-points} do not hold for non-symmetric convex sets. (Example: $1 \tensor 2$ is not an extreme point of $\conv([-1,1] \settensor [2,3]) \subseteq \R \tensor \R = \R$.)
\end{remark}

\begin{remark}
	\label{rmk:absolutely-convex-closed}
	In many applications it is natural to start with \emph{closed} absolutely convex sets, and take the \emph{closed} convex hull of their tensor product (e.g.~\cite[Remark 3.19]{Paulsen-Todorov-Tomforde}, \cite[\S 4.1.4]{Aubrun-Szarek}, or when computing the closed unit ball of the projective norm). Our methods are not equipped to deal with closures.
	
	If $E$, $F$ are finite\-/dimensional and if $C$, $D$ are compact, then $\conv(C \settensor D)$ is automatically compact, so here taking closures is not necessary. In particular:
\end{remark}
	
\begin{corollary}
	\label{cor:projective-norm-faces}
	Let $E$ and $F$ be \textup(real\textup) finite\-/dimensional normed spaces. Then the closed unit ball of the projective norm preserves proper faces: if $M \subset B_E$, $N \subset B_F$ are proper faces, then $\conv(M \settensor N)$ is a face of $B_{E \tensor_\pi F}$.
\end{corollary}

This had already been known for extreme points. More generally, if $E$ and $F$ are Banach spaces, then it follows from a result of Tseitlin \cite{Tseitlin} (see also \cite{Ruess-Stegall-extreme}) that the closed unit ball of the \emph{completed} projective tensor product $E\topdual \hattensor_\pi F\topdual$ preserves extreme points, provided that $E\topdual$ or $F\topdual$ has the approximation property and $E\topdual$ or $F\topdual$ has the Radon--Nikodym property.%
	\hair\footnote{The cited results relate to extreme points in duals of operator spaces. Our assumptions on $E\topdual$ and $F\topdual$ ensure that $E\topdual \hattensor_\pi F\topdual \cong (E \hattensor_\varepsilon F)\topdual$ isometrically; see \cite[Theorem 16.6]{Defant-Floret}.}
In particular, this settles the finite\-/dimensional case, proving \autoref{cor:projective-norm-faces} for extreme points.

\begin{remark}
	\label{rmk:projective-norm-extreme-points}
	We do not know whether the closed unit ball of the projective norm always preserves extreme points, even in the algebraic tensor product.
	This does not follow from \autoref{cor:absolutely-convex-extreme-points}, because the closed unit ball of $E \tensor_\pi F$ is the \emph{closure} of $\conv(B_E \settensor B_F)$.
	Known results in this direction usually start with something stronger than an extreme point, such as a \emph{denting point} (see \cite[Theorem~5]{Ruess-Stegall-weak*-denting}, \cite[Corollary 4]{Werner}).
	
	We suspect that there are Banach spaces $E$ and $F$ such that the projective norm does not preserve all extreme points of their closed unit balls, but we have not been able to construct such an example.
	To our knowledge, no such examples are known in the literature either.
	
	Finally, we should point out that the injective norm does not preserve extreme points; see \autoref{rmk:injective-norm-extreme-points}.
\end{remark}

\chapter{The injective cone}
\label{chp:injective}

In this chapter, we carry out an in-depth study of the properties of the injective cone.
This cone depends not only on the vector spaces $E,F$ and the cones $E_+,F_+$, but also on the dual spaces $E',F'$, so we will work with dual pairs.

Let $\langle E,E'\rangle$, $\langle F,F'\rangle$ be dual pairs of (real) vector spaces, and let $E_+ \subseteq E$, $F_+ \subseteq F$ be convex cones in the primal spaces. The \index{injective cone}\emph{injective cone}%
	\footnote{A note about terminology: in the literature, $\maxwedge{E_+}{F_+}$ is usually called the \emph{biprojective cone} (see e.g.~\cite{Merklen,Peressini-Sherbert,Birnbaum}). The results in this section show that this cone is in many ways analogous to the injective topology, and as such deserves the name \emph{injective cone}. The only prior use of this name (that we are aware of) is in \cite{Wittstock} and \cite{Mulansky}.}
in $E \tensor F$ is defined as
\[ \glsadd{maxwedge}\maxwedge{E_+}{F_+} := \big\{u \in E \tensor F \, : \, \langle u , \varphi \tensor \psi \rangle \geq 0\ \text{for all $\varphi \in E_+'$, $\psi \in F_+'$}\big\}. \]
The notation causes some ambiguity, because $\maxwedge{E_+}{F_+}$ does not only depend on $E_+$ and $F_+$, but also on the dual pairs $\langle E,E'\rangle$ and $\langle F,F'\rangle$. To be fully precise, the injective cone should be denoted as something like $\maxwedge{(\langle E,E'\rangle,E_+)}{(\langle F,F'\rangle,F_+)}$. We forego this cumbersome notation for the sake of clarity; it will always be clear what is meant.

If $E$ and $F$ are locally convex and if $E \tensor F$ is equipped with a compatible topology $\alpha$ (in the sense of Grothendieck \cite[p.~89]{Grothendieck-top}, see also \cite[\S 44.1]{Kothe-II}), then for every $\varphi \in E\topdual$, $\psi \in F\topdual$ the tensor product $\varphi \tensor \psi : E \tensor_\alpha F \to \R$ is continuous, and as such has a unique extension to $E \hattensor_\alpha F$. In this setting we may likewise define the injective cone as
\[ \glsadd{hatmaxwedge}\hatmaxwedge[\alpha]{E_+}{F_+} := \big\{ u \in E \hattensor_\alpha F \, : \, (\varphi \hattensor_\alpha \psi)(u) \geq 0\ \text{for all $\varphi \in E_+\topdual$, $\psi \in F_+\topdual$}\big\}. \]
Clearly $\maxwedge{E_+}{F_+} = (\hatmaxwedge[\alpha]{E_+}{F_+}) \cap (E \tensor F)$. Note that, unlike the projective cone, the injective cone typically becomes larger when passing from the algebraic tensor product $E \tensor F$ to the completion $E \hattensor_\alpha F$.

\begin{remark}
	\label{rmk:max-is-dual-of-min}
	Let $G$ be any reasonable dual of $E \tensor F$ (cf.~\autopageref{subsec:reasonable-dual}).
	It is clear from the definition that $\maxwedge{E_+}{F_+}$ is the predual cone of $\minwedge{E_+'}{F_+'}$ under the dual pairing $\langle E \tensor F , G \rangle$.
	Likewise, $\hatmaxwedge[\alpha]{E_+}{F_+} \subseteq E \hattensor_\alpha F$ is the predual cone of $\minwedge{E_+\topdual}{F_+\topdual} \subseteq (E \hattensor_\alpha F)\topdual$.
	
	An immediate consequence is that the injective cone is always weakly closed. Furthermore, by the bipolar theorem, the dual cone of $\maxwedge{E_+}{F_+}$ with respect to the dual pair $\langle E \tensor F , G\rangle$ is the $\sigma(G , E \tensor F)$-closure of $\minwedge{E_+'}{F_+'}$. (Note that this need not be contained in $E' \tensor F'$.) Similarly, in the locally convex setting, the dual cone of $\hatmaxwedge[\alpha]{E_+}{F_+}$ is the weak\nobreakdash-$*$ closure of $\minwedge{E_+\topdual}{F_+\topdual} \subseteq (E \hattensor_\alpha F)\topdual$.
\end{remark}

In this chapter, we give a detailed study of the properties of the injective cone.
In \mysecref{sec:injective-characteristic-property}, we establish the characteristic property of the injective cone.
In \mysecref{sec:injective-positive-linear-maps}, we show that the injective cone preserves positive maps, bipositive maps, and retracts, but not pushforwards.
In \mysecref{sec:injective-proper-cone} we determine necessary and sufficient conditions for the injective cone to be proper.
Finally, in \mysecref{sec:injective-faces}--\mysecref{sec:injective-extremal-rays} we show how faces in $E_+$ and $F_+$ determine faces of $\maxwedge{E_+}{F_+}$.

\section{The characteristic property of the injective cone}
\label{sec:injective-characteristic-property}

We show that the injective cone can be identified with a cone of positive bilinear forms. Let $E \SCBiltensor F$ denote the space of separately weak\nobreakdash-$*$ continuous bilinear forms on $E' \times F'$:
\[ \glsadd{SCBiltensor}E \SCBiltensor F \, := \, \SCBil\big(E_\weakstar' \times F_\weakstar'). \]
(K\"othe \cite[\S 44.4]{Kothe-II} uses the symbol $\boxtimes$ instead of $\SCBiltensor$.)

We shall understand $E \SCBiltensor F$ to be equipped with the cone it inherits from $\Bil(E' \times F')$. In other words, $b \in E \SCBiltensor F$ is positive if and only if $b(\varphi,\psi) \geq 0$ for all $\varphi \in E_+'$, $\psi \in F_+'$.

The characteristic property of the injective cone is that it is given by a bipositive map to $E \SCBiltensor F$ (algebraic case) or $\tilde E \SCBiltensor \tilde F$ (completed locally convex case).

\begin{remark}
	\label{rmk:injective-property}
	Statements about positive bilinear forms can be turned into equivalent statements about positive linear operators in the following way. Recall that $\SCBil({E_\weakstar' \times F_\weakstar'})$ is naturally isomorphic to $\CL(E_\weakstar' , F_\weak)$. Under this correspondence, the positive cone of $\SCBil(E_\weakstar' \times F_\weakstar')$ is the cone of \index{approximately positive linear map}\emph{approximately positive} operators $E_\weakstar' \to F_\weak$, i.e.~those operators $T$ that satisfy $T[E_+'] \subseteq \overline{F_+}^{\,\weak}$. In particular, if $F_+$ is weakly closed, then this is just the cone of positive operators $E_\weakstar' \to F_\weak$. Similarly, $\SCBil(E_\weakstar' \times F_\weakstar') \cong \CL(F_\weakstar' , E_\weak)$, and the positive cone of $\SCBil(E_\weakstar' \times F_\weakstar')$ corresponds with the approximately positive cone of $\CL(F_\weakstar' , E_\weak)$.

	The advantage of sticking to bilinear forms is twofold: it keeps the theory symmetric in $E$ and $F$, and it avoids the nuisance of having to take the weak closure of $F_+$ (or $E_+$).
\end{remark}

We proceed to prove the characteristic property in three settings: the algebraic tensor product, the completed injective tensor product, and arbitrary completed tensor products.

\subsection{Situation I: the algebraic tensor product}
Let $\langle E,E'\rangle$ and $\langle F,F'\rangle$ be dual pairs. Equip $E'$ and $F'$ with their respective weak\nobreakdash-$*$ topologies, and denote these spaces as $E_\weakstar'$ and $F_\weakstar'$. The dual pairing $\langle E \tensor F , E' \tensor F'\rangle$ yields a natural map $E \tensor F \hookrightarrow ({E' \tensor F'})\algdual \cong \Bil({E' \times F'})$. Note that the elements of ${E \tensor F}$ give rise to jointly continuous bilinear maps ${E_\weakstar' \times F_\weakstar'} \to \R$. Indeed, an elementary tensor $x_0 \tensor y_0 \in E \tensor F$ defines the bilinear map $(\varphi , \psi) \mapsto \langle x_0,\varphi\rangle\langle y_0,\psi\rangle$, which is easily seen to be jointly continuous (use that $\varphi \mapsto \langle x_0,\varphi\rangle$ and $\psi \mapsto \langle y_0,\psi\rangle$ are continuous). Consequently, finite sums of elementary tensors also define jointly continuous bilinear maps, and the claim follows. This gives us natural inclusions
\begin{equation}
	E \tensor F \subseteq \CBil(E_\weakstar' \times F_\weakstar') \subseteq E \SCBiltensor F \subseteq \Bil(E' \times F'). \label{eqn:bilinear-inclusions:algebraic}
\end{equation}
From left to right, these are the spaces of (continuous) finite rank, jointly continuous, separately continuous, and all bilinear forms on $E_\weakstar' \times F_\weakstar'$.

\begin{proposition}
	\label{prop:injective-property:algebraic}
	The elements of $\maxwedge{E_+}{F_+}$ are precisely those elements in $E \tensor F$ which define a positive bilinear map $E' \times F' \to \R$; that is:
	\[ \maxwedge{E_+}{F_+} = \CBil(E_\weakstar' \times F_\weakstar')_+ \cap (E \tensor F). \]
\end{proposition}
\begin{proof}
	By \autoref{rmk:max-is-dual-of-min}, $\maxwedge{E_+}{F_+}$ is the dual cone of $\minwedge{E_+'}{F_+'}$ with respect to the dual pair $\langle E \tensor F , E' \tensor F' \rangle$, so we have $\maxwedge{E_+}{F_+} = (E' \tensor F')_+\algdual \cap (E \tensor F)$. It follows from \autoref{prop:projective-property} that $u \in E \tensor F$ belongs to $\maxwedge{E_+}{F_+}$ if and only if $u$ defines a positive bilinear map $E' \times F' \to \R$.
\end{proof}
\begin{corollary}
	\label{cor:injective-property:algebraic}
	All inclusions in \eqref{eqn:bilinear-inclusions:algebraic} are bipositive.
\end{corollary}

\subsection{Situation II: injective topology, completed}
Let $E$ and $F$ be locally convex. Let $E \SCBiltensor_\varepsilon F$ denote the space $E \SCBiltensor F$ ($= \SCBil(E_\weakstar\topdual \times F_\weakstar\topdual)$) equipped with the \index{topology!bi-equicontinuous|see {injective}}\emph{bi-equicontinuous} (or \index{topology!injective}\emph{injective}) topology $\varepsilon$, that is, the locally convex topology given by the family of seminorms
\[ \big.p\big._{M,N}(b) = \sup_{\varphi\in M,\psi\in N} |b(\varphi,\psi)|,\qquad\text{($M \subseteq E\topdual$ and $N \subseteq F\topdual$ equicontinuous)}. \]
If $E$ and $F$ are complete, then $E \SCBiltensor_\varepsilon F$ is also complete (see \cite[\S 40.4.(5)]{Kothe-II}), so in this case we may identify $E \hattensor_\varepsilon F$ with the closure of $E \tensor_\varepsilon F$ in $E \SCBiltensor_\varepsilon F$, and we have the following inclusions of vector spaces:
\begin{equation}
	E \tensor F \subseteq E \hattensor_\varepsilon F \subseteq E \SCBiltensor_\varepsilon F \subseteq \Bil(E\topdual \times F\topdual),\qquad\text{($E$ and $F$ complete)}. \label{eqn:bilinear-inclusions:completed-injective-i}
\end{equation}
This may fail if $E$ or $F$ is not complete. (In particular, $E \tensor \R = E \SCBiltensor \R = E$, but $E \hattensor_\varepsilon \R = \tilde E$.) However, in general we have $E \hattensor_\varepsilon F = \tilde E \hattensor_\varepsilon \tilde F$ (see \cite[\S 44.5.(1)]{Kothe-II}), hence
\begin{equation}
	E \tensor F \subseteq E \hattensor_\varepsilon F = \tilde E \hattensor_\varepsilon \tilde F \subseteq \tilde E \SCBiltensor_\varepsilon \tilde F \subseteq \Bil(E\topdual \times F\topdual). \label{eqn:bilinear-inclusions:completed-injective-ii}
\end{equation}

\begin{proposition}
	\label{prop:injective-property:completed-injective}
	Let $E$, $F$ be locally convex. Then the natural inclusion $E \hattensor_\varepsilon F \hookrightarrow \tilde E \SCBiltensor_\varepsilon \tilde F$ is bipositive; that is:
	\[ \hatmaxwedge[\varepsilon]{E_+}{F_+} = \SCBil\left(E_{\sigma(E\topdual, \tilde E)}\topdual \times F_{\sigma(F\topdual , \tilde F)}\topdual\right)_+ \cap (E \hattensor_\varepsilon F). \]
\end{proposition}
\begin{proof}
	Continuous linear functionals $\varphi \in E\topdual$ and $\psi \in F\topdual$ define a functional on $E \hattensor_\varepsilon F$ in two different ways: either as the (unique) extension of $\varphi \tensor \psi$ to the completion $E \hattensor_\varepsilon F$, or as the restriction of the evaluation functional $f_{\varphi,\psi} : \Bil(E\topdual \times F\topdual) \to \R$, $b \mapsto b(\varphi,\psi)$ to the subspace $E \hattensor_\varepsilon F$. We claim that these two functionals coincide on $E \hattensor_\varepsilon F$. The inclusion $E \tensor F \hookrightarrow \Bil(E\topdual \times F\topdual)$ is such that $(\varphi \tensor \psi)(u) = u(\varphi , \psi)$, so the functionals coincide on $E \tensor F$. Furthermore, the functional $f_{\varphi,\psi}$ is easily seen to be continuous on $\tilde E \SCBiltensor_\varepsilon \tilde F$ (use that the sets $\{\varphi\} \subseteq E\topdual$, $\{\psi\} \subseteq F\topdual$ are equicontinuous). Hence $\varphi \tensor \psi = f_{\varphi,\psi}$ on $E \tensor F$, and by continuity also on $E \hattensor_\varepsilon F$, which proves our claim.
	
	It follows from the claim and the definition of $\hatmaxwedge[\varepsilon]{E_+}{F_+}$ that an element $u \in E \hattensor_\varepsilon F$ belongs to $\hatmaxwedge[\varepsilon]{E_+}{F_+}$ if and only if it defines a positive bilinear form $E' \times F' \to \R$.
\end{proof}
\begin{corollary}
	\label{cor:injective-property:completed-injective}
	All inclusions in \eqref{eqn:bilinear-inclusions:completed-injective-i} and \eqref{eqn:bilinear-inclusions:completed-injective-ii} are bipositive.
\end{corollary}

We only needed the bi-equicontinuous topology on $E \SCBiltensor F$ for the proof of \autoref{prop:injective-property:completed-injective}. From here on out we can forget about it.

\subsection{Situation III: arbitrary compatible topology, completed}
Now let $\alpha$ be an arbitrary compatible topology on $E \tensor F$ ($E$ and $F$ locally convex). Since the injective topology is the weakest compatible topology, we have a natural map $E \hattensor_\alpha F \to E \hattensor_\varepsilon F$, so here the picture is as follows:
\begin{equation}
	E \tensor F \: \hookrightarrow \: E \hattensor_\alpha F \: \to \: E \hattensor_\varepsilon F \: \hookrightarrow \: \tilde E \SCBiltensor \tilde F \: \hookrightarrow \: \Bil(E\topdual \times F\topdual). \label{eqn:bilinear-inclusions:completed-compatible}
\end{equation}
The map $E \hattensor_\alpha F \to E \hattensor_\varepsilon F$ need not be injective (this is related to the approximation property; see e.g.~\cite[Theorem 5.6]{Defant-Floret}). However, it remains bipositive.

\begin{proposition}
	\label{prop:injective-property:completed-compatible}
	Let $E$, $F$ be locally convex, and let $\alpha$ be a compatible topology on $E \tensor F$. Then the natural map $\Phi_{\alpha \to \varepsilon} : E \hattensor_\alpha F \to E \hattensor_\varepsilon F$ is bipositive; that is:
	\[ \hatmaxwedge[\alpha]{E_+}{F_+} = \Phi_{\alpha \to \varepsilon}^{-1}[\hatmaxwedge[\varepsilon]{E_+}{F_+}]. \]
\end{proposition}
\begin{proof}
	Note that $\varphi \hattensor_\alpha \psi = (\varphi \hattensor_\varepsilon \psi) \circ \Phi_{\alpha \to \varepsilon}$, as they coincide on $E \tensor F$. Hence: $u \in \hatmaxwedge[\alpha]{E_+}{F_+}$ if and only if $\Phi_{\alpha \to \varepsilon}(u) \in \hatmaxwedge[\varepsilon]{E_+}{F_+}$.
\end{proof}
\begin{corollary}
	\label{cor:injective-property:completed-compatible}
	All maps in \eqref{eqn:bilinear-inclusions:completed-compatible} are bipositive.
\end{corollary}

\section{Mapping properties of the injective cone}
\label{sec:injective-positive-linear-maps}
We show that the injective cone preserves all positive maps, bipositive maps (provided the cones are closed), and retracts, and show that it fails to preserve quotients, pushforwards, and approximate pushforwards.

Let $\langle E,E'\rangle$, $\langle F,F'\rangle$, $\langle G,G'\rangle$, $\langle H,H'\rangle$ be dual pairs, equipped with convex cones $E_+,F_+,G_+,H_+$ in the primal spaces. Given $T \in \CL(E_\weak , G_\weak)$ and $S \in \CL(F_\weak , H_\weak)$, we define \glsadd{T Biltensor S}$T \Biltensor S : \Bil({E' \times F'}) \to \Bil({G' \times H'})$ by
\[ b \mapsto \Big( (\varphi,\psi) \mapsto b(T'\varphi , S'\psi)\Big), \]
where $T' \in \CL(G_\weakstar' , E_\weakstar')$, $S' \in \CL(H_\weakstar' , F_\weakstar')$ denote the respective adjoints.

Note that $(T \Biltensor S)b$ is separately weak\nobreakdash-$*$ continuous whenever $b$ is, so $T \Biltensor S$ restricts to a map \glsadd{T SCBiltensor S}$T \SCBiltensor S : E \SCBiltensor F \to G \SCBiltensor H$.
\begin{proposition}
	\label{prop:linear-maps:algebraic-diagram}
	\glsadd{T tensor S}The following diagram commutes:
	\begin{center}
		\begin{tikzcd}[row sep=large]
			E \tensor F \arrow[r, hook] \arrow[d, "T \tensor S"] & E \SCBiltensor F \arrow[r, hook] \arrow[d, "T \SCBiltensor S"] & \Bil(E' \times F') \arrow[d, "T \Biltensor S"] \\
			G \tensor H \arrow[r, hook] & G \SCBiltensor H \arrow[r, hook] & \Bil(G' \times H').
		\end{tikzcd}
	\end{center}
\end{proposition}
\begin{proof}
	The rightmost square commutes by definition ($T \SCBiltensor S$ is the restriction of $T \Biltensor S$).
	For the leftmost square, note that $x \tensor y \in E \tensor F$ defines the bilinear map $(\varphi,\psi) \mapsto \langle x , \varphi \rangle \langle y , \psi \rangle$, and $Tx \tensor Sy$ defines the bilinear map $(\varphi,\psi) \mapsto \langle Tx , \varphi \rangle \langle Sy , \psi \rangle = \langle x , T'\varphi \rangle \langle y , S'\varphi \rangle$.
\end{proof}
\begin{proposition}
	\label{prop:linear-maps:completions-diagram}
	If $E$, $F$, $G$, $H$ are locally convex, if $T \in \CL(E,G)$, $S \in \CL(F,H)$, and if $\alpha$ and $\beta$ are compatible topologies on $E \tensor F$ and $G \tensor H$ for which the map $T \tensor_{\alpha \to \beta} S : E \tensor_\alpha F \to G \tensor_\beta H$ is continuous, then the following diagram commutes:
	\begin{center}
		\begin{tikzcd}[row sep=large]
			E \tensor F \arrow[r, hook]\arrow[rr, hook, bend left=25] \arrow[d, "T \tensor S"] & E \hattensor_\alpha F \arrow[r] \arrow[d, "T \hattensor_{\alpha \to \beta} S"] & E \hattensor_\varepsilon F \arrow[r, hook] \arrow[d, "T \hattensor_\varepsilon S"] & \tilde E \SCBiltensor \tilde F \arrow[r, hook] \arrow[d, "\tilde T \SCBiltensor \tilde S"] & \Bil(E\topdual \times F\topdual) \arrow[d, "T \Biltensor S"] \\
			G \tensor H \arrow[r, hook]\arrow[rr, hook, bend right=25] & G \hattensor_\beta H \arrow[r] & G \hattensor_\varepsilon H \arrow[r, hook] & \tilde G \SCBiltensor \tilde H \arrow[r, hook] & \Bil(G\topdual \times H\topdual).
		\end{tikzcd}
	\end{center}
\end{proposition}
Here the horizontal maps are the ones from \eqref{eqn:bilinear-inclusions:completed-compatible}, which are bipositive by \autoref{cor:injective-property:completed-compatible}.
\begin{proof}
	The rightmost square commutes since $T \Biltensor S = \tilde T \Biltensor \tilde S$ (use that $T : E \to G$ and its completion $\tilde T : \tilde E \to \tilde G$ have the same adjoint $T\topdual = \tilde T\topdual : G\topdual \to E\topdual$), and $\tilde T \SCBiltensor \tilde S$ is a restriction of $\tilde T \Biltensor \tilde S$. (However, $\tilde T \SCBiltensor \tilde S \neq T \SCBiltensor S$, as the domain and codomain are different!)
	
	The other squares (and the triangles) commute because the respective compositions agree on the dense subspace $E \tensor F$ (or $G \tensor H$).
\end{proof}

\begin{lemma}
	\label{lem:injective-positive-linear-maps}
	Let $\langle E,E'\rangle$, $\langle F,F'\rangle$, $\langle G,G'\rangle$, $\langle H,H'\rangle$ be dual pairs, and let $T \in \CL(E_\weak , G_\weak)$ and $S \in \CL(F_\weak , H_\weak)$.
	\begin{enumerate}[label=(\alph*)]
		\item\label{itm:injective-positive-maps} If $T$ and $S$ are positive, then $T \Biltensor S$ is positive.
		\item\label{itm:injective-bipositive-maps}
		If $\overline{E_+}^{\,\weak} = T^{-1}[\overline{G_+}^{\,\weak}]$ and $\overline{F_+}^{\,\weak} = S^{-1}[\overline{H_+}^{\,\weak}]$ \textup(i.e.{} $T$ and $S$ are approximately bipositive\textup), then $T \SCBiltensor S$ is bipositive.
	\end{enumerate}
\end{lemma}
\begin{proof}\ \par
	\begin{enumerate}[label=(\alph*)]
		\item Let $b \in \Bil(E' \times F')$ be positive. If $\varphi \in G_+'$ and $\psi \in H_+'$, then $\varphi \circ T \in E_+'$ and $\psi \circ S \in F_+'$ (the composition of positive linear maps is positive), so $(T \Biltensor S)(b)(\varphi, \psi) \geq 0$. It follows that $(T \Biltensor S)(b)$ is a positive bilinear map on $G' \times H'$, so $T \Biltensor S$ is positive.
		
		\item By the duality between approximate pushforwards and approximate pullbacks (see \autopageref{p:approximate-duality}), the adjoints $T' \in \CL(G_\weakstar' , E_\weakstar')$ and $S' \in \CL(H_\weakstar' , F_\weakstar')$ are weak\nobreakdash-$*$ approximate pushforwards.
		Since $T \SCBiltensor S$ is precisely the map $(T' \tensor S')'$ from \autoref{lem:tensor-product-of-approximate-pushforwards}, it follows from said lemma that $T \SCBiltensor S$ is bipositive. \qedhere
	\end{enumerate}
\end{proof}
\begin{theorem}
	\label{thm:injective-positive-linear-maps:algebraic}
	Let $T \in \CL(E_\weak , G_\weak)$ and $S \in \CL(F_\weak , H_\weak)$.
	\begin{enumerate}[label=(\alph*)]
		\item\label{itm:algebraic-injective-positive-maps} If $T$ and $S$ are positive, then $(T \tensor S)[\maxwedge{E_+}{F_+}] \subseteq \maxwedge{G_+}{H_+}$.
		
		\item\label{itm:algebraic-injective-bipositive-maps}
		If $\overline{E_+}^{\,\weak} = T^{-1}[\overline{G_+}^{\,\weak}]$ and $\overline{F_+}^{\,\weak} = S^{-1}[\overline{H_+}^{\,\weak}]$ \textup(i.e.{} $T$ and $S$ are approximately bipositive\textup), then $\maxwedge{E_+}{F_+} = (T \tensor S)^{-1}[\maxwedge{G_+}{H_+}]$.
	\end{enumerate}
	In summary: the algebraic injective cone preserves continuous positive maps and \textup(continuous\footnotehack{By our definition, approximately bipositive maps are already required to be continuous.}\textup) approximately bipositive maps.
\end{theorem}
\begin{proof}
	All horizontal arrows in the diagram from \autoref{prop:linear-maps:algebraic-diagram} are bipositive (by \autoref{cor:injective-property:algebraic}), so \ref{itm:algebraic-injective-positive-maps} and \ref{itm:algebraic-injective-bipositive-maps} follow easily from \autoref{lem:injective-positive-linear-maps}. For the summary, recall from \autoref{rmk:max-is-dual-of-min} that $\maxwedge{E_+}{F_+}$ and $\maxwedge{G_+}{H_+}$ are weakly closed, so in \ref{itm:algebraic-injective-bipositive-maps} we find that $T \tensor S$ is approximately bipositive (in addition to being bipositive).
\end{proof}

\begin{theorem}
	\label{thm:injective-positive-linear-maps:completed}
	Let $E$, $F$, $G$, $H$ be locally convex, let $T \in \CL(E, G)$ and $S \in \CL(F, H)$, and let $\alpha$ and $\beta$ be compatible topologies on respectively $E \tensor F$ and $G \tensor H$ for which the map $T \tensor_{\alpha \to \beta} S : {E \tensor_\alpha F} \to {G \tensor_\beta H}$ is continuous.
	\begin{enumerate}[label=(\alph*)]
		\item\label{itm:completed-injective-positive-maps} If $T$ and $S$ are positive, then $(T \hattensor_{\alpha\to\beta} S)[\hatmaxwedge[\alpha]{E_+}{F_+}] \subseteq \hatmaxwedge[\beta]{G_+}{H_+}$.
		
		\item\label{itm:completed-injective-bipositive-maps} If $E$ and $F$ are complete and $\overline{E_+} = T^{-1}[\overline{G_+}]$ and $\overline{F_+} = S^{-1}[\overline{H_+}]$, then $\hatmaxwedge[\alpha]{E_+}{F_+} = (T \hattensor_{\alpha\to\beta} S)^{-1}[\hatmaxwedge[\beta]{G_+}{H_+}]$.
	\end{enumerate}
	In summary: the completed injective cone preserves continuous positive maps, and \textup(continuous\footnotemark\textup) approximately bipositive maps if $E$ and $F$ are complete.
\end{theorem}
\begin{proof}
	\ \par
	\begin{enumerate}[label=(\alph*)]
		\item All horizontal arrows in the diagram from \autoref{prop:linear-maps:completions-diagram} are bipositive (by \autoref{cor:injective-property:completed-compatible}), so the result follows from \myautoref{lem:injective-positive-linear-maps}{itm:injective-positive-maps}.
		
		\item Recall: in a locally convex space, the weak closure and original closure of a convex cone coincide. Moreover, note that we may assume without loss of generality that $G$ and $H$ are also complete. (Extend $T$ to the map $\tilde T : E \to \tilde G$, and let $\widetilde{G_+}$ denote the closure of $G_+$ in $\tilde G$. Then $\tilde T^{-1}[\widetilde{G_+}] = T^{-1}[\overline{G_+}]$, since $\ran(\tilde T) \subseteq G$.)
		
		We refer again to the diagram from \autoref{prop:linear-maps:completions-diagram}. All horizontal arrows in are bipositive, and the vertical arrow $T \SCBiltensor S = \tilde T \SCBiltensor \tilde S$ is bipositive by \myautoref{lem:injective-positive-linear-maps}{itm:injective-bipositive-maps}. The result is easily deduced. \qedhere
	\end{enumerate}
\end{proof}

\begin{remark}
	\label{rmk:injective-mapping-property}
	We get one of the characteristic properties of the injective topology for free: if $E$, $F$, $G$, $H$ are locally convex, $E$ and $F$ complete, and if $T \in \CL(E,G)$ and $S \in \CL(F,H)$ are injective, then so is $T \hattensor_\varepsilon S \in \CL(E \hattensor_\varepsilon F , G \hattensor_\varepsilon H)$. Indeed, equip all spaces with the trivial cone $\{0\}$, then every dual cone is the entire dual space, so $\Bil(E\topdual \times F\topdual)_+ = \{0\}$. Therefore $E \hattensor_\varepsilon F$ and $G \hattensor_\varepsilon H$ are also equipped with the zero cone (since $E \hattensor_\varepsilon F \to \Bil(E\topdual \times F\topdual)$ is bipositive and injective). Since $T \hattensor_\varepsilon S$ is bipositive, we have $(T \hattensor_\varepsilon S)^{-1}[\{0\}] = \{0\}$, so $T \hattensor_\varepsilon S$ is injective.
	
	This shows immediately that the completeness assumptions in \myautoref{thm:injective-positive-linear-maps:completed}{itm:completed-injective-bipositive-maps} cannot be omitted. (After all, $T \hattensor_\varepsilon \id_\R : E \hattensor_\varepsilon \R \to G \hattensor_\varepsilon \R$ is simply the completion $\tilde T : \tilde E \to \tilde G$, which may fail to be injective even if $T$ is injective.)
	
	A similar argument shows that the weak closures in \myautoref{lem:injective-positive-linear-maps}{itm:injective-bipositive-maps} and subsequent theorems cannot be omitted: the map $T \tensor_\varepsilon \id_\R : E \tensor \R \to G \tensor \R$ is simply $T$, but with the positive cones $E_+$, $G_+$ replaced by their weak closures. But one does not necessarily have $T^{-1}[\overline{G_+}^{\,\weak}] = \overline{E_+}^{\,\weak}$ whenever $T^{-1}[G_+] = E_+$. (Concrete example: let $G = \R^2$ with $G_+ = \{(x,y) \, : \, x > 0\} \cup \{(0,0)\}$, let $E := \spn\{(0,1)\} \subseteq G$ with $E_+ := G_+ \cap E$, and let $T$ be the inclusion $E \hookrightarrow G$.)
\end{remark}

\begin{remark}
	A topological order retract $G \subseteq E$ is given by two continuous positive linear maps $E \twoheadrightarrow G \hookrightarrow E$, so it follows at once that the injective cone (in all its incarnations) preserves all topological order retracts, without any assumptions on completeness or weak closures. The argument is analogous to that of \myautoref{prop:projective-positive-linear-maps}{itm:projective-retracts}.
\end{remark}

The following example shows that the injective cone does not preserve pushforwards, not even approximately.

\begin{example}[{Dual to \autoref{xmpl:projective-bipositive-fail-2}; cf.~\cite[Situation 4]{Dobben-extensions}}]
	\label{xmpl:injective-pushforward-fail-2}
	Let $E$ be a finite\-/dimensional space equipped with a proper, generating, polyhedral cone which is \emph{not} a simplex cone. Let $x_1,\ldots,x_m$ be representatives of the extremal rays of $E_+$, and let $\R^m$ be equipped with the standard cone $\R_{\geq 0}^m$. Then the map $T : \R^m \to E$, $(\lambda_1,\ldots,\lambda_m) \mapsto \lambda_1 x_1 + \ldots + \lambda_m x_m$ is a pushforward (i.e.~$T[\R_{\geq 0}^m] = E_+$).
	
	Since $E_+$ is not a simplex cone, it follows from \cite[Proposition 3.1]{Barker-Loewy} (see also \autoref{thm:min-equals-max} below) that $\minwedge{E_+}{E_+\algdual} \neq \maxwedge{E_+}{E_+\algdual}$. On the other hand, we have $\minwedge{\R_{\geq 0}^m}{E_+\algdual} = \maxwedge{\R_{\geq 0}^m}{E_+\algdual}$, and it follows from \myautoref{prop:projective-positive-linear-maps}{itm:projective-pushforwards} that $T \tensor \id_{E\algdual}$ is a pushforward for the projective cone. Therefore:
	\[ (T \tensor \id_{E\algdual})[\maxwedge{\R_{\geq 0}^m}{E_+\algdual}] = (T \tensor \id_{E\algdual})[\minwedge{\R_{\geq 0}^m}{E_+\algdual}] = \minwedge{E_+}{E_+\algdual} \neq \maxwedge{E_+}{E_+\algdual}. \]
	This shows that $T \tensor \id_{E\algdual}$ is not a pushforward for the injective cone.
	
	Note that all cones in this example are polyhedral, and therefore closed. In particular, the situation is not resolved by adding closures, which shows that the injective cone does not preserve approximate pushforwards.
\end{example}

The finite\-/dimensional techniques used in \autoref{xmpl:injective-pushforward-fail-2} will be discussed in more detail in \mychpref{chp:finite-dimensional} and \mychpref{chp:many-examples}.

\section{When is the injective cone proper?}
\label{sec:injective-proper-cone}
We determine the lineality space of $E \SCBiltensor F$, and we use this to give necessary and sufficient conditions for the injective cone (in all its incarnations) to be proper. Direct formulas for the lineality space (under certain topological assumptions) will be given in \mysecref{sec:injective-ideals}.

As before, let $\langle E,E'\rangle$, $\langle F,F'\rangle$ be dual pairs, equipped with convex cones $E_+ \subseteq E$, $F_+ \subseteq F$ in the primal spaces.

\begin{proposition}
	\label{prop:injective-lineality-space}
	The lineality space of $(E \SCBiltensor F)_+$ is the set of those bilinear forms in $E \SCBiltensor F$ that vanish on $\overline{\spn(E_+')}^{\,\weakstar} \times \overline{\spn(F_+')}^{\,\weakstar} = \lineal(\overline{E_+}^{\,\weak})^\perp \times \lineal(\overline{F_+}^{\,\weak})^\perp$.
\end{proposition}
\begin{proof}
	If $b \in E \SCBiltensor F$ vanishes on $\overline{\spn(E_+')}^{\,\weakstar} \times \overline{\spn(F_+')}^{\,\weakstar}$, then in particular it vanishes on $E_+' \times F_+'$, so evidently both $b$ and $-b$ define positive bilinear forms. Conversely, if $b \in \lineal((E \SCBiltensor F)_+)$, then both $b$ and $-b$ are positive on $E_+' \times F_+'$, so it follows that $b$ must vanish on $E_+' \times F_+'$. Therefore $b$ also vanishes on $\spn(E_+') \times \spn(F_+')$, and consequently on $\overline{\spn(E_+')}^{\,\weakstar} \times \overline{\spn(F_+')}^{\,\weakstar}$. (Use weak\nobreakdash-$*$ continuity in one variable at a time, as we did in the proof of \autoref{lem:tensor-product-of-approximate-pushforwards}.)
	
	Since $\lineal(\overline{E_+}^{\,\weak}) = {}^\perp (E_+')$ (see \mysecref{sec:convex-cones}), we have $\overline{\spn(E_+')}^{\,\weakstar} = \lineal(\overline{E_+}^{\,\weak})^\perp$.
\end{proof}

Direct formulas for the lineality space of the injective cone will be given in \myautoref{cor:tensor-subspace}{itm:ts:lin-space} (in $E \tensor F$) and \myautoref{cor:SCBiltensor-decomposition}{itm:SCBil-lin-space} (in $E \SCBiltensor F$).
For now, we focus on conditions for the injective cone to be proper.

\begin{theorem}
	\label{thm:injective-proper-cone}
	The following are equivalent:
	\begin{enumerate}[label=(\roman*)]
		\item\label{itm:injective-proper-cone} $\maxwedge{E_+}{F_+}$ is a proper cone;
		\item\label{itm:intermediate-subspace-proper-cone} For every subspace $E \tensor F \subseteq G \subseteq E \SCBiltensor F$, the cone $G_+ := G \cap (E \SCBiltensor F)_+$ is proper.
		\item\label{itm:SCBil-proper-cone} $(E \SCBiltensor F)_+$ is a proper cone;
		\item\label{itm:injective-criteria} $E = \{0\}$, or $F = \{0\}$, or both $\overline{E_+}^{\,\weak}$ and $\overline{F_+}^{\,\weak}$ are proper cones.
	\end{enumerate}
	In particular, the injective tensor product of weakly closed proper cones is a proper cone.
\end{theorem}
Note that the equivalence $\myref{itm:injective-proper-cone} \Longleftrightarrow \myref{itm:injective-criteria}$ is very similar to \autoref{thm:projective-proper-cone}. However, we should point out that the corner case is slightly different now. In \autoref{thm:projective-proper-cone}, the corner case is when one of the \emph{cones} is trivial; here the corner case is when one of the \emph{spaces} is trivial.

\begin{proof}[Proof of \autoref{thm:injective-proper-cone}]
	$\myref{itm:SCBil-proper-cone} \Longrightarrow \myref{itm:intermediate-subspace-proper-cone}$. Trivial.
	
	$\myref{itm:intermediate-subspace-proper-cone} \Longrightarrow \myref{itm:injective-proper-cone}$. Immediate, since $\maxwedge{E_+}{F_+} = (E \tensor F) \cap (E \SCBiltensor F)_+$.
	
	$\myref{itm:injective-criteria} \Longrightarrow \myref{itm:SCBil-proper-cone}$. If $E = \{0\}$, then clearly $E \SCBiltensor F = \{0\}$, so $(E \SCBiltensor F)_+$ is a proper cone regardless of any properties of $F_+$ (and similarly if $F = \{0\}$). If $\overline{E_+}^{\,\weak}$ and $\overline{F_+}^{\,\weak}$ are proper cones, then $\lineal(\overline{E_+}^{\,\weak}) = \lineal(\overline{F_+}^{\,\weak}) = \{0\}$, so it follows from \autoref{prop:injective-lineality-space} that $\lineal((E \SCBiltensor F)_+) = \{0\}$.
	
	$\myref{itm:injective-proper-cone} \Longrightarrow \myref{itm:injective-criteria}$. We prove the contrapositive: suppose that $E,F \neq \{0\}$ and that $\overline{E_+}^{\,\weak}$ is not a proper cone. Then we may choose $x \in E \setminus \{0\}$ with $\pm x \in \overline{E_+}^{\,\weak}$. Note that $(\overline{E_+}^{\,\weak})' = E_+'$, so for every $\varphi \in E_+'$ we have $\varphi(x),\varphi(-x) \geq 0$, and therefore $\varphi(x) = 0$. Now choose any $y \in F \setminus \{0\}$ (here we use that $F \neq \{0\}$), then for all $\varphi \in E_+'$, $\psi \in F_+'$ we have $\langle x \tensor y , \varphi \tensor \psi \rangle = \varphi(x)\psi(y) = 0\cdot \psi(y) = 0$, so we find $\pm x \tensor y \in \maxwedge{E_+}{F_+}$. Since $x$ and $y$ are non-zero, we have $x \tensor y \neq 0$, and we conclude that $\maxwedge{E_+}{F_+}$ fails to be proper.
\end{proof}

To tell whether $\hatmaxwedge[\alpha]{E_+}{F_+}$ is a proper cone, we need to assume that $E$ and $F$ are complete.
In the case where $E$ and $F$ are not complete, an answer can be found by first passing to the completions $\tilde E$, $\tilde F$.
See also \autoref{rmk:need-completeness} below.

\begin{corollary}
	\label{cor:completed-injective-proper-cone}
	Let $E,F$ be complete locally convex spaces, $E_+ \subseteq E$, $F_+ \subseteq F$ convex cones, and $\alpha$ a compatible locally convex topology on $E \tensor F$. Then the following are equivalent:
	\begin{enumerate}[label=(\roman*)]
		\item\label{itm:completed-injective-proper-cone} $\hatmaxwedge[\alpha]{E_+}{F_+} \subseteq E \hattensor_\alpha F$ is a proper cone;
		\item\label{itm:completed-injective-criteria} $E = \{0\}$, or $F = \{0\}$, or both $\overline{E_+}$ and $\overline{F_+}$ are proper cones and the natural map $E \hattensor_\alpha F \to E \hattensor_\varepsilon F$ is injective.
	\end{enumerate}
\end{corollary}
\begin{proof}
	First of all, recall that $\overline{E_+} = \overline{E_+}^{\,\weak}$, since $E_+$ is convex and $E$ is locally convex. Likewise, $\overline{F_+} = \overline{F_+}^{\,\weak}$.
	
	For the injective topology, recall from \eqref{eqn:bilinear-inclusions:completed-injective-i} that we have $E \tensor F \subseteq E \hattensor_\varepsilon F \subseteq  E \SCBiltensor F$, since $E$ and $F$ are complete. Hence for $\alpha = \varepsilon$ the result follows from \autoref{thm:injective-proper-cone}.
	
	For general $\alpha$, recall that $E \hattensor_\alpha F \to E \hattensor_\varepsilon F$ is bipositive. Therefore:
	
	$\myref{itm:completed-injective-proper-cone} \Longrightarrow \myref{itm:completed-injective-criteria}$. If $\hatmaxwedge[\alpha]{E_+}{F_+}$ is proper, then the bipositive map $E \hattensor_\alpha F \to E \hattensor_\varepsilon F$ is automatically injective (see \autoref{rmk:bipositive-injective}). Furthermore, the subcone $\maxwedge{E_+}{F_+} \subseteq \hatmaxwedge[\alpha]{E_+}{F_+}$ is also proper, so it follows from \autoref{thm:injective-proper-cone} that \myref{itm:completed-injective-criteria} holds.
	
	$\myref{itm:completed-injective-criteria} \Longrightarrow \myref{itm:completed-injective-proper-cone}$. It follows from the assumptions that $\hatmaxwedge[\varepsilon]{E_+}{F_+}$ is a proper cone and that $E \hattensor_\alpha F \to E \hattensor_\varepsilon F$ is injective. (The latter statement is trivially true if $E = \{0\}$ or $F = \{0\}$; otherwise it holds by assumption.) Since $E \hattensor_\alpha F \to E \hattensor_\varepsilon F$ is bipositive and injective, it follows that $\hatmaxwedge[\alpha]{E_+}{F_+}$ is also proper.
\end{proof}

\begin{remark}
	\label{rmk:need-completeness}
	In \autoref{cor:completed-injective-proper-cone}, the assumption that $E$ and $F$ are complete cannot be omitted. Under the natural isomorphism $E \hattensor_\alpha \R \cong \tilde E$, the injective cone $\hatmaxwedge[\alpha]{E_+}{\R_+}$ corresponds with $\widetilde{E_+}$ (the closure of $E_+$ in $\tilde E$). However, it can happen that $\overline{E_+}$ is proper but $\widetilde{E_+}$ is not (e.g.~\cite[Example 6.4]{Dobben-semisimplicity}).
\end{remark}

\begin{remark}
	The natural map $E \hattensor_\alpha F \to E \hattensor_\varepsilon F$ is not always injective; this is related to the approximation property. Further remarks along this line can be found in \mysecref{sec:semisimple-completed} below; see also \cite[Theorem 5.6]{Defant-Floret}.
\end{remark}

\section{Faces of the injective cone}
\label{sec:injective-faces}
In this section, we present a general way to construct faces of the space $E \SCBiltensor F = \SCBil(E_\weakstar' \times F_\weakstar')$ of separately weak\nobreakdash-$*$ continuous bilinear forms.
This will be used in \mysecref{sec:injective-ideals} to obtain ideals in for the injective cone.

Since the injective cone is characterized by bipositive maps $E \tensor F \to E \SCBiltensor F$ and $E \hattensor_\alpha F \to \tilde E \SCBiltensor \tilde F$ (see \mysecref{sec:injective-characteristic-property}, the inverse image of a face in $(E \SCBiltensor F)_+$ (resp.~$(\tilde E \SCBiltensor \tilde F)_+$) immediately gives us a face in $\maxwedge{E_+}{F_+}$ (resp.~$\hatmaxwedge[\alpha]{E_+}{F_+}$). Therefore we focus on faces in $E \SCBiltensor F$. For ideals in $E \tensor F$ and $E \hattensor_\alpha F$, see \mysecref{sec:injective-ideals}.

\begin{definition}
	\label{def:MN-face-set}
	Let $\langle E,E'\rangle$, $\langle F,F'\rangle$ be dual pairs, and let $E_+ \subseteq E$, $F_+ \subseteq F$ be convex cones. Given $b \in E \SCBiltensor F$ and subsets $M' \subseteq E'$, $N' \subseteq F'$, let us write
	\begin{align*}
		\glsadd{bMcdot}b(M' , \,\cdot\,) &:= \{b(\varphi , \,\cdot\,) \, : \, \varphi \in M'\} \subseteq (F_\weakstar')\topdual = F;\\\noalign{\smallskip}
		\glsadd{bcdotN}b(\,\cdot\, , N') &:= \{b(\,\cdot\, , \psi) \, : \, \psi \in N'\} \subseteq (E_\weakstar')\topdual = E.
	\end{align*}
	Given subsets $M \subseteq E$, $M' \subseteq E'$, $N \subseteq F$, $N' \subseteq F'$, we define
	\begin{align*}
		\glsadd{MsetNface}\MsetNface{M'}{N} &:= \big\{ b \in E \SCBiltensor F \, : \, b(M' , \,\cdot\,) \subseteq N\big\};\\\noalign{\smallskip}
		\glsadd{MfaceNset}\MfaceNset{M}{N'} &:= \big\{ b \in E \SCBiltensor F \, : \, b(\,\cdot\, , N') \subseteq M\big\}.
	\end{align*}
	Under the natural isomorphism $E \SCBiltensor F = \SCBil(E_\weakstar' \times F_\weakstar') \cong \CL(E_\weakstar' , F_\weak)$, the set $\MsetNface{M'}{N}$ is simply the set of operators $T : E_\weakstar' \to F_\weak$ satisfying $T[M'] \subseteq N$. Likewise, $\MfaceNset{M}{N'}$ corresponds with the set of operators $S : F_\weakstar' \to E_\weak$ satisfying $S[N'] \subseteq M$.
\end{definition}

Note that the positive cone can be described as $(E \SCBiltensor F)_+ = \MsetNface{E_+'}{\overline{F_+}^{\,\weak}} = \MfaceNset{\overline{E_+}^{\,\weak}}{F_+'}$.

The following lemma will be central to the remainder of this chapter.
The special case where $M'$ and $N'$ are also faces has already been studied in the finite-dimensional setting (e.g.~\cite[\S 4]{Barker-perfect} and \cite[\S 4]{Tam-faces}), but we will need the greater generality presented here.

\begin{lemma}
	\label{lem:injective-faces}
	If $M' \subseteq E_+'$, $N' \subseteq F_+'$ are subsets of the dual cones and if $M \subseteq \overline{E_+}^{\,\weak}$, $N \subseteq \overline{F_+}^{\,\weak}$ are faces, then $(\MsetNface{M'}{N}) \cap (E \SCBiltensor F)_+$ and $(\MfaceNset{M}{N'}) \cap (E \SCBiltensor F)_+$ are faces of $(E \SCBiltensor F)_+$.
\end{lemma}
\begin{proof}
	Given $\varphi \in E'$, let $L_\varphi : E \SCBiltensor F \to (F_\weakstar')\topdual = F$ denote the map $b \mapsto b(\varphi , \,\cdot\,)$. If $\varphi \in E_+'$, then $L_\varphi$ is a positive linear map in the sense that $L_\varphi[(E \SCBiltensor F)_+] \subseteq \overline{F_+}^{\,\weak}$. Therefore $L_\varphi^{-1}[N] \cap (E \SCBiltensor F)_+$ defines a face of $(E \SCBiltensor F)_+$. Since $(\MsetNface{M'}{N}) \cap (E \SCBiltensor F)_+$ can be written as an intersection of faces,
	\[ (\MsetNface{M'}{N}) \cap (E \SCBiltensor F)_+ \, = \, \bigcap_{\varphi \in M'} L_\varphi^{-1}[N] \cap (E \SCBiltensor F)_+, \]
	it also a face of $(E \SCBiltensor F)_+$. The conclusion for $(\MfaceNset{M}{N'}) \cap (E \SCBiltensor F)_+$ follows by symmetry.
\end{proof}

As a first application of \autoref{lem:injective-faces}, we study a construction of faces in the injective cone that is dual to the construction in the projective cone (see \mysecref{sec:projective-faces}).
A slightly different construction, based again on \autoref{lem:injective-faces}, will be used in \mysecref{sec:injective-ideals} below to construct ideals for the injective cone.

\begin{theorem}
	\label{thm:injective-faces}
	Let $M \subseteq \overline{E_+}^{\,\weak}$, $N \subseteq \overline{F_+}^{\,\weak}$ be faces, and define
	\begin{align*}
		\glsadd{SCorface}\SCorface{M}{N} &:= (\MfaceNset{M}{N^\dualface}) \cap (\MsetNface{M^\dualface}{N}) \cap (E \SCBiltensor F)_+;\\\noalign{\smallskip}
		\glsadd{SCandface}\SCandface{M}{N} &:=  (\MfaceNset{M}{F_+'}) \cap (\MsetNface{E_+'}{N}).
	\end{align*}
	Then:
	\begin{enumerate}[label=(\alph*),series=injective-faces]
		\item\label{itm:injective-faces} $\SCorface{M}{N}$ and $\SCandface{M}{N}$ are faces of $(E \SCBiltensor F)_+$.
		
		\item\label{itm:injective-sublattice}
		The face lattice of $(E \SCBiltensor F)_+$ contains the following partially ordered subset:
		\begin{center}
			\begin{tikzcd}
				& \SCorface{M}{N} \arrow[dl, no head] \arrow[dr, no head] & \\
				\hspace*{-10mm} \SCorface{M}{\lineal(\overline{F_+}^{\,\weak})} = \SCandface{M}{\overline{F_+}^{\,\weak}} = \MfaceNset{M}{F_+'} \arrow[dr, no head] \hspace*{-10mm} & & \hspace*{-10mm} \SCorface{\lineal(\overline{E_+}^{\,\weak})}{N} = \SCandface{\overline{E_+}^{\,\weak}}{N} = \MsetNface{E_+'}{N} \hspace*{-10mm} \arrow[dl, no head] \\
				& \SCandface{M}{N} &
			\end{tikzcd}
		\end{center}
		This subset respects meets from the face lattice:
		\[ \SCandface{M}{N} = (\SCorface{M}{\lineal(\overline{F_+}^{\,\weak})}) \cap (\SCorface{\lineal(\overline{E_+}^{\,\weak})}{N}) = (\SCandface{M}{\overline{F_+}^{\,\weak}}) \cap (\SCandface{\overline{E_+}^{\,\weak}}{N}). \]
		
		\item\label{itm:injective-dual-faces} If $M$ and $N$ are dual faces, then so are $\SCorface{M}{N}$ and $\SCandface{M}{N}$, and one has
		\begin{align*}
			\SCorface{M}{N} &= {}^\predualface (\andface{M^\dualface}{N^\dualface}) = (\MfaceNset{M}{N^\dualface}) \cap (E \SCBiltensor F)_+ = (\MsetNface{M^\dualface}{N}) \cap (E \SCBiltensor F)_+;\\\noalign{\smallskip}
			\SCandface{M}{N} &= {}^\predualface (\orface{M^\dualface}{N^\dualface}).
		\end{align*}
		If this is the case, then the subset from \ref{itm:injective-sublattice} respects meets and joins from the lattice of $\langle (E \SCBiltensor F)_+ \: , \, \minwedge{E_+'}{F_+'} \rangle$-dual faces \textup(as defined in \mysecref{app:dual-and-exposed}\textup).
		
		\item\label{itm:injective-exposed-or} If $M$ and $N$ are exposed faces, then so is $\SCorface{M}{N}$.
		
		\item\label{itm:injective-exposed-and} If $M$ and $N$ as well as $\lineal(\overline{E_+}^{\,\weak})$ and $\lineal(\overline{F_+}^{\,\weak})$ are exposed faces, then so is $\SCandface{M}{N}$.
	\end{enumerate}
\end{theorem}

Note: in the finite\-/dimensional case, the conclusion in \ref{itm:injective-dual-faces} is simply that the four-element subset from \ref{itm:injective-sublattice} respects the operations of the lattice of exposed faces. (Here we use that $(\maxwedge{E_+}{F_+})\algdual = \minwedge{E_+\algdual}{F_+\algdual}$ because $\minwedge{E_+\algdual}{F_+\algdual}$ is closed; see \myautoref{cor:closures}{itm:closed-duality}.)

\begin{proof}[Proof of \autoref{thm:injective-faces}]
	\ \par
	\begin{enumerate}[label=(\alph*)]
		\item Note that everything in $\MfaceNset{M}{F_+'}$ is automatically positive, for if $b(\,\cdot\, , F_+') \subseteq M$ then certainly $b(\,\cdot\, , F_+') \subseteq \overline{E_+}^{\,\weak}$. This shows that $\MfaceNset{M}{F_+'} = (\MfaceNset{M}{F_+'}) \cap (E \SCBiltensor F)_+$. Now the result follows from \autoref{lem:injective-faces}, since the intersection of two faces is again a face.
		
		\item If $b \in \MfaceNset{M}{F_+'}$, then $b(\,\cdot\, , F_+') \subseteq M$, so in particular $b$ vanishes on $M^\dualface \times F_+'$. Therefore $b(M^\dualface , \,\cdot\,) \subseteq {}^\perp (F_+') = \lineal(\overline{F_+}^{\,\weak})$, which shows that $\MfaceNset{M}{F_+'} \subseteq \MsetNface{M^\dualface}{\lineal(\overline{F_+}^{\,\weak})}$. Since we also have $\MfaceNset{M}{F_+'} \subseteq (E \SCBiltensor F)_+$ (see \ref{itm:injective-faces}), it follows from the definition that
		\[ \SCorface{M}{\lineal(\overline{F_+}^{\,\weak})} = (\MfaceNset{M}{F_+'}) \cap (\MsetNface{M^\dualface}{\lineal(\overline{F_+}^{\,\weak})}) \cap (E \SCBiltensor F)_+ = \MfaceNset{M}{F_+'}. \]
		Similarly, since $\MsetNface{E_+'}{\overline{F_+}^{\,\weak}} = (E \SCBiltensor F)_+$, it follows again from the definition that
		\[ \SCandface{M}{\overline{F_+}^{\,\weak}} = (\MfaceNset{M}{F_+'}) \cap (\MsetNface{E_+'}{\overline{F_+}^{\,\weak}}) \cap (E \SCBiltensor F)_+ = \MfaceNset{M}{F_+'}. \]
		The equality $\SCorface{\lineal(\overline{E_+}^{\,\weak})}{N} = \SCandface{\overline{E_+}^{\,\weak}}{N} = \MsetNface{E_+'}{N}$ follows analogously. As a consequence, the intersection formula follows immediately from the definition of $\SCandface{M}{N}$. Finally, the upwards inclusions follow by noting that if $M_1 \subseteq M_2 \subseteq \overline{E_+}^{\,\weak}$ and $N_1 \subseteq N_2 \subseteq \overline{F_+}^{\,\weak}$ are faces, then $\SCorface{M_1}{N_1} \subseteq \SCorface{M_2}{N_2}$.
		
		\item If $b \in \MfaceNset{M}{N^\dualface}$, then $b(\,\cdot\, , N^\dualface) \subseteq M$, so in particular $b$ vanishes on $M^\dualface \times N^\dualface$. Conversely, if $b \in (E \SCBiltensor F)_+$ vanishes on $M^\dualface \times N^\dualface$, then $b(\,\cdot\, , N^\dualface) \subseteq {}^\predualface (M^\dualface) = M$, so $b \in \MfaceNset{M}{N^\dualface}$. This proves that
		\begin{align*}
			(\MfaceNset{M}{N^\dualface}) \cap (E \SCBiltensor F)_+ &= \big\{ b \in (E \SCBiltensor F)_+ \, : \, b(\varphi,\psi) = 0\ \text{for all $\varphi \in M^\dualface$, $\psi \in N^\dualface$}\big\}\\\noalign{\smallskip}
			&= {}^\predualface (M^\dualface \settensor N^\dualface).
		\end{align*}
		By symmetry, the same is true of $\MsetNface{M^\dualface}{N}$, so we find
		\[ \SCorface{M}{N} = (\MfaceNset{M}{N^\dualface}) \cap (E \SCBiltensor F)_+ = (\MsetNface{M^\dualface}{N}) \cap (E \SCBiltensor F)_+ = {}^\predualface (M^\dualface \settensor N^\dualface). \]
		Since $\andface{M^\dualface}{N^\dualface}$ is the face (of $\minwedge{E_+'}{F_+'}$) generated by $M^\dualface \settensor N^\dualface$, it follows that $\SCorface{M}{N} = {}^\predualface (\andface{M^\dualface}{N^\dualface})$. This shows that $\SCorface{M}{N}$ is a dual face.
		
		Since $\lineal(\overline{E_+}^{\,\weak}) = {}^\predualface (E_+')$ and $\lineal(\overline{F_+}^{\,\weak}) = {}^\predualface (F_+')$ are dual faces, it follows from the intersection formula from \ref{itm:injective-sublattice} that $\SCandface{M}{N}$ is also a dual face. Furthermore, since $\lineal(\overline{E_+}^{\,\weak})^\dualface = E_+'$ and $\lineal(\overline{F_+}^{\,\weak})^\dualface = F_+'$, it follows that
		\begin{align*}
			\SCandface{M}{N} &= (\SCorface{M}{\lineal(\overline{F_+}^{\,\weak})}) \cap (\SCorface{\lineal(\overline{E_+}^{\,\weak})}{N}) \\\noalign{\smallskip}
			&= {}^\predualface (\andface{M^\dualface}{F_+'}) \, \cap \, {}^\predualface (\andface{E_+'}{N^\dualface}) \\\noalign{\smallskip}
			&= {}^\predualface \big((\andface{M^\dualface}{F_+'}) + (\andface{E_+'}{N^\dualface})\big) \\\noalign{\smallskip}
			&= {}^\predualface (\orface{M^\dualface}{N^\dualface}),
		\end{align*}
		where the last step uses that $\orface{M^\dualface}{N^\dualface} = (\andface{M^\dualface}{F_+'}) + (\andface{E_+'}{N^\dualface})$, by \myautoref{thm:projective-faces}{itm:projective-sublattice}.
		
		That the diagram from \ref{itm:injective-sublattice} respects joins from the lattice of $\langle (E \SCBiltensor F)_+ , \minwedge{E_+'}{F_+'} \rangle$-dual faces follows from duality. Indeed, by \myautoref{thm:projective-faces}{itm:projective-exposed-or} and \myautoref{thm:projective-faces}{itm:projective-exposed-and}, $\orface{M^\dualface}{N^\dualface}$ and $\andface{M^\dualface}{N^\dualface}$ are $\langle \minwedge{E_+'}{F_+'} , (E \SCBiltensor F)_+ \rangle$-dual faces (use that $\lineal(E_+')$ and $\lineal(F_+')$ are automatically dual faces, because $E_+'$ and $F_+'$ are weak\nobreakdash-$*$ closed; see \autoref{rmk:lineality-space-dual-exposed-i}), so it follows that
		\begin{align*}
			(\SCorface{M}{N})^\dualface &= \andface{M^\dualface}{N^\dualface};\\\noalign{\smallskip}
			(\SCandface{M}{N})^\dualface &= \orface{M^\dualface}{N^\dualface}.
		\end{align*}
		Therefore the join of $\SCorface{M}{\lineal(\overline{F_+}^{\,\weak})}$ and $\SCorface{\lineal(\overline{E_+}^{\,\weak})}{N}$ in the lattice of $\langle (E \SCBiltensor F)_+ , \minwedge{E_+'}{F_+'}\rangle$-dual faces is given by
		\begin{align*}
			\Big.{}\Big.^\predualface \Big(\big(\SCorface{M}{\lineal(\overline{F_+}^{\,\weak})}\big)^\dualface \, \cap \, \big(\SCorface{\lineal(\overline{E_+}^{\,\weak})}{N}\big)^\dualface \, \Big) \, &= \, \big.{}\big.^\predualface \big( (\andface{M^\dualface}{F_+'}) \, \cap \, (\andface{E_+'}{N^\dualface}) \big) \\
			&= \, \big.{}\big.^\predualface \big( \andface{M^\dualface}{N^\dualface} \big) \\\noalign{\smallskip}
			&= \, \SCorface{M}{N}.
		\end{align*}
		
		\item Suppose that $M = {}^\predualface \{\varphi_0\}$ and $N = {}^\predualface \{\psi_0\}$. Then in particular $M$ and $N$ are dual faces, so by \ref{itm:injective-dual-faces} we have
		\[ \SCorface{M}{N} = {}^\predualface (\andface{M^\dualface}{N^\dualface}) = (\MfaceNset{M}{N^\dualface}) \cap (E \SCBiltensor F)_+ = (\MsetNface{M^\dualface}{N}) \cap (E \SCBiltensor F)_+. \]
		We prove that $\SCorface{M}{N} = {}^\predualface \{\varphi_0 \tensor \psi_0\}$. Evidently one has $\{\varphi_0 \tensor \psi_0\} \subseteq \andface{M^\dualface}{N^\dualface}$, so ${}^\predualface\{\varphi_0 \tensor \psi_0\} \supseteq {}^\predualface (\andface{M^\dualface}{N^\dualface}) = \SCorface{M}{N}$. For the converse, suppose that $b \in (E \SCBiltensor F)_+$ is such that $b(\varphi_0 , \psi_0) = 0$. Then $b(\,\cdot\, , \psi_0) \in {}^\predualface \{\varphi_0\} = M$, so $b$ vanishes on $M^\dualface \times \{\psi_0\}$. It follows that $b(M^\dualface , \,\cdot\,) \subseteq {}^\predualface \{\psi_0\} = N$, so $b \in (\MsetNface{M^\dualface}{N}) \cap (E \SCBiltensor F)_+ = \SCorface{M}{N}$.
		
		\item This follows from \ref{itm:injective-exposed-or} and the intersection formula from \ref{itm:injective-sublattice}. \qedhere
	\end{enumerate}
\end{proof}

\begin{remark}
	\label{rmk:lineality-space-dual-exposed-ii}
	In \myautoref{thm:injective-faces}{itm:injective-exposed-and}, it is required that $\lineal(\overline{E_+}^{\,\weak})$ and $\lineal(\overline{F_+}^{\,\weak})$ are exposed. Recall that this is automatically the case if $E$ and $F$ are separable normed spaces; see \autoref{rmk:lineality-space-dual-exposed-i} and \autoref{cor:quasi-semisimple-cone-in-separable-normed-space}.
	
	Much as in the projective case, this assumption on $\lineal(\overline{E_+}^{\,\weak})$ and $\lineal(\overline{F_+}^{\,\weak})$ cannot be omitted. The example runs along the same lines as the example in \autoref{rmk:lineality-space-dual-exposed-i}, except we need a much larger space. Let $E_+$ be a weakly closed proper cone for which $\{0\}$ is not exposed (see \autoref{xmpl:non-exposed-i}, \autoref{xmpl:non-exposed-ii}), and let $F := \R$ with the standard cone, so that $E \SCBiltensor F \cong E$. Take some exposed face $M \subseteq E_+$, and let $N := \{0\} \subseteq \R$ be the minimal face. Then $\SCandface{M}{N} = \{0\}$, which is not exposed by assumption.
\end{remark}

\begin{remark}
	\label{rmk:not-the-sum}
	\myautoref{thm:injective-faces}{itm:injective-dual-faces} presents a duality between the four-element sublattices from \myautoref{thm:projective-faces}{itm:projective-sublattice} and \myautoref{thm:injective-faces}{itm:injective-sublattice}. In the projective diagram, the top face $\orface{M}{N}$ is not merely the join, but even the sum of the left and right faces $\andface{M}{F_+}$ and $\andface{E_+}{N}$. Given that the injective diagram is dual to the projective diagram, could the same be true here?
	
	Unfortunately, this is not the case, and it already fails for proper, generating, polyhedral cones in finite\-/dimensional spaces. In this setting, all faces are exposed, so by \myautoref{thm:injective-faces}{itm:injective-dual-faces} an equivalent question is the following: if $f : E\algdual \to F$ is positive with $f[M^\dualface] \subseteq N$, then can $f$ be written as $f = g + h$ with $g$ and $h$ positive and $g[M^\dualface] = \{0\}$ and $h[E_+\algdual] \subseteq N$?
	
	Counterexample: let $F_+$ be a proper, generating, polyhedral cone with a facet $N \subseteq F_+$ such that at least two extremal rays of $F_+$ are not contained in $N$. Furthermore, let $E_+ := F_+\algdual$ with $M := N^\dualface$, and let $f : E\algdual = F \to F$ be the identity. Then one has $f[M^\dualface] \subseteq N$. However, if $f = g + h$ is the desired decomposition, then $\rank(g) \leq 1$, because $\ker(g)$ contains a facet, so $g[F_+]$ is either a ray or $\{0\}$. But now every $x \in F_+$ can be written as $x = g(x) + h(x) \in g[F_+] + N$, contradicting our assumption that at least two extremal rays of $F_+$ are not contained in $N$.
\end{remark}

\section{Order ideals for the injective cone}
\label{sec:injective-ideals}
Recall that $I \mapsto I_+$ defines a surjective many-to-one correspondence between order ideals and faces (see \mysecref{app:faces-ideals}).
In order to get more convenient formulas for the faces of the injective cone, it is helpful to formulate these results in terms of ideals.
The main aim in this section is to provide sufficient conditions so that $I \tensor J$ and $(I \tensor F) + (E \tensor J)$ are ideals for the injective cone, given that $I \subseteq E$ and $J \subseteq F$ are ideals in the base spaces.
(Similar questions in $E \SCBiltensor F$ and $E \hattensor_\alpha F$ are also addressed.)

Recall from \mysecref{sec:injective-characteristic-property} that the injective cone is characterized by bipositive maps $E \tensor F \hookrightarrow E \SCBiltensor F$ and $E \hattensor_\alpha F \to \tilde E \SCBiltensor \tilde F$.
Given subsets $X \subseteq E \SCBiltensor F$ and $Y \subseteq \tilde E \SCBiltensor \tilde F$, we denote by $X \cap (E \tensor F)$ and $Y \cap (E \hattensor_\alpha F)$ the inverse images of $X$ and $Y$ under these maps. (This is a slight abuse of notation, for the map $E \hattensor_\alpha F \to \tilde E \SCBiltensor \tilde F$ might fail to be injective in the absence of the approximation property, but this will cause no confusion.)
It is not hard to see that the inverse image of an ideal (resp.~face) under a bipositive map is again an ideal (reps.~face) (see \myautoref{prop:ideals-obtained}{itm:ideal-pullback}), so $(E \tensor F) \cap X$ and $(E \hattensor_\alpha F) \cap Y$ are ideals (resp.~faces) whenever $X$ and $Y$ are ideals (resp.~faces).
This is the approach that we will take: we establish ideals in $E \SCBiltensor F$ and restrict these to ideals in the algebraic/completed tensor product.

In order to obtain ideals in $E \SCBiltensor F$, we note that the faces obtained in \autoref{lem:injective-faces} can sometimes be written as the positive part of a linear subspace.

\begin{lemma}
	\label{lem:SCBiltensor-subspaces}
	In the notation from \mysecref{sec:injective-faces}:
	\begin{enumerate}[label=(\alph*)]
		\item\label{itm:SCBil-subspace} If $M' \subseteq E'$ and $N' \subseteq F'$ are subsets and if $M \subseteq E$ and $N \subseteq F$ are linear subspaces, then $\MsetNface{M'}{N}$ and $\MfaceNset{M}{N'}$ are linear subspaces.
		
		\item\label{itm:SCBil-oneperp} If $I \subseteq E$ and $J \subseteq F$ are subspaces and if $I$ is weakly closed, then $\MsetNface{I^\perp}{J} \subseteq \MfaceNset{I}{J^\perp}$.
		
		\item\label{itm:SCBil-perp} If $I \subseteq E$ and $J \subseteq F$ are weakly closed subspaces, then $\MsetNface{I^\perp}{J} = \MfaceNset{I}{J^\perp} = {}^\perp (I^\perp \tensor J^\perp)$, where the orthogonal complement is taken with respect to the dual pair $\langle E \SCBiltensor F , E' \tensor F' \rangle$.
	\end{enumerate}
\end{lemma}
Note that ${}^\perp (I^\perp \tensor J^\perp) \subseteq E \SCBiltensor F$ is the set of separately weak\nobreakdash-$*$ continuous bilinear forms $E' \times F' \to \R$ that vanish on $I^\perp \times J^\perp$.
\begin{proof}[Proof of \autoref{lem:SCBiltensor-subspaces}]
	\leavevmode
	\begin{enumerate}[label=(\alph*)]
		\item If $T_1,T_2 : E_\weakstar' \to F_\weak$ map the subset $M' \subseteq E'$ in the subspace $N \subseteq F$, and if $\lambda,\mu\in\R$ are arbitrary, then $\lambda T_1 + \mu T_2$ also maps $M'$ in $N$.
		
		\item If $b(I^\perp , \,\cdot\, ) \subseteq J$, then $b(I^\perp , J^\perp) = \{0\}$, hence $b(\,\cdot\, , J^\perp) \subseteq {}^\perp (I^\perp) = I$, since $I$ is weakly closed.
		
		\item Since $J$ is weakly closed, one has $b(I^\perp , \,\cdot\,) \subseteq J$ if and only if $b(I^\perp , J^\perp) = \{0\}$, i.e.~$b$ vanishes on $I^\perp \times J^\perp$. Therefore $\MsetNface{I^\perp}{J} = {}^\perp (I^\perp \tensor J^\perp)$. The other equality follows analogously.
		\qedhere
	\end{enumerate}
\end{proof}

\smallskip\noindent
We can now formulate the following ``linearization'' of \autoref{lem:injective-faces}.

\begin{lemma}
	\label{lem:injective-ideals}
	Let $M' \subseteq E_+'$ be a set of positive linear functionals, and let $N \subseteq \overline{F_+}^{\,\weak}$ be a face.
	\begin{enumerate}[label=(\alph*)]
		\item\label{itm:inj-id:SCBiltensor} If $J \subseteq F$ is a weakly closed subspace such that $J \cap \overline{F_+}^{\,\weak} = N$, then
		\[ (\MsetNface{M'}{N}) \cap (E \SCBiltensor F)_+ \, = \, \big(\,\MsetNface{\overline{\spn(M')}^{\,\weakstar}}{J}\big) \cap (E \SCBiltensor F)_+. \]
		In particular, $\,\MsetNface{\overline{\spn(M')}^{\,\weakstar}}{J}$ is an ideal in $E \SCBiltensor F$.
		
		\item\label{itm:inj-id:tensor} If $J \subseteq F$ is a subspace such that $J \cap \overline{F_+}^{\,\weak} = N$, then
		\[ (\MsetNface{M'}{N}) \cap (E \SCBiltensor F)_+ \cap (E \tensor F) \, = \, \big(\,\MsetNface{\overline{\spn(M')}^{\,\weakstar}}{J}\big) \cap (E \SCBiltensor F)_+ \cap (E \tensor F). \]
		In particular, $\big(\,\MsetNface{\overline{\spn(M')}^{\,\weakstar}}{J}\big) \cap (E \tensor F)$ is an ideal in $E \tensor F$.
	\end{enumerate}
	Interchanging $E$ and $F$ yields corresponding statements for ideals of the form $\MfaceNset{I}{\overline{\spn(N')}^{\,\weakstar}}$ and $(\MfaceNset{I}{\overline{\spn(N')}^{\,\weakstar}}) \cap (E \tensor F)$.
\end{lemma}
\begin{proof}
	\leavevmode
	\begin{enumerate}[label=(\alph*)]
		\item ``$\subseteq$''. If $b \in \MsetNface{M'}{N}$, then we have $b(M' , \,\cdot\,) \subseteq N \subseteq J$, so it follows by linearity and continuity that $b(\overline{\spn(M')}^{\,\weakstar} , \,\cdot\, ) \subseteq J$. This shows that $\MsetNface{M'}{N} \subseteq \MsetNface{\overline{\spn(M')}^{\,\weakstar}}{J}$.
		
		``$\supseteq$''. If $b \in \big(\,\MsetNface{\overline{\spn(M')}^{\,\weakstar}}{J}\big) \cap (E \SCBiltensor F)_+$, then $b(M' , \,\cdot\,) \subseteq b(\overline{\spn(M')}^{\,\weakstar} , \,\cdot\,) \subseteq J$, but also $b(M' , \,\cdot\,) \subseteq b(E_+' , \,\cdot\,) \subseteq \overline{F_+}^{\,\weak}$ by positivity, so we find $b(M' , \,\cdot\,) \subseteq J \cap \overline{F_+}^{\,\weak} = N$.
		
		To conclude that $\,\MsetNface{\overline{\spn(M')}^{\,\weakstar}}{J}$ is an ideal, note that it is a linear subspace (by \myautoref{lem:SCBiltensor-subspaces}{itm:SCBil-subspace}) whose positive part is a face (by \autoref{lem:injective-faces}).
		
		\item ``$\subseteq$''. If $b \in (\MsetNface{M'}{N}) \cap (E \tensor F)$, then $b(M' , \,\cdot\,) \subseteq N \subseteq J$. But $b$ has finite rank, so there is a finite\-/dimensional (hence closed) subspace $Y \subseteq J$ such that $b(M' , \,\cdot\,) \subseteq Y$. By linearity and continuity, it follows that $b(\overline{\spn(M')}^{\,\weakstar} , \,\cdot\, ) \subseteq Y \subseteq J$, which shows that $(\MsetNface{M'}{N}) \cap (E \tensor F) \subseteq \MsetNface{\overline{\spn(M')}^{\,\weakstar}}{J}$.
		
		The reverse inclusion ``$\supseteq$'' and the conclusion follow as in \ref{itm:inj-id:SCBiltensor}.
		\qedhere
	\end{enumerate}
\end{proof}

Recall that we call a convex cone $E_+ \subseteq E$ in a topological vector space \emph{semisimple} if $\overline{E_+}^{\,\weak}$ is a proper cone, or equivalently, if $\spn(E_+')$ is weak\nobreakdash-$*$ dense in $E'$ (see \mysecref{sec:convex-cones} and \cite{Dobben-semisimplicity}).
Furthermore, if $I \subseteq E$ is a weakly closed subspace, then the quotient $E/I$ belongs to the dual pair $\langle E/I , I^\perp \rangle$, the weak topology of $E/I$ coincides with the quotient topology $E_\weak/I$, and the weak\nobreakdash-$*$ topology on $(E/I)\topdual = I^\perp \subseteq E'$ coincides with the relative $\sigma(E',E)$-topology (see \mysecref{sec:dual-pairs}).
In particular, we may unambiguously refer to this as the weak\nobreakdash-$*$ topology on $(E/I)\topdual \cong I^\perp$.

\begin{theorem}
	\label{thm:injective-ideals}
	Let $\langle E,E' \rangle$, $\langle F,F' \rangle$ be dual pairs, and let $E_+ \subseteq E$, $F_+ \subseteq F$ be convex cones. Given subspaces $I \subseteq E$ and $J \subseteq F$, we define
	\begin{align*}
		\glsadd{SCorideal}\SCorideal{I}{J} &:= (\MsetNface{I^\perp}{J}) \cap (\MfaceNset{I}{J^\perp});\\\noalign{\smallskip}
		\glsadd{SCandideal}\SCandideal{I}{J} &:= (\MsetNface{\lineal(\overline{E_+}^{\,\weak})^\perp}{J}) \cap (\MfaceNset{I}{\lineal(\overline{F_+}^{\,\weak})^\perp}).
	\end{align*}
	Suppose that $I$ and $J$ are ideals with respect to $\overline{E_+}^{\,\weak}$ and $\overline{F_+}^{\,\weak}$, respectively.%
		\hair\footnote{In other words, $I \cap \overline{E_+}^{\,\weak}$ and $J \cap \overline{F_+}^{\,\weak}$ are faces of $\overline{E_+}^{\,\weak}$ and $\overline{F_+}^{\,\weak}$, respectively.}
	Then:
	\begin{enumerate}[label=(\alph*)]
		\item\label{itm:SCandideal-tensor} $(\SCandideal{I}{J}) \cap (E \tensor F)$ is an ideal in $E \tensor F$ \textup(with respect to the injective cone\textup);
		\item\label{itm:SCandideal-SCBiltensor} If $I$ and $J$ are weakly closed, then $\SCandideal{I}{J}$ is an ideal in $E \SCBiltensor F$;
		\item\label{itm:SCorideal-tensor} If $I$ is weakly closed and $(E/I)_+$ is semisimple, or if $J$ is weakly closed and $(F/J)_+$ is semisimple, then $(\SCorideal{I}{J}) \cap (E \tensor F)$ is an ideal in $E \tensor F$ \textup(with respect to the injective cone\textup);
		\item\label{itm:SCorideal-SCBiltensor} If $I$ and $J$ are weakly closed, and if at least one of $(E/I)_+$ and $(F/J)_+$ is semisimple, then $\SCorideal{I}{J}$ is an ideal in $E \SCBiltensor F$.
	\end{enumerate}
\end{theorem}
\begin{proof}
	\leavevmode
	\begin{enumerate}[label=(\alph*)]
		\item Since $\lineal(\overline{E_+}^{\,\weak}) = {}^\perp (E_+')$ (see \mysecref{sec:convex-cones}), we have $\overline{\spn(E_+')}^{\,\weakstar} = \lineal(\overline{E_+}^{\,\weak})^\perp$. Hence it follows from \myautoref{lem:injective-ideals}{itm:inj-id:tensor} that $(\MsetNface{\lineal(\overline{E_+}^{\,\weak})^\perp}{J}) \cap (E \tensor F)$ is an ideal in $E \tensor F$. Analogously, $(\MfaceNset{I}{\lineal(\overline{F_+}^{\,\weak})^\perp}) \cap (E \tensor F)$ is an ideal in $E \tensor F$. The conclusions follows since the intersection of two ideals is an ideal.
		
		\item Analogous to \ref{itm:SCandideal-tensor}, using \myautoref{lem:injective-ideals}{itm:inj-id:SCBiltensor} instead of \myautoref{lem:injective-ideals}{itm:inj-id:tensor}.
		
		\item Assume that $I$ is weakly closed and $(E/I)_+$ is semisimple (the other case is analogous). Since $I$ is weakly closed, it follows from \myautoref{lem:SCBiltensor-subspaces}{itm:SCBil-oneperp} that $\SCorideal{I}{J} = \MsetNface{I^\perp}{J}$. Furthermore, by duality, the adjoint of the pushforward $E \to E/I$ is the pullback (bipositive map) $(E/I)\topdual \cong I^\perp \to E'$ (see \cite[Proposition 2]{Dobben-semisimplicity}), so we have $(E/I)_+\topdual = I^\perp \cap E_+'$. Since $(E/I)_+$ is semisimple, its dual cone $(E/I)_+\topdual$ separates points on $E/I$. Equivalently, the subspace $\spn((E/I)_+\topdual) = \spn(I^\perp \cap E_+')$ is weak\nobreakdash-$*$ dense in $I^\perp$. Hence it follows from \myautoref{lem:injective-ideals}{itm:inj-id:tensor} that $(\MsetNface{I^\perp}{J}) \cap (E \tensor F)$ is an ideal in $E \tensor F$.
		
		\item Analogous to \ref{itm:SCorideal-tensor}, using \myautoref{lem:injective-ideals}{itm:inj-id:SCBiltensor} instead of \myautoref{lem:injective-ideals}{itm:inj-id:tensor}.
		\qedhere
	\end{enumerate}
\end{proof}

\begin{remark}
	In terms of the mapping properties, it is not surprising that the semisimplicity of $(E/I)_+$ and $(F/J)_+$ plays a role in \autoref{thm:injective-ideals}. Let $\pi_I : E \to E/I$ and $\pi_J : F \to F/J$ denote the canonical maps. If both $(E/I)_+$ and $(F/J)_+$ are semisimple, then $(E/I) \SCBiltensor (F/J)$ is a proper cone (by \autoref{thm:injective-proper-cone}), so now evidently $\ker(\pi_I \SCBiltensor \pi_J) = {}^\perp (I^\perp \tensor J^\perp)$ is an ideal in $E \SCBiltensor F$.
	
	What is surprising in \autoref{thm:injective-ideals} is that it is sufficient for only one of $(E/I)_+$ and $(F/J)_+$ to be semisimple. This could not have been predicted solely on the basis of the mapping properties.
	The following example shows that we need at least one of the quotients to be semisimple, even in the finite\-/dimensional case.
\end{remark}
\begin{example}
	\label{xmpl:SCorface-ideal-fail}
	Let $E := \R^3$ and let $E_+ \subseteq E$ be the second-order cone $E_+ = \{(x_1,x_2,x_3) \, : \, \sqrt{x_1^2 + x_2^2} \leq x_3\}$. The injective cone $\maxwedge{E_+\algdual}{E_+}$ can be identified with the cone $\LL_+(E,E)$ of positive linear operators $E \to E$.
	If we identify $E\algdual$ with $\R^3$ via the standard inner product, then $E_+$ is self-dual.
	The vectors $(1,0,1),(-1,0,1) \in \R^3$ define extremal rays of $E_+$, so the subspaces $I := \spn\{(-1,0,1)\} \subseteq E\algdual$ and $J := \spn\{(1,0,1)\} \subseteq E$ are ideals (see \myautoref{prop:ideals-obtained}{itm:ideal-face}).
	It follows from \myautoref{lem:SCBiltensor-subspaces}{itm:SCBil-perp} that $\SCorideal{I}{J} = \MsetNface{I^\perp}{J}$. We show that this is not an ideal.
	
	Let $b_1 \in E\algdual \tensor E = \Bil(E,E\algdual) \cong L(E,E)$ correspond to the identity $E \to E$, and let $b_2 \in E\algdual \tensor E$ be the bilinear form $E \times E\algdual \to \R$ corresponding with the linear map $(x_1,x_2,x_3) \mapsto (x_1,-x_2,x_3)$. Clearly $b_1$ and $b_2$ are positive. However, since $\dim(I^\perp) = 2$ and $\dim(J) = 1$, maps in $\MsetNface{I^\perp}{J}$ cannot be invertible, so in particular we have $b_1,b_2 \notin \MsetNface{I^\perp}{J}$.
	
	It is not hard to see that $b_1 + b_2 \in \MsetNface{I^\perp}{J}$, and evidently we have $0 \leq b_1 \leq b_1 + b_2$. This shows that $\MsetNface{I^\perp}{J}$ is not an ideal.
\end{example}

We conclude this section by providing more convenient direct formulas for the ideals $\SCorideal{I}{J}$ and $\SCandideal{I}{J}$ and their restrictions to $E \tensor F$ or $E \hattensor_\alpha F$.
Roughly speaking, under certain topological assumptions we have $\SCorideal{I}{J} = (I \SCBiltensor F) + (E \SCBiltensor J)$ and $\SCandideal{I}{J} = I \SCBiltensor J$.

\subsection{Ideals in the algebraic tensor product}
We show that the ideals $(\SCandideal{I}{J}) \cap (E \tensor F)$ and $(\SCorideal{I}{J}) \cap (E \tensor F)$ from \autoref{thm:injective-ideals} are always equal to $(I \tensor J) + \lineal(\maxwedge{E_+}{F_+})$ and $(I \tensor F) + (E \tensor J)$, respectively

\begin{lemma}
	\label{lem:tensor-subspace}
	If $I \subseteq E$ and $J \subseteq F$ are subspaces, then
	\[ (\MsetNface{I^\perp}{J}) \cap (\MfaceNset{I}{J^\perp}) \cap (E \tensor F) \, = \, (I \tensor F) + (E \tensor J) + (\overline{I\,}^{\,\weak} \tensor \overline{J\,}^{\,\weak}). \]
\end{lemma}
\begin{proof}
	Choose an algebraic decomposition $E \cong E_1 \oplus E_2 \oplus E_3$ with $E_1 \cong I$ and $E_1 \oplus E_2 \cong \overline{I\,}^{\,\weak}$, and likewise for $F \cong F_1 \oplus F_2 \oplus F_3$. Then $E \tensor F \cong \bigoplus_{i=1}^3 \bigoplus_{j=1}^3 (E_i \tensor F_j)$. Under this identification, $(\MsetNface{I^\perp}{J}) \cap (E \tensor F)$ corresponds with those elements that are zero in the $E_3 \tensor F_2$ and $E_3 \tensor F_3$ components. Likewise, $(\MfaceNset{I}{J^\perp}) \cap (E \tensor F)$ corresponds with those elements that are zero in the $E_2 \tensor F_3$ and $E_3 \tensor F_3$ components, and the conclusion follows.
\end{proof}

\begin{corollary}
	\label{cor:tensor-subspace}
	Let $I \subseteq E$ and $J \subseteq F$ be subspaces.
	\begin{enumerate}[label=(\alph*)]
		\item\label{itm:ts:oneperp} If at least one of $I$ and $J$ is weakly closed, then
		\[ (\MsetNface{I^\perp}{J}) \cap (\MfaceNset{I}{J^\perp}) \cap (E \tensor F) \, = \, (I \tensor F) + (E \tensor J). \]
		\item\label{itm:ts:perp} If both $I$ and $J$ are weakly closed, then
		\[ ({}^\perp (I^\perp \tensor J^\perp)) \cap (E \tensor F) \, = \, (I \tensor F) + (E \tensor J). \]
		\item\label{itm:ts:lin-space} The lineality space of the injective cone \textup(in $E \tensor F$\textup) is
		\[ \lineal(\maxwedge{E_+}{F_+}) \, = \, (\lineal(\overline{E_+}^{\,\weak}) \tensor F) + (E \tensor \lineal(\overline{F_+}^{\,\weak})). \]
	\end{enumerate}
\end{corollary}
\begin{proof}
	\leavevmode
	\begin{enumerate}[label=(\alph*)]
		\item If $I$ is weakly closed, then $\overline{I\,}^{\,\weak} \tensor \overline{J\,}^{\,\weak} \subseteq I \tensor F$, so the result follows from \autoref{lem:tensor-subspace}.
		
		\item Immediate, for now we have ${}^\perp (I^\perp \tensor J^\perp) = \MsetNface{I^\perp}{J} = \MfaceNset{I}{J^\perp}$, by \myautoref{lem:SCBiltensor-subspaces}{itm:SCBil-perp}.
		
		\item By \autoref{prop:injective-lineality-space}, we have $\lineal(\maxwedge{E_+}{F_+}) = ({}^\perp (\lineal(\overline{E_+}^{\,\weak})^\perp \tensor \lineal(\overline{F_+}^{\,\weak})^\perp)) \cap (E \tensor F)$, where we note that $\lineal(\overline{E_+}^{\,\weak})$ and $\lineal(\overline{F_+}^{\,\weak})$ are weakly closed subspaces.
		\qedhere
	\end{enumerate}
\end{proof}

\begin{theorem}
	\label{thm:injective-ideals-algebraic}
	Let $I \subseteq E$ and $J \subseteq F$ be ideals with respect to $\overline{E_+}^{\,\weak}$ and $\overline{F_+}^{\,\weak}$.
	\begin{enumerate}[label=(\alph*)]
		\item One has $(\SCandideal{I}{J}) \cap (E \tensor F) = (I \tensor J) + \lineal(\maxwedge{E_+}{F_+})$, and this is an ideal in $E \tensor F$ \textup(with respect to the injective one\textup);
		
		\item If $I$ is weakly closed and $(E/I)_+$ is semisimple, or if $J$ is weakly closed and $(F/J)_+$ is semisimple, then one has $(\SCorideal{I}{J}) \cap (E \tensor F) = (I \tensor F) + (E \tensor J)$, and this is an ideal in $E \tensor F$ \textup(with respect to the injective cone\textup).
	\end{enumerate}
\end{theorem}
\begin{proof}
	\leavevmode
	\begin{enumerate}[label=(\alph*)]
		\item Every ideal contains the lineality space, so we may choose a decomposition $E \cong E_1 \oplus E_2 \oplus E_3$ with $E_1 \cong \lineal(\overline{E_+}^{\,\weak})$ and $E_1 \oplus E_2 \cong I$, and likewise for $F \cong F_1 \oplus F_2 \oplus F_3$. With respect to the decomposition $E \tensor F \cong \bigoplus_{i=1}^3 \bigoplus_{j=1}^3 (E_i \tensor F_j)$, the subspace $(\MsetNface{\lineal(\overline{E_+}^{\,\weak})^\perp}{J}) \cap (E \tensor F)$ corresponds with those elements that are zero in the $E_2 \tensor F_3$ and $E_3 \tensor F_3$ components, and $(\MfaceNset{I}{\lineal(\overline{F_+}^{\,\weak})^\perp}) \cap (E \tensor F)$ corresponds with those elements that are zero in the $E_3 \tensor F_2$ and $E_3 \tensor F_3$ components. Since $\lineal(\maxwedge{E_+}{F_+}) = (E_1 \tensor F) + (E \tensor F_1)$ (by \myautoref{cor:tensor-subspace}{itm:ts:lin-space}) and $I \tensor J = (E_1 \oplus E_2) \tensor (F_1 \oplus F_2)$, the conclusion follows. (This is an ideal by \myautoref{thm:injective-ideals}{itm:SCandideal-tensor}.)
		
		\item The formula follows from \myautoref{cor:tensor-subspace}{itm:ts:oneperp}, and this is an ideal by \myautoref{thm:injective-ideals}{itm:SCorideal-tensor}.
		\qedhere
	\end{enumerate}
\end{proof}

\subsection{Ideals in the space of separately \texorpdfstring{weak-$*$}{weak-*} continuous bilinear forms}
Let $I \subseteq E$ and $J \subseteq F$ be subspaces, and write $I_+ := I \cap \overline{E_+}^{\,\weak}$ and $J_+ := J \cap \overline{F_+}^{\,\weak}$.
If we let $I$ and $J$ belong to the dual pairs $\langle I, E' / I^\perp \rangle$, $\langle J , F' / J^\perp \rangle$, then the inclusions $T : I \hookrightarrow E$, $S : J \hookrightarrow F$ are weakly continuous (weak homomorphisms in fact; see \mysecref{sec:dual-pairs}) and approximately bipositive.
Therefore $T \SCBiltensor S : I \SCBiltensor J \to E \SCBiltensor F$ is injective and bipositive, by \myautoref{lem:injective-positive-linear-maps}{itm:injective-bipositive-maps}.%
	\hair\footnote{That $T \SCBiltensor S$ is injective follows from \autoref{rmk:injective-mapping-property}. This is a classical result; see also \cite[\S 44.4.(5)]{Kothe-II}.}
In other words, we may interpret $I \SCBiltensor J$ as a subspace of $E \SCBiltensor F$, and moreover $(I \SCBiltensor J)_+ = (I \SCBiltensor J) \cap (E \SCBiltensor F)_+$.

\begin{lemma}
	\label{lem:SCBiltensor-subspace}
	The image of $I \SCBiltensor J$ under the natural inclusion $T \SCBiltensor S : I \SCBiltensor J \hookrightarrow E \SCBiltensor F$ is equal to $(\MsetNface{E'}{J}) \cap (\MfaceNset{I}{F'})$.
\end{lemma}
\begin{proof}
	By definition (see \mysecref{sec:injective-positive-linear-maps}), $\id_E \SCBiltensor S : E \SCBiltensor J \hookrightarrow E \SCBiltensor F$ is given by $((\id_E \SCBiltensor S)b)(\varphi,\psi) = b(\varphi , S'\psi)$. Therefore the following diagram commutes:
	\begin{center}
		\begin{tikzcd}[row sep=huge,column sep=tiny]
			E \SCBiltensor J \arrow[r, equal, shorten <=.6pt, shorten >=.6pt] & \SCBil(E_\weakstar' \times J_\weakstar') \arrow[r, "\sim"]\arrow[d, "\id_E \SCBiltensor S"] &[2em] \CL(E_\weakstar' , J_\weak) \arrow[d, "R \mapsto SR"] \\
			E \SCBiltensor F \arrow[r, equal] & \SCBil(E_\weakstar' \times F_\weakstar') \arrow[r, "\sim"] &[2em] \CL(E_\weakstar' , F_\weak).
		\end{tikzcd}
	\end{center}
	An operator $T \in \CL(E_\weakstar' , F_\weak)$ lies in the image of $\CL(E_\weakstar' , J_\weak)$ if and only if $T[E'] \subseteq J$. Therefore a bilinear form $b \in E \SCBiltensor F$ is the extension of a bilinear form in $E \SCBiltensor J$ if and only if $b \in \MsetNface{E'}{J}$. By the same argument, $I \SCBiltensor J = (E \SCBiltensor J) \cap (\MfaceNset{I}{F'})$, and the conclusion follows.
\end{proof}

We will henceforth identify $I \SCBiltensor J$ with the subspace $(\MsetNface{E'}{J}) \cap (\MfaceNset{I}{F'}) \subseteq E \SCBiltensor F$.

Next, we turn to complementary decompositions. We say that a subspace $E_1 \subseteq E$ is \emph{weakly complemented} if it is complemented in the weak topology. (Recall that complemented subspaces and their complements are automatically closed: if $P : E \to E$ is a continuous projection, then $\ker(P)$ and $\ran(P) = \ker(\id_E - P)$ are closed.)

\begin{lemma}
	\label{lem:SCBiltensor-decomposition}
	If $E_1 \subseteq E$ and $F_1 \subseteq F$ are weakly complemented subspaces with complements $E_2 \subseteq E$ and $F_2 \subseteq F$, respectively, then $E \SCBiltensor F$ decomposes as the internal \textup(algebraic\textup) direct sum
	\[ E \SCBiltensor F = (E_1 \SCBiltensor F_1) \oplus (E_1 \SCBiltensor F_2) \oplus (E_2 \SCBiltensor F_1) \oplus (E_2 \SCBiltensor F_2). \]
\end{lemma}
\begin{proof}
	The complementary pairs give rise to weakly continuous complementary decompositions $E_\weak \cong (E_1)_\weak \times (E_2)_\weak$ and $F_\weak \cong (F_1)_\weak \times (F_2)_\weak$ (topological products). Dualizing the first of these, we obtain a weak\nobreakdash-$*$ continuous complementary decomposition $E_\weakstar' \cong (E_2^\perp)_\weakstar \oplus (E_1^\perp)_\weakstar$ (locally convex sum).%
		\hair\footnote{Note that the indices are reversed when passing to the dual: we have $(E_1)_\weakstar' \cong E_\weakstar'/E_1^\perp \cong (E_2^\perp)_\weakstar$ and vice versa.}
	Using the mapping properties of locally convex sums and topological products (see e.g.~\cite[\S 39.8]{Kothe-II}), we find
	\begin{align*}
		E \SCBiltensor F &= \SCBil(E_\weakstar' \times F_\weakstar') \\\noalign{\smallskip}
		&\cong \CL(E_\weakstar' , F_\weak) \\\noalign{\smallskip}
		&= \CL\big((E_2^\perp)_\weakstar \oplus (E_1^\perp)_\weakstar \, , \, (F_1)_\weak \times (F_2)_\weak\big) \\\noalign{\smallskip}
		&\cong \prod_{i\in\{2,1\}} \prod_{j\in\{1,2\}} \CL\big((E_i^\perp)_\weakstar , (F_j)_\weak\big) \\
		&\cong \prod_{i\in\{1,2\}} \prod_{j\in\{1,2\}} (E_i)_\weak \SCBiltensor (F_j)_\weak. \qedhere
	\end{align*}
\end{proof}
\begin{corollary}
	\label{cor:SCBiltensor-decomposition}
	\leavevmode
	\begin{enumerate}[label=(\alph*)]
		\item\label{itm:SCBil-compl-perp} If $E_1 \subseteq E$ and $F_1 \subseteq F$ are weakly complemented subspaces with complements $E_2 \subseteq E$ and $F_2 \subseteq F$, respectively, then
		\[ {}^\perp (E_1^\perp \tensor F_1^\perp) \, = \, (E_1 \SCBiltensor F_1) \oplus (E_1 \SCBiltensor F_2) \oplus (E_2 \SCBiltensor F_1), \]
		where the orthogonal complement is taken with respect to the dual pair $\langle E \SCBiltensor F , E' \tensor F' \rangle$.
		
		\item\label{itm:SCBil-lin-space} If $\lineal(\overline{E_+}^{\,\weak})$ and $\lineal(\overline{F_+}^{\,\weak})$ are weakly complemented with complements $X$ and $Y$, then
		\[ \lineal((E \SCBiltensor F)_+) \, = \, \big(\lineal(\overline{E_+}^{\,\weak}) \SCBiltensor \lineal(\overline{F_+}^{\,\weak})\big) \oplus \big(\lineal(\overline{E_+}^{\,\weak}) \SCBiltensor Y\big) \oplus \big(X \SCBiltensor \lineal(\overline{F_+}^{\,\weak})\big). \]
	\end{enumerate}
\end{corollary}

\noindent
Now we come to concrete descriptions of the subspaces $\SCorideal{I}{J}$ and $\SCandideal{I}{J}$ from \autoref{thm:injective-ideals}.

\begin{theorem}
	\label{thm:SCBiltensor-ideals}
	Let $I \subseteq E$ and $J \subseteq F$ be weakly closed ideals with respect to $\overline{E_+}^{\,\weak}$ and $\overline{F_+}^{\,\weak}$.
	\begin{enumerate}[label=(\alph*)]
		\item If $\lineal(\overline{E_+}^{\,\weak})$ and $\lineal(\overline{F_+}^{\,\weak})$ are weakly complemented, then $\SCandideal{I}{J} = (I \SCBiltensor J) + \lineal((E \SCBiltensor F)_+)$, and this is an ideal in $E \SCBiltensor F$.
		
		\item If $I$ and $J$ are weakly complemented, then $\SCorideal{I}{J} = (I \SCBiltensor F) + (E \SCBiltensor J)$. This is an ideal in $E \SCBiltensor F$ if at least one of $(E/I)_+$ and $(F/J)_+$ is semisimple.
	\end{enumerate}
\end{theorem}
\begin{proof}
	\leavevmode
	\begin{enumerate}[label=(\alph*)]
		\item ``$\supseteq$''. It follows from \myautoref{thm:injective-ideals}{itm:SCandideal-SCBiltensor} that $\SCandideal{I}{J}$ is an ideal. Since every ideal contains the lineality space, we have $\lineal((E \SCBiltensor F)_+) \subseteq \SCandideal{I}{J}$.
		Furthermore, we have $\MsetNface{E'}{J} \subseteq \MsetNface{\lineal(\overline{E_+}^{\,\weak})^\perp}{J}$ and $\MfaceNset{I}{F'} \subseteq \MfaceNset{I}{\lineal(\overline{F_+}^{\,\weak})^\perp}$, and therefore $I \SCBiltensor J \subseteq \SCandideal{I}{J}$ (by \autoref{lem:SCBiltensor-subspace}).
		
		``$\subseteq$''. The orthogonal complement of a weakly complemented subspace is weak\nobreakdash-$*$ complemented in the dual, so we may choose weak\nobreakdash-$*$ continuous projections $P : E' \to \lineal(\overline{E_+}^{\,\weak})^\perp \hookrightarrow E'$ and $Q : F' \to \lineal(\overline{F_+}^{\,\weak})^\perp \hookrightarrow F'$. Let $b_1 \in \SCandideal{I}{J}$ be given, and define $b_2(\varphi,\psi) = b_1(P\varphi,Q\psi)$. Evidently $b_2$ is separately weak\nobreakdash-$*$ continuous, so $b_2 \in E \SCBiltensor F$. Furthermore, $b_1$ and $b_2$ agree on $\lineal(\overline{E_+}^{\,\weak})^\perp \times \lineal(\overline{F_+}^{\,\weak})^\perp = \overline{\spn(E_+')}^{\,\weakstar} \times \overline{\spn(F_+')}^{\,\weakstar}$, so $b_2$ is positive and $b_1 - b_2 \in \lineal((E \SCBiltensor F)_+)$ (by \autoref{prop:injective-lineality-space}). Finally, since $b_1 \in (\MsetNface{\lineal(\overline{E_+}^{\,\weak})^\perp}{J}) \cap (\MfaceNset{I}{\lineal(\overline{F_+}^{\,\weak})^\perp})$, we have $b_2 \in (\MsetNface{E'}{J}) \cap (\MfaceNset{I}{F'}) = I \SCBiltensor J$. Therefore, $b_1 = b_2 + (b_1 - b_2) \in (I \SCBiltensor F) + \lineal((E \SCBiltensor F)_+)$.
		
		\item By \myautoref{lem:SCBiltensor-subspaces}{itm:SCBil-perp}, we have $\SCorideal{I}{J} = {}^\perp (I^\perp \tensor J^\perp)$, so the direct formula follows from \myautoref{cor:SCBiltensor-decomposition}{itm:SCBil-compl-perp}.
		The conditions for $\SCorideal{I}{J}$ to be an ideal follow from \myautoref{thm:injective-ideals}{itm:SCorideal-SCBiltensor}.
		\qedhere
	\end{enumerate}
\end{proof}

\begin{corollary}
	\label{cor:SCBiltensor-ideals}
	If $\overline{E_+}^{\,\weak}$ and $\overline{F_+}^{\,\weak}$ are proper cones, and if $I \subseteq E$ and $J \subseteq F$ are weakly closed ideals with respect to $\overline{E_+}^{\,\weak}$ and $\overline{F_+}^{\,\weak}$, then $I \SCBiltensor J$ is an ideal in $E \SCBiltensor F$.
\end{corollary}

\subsection{Ideals in completed locally convex tensor products}
Finally, we turn our attention to ideals in the completed tensor product $E \hattensor_\alpha F$.
The ideals $\SCorideal{I}{J}$ and $\SCandideal{I}{J}$ obtained in \autoref{thm:injective-ideals} can be restricted to ideals in $E \hattensor_\alpha F$ (with respect to the injective cone).
However, although we were able to find more convenient formulas for the intersections of $\SCorideal{I}{J}$ and $\SCandideal{I}{J}$ with the algebraic tensor product $E \tensor F$ (see \autoref{thm:injective-ideals-algebraic}),
there are no similar formulas for the intersections with $E \hattensor_\alpha F$.
To illustrate the obstruction, we first rephrase the problem in the more common terminology of normed tensor products.

Let $E$ and $F$ be Banach spaces, let $E_+ \subseteq E$ and $F_+ \subseteq F$ be closed proper cones, and let $J \subseteq F$ be a closed order ideal.
Then $E \SCBiltensor F \cong \CL(E_\weakstar\topdual , F_\weak) \subseteq \CL(E\topdual , F)$ is the subspace of those operators $T : E\topdual \to F$ for which the range of the adjoint $T\topdual : F\topdual \to E\topdualdual$ is contained in $E$.
By \autoref{thm:injective-ideals}, the subspace $\SCorideal{\{0\}}{J} = \SCandideal{E}{J} = \MsetNface{E\topdual}{J}$ is an ideal in $E \SCBiltensor F$.
The elements of this ideal are simply the weak-$*$-to-weak continuous operators $E\topdual \to F$ whose range is contained in $J$.
In particular, if $\alpha$ is a tensor norm, then $(\MsetNface{E\topdual}{J}) \cap (E \hattensor_\alpha F) = \CL(E\topdual,J) \cap (E \hattensor_\alpha F)$.
It is well-known that this can be different from $E \hattensor_\alpha J$.
We give two examples.

\begin{example}
	If $F$ has the approximation property but $J$ does not, then one has $E\topdual \hattensor_\varepsilon J \neq \CK(E , J)$ for some appropriate Banach space $E$, but also $E\topdual \hattensor_\varepsilon F = \CK(E, F)$ (see \cite[\S 5.3]{Defant-Floret}). Therefore $\CL(E , J) \cap (E\topdual \hattensor_\varepsilon F) = \CK(E,J)$ is strictly larger than $E\topdual \hattensor_\varepsilon F$.
\end{example}

\begin{example}
	It is well-known that the operator ideal of nuclear operators is not injective: if $J \subseteq F$ is a closed subspace and if $T : E \to J$ is nuclear as a map $E \to F$, then it does not necessarily follow that $T$ is also nuclear as a map $E \to J$ (see \cite[\S 9.7]{Defant-Floret}).
	Even if all spaces have the approximation property, so that $E\topdual \hattensor_\pi F = \CN(E,F)$ and $E\topdual \hattensor_\pi J = \CN(E,J)$, it can happen that $\CL(E,J) \cap \CN(E,F) \neq \CN(E,J)$, so that $\CL(E,J) \cap (E\topdual \hattensor_\pi F) \neq E\topdual \hattensor_\pi J$.
\end{example}

The obstruction is a purely topological one, and has nothing to do with cone-theoretic issues. Therefore we only sketch the proofs of the following special cases, where a convenient formula can be obtained.

\begin{theorem}[Injective topology; approximation property]
	\label{thm:injective-ideals-completed-i}
	Let $E$ and $F$ be complete locally convex spaces, let $E_+ \subseteq E$, $F_+ \subseteq F$ be closed proper cones, and let $I \subseteq E$, $J \subseteq F$ be closed ideals.
	If $I$ or $J$ has the approximation property, then $I \hattensor_\varepsilon J$ is an ideal in $E \hattensor_\varepsilon F$.
\end{theorem}
\begin{proof}[Proof sketch]
	The \emph{$\varepsilon$-product} $E \epsprod F$ is the subspace of $E \SCBiltensor F$ consisting of those operators $b \in \SCBil(E_\weakstar\topdual , F_\weakstar\topdual) \cong \CL(E_\weakstar\topdual , F_\weak) \cong \CL(F_\weakstar\topdual , E_\weak)$ that map equicontinuous subsets of $E\topdual$ in relatively compact sets in $F$, or equivalently, that map equicontinuous subsets of $F\topdual$ in relatively compact sets in $E$ (see \cite[\S 43.3.(2)]{Kothe-II}).
	This property is passed to subspaces, so $(E \epsprod F) \cap (I \SCBiltensor J) \subseteq I \epsprod J$.
	
	Since $E$ and $F$ are complete, we have $E \hattensor_\varepsilon F \subseteq E \epsprod F$ (see \cite[\S 43.3.(5)]{Kothe-II}). Furthermore, since $I$ and $J$ are complete and $I$ or $J$ has the approximation property, we have $I \hattensor_\varepsilon J = I \epsprod J$ (see \cite[\S 43.3.(7)]{Kothe-II}). It follows that
	\[ (I \SCBiltensor J) \cap (E \hattensor_\varepsilon F) \subseteq (I \SCBiltensor J) \cap (E \epsprod F) \subseteq I \epsprod J = I \hattensor_\varepsilon J. \]
	On the other hand, one clearly has $I \hattensor_\varepsilon J \subseteq (I \SCBiltensor J) \cap (E \hattensor_\varepsilon F)$, so we have equality.
	By \autoref{cor:SCBiltensor-ideals}, $I \SCBiltensor J$ is an ideal in $E \SCBiltensor F$, so it follows that $I \hattensor_\varepsilon J$ is an ideal in $E \hattensor_\varepsilon F$.
\end{proof}

If $E$ and $F$ are Banach spaces, then the $\varepsilon$-product in the proof of \autoref{thm:injective-ideals-completed-i} can be replaced by a suitable space of compact operators.

\bigskip
The second situation where a more convenient formula can be obtained is if the subspaces are complemented. Let us say that a \emph{locally convex tensor topology} is a locally convex topology $\alpha$ defined on $E \tensor F$ for every pair $(E,F)$ of locally convex spaces such that:
\begin{enumerate}[label=(\roman*)]
	\item $\alpha$ is a compatible topology on $E \tensor F$ for every pair $(E,F)$;
	\item $\alpha$ satisfies the \emph{continuous mapping property}: if $T : E \to G$ and $S : F \to H$ are continuous, then $T \tensor S : E \tensor_\alpha F \to G \tensor_\alpha H$ is also continuous.
\end{enumerate}
Examples of locally convex tensor topologies include the projective topology $\pi$ and the injective topology $\varepsilon$. More generally, every tensor norm gives rise to a locally convex tensor topology that even satisfies the equicontinuous mapping property; see \cite[\S 35.2]{Defant-Floret}.

\begin{theorem}[Arbitrary topology; complemented subspaces]
	\label{thm:injective-ideals-completed-ii}
	Let $E$ and $F$ be complete locally convex spaces, let $E_+ \subseteq E$, $F_+ \subseteq F$ be convex cones, let $\alpha$ be a locally convex tensor topology, and let $I \subseteq E$, $J \subseteq F$ be closed ideals with respect to $\overline{E_+}$ and $\overline{F_+}$.
	\begin{enumerate}[label=(\alph*)]
		\item\label{itm:completed-SCandideal} If $I$, $J$, $\lineal(\overline{E_+})$ and $\lineal(\overline{F_+})$ are complemented, then $(\SCandideal{I}{J}) \cap (E \hattensor_\alpha F) = (I \hattensor_\alpha F) + \lineal(\hatmaxwedge[\alpha]{E_+}{F_+})$, and this is an ideal in $E \hattensor_\alpha F$ \textup(with respect to the injective cone\textup).
		\item\label{itm:completed-SCorideal} If $I$ and $J$ are complemented, then $(\SCorideal{I}{J}) \cap (E \hattensor_\alpha F) = (I \hattensor_\alpha F) + (E \hattensor_\alpha J)$. This is an ideal in $E \hattensor_\alpha F$ \textup(with respect to the injective cone\textup) if at least one of $(E/I)_+$ and $(F/J)_+$ is semisimple.
	\end{enumerate}
\end{theorem}
\begin{proof}[Proof sketch]
	Given closed subspaces $E_1,\ldots,E_n \subseteq E$, we say that \emph{$E \cong \bigoplus_{i=1}^n E_i$ topologically} if the canonical map $\bigoplus_{i=1}^n E_i \to E$ is a topological isomorphism. Equivalently, if $E \cong \bigoplus_{i=1}^n E_i$ algebraically, then one has $E \cong \bigoplus_{i=1}^n E_i$ topologically if and only if every projection $E \to E_i$ is continuous (see \cite[Theorem 2.2]{Schaefer}).
	
	If $E \cong \bigoplus_{i=1}^n E_i$ topologically and $F \cong \bigoplus_{j=1}^m F_j$ topologically, then $E \hattensor_\alpha F \cong \bigoplus_{i,j} (E_i \hattensor_\alpha F_j)$ topologically. Analogously, \autoref{lem:SCBiltensor-decomposition} can be extended to prove that $E \SCBiltensor F \cong \bigoplus_{i,j} (E_i \SCBiltensor F_j)$, and the following diagram commutes:
	\begin{center}
		\begin{tikzcd}[row sep=large]
			E \hattensor_\alpha F \arrow[r] \arrow[d, "\rotatebox{90}{$\sim$}"] & E \hattensor_\varepsilon F \arrow[r] \arrow[d, "\rotatebox{90}{$\sim$}"] & E \SCBiltensor F \arrow[d, "\rotatebox{90}{$\sim$}"] \\
			\bigoplus_{i,j} (E_i \hattensor_\alpha F_j) \arrow[r] & \bigoplus_{i,j} (E_i \hattensor_\varepsilon F_j) \arrow[r] & \bigoplus_{i,j} (E_i \SCBiltensor F_j).
		\end{tikzcd}
	\end{center}
	In particular, for every subset $A \subseteq [n] \times [m]$ we have
	\[ \Bigg(\bigoplus_{(i,j) \in A} (E_i \SCBiltensor F_j)\Bigg) \cap (E \hattensor_\alpha F) = \bigoplus_{(i,j) \in A} (E_i \hattensor_\alpha F_j). \]
	Therefore the result follows from \autoref{thm:SCBiltensor-ideals}, using the following decompositions:
	\begin{enumerate}[label=(\alph*)]
		\item $E \cong E_1 \oplus E_2 \oplus E_3$ and $F \cong F_1 \oplus F_2 \oplus F_3$ (topologically), where $E_1 = \lineal(\overline{E_+})$, $E_1 \oplus E_2 = I$, $F_1 = \lineal(\overline{F_+})$, $F_1 \oplus F_2 = J$; and $A = \{(1,1),(1,2),(1,3),(2,1),(2,2),(3,1)\}$.
		
		\item $E \cong E_1 \oplus E_2$ and $F \cong F_1 \oplus F_2$, where $E_1 = I$, $F_1 = J$; and $A = \{(1,1),(1,2),(2,1)\}$.\!\qedhere
	\end{enumerate}
\end{proof}

\section{Extremal rays of the injective cone}
\label{sec:injective-extremal-rays}
As an application of the results from \mysecref{sec:injective-ideals}, we show that the tensor product of two extremal rays defines an extremal ray of the injective cone. In \mysecref{sec:rank-one} we will prove that all rank one extremal directions of the injective cone are of this form, but \autoref{xmpl:injective-extremal-rays} will show that there might be extremal directions of larger rank.

\begin{proposition}
	\label{prop:SCBiltensor-extremal-rays}
	If $x_0 \in \overline{E_+}^{\,\weak} \setminus \{0\}$ and $y_0 \in \overline{F_+}^{\,\weak} \setminus \{0\}$ define extremal rays of $\overline{E_+}^{\,\weak}$ and $\overline{F_+}^{\,\weak}$, then $x_0 \tensor y_0 \in E \tensor F \subseteq E \SCBiltensor F$ defines an extremal ray of $(E \SCBiltensor F)_+$. In other words:
	\[ \rext(\overline{E_+}^{\,\weak}) \settensor \rext(\overline{F_+}^{\,\weak}) \subseteq \rext((E \SCBiltensor F)_+). \]
\end{proposition}
\begin{proof}
	Let $M := \{\lambda x_0 \, : \, \lambda \geq 0\}$ denote the ray generated by $x_0$. Then $M$ is an extremal ray, so in particular a face. Every face contains the minimal face $\lineal(\overline{E_+}^{\,\weak})$, but $M$ does not contain a non-trivial subspace, so $\overline{E_+}^{\,\weak}$ is a proper cone. Furthermore, $I := \spn(M) = \spn(x_0)$ is an ideal by \myautoref{prop:ideals-obtained}{itm:ideal-face}, and is weakly closed because it is finite\-/dimensional. Analogously, $J := \spn(y_0)$ is a weakly closed ideal in $F$, so it follows from \autoref{cor:SCBiltensor-ideals} that $I \tensor J$ defines an ideal in $E \SCBiltensor F$. To complete the proof, note that $x_0 \tensor y_0 \in (E \SCBiltensor F)_+$, and that $-x_0 \tensor y_0 \notin (E \SCBiltensor F)_+$ because $(E \SCBiltensor F)_+$ is a proper cone. In other words, $(I \tensor J)_+$ is the ray generated by $x_0 \tensor y_0$.
\end{proof}

\begin{corollary}
	\label{cor:injective-extremal-rays}
	If $\langle E,E'\rangle$, $\langle F,F'\rangle$ are dual pairs and if $E_+ \subseteq E$, $F_+ \subseteq F$ are convex cones, then
	\[ \rext(\overline{E_+}^{\,\weak}) \settensor \rext(\overline{F_+}^{\,\weak}) \subseteq \rext(\maxwedge{E_+}{F_+}). \]
\end{corollary}
\begin{corollary}
	\label{cor:injective-completed-extremal-rays}
	If $E$, $F$ are complete locally convex spaces, if $E_+ \subseteq E$, $F_+ \subseteq F$ are convex cones, and if $\alpha$ is a compatible locally convex topology on $E \tensor F$ for which the natural map $E \hattensor_\alpha F \to E \hattensor_\varepsilon F$ is injective, then
	\[ \rext(\overline{E_+}^{\,\weak}) \settensor \rext(\overline{F_+}^{\,\weak}) \subseteq \rext(\hatmaxwedge[\alpha]{E_+}{F_+}). \]
\end{corollary}
Note: if $E \hattensor_\alpha F \to E \hattensor_\varepsilon F$ is not injective, then $\hatmaxwedge[\alpha]{E_+}{F_+}$ does not have extremal rays, since it is not a proper cone (see \autoref{cor:completed-injective-proper-cone}).

In \autoref{thm:projective-extremal-rays}, we found that the extremal rays of the projective cone are precisely the tensor products of the extremal rays of the base cones. This is not true for the injective cone; the following example shows that the inclusion from \autoref{cor:injective-extremal-rays} can be strict.

\begin{example}[{cf.~\autoref{xmpl:projective-bipositive-fail-2}, \autoref{xmpl:injective-pushforward-fail-2}}]
	\label{xmpl:injective-extremal-rays}
	Let $E$ be finite\-/dimensional, and let $E_+ \subseteq E$ be a proper, generating, polyhedral cone which is \emph{not} a simplex cone.
	Then both $\minwedge{E_+\algdual}{E_+}$ and $\maxwedge{E_+\algdual}{E_+}$ are proper, generating, polyhedral cones (use that the class of proper, generating, polyhedral cones is closed under taking duals and projective tensor products). As such, they are generated by their extremal rays.
	However, it follows from \cite[Proposition 3.1]{Barker-Loewy} (see also \autoref{thm:min-equals-max} below) that $\minwedge{E_+\algdual}{E_+} \neq \maxwedge{E_+\algdual}{E_+}$, so in particular $\rext(\maxwedge{E_+\algdual}{E_+}) \neq \rext(\minwedge{E_+\algdual}{E_+}) = \rext(E_+\algdual) \settensor \rext(E_+)$.
	
	It will follow from \myautoref{cor:rank-one}{itm:injective-rank-one} below that the additional extremal directions of the injective cone must necessarily have rank $\geq 2$.
\end{example}

\begin{remark}
	\label{rmk:injective-norm-extreme-points}
	It is somewhat remarkable that the injective cone preserves extremal rays, because the injective norm does not preserve extreme points (of the closed unit ball). Indeed, if $E = F = \R^2$ with the Euclidean norm, then $E \tensor_\varepsilon F \cong \R^{2\times 2}$ with the operator norm (i.e.~the Schatten $\infty$-norm). But the extreme points of the unit ball for the operator norm are the orthogonal matrices, which in particular have full rank. In other words, no rank $1$ operator is an extreme point, so in this case $\ext(B_{E \tensor_\varepsilon F})$ is \emph{disjoint} from $\ext(B_E) \settensor \ext(B_F)$.
	
	This discrepancy can be explained as follows. In \mysecref{sec:application-to-convex-sets}, we proved that the projective unit ball preserves extreme points (at least in the finite\-/dimensional case).
	The proof used homogenization: given finite\-/dimensional normed spaces $E$ and $F$, we considered the respective ``ice cream cones'' in $E \oplus \R$ and $F \oplus \R$.
	However, the tensor product $(E \oplus \R) \tensor (F \oplus \R) \cong (E \tensor F) \oplus E \oplus F \oplus \R$ is larger than $(E \tensor F) \oplus \R$, so the projective cone is larger than the homogenization of the projective unit ball.
	In order to recover an extreme point of the projective unit ball, we had to work with a two-dimensional face of the projective cone.
	Thus, extremal rays of the projective cone do not correspond directly with extreme points of the projective unit ball.
	Apparently, the two-dimensional face used in this argument has no analogue in the injective cone.
\end{remark}

\chapter{Reasonable cross-cones}
\label{chp:reasonable}

Apart from the projective and injective cone, there are many other reasonable cones in the tensor product.
Depending on the application, other choices may be appropriate as well.
For instance, the tensor product of two spaces of hermitian matrices is again a space of hermitian matrices, but if the spaces are equipped with the positive semidefinite cone, then neither the projective nor the injective cone is equal to the positive semidefinite cone in the tensor product (see e.g.{} \cite[\S 2]{Ando}).

In this chapter, mirroring an analogous definition in the normed theory, we study the broader class of `reasonable crosscones' in the tensor product.

\begin{definition}
	\label{def:reasonable}
	Let $\langle E,E'\rangle$ and $\langle F,F'\rangle$ be dual pairs, and let $E_+ \subseteq E$ and $F_+ \subseteq F$ be convex cones.
	We say that a convex cone $\mywedge \subseteq E \tensor F$ is a \emph{reasonable cross-cone} if it satisfies the following criteria:
	\begin{enumerate}[label=(\roman*)]
		\item\label{itm:reasonable:primal} For all $x \in E_+$ and $y \in F_+$ one has $x \tensor y \in \mywedge$;
		\item\label{itm:reasonable:dual} For all $\varphi \in E_+'$ and $\psi \in F_+'$, one has $\varphi \tensor \psi \in \mywedge'$.
	\end{enumerate}
	Here $\mywedge'$ denotes the dual cone of $\mywedge$ with respect to any reasonable dual $G$ of $E \tensor F$ (that is, $E' \tensor F' \subseteq G \subseteq \SCBil(E \times F)$; see \mysecref{sec:dual-pairs}). The definition does not depend on the choice of reasonable dual, because $\varphi \tensor \psi \in E' \tensor F'$.
	
	Reasonable cross-cones in the completed tensor product $E \hattensor_\alpha F$ ($E$ and $F$ locally convex, $\alpha$ a compatible locally convex topology on $E \tensor F$) are defined analogously.
\end{definition}

The following proposition shows that a cone in $E \tensor F$ is a reasonable crosscone if and only if it lies somewhere between the projective and injective cones.

\begin{proposition}
	\label{prop:min-max-cone}
	Let $\langle E,E'\rangle$, $\langle F,F'\rangle$ be dual pairs and let $E_+ \subseteq E$, $F_+ \subseteq F$ be convex cones. Then $\minwedge{E_+}{F_+}$ and $\maxwedge{E_+}{F_+}$ are reasonable cross-cones, and $\minwedge{E_+}{F_+} \subseteq \maxwedge{E_+}{F_+}$.
	Furthermore, a convex cone $\mywedge \subseteq E \tensor F$ is a reasonable cross-cone if and only if $\minwedge{E_+}{F_+} \subseteq \mywedge \subseteq \maxwedge{E_+}{F_+}$.
\end{proposition}
\begin{proof}
	For $x \in E_+$, $y \in F_+$, $\varphi \in E_+'$, $\psi \in F_+'$ we have $\langle x , \varphi\rangle \geq 0$ and $\langle y , \psi \rangle \geq 0$, and therefore $\langle x \tensor y , \varphi \tensor \psi \rangle = \langle x , \varphi\rangle \cdot \langle y , \psi \rangle \geq 0$. It follows that $\minwedge{E_+}{F_+}$ and $\maxwedge{E_+}{F_+}$ are reasonable cross-cones and that $\minwedge{E_+}{F_+} \subseteq \maxwedge{E_+}{F_+}$.
	
	For a general convex cone $\mywedge \subseteq E \tensor F$, clearly \myautoref{def:reasonable}{itm:reasonable:primal} is equivalent to $\minwedge{E_+}{F_+} \subseteq \mywedge$, and \myautoref{def:reasonable}{itm:reasonable:dual} is equivalent to $\mywedge \subseteq \maxwedge{E_+}{F_+}$.
\end{proof}

Likewise, a convex cone $\mywedge \subseteq E \hattensor_\alpha F$ is a reasonable cross-cone if and only if $\hatminwedge[\alpha]{E_+}{F_+} \subseteq \mywedge \subseteq \hatmaxwedge[\alpha]{E_+}{F_+}$. If this is the case, then in particular $\mywedge \cap (E \tensor F)$ is a reasonable cross-cone in the algebraic tensor product $E \tensor F$.

In this chapter, we give three applications of the results from the previous chapters to arbitrary reasonable crosscones.
We show that all reasonable crosscones have the same rank one tensors whenever $E_+$ and $F_+$ are weakly closed and proper (\mysecref{sec:rank-one}), we show that ideals and extremal rays are preserved by reasonable crosscones (\mysecref{sec:reasonable-faces-ideals}), and we show that every reasonable crosscone in $E \tensor F$ is semisimple with respect to any reasonable dual space if $E_+$ and $F_+$ are semisimple (\mysecref{sec:semisimple-algebraic}).
Finally, in \mysecref{sec:semisimple-completed}, we study the related problem of semisimplicity in completed locally convex tensor products, but there things are a bit more complicated.

\section{Rank one tensors of reasonable crosscones}
\label{sec:rank-one}
The definition of reasonable crosscones is based on two criteria regarding rank one tensors in $E \tensor F$ and $E' \tensor F'$.
We show that, if $E_+$ and $F_+$ are sufficiently nice, then all reasonable crosscones contain the same rank one tensors (\autoref{cor:rank-one}).
Using this, we will classify all rank one tensors in the projective and injective cones (\autoref{prop:rank-one-proj-inj}).

If $\langle E,E'\rangle$ is a dual pair, then we say that a convex cone $E_+ \subseteq E$ is \index{convex cone!approximately generating}\emph{approximately generating} (or \emph{total}) if $\spn(E_+)$ is weakly dense in $E$.

If $\mywedge \subseteq E \tensor F$ is a convex cone, then we understand $\mywedge'$ to be the dual cone with respect to some reasonable dual $G$ of $E \tensor F$. The choice of $G$ does not matter, for we will restrict our attention to $\mywedge' \cap (E' \tensor F')$.

The following result is an extension of \cite[Theorem 3.3]{Barker-monotone}.

\begin{proposition}
	\label{prop:rank-one}
	Let $\langle E, E' \rangle$, $\langle F, F' \rangle$ be dual systems, let $E_+ \subseteq E$, $F_+ \subseteq F$ be convex cones, and let $\mywedge \subseteq E \tensor F$ be a reasonable crosscone.
	\begin{enumerate}[label=(\alph*)]
		\item\label{itm:rank-one-primal} If $E_+$ and $F_+$ are weakly closed proper cones, then a rank one tensor $x_0 \tensor y_0 \in E \tensor F$ belongs to $\mywedge$ if and only if either $x_0 \in E_+$ and $y_0 \in F_+$ or $-x_0 \in E_+$ and $-y_0 \in F_+$.
		
		\item\label{itm:rank-one-dual} If $E_+$ and $F_+$ are approximately generating, then a rank one tensor $\varphi_0 \tensor \psi_0 \in E' \tensor F'$ defines a positive linear functional on $\mywedge$ if and only if either $\varphi_0 \in E_+'$ and $\psi_0 \in F_+'$ or $-\varphi_0 \in E_+'$ and $-\psi_0 \in F_+'$.
	\end{enumerate}
\end{proposition}
\begin{proof}
	\leavevmode
	\begin{enumerate}[label=(\alph*)]
		\item ``$\Longleftarrow$''. If $x_0 \in E_+$ and $y_0 \in F_+$, then $x_0 \tensor y_0 \in \mywedge$ by definition. If $-x_0 \in E_+$ and $-y_0 \in F_+$, note that $x_0 \tensor y_0 = -x_0 \tensor -y_0 \in \mywedge$.
	
		``$\Longrightarrow$''. Let $x_0 \tensor y_0 \in \mywedge$ be of rank one (i.e.~with $x_0,y_0 \neq 0$). Weakly closed proper cones are semisimple, so the dual cones $E_+'$ and $F_+'$ separate points on $E_+$ and $F_+$, respectively. Choose $\varphi_0 \in E_+'$, $\psi_0 \in F_+'$ such that $\langle x_0, \varphi_0\rangle \neq 0$ and $\langle y_0 , \psi_0 \rangle \neq 0$. Then $\varphi_0 \tensor \psi_0$ defines a positive linear functional on $\mywedge$, so we have $\langle x_0, \varphi_0 \rangle \langle y_0 , \psi_0 \rangle  = \langle x_0 \tensor y_0 , \varphi_0 \tensor \psi_0 \rangle \geq 0$. It follows that $\langle x_0 , \varphi_0\rangle$ and $\langle y_0 , \psi_0\rangle$ have the same sign. Since $-x_0 \tensor -y_0 = x_0 \tensor y_0$, we may assume without loss of generality that $\langle x_0 , \varphi_0 \rangle , \langle y_0 , \psi_0 \rangle > 0$.
	
		If $\varphi \in E_+'$ is arbitrary, then $\varphi \tensor \psi_0$ is a positive linear functional on $\mywedge$, so we have $\langle x_0 , \varphi\rangle \langle y_0 , \psi_0 \rangle = \langle x_0 \tensor y_0 , \varphi \tensor \psi_0 \rangle \geq 0$. Since $\langle y_0 , \psi_0 \rangle > 0$, it follows that $\langle x_0 , \varphi\rangle \geq 0$ for all $\varphi \in E_+'$, which shows that $x_0 \in (E_+')' = \overline{E_+}^{\,\weak} = E_+$. Analogously, we find $y_0 \in F_+$.
		
		\item In this case $E_+$ and $F_+$ separate points on $E'$ and $F'$, so the same proof can be carried out.
		(If $\varphi_0 \tensor \psi_0 \in \mywedge'$ has rank one, then we may choose $x_0\in E_+$, $y_0 \in F_+$ such that $\langle x_0 , \varphi_0 \rangle \langle y_0 , \psi_0 \rangle > 0$, and use these to show that $\varphi_0 \in E_+'$ and $\psi_0 \in F_+'$ or $-\varphi_0 \in E_+'$ and $-\psi_0 \in F_+'$.)
		\qedhere
	\end{enumerate}
\end{proof}

\begin{corollary}
	\label{cor:rank-one}
	Let $\langle E, E' \rangle$, $\langle F, F' \rangle$ be dual systems, and let $E_+ \subseteq E$, $F_+ \subseteq F$ be convex cones.
	\begin{enumerate}[label=(\alph*)]
		\item\label{itm:rank-one-agree} If $E_+$ and $F_+$ are weakly closed proper cones, then all reasonable crosscones in $E \tensor F$ agree on the rank one tensors.
		
		\item\label{itm:injective-rank-one} The set of rank one extremal directions of the injective cone $\maxwedge{E_+}{F_+} \subseteq E \tensor F$ is given by $\rext(\overline{E_+}^{\,\weak}) \settensor \rext(\overline{F_+}^{\,\weak})$.

		\item If $E_+$ and $F_+$ are weakly closed proper cones, and if $\mywedge \subseteq E \tensor F$ is a reasonable crosscone, then the set of rank one extremal directions of $\mywedge$ is given by $\rext(E_+) \settensor \rext(F_+)$.
	\end{enumerate}
\end{corollary}
\begin{proof}
	\leavevmode
	\begin{enumerate}[label=(\alph*)]
		\item Immediate from \myautoref{prop:rank-one}{itm:rank-one-primal}.
		
		\item If $x_0 \in \overline{E_+}^{\,\weak}$ and $y_0 \in \overline{F_+}^{\,\weak}$ are extremal directions, then $x_0 \tensor y_0$ is an extremal direction of $\maxwedge{E_+}{F_+}$, by \autoref{cor:injective-extremal-rays}. For the converse, suppose that $x_0 \tensor y_0$ is a rank one extremal direction of $\maxwedge{E_+}{F_+}$. Then $E \tensor F \neq \{0\}$ (since there exist rank one tensors), so $E \neq \{0\}$ and $F \neq \{0\}$. Similarly, $\maxwedge{E_+}{F_+}$ is a proper cone (since it has extremal directions), so now it follows from \autoref{thm:injective-proper-cone} that $\overline{E_+}^{\,\weak}$ and $\overline{F_+}^{\,\weak}$ are proper cones. Since $\maxwedge{E_+}{F_+} = \maxwedge{\overline{E_+}^{\,\weak}}{\overline{F_+}^{\,\weak}}$, it follows from \ref{itm:rank-one-agree} that $x_0 \tensor y_0 \in \minwedge{\overline{E_+}^{\,\weak}}{\overline{F_+}^{\,\weak}}$. Clearly $x_0 \tensor y_0$ is automatically extremal in this (smaller) cone, so it follows from \autoref{thm:projective-extremal-rays} that $x_0$ and $y_0$ (or $-x_0$ and $-y_0$) are extremal directions of $\overline{E_+}^{\,\weak}$ and $\overline{F_+}^{\,\weak}$.
		
		\item By \ref{itm:rank-one-agree}, every rank one extremal direction of a reasonable crosscone is also an extremal direction of every smaller reasonable crosscone. By \ref{itm:injective-rank-one} and \autoref{thm:projective-extremal-rays}, the projective and injective cones have the same rank one extremal directions.
		\qedhere
	\end{enumerate}
\end{proof}

\begin{remark}
	In general $\minwedge{E_+}{F_+}$ and $\maxwedge{E_+}{F_+}$ do not agree on the rank one tensors.
	For example, if $E_+$ is not weakly closed, then all non-zero tensors in $E \tensor \R \cong E$ have rank one, but $\minwedge{E_+}{\R_{\geq 0}} = E_+$ whereas $\maxwedge{E_+}{\R_{\geq 0}} = \overline{E_+}^{\,\weak}$.
	As a more extreme example, let $E_+ = E$ and $F_+ = \{0\}$; then $\minwedge{E_+}{F_+} = \{0\}$, whereas $\maxwedge{E_+}{F_+} = E \tensor F$.
\end{remark}

Using \autoref{prop:rank-one}, we can determine exactly which rank one tensors belong to the projective and injective cones (without additional assumptions on $E_+$ and $F_+$).

\begin{proposition}
	\label{prop:rank-one-proj-inj}
	Let $\langle E, E'\rangle$, $\langle F, F' \rangle$ be dual pairs and let $E_+\subseteq E$, $F_+ \subseteq F$ be convex cones.
	\begin{enumerate}[label=(\alph*)]
		\item A rank one tensor $x_0 \tensor y_0 \in E \tensor F$ belongs to the projective cone $\minwedge{E_+}{F_+}$ if and only if at least one of the following applies:
		\begin{enumerate}[label=(\roman*)]
			\item $x_0 \in \lineal(E_+)$ and $y_0 \in \spn(F_+)$;
			\item $x_0 \in \spn(E_+)$ and $y_0 \in \lineal(F_+)$;
			\item $x_0 \in E_+$ and $y_0 \in F_+$;
			\item $-x_0 \in E_+$ and $-y_0 \in F_+$.
		\end{enumerate}
		
		\item A rank one tensor $x_0 \tensor y_0 \in E \tensor F$ belongs to the injective cone $\maxwedge{E_+}{F_+}$ if and only if at least one of the following applies:
		\begin{enumerate}[label=(\roman*)]
			\item $x_0 \in \lineal(\overline{E_+}^{\,\weak})$;
			\item $y_0 \in \lineal(\overline{F_+}^{\,\weak})$;
			\item $x_0 \in \overline{E_+}^{\,\weak}$ and $y_0 \in \overline{F_+}^{\,\weak}$;
			\item $-x_0 \in \overline{F_+}^{\,\weak}$ and $-y_0 \in \overline{F_+}^{\,\weak}$.
		\end{enumerate}
	\end{enumerate}
\end{proposition}
\begin{proof}
	\leavevmode
	\begin{enumerate}[label=(\alph*)]
		\item ``$\Longleftarrow$''. If $x_0 \in E_+$ and $y_0 \in F_+$, then clearly $x_0 \tensor y_0 \in \minwedge{E_+}{F_+}$. If $x_0 \in \lineal(E_+)$ and $y_0 \in \spn(F_+)$, then it follows from \autoref{cor:projective-lineality-space} that $x_0 \tensor y_0 \in \lineal(\minwedge{E_+}{F_+}) \subseteq \minwedge{E_+}{F_+}$. The other two cases are analogous.
		
		``$\Longrightarrow$''. Let $x_0 \tensor y_0 \in \minwedge{E_+}{F_+}$ be of rank one (i.e.~with $x_0,y_0 \neq 0$), and write $x_0 \tensor y_0 = \sum_{i=1}^k x_i \tensor y_i$ with $x_1,\ldots,x_k \in E_+$, $y_1,\ldots,y_k \in F_+$. Note that we must have $y_0 \in \spn(F_+)$: choose $\varphi \in E'$ such that $\varphi(x_0) = 1$, then
		\[ y_0 = (\varphi \tensor \id_F)(x_0 \tensor y_0) = (\varphi \tensor \id_F)\left(\sum_{i=1}^k x_i \tensor y_i\right) = \sum_{i=1}^k \varphi(x_i)y_i \in \spn(F_+). \]
		Analogously, $x_0 \in \spn(E_+)$.
		
		Let $\pi_{\lineal(E_+)} : E \to E/\lineal(E_+)$ and $\pi_{\lineal(F_+)} : F \to F/\lineal(F_+)$ be the canonical maps. Since $\lineal(E_+)$ and $\lineal(F_+)$ are ideals, the quotient cones are proper (see \mysecref{app:faces-ideals}). For notational convenience, let $x_0',\ldots,x_k'$ and $y_0',\ldots,y_k'$ denote the images of $x_0,\ldots,x_k$ and $y_0,\ldots,y_k$ in the respective quotients. Now $x_0' \tensor y_0' \in \minwedge{(E/\lineal(E_+))_+}{(F/\lineal(F_+))_+}$ has rank at most one. If it has rank zero, then $x_0 \in \lineal(E_+)$ or $y_0 \in \lineal(F_+)$, so we are done. Assume therefore that $x_0' \tensor y_0'$ has rank one.
		
		Define $X := \spn\{x_0',\ldots,x_k'\} \subseteq E/\lineal(E_+)$, and let $X_+ \subseteq X \cap (E/\lineal(E_+))_+$ be the convex cone generated by $x_1',\ldots,x_k'$.
		Then $X$ is finite\-/dimensional and $X_+$ is closed (because it is finitely generated) and proper (since it is contained in the proper cone $(E/\lineal(E_+))_+$).
		Define $Y_+ \subseteq F/\lineal(F_+)$ and $Y \subseteq F/\lineal(F_+)$ analogously.
		
		Since $x_0',\ldots,x_k'$ and $y_0',\ldots,y_k'$ belong to $X$ and $Y$, it follows that $x_0' \tensor y_0' = \sum_{i=1}^k x_i' \tensor y_i'$ holds in $X \tensor Y$, so we have $x_0' \tensor y_0' \in \minwedge{X_+}{Y_+}$. Since $X_+$ and $Y_+$ are closed and proper, it follows from \myautoref{prop:rank-one}{itm:rank-one-primal} that $x_0' \in X_+$ and $y_0' \in Y_+$ or $-x_0' \in X_+$ and $-y_0' \in Y_+$. Since the quotient maps $\pi_{\lineal(E_+)}$ and $\pi_{\lineal(F_+)}$ are bipositive (see \autoref{prop:quotient-bipositive}), it follows that $x_0 \in E_+$ and $y_0 \in F_+$ or $-x_0 \in E_+$ and $-y_0 \in F_+$.
		
		\item ``$\Longleftarrow$''. If $x_0 \in \overline{E_+}^{\,\weak}$ and $y_0 \in \overline{F_+}^{\,\weak}$, then
		\[ x_0 \tensor y_0 \in \minwedge{\overline{E_+}^{\,\weak}}{\overline{F_+}^{\,\weak}} \subseteq \maxwedge{\overline{E_+}^{\,\weak}}{\overline{F_+}^{\,\weak}} = \maxwedge{E_+}{F_+}. \]
		If $x_0 \in \lineal(\overline{E_+}^{\,\weak})$, $y_0 \in F$, then $x_0 \tensor y_0 \in \lineal(\maxwedge{E_+}{F_+}) \subseteq \maxwedge{E_+}{F_+}$, by \myautoref{cor:tensor-subspace}{itm:ts:lin-space}. The other two cases are analogous.
		
		``$\Longrightarrow$''. Let $x_0 \tensor y_0 \in \maxwedge{E_+}{F_+}$ be of rank one (i.e.~with $x_0,y_0 \neq 0$). If $x_0 \in \lineal(\overline{E_+}^{\,\weak})$ or $y_0 \in \lineal(\overline{F_+}^{\,\weak})$, then we are done, so assume $x_0 \notin \lineal(\overline{E_+}^{\,\weak})$ and $y_0 \notin \lineal(\overline{F_+}^{\,\weak})$. Since $\lineal(\overline{E_+}^{\,\weak}) = {}^\perp (E_+')$, this means that we may choose $\varphi_0 \in E_+'$, $\psi_0 \in F_+'$ such that $\langle x_0 , \varphi_0 \rangle \neq 0$ and $\langle y_0 , \psi_0 \rangle \neq 0$. Now it follows from the argument of \autoref{prop:rank-one} that either $x_0 \in \overline{E_+}^{\,\weak}$ and $y_0 \in \overline{F_+}^{\,\weak}$, or $-x_0 \in \overline{E_+}^{\,\weak}$ and $-y_0 \in \overline{F_+}^{\,\weak}$.
		\qedhere
	\end{enumerate}
\end{proof}

\autoref{prop:rank-one-proj-inj} can be paraphrased as follows: every rank one tensor in the projective or injective cone is either positive for obvious reasons (conditions \textit{(iii)} and \textit{(iv)}) or belongs to the lineality space (conditions \textit{(i)} and \textit{(ii)}).

\section{Ideals and faces of reasonable crosscones}
\label{sec:reasonable-faces-ideals}
An ideal with respect to the injective cone is also an ideal with respect to every smaller cone, and a face of the injective cone is also a face of every smaller cone containing that face.
Therefore the results from \mychpref{chp:injective} immediately give rise to the following consequences (among others).

\begin{proposition}
	\label{prop:reasonable-faces-ideals}
	Let $\langle E,E'\rangle$, $\langle F,F'\rangle$ be dual pairs, let $E_+ \subseteq E$, $F_+ \subseteq F$ be convex cones, and let $\mywedge \subseteq E \tensor F$ be a reasonable crosscone. Then:
	\begin{enumerate}[label=(\alph*)]
		\item\label{itm:reasonable-ideals} If $E_+$ and $F_+$ are weakly closed and if $I \subseteq E$, $J \subseteq F$ are ideals, then $(I \tensor J) + \lineal(\maxwedge{E_+}{F_+})$ is an ideal with respect to $\mywedge$. Additionally, if $I$ is weakly closed and $(E/I)_+$ is semisimple, or if $J$ is weakly closed and $(F/J)_+$ is semisimple, then $(I \tensor F) + (E \tensor J)$ is an ideal with respect to $\mywedge$.
		
		\item The lineality space of $\mywedge$ satisfies
		\begin{align*}
			(\lineal(E_+) \tensor \spn(F_+)) &+ (\spn(E_+) \tensor \lineal(F_+)) \subseteq  \lineal(\mywedge)\\
			& \hspace*{3em} \subseteq (\lineal(\overline{E_+}^{\,\weak}) \tensor F) + (E \tensor \lineal(\overline{F_+}^{\,\weak})).
		\end{align*}
		
		\item If $E_+$ and $F_+$ are weakly closed and if $x_0 \in E_+$, $y_0 \in F_+$ define extremal rays, then $x_0 \tensor y_0$ defines an extremal ray of $\mywedge$.
	\end{enumerate}
\end{proposition}

\section[Semisimplicity in the algebraic tensor product]{Semisimplicity of reasonable crosscones in the algebraic tensor product}
\label{sec:semisimple-algebraic}
Recall that a convex cone $E_+$ is \emph{semisimple} if it is contained in a weakly closed proper cone, or equivalently, if $E_+'$ separates points on $E$. In this section, we prove that every reasonable crosscone in $E \tensor F$ is semisimple if $E_+$ and $F_+$ are semisimple, and we determine necessary and sufficient criteria for the projective and injective cones to be semisimple. Similar results in completed locally convex tensor products will be discussed in \mysecref{sec:semisimple-completed}.

We start by setting up a partial converse, using the following proposition.

\begin{proposition}
	\label{prop:minwedge-weak-closure}
	Let $\langle E,E' \rangle$, $\langle F, F' \rangle$ be dual pairs, and let $E_+ \subseteq E$, $F_+ \subseteq F$ be convex cones. If $G$ is a reasonable dual of $E \tensor F$ and if $\mywedge \subseteq E \tensor F$ is a reasonable crosscone, then
	\[ \minwedge{\overline{E_+}^{\,\weak}}{\overline{F_+}^{\,\weak}} \subseteq \overline{\mywedge}^{\,\weak}, \]
	where $\overline{\mywedge}^{\,\weak}$ denotes the $\sigma(E \tensor F , G)$-closure of $\mywedge$.
\end{proposition}
\begin{proof}
	Let $\lintop_\ind$ denote the finest compatible topology on the tensor product $E_\weak \tensor F_\weak$ (known as the \index{topology!inductive}\emph{inductive} topology, not to be confused with the injective topology). Then the natural map $E \times F \to E \tensor F$ is separately continuous as a map $E_\weak \times F_\weak \to (E \tensor F,\lintop_\ind)$, and the dual of $(E \tensor F,\lintop_\ind)$ is $\SCBil(E_\weak \times F_\weak) = \SCBil(E \times F)$; see \cite[\S 44.1.(5)]{Kothe-II}. In particular, since $G \subseteq \SCBil(E \times F)$, it follows that $\weak = \sigma(E \tensor F , G)$ is weaker than $\lintop_\ind$. Therefore: $\overline{\mywedge}^{\,\ind} \subseteq \overline{\mywedge}^{\,\weak}$.
	
	To complete the proof, we show that $\minwedge{\overline{E_+}^{\,\weak}}{\overline{F_+}^{\,\weak}} \subseteq \overline{\mywedge}^{\,\ind}$. Since $\mywedge$ is a reasonable crosscone, we have $\minwedge{E_+}{F_+} \subseteq \mywedge \subseteq \overline{\mywedge}^{\,\ind}$. Since $E_\weak \times F_\weak \to (E \tensor F,\lintop_\ind)$ is separately continuous, for every $x_0 \in E_+$ one has $x_0 \tensor \overline{F_+}^{\,\weak} \subseteq \overline{\mywedge}^{\,\ind}$. (The inverse image of $\overline{\mywedge}^{\,\ind}$ under the map $y \mapsto x_0 \tensor y$ contains $F_+$, and therefore $\overline{F_+}^{\,\weak}$.) Then, by the same argument, for every $y_0 \in \overline{F_+}^{\,\weak}$ we have $\overline{E_+}^{\,\weak} \tensor y_0 \subseteq \overline{\mywedge}^{\,\ind}$. It follows that $\overline{E_+}^{\,\weak} \settensor \overline{F_+}^{\,\weak} \subseteq \overline{\mywedge}^{\,\ind}$, and the result follows by taking positive combinations.
\end{proof}

For clarity, let us say that $\mywedge$ is \index{convex cone!G-semisimple@$G$-semisimple}\emph{$G$-semisimple} if it is semisimple for the dual pair $\langle E \tensor F , G \rangle$ (i.e.~if $\overline{\mywedge}^{\,\sigma(E \tensor F , G)}$ is a proper cone), where $G$ is a reasonable dual of $E \tensor F$.

\begin{theorem}
	\label{thm:reasonable-semisimple}
	Let $\langle E, E' \rangle$, $\langle F, F' \rangle$ be dual pairs, let $E_+ \subseteq E$, $F_+ \subseteq F$ be convex cones, let $G$ be a reasonable dual of $E \tensor F$, and let $\mywedge \subseteq E \tensor F$ be a reasonable crosscone.
	\begin{enumerate}[label=(\alph*)]
		\item If $E_+$ and $F_+$ are semisimple, then $\mywedge$ is $G$-semisimple.
		\item\label{itm:rs:partial-converse} If $E_+ \neq \{0\}$ and $F_+ \neq \{0\}$, and if $\mywedge$ is $G$-semisimple, then $E_+$ and $F_+$ are semisimple.
	\end{enumerate}
\end{theorem}
\begin{proof}\ \par
	\begin{enumerate}[label=(\alph*)]
		\item Semisimplicity means that $\overline{E_+}^{\,\weak}$ and $\overline{F_+}^{\,\weak}$ are proper cones, so it follows from \autoref{thm:injective-proper-cone} that $\maxwedge{E_+}{F_+}$ is a proper cone. Furthermore, $\maxwedge{E_+}{F_+}$ is weakly closed (see \autoref{rmk:max-is-dual-of-min}), so it follows that $\mywedge$ is contained in a weakly closed proper cone.
		
		\item It follows from \autoref{prop:minwedge-weak-closure} that $\minwedge{\overline{E_+}^{\,\weak}}{\overline{F_+}^{\,\weak}} \subseteq \overline{\mywedge}^{\,\weak}$, where $\overline{\mywedge}^{\,\weak}$ is a proper cone (by semisimplicity). In particular, $\minwedge{\overline{E_+}^{\,\weak}}{\overline{F_+}^{\,\weak}}$ is a proper cone. By assumption, we have $\overline{E_+}^{\,\weak},\overline{F_+}^{\,\weak} \neq \{0\}$, so it follows from \autoref{thm:projective-proper-cone} that $\overline{E_+}^{\,\weak}$ and $\overline{F_+}^{\,\weak}$ must be proper cones as well. Equivalenty: $E_+$ and $F_+$ are semisimple. \qedhere
	\end{enumerate}
\end{proof}

\begin{remark}
	\label{rmk:partial-converse}
	We note that the partial converse given in \myautoref{thm:reasonable-semisimple}{itm:rs:partial-converse} is the best we can do. If one of the cones is trivial, then the outcome depends on the other cone. Indeed, let $E,F \neq \{0\}$ with convex cones $E_+ \subseteq E$, $F_+ \subseteq F$, such that $E_+ = \{0\}$ and $F_+$ is not semisimple. Then $\minwedge{E_+}{F_+} = \{0\}$, which is semisimple, but $\maxwedge{E_+}{F_+}$ is not semisimple by \autoref{thm:injective-proper-cone}.

	More can be said if we choose the cone beforehand. The injective cone is already weakly closed with respect to any reasonable dual, so \autoref{thm:injective-proper-cone} tells us exactly when $\maxwedge{E_+}{F_+}$ is semisimple. For the projective cone, we obtain necessary and sufficient criteria very similar to those in \autoref{thm:projective-proper-cone}.
\end{remark}

\begin{corollary}
	\label{cor:projective-semisimple}
	Let $\langle E , E' \rangle$, $\langle F , F' \rangle$ be dual pairs, let $E_+ \subseteq E$, $F_+ \subseteq F$ be convex cones, and let $G$ be a reasonable dual of $E \tensor F$. Then $\minwedge{E_+}{F_+}$ is $G$-semisimple if and only if $E_+ = \{0\}$, or $F_+ = \{0\}$, or both $E_+$ and $F_+$ are semisimple.
\end{corollary}
\begin{proof}
	If $E_+ = \{0\}$ or $F_+ = \{0\}$, then $\minwedge{E_+}{F_+} = \{0\}$, which is semisimple. The rest follows from \autoref{thm:reasonable-semisimple}.
\end{proof}

\begin{remark}
	\label{rmk:min-not-dense-in-max}
	Barring corner cases, we find that $\minwedge{E_+}{F_+}$ is semisimple if and only if $\maxwedge{E_+}{F_+}$ is a proper cone. It is tempting to conjecture that the projective cone is always dense in the injective cone. For locally convex lattices, Birnbaum \cite[Proposition 3]{Birnbaum} found a positive answer, but in general this is far from being true. Counterexamples have been known for a long time (e.g.~\cite[Example following Proposition 3]{Birnbaum}; \cite[Proposition 3.1]{Barker-Loewy}). Very recently, Aubrun, Lami, Palazuelos and Pl\'avala \cite{Aubrun-et-al-ii} proved that this fails for all closed, proper, generating cones in finite\-/dimensional spaces, unless at least one of the cones is a simplex cone.
	We will prove a large class of special cases of this result in \mychpref{chp:many-examples}.
\end{remark}

\begin{remark}
	Fremlin \cite{Fremlin} developed a theory of tensor products of Archimedean Riesz spaces, which was further developed by Grobler and Labuschagne \cite{Grobler-Labuschagne}, and van Gaans and Kalauch \cite{vanGaans-Kalauch} to a theory of tensor products of Archimedean cones. In this setting, the challenge is to extend the projective cone to a proper Archimedean cone. In \cite{vanGaans-Kalauch}, van Gaans and Kalauch showed that the projective tensor product of two generating Archimedean cones is always contained in a proper Archimedean cone (see \cite[Lemma 4.2]{vanGaans-Kalauch}).
	
	Our results are parallel to this. If the given cones $E_+ \subseteq E$ and $F_+ \subseteq F$ are not only Archimedean but also closed in some locally convex topology (this is a stronger assumption), then their projective tensor product is contained in a closed (hence Archimedean) proper cone. In other words, we start with a stronger assumption, and end up with a stronger conclusion.
	
	The preceding results are no substitute for the methods developed in \cite{vanGaans-Kalauch}. For example, the space $L^p[0,1]$ with $p\in (0,1)$ does not admit a non-trivial positive linear functional, so here we have an Archimedean cone which fails to be semisimple in a rather dramatic way. Consequently, our results fail to prove that the projective tensor product $\minwedge{L_+^p[0,1]}{L_+^p[0,1]}$ is contained in a proper Archimedean cone, which we know to be true by the results of \cite{vanGaans-Kalauch}. (In fact, since $L_+^p[0,1]$ is a lattice cone, this follows already from Fremlin's original result \cite[Theorem 4.2]{Fremlin}).
\end{remark}

\section[Semisimplicity in completed locally convex tensor products]{Semisimplicity of reasonable crosscones in completed locally convex tensor products}
\label{sec:semisimple-completed}
In the completed setting, semisimplicity turns out to be more subtle. This is because there is one additional requirement for the injective cone to be proper: not only do $\overline{E_+}^{\,\weak}$ and $\overline{F_+}^{\,\weak}$ need to be proper, but the natural map $E \hattensor_\alpha F \to E \hattensor_\varepsilon F$ must be injective. (See \autoref{cor:completed-injective-proper-cone}.) This leads to the following analogue of \autoref{thm:reasonable-semisimple}.

\begin{theorem}
	\label{thm:completed-semisimple}
	Let $E$, $F$ be complete locally convex spaces, $E_+ \subseteq E$, $F_+ \subseteq F$ convex cones, $\alpha$ a compatible locally convex topology on $E \tensor F$, and $\mywedge \subseteq E \hattensor_\alpha F$ a reasonable crosscone.
	\begin{enumerate}[label=(\alph*)]
		\item If $E_+$ and $F_+$ are semisimple and if $E \hattensor_\alpha F \to E \hattensor_\varepsilon F$ is injective, then $\mywedge$ is semisimple.
		
		\item If $E_+ \neq \{0\}$ and $F_+ \neq \{0\}$, and if $\mywedge$ is semisimple, then $E_+$ and $F_+$ are semisimple.
	\end{enumerate}
\end{theorem}
\begin{proof}\ \par
	\begin{enumerate}[label=(\alph*)]
		\item It follows from the assumptions and \autoref{cor:completed-injective-proper-cone} that the injective cone $\hatmaxwedge[\alpha]{E_+}{F_+}$ is proper, so $\mywedge$ is contained in a closed proper cone.
		
		\item If $\mywedge$ is semisimple, then in particular $\mywedge \cap (E \tensor F)$ is semisimple, so the result follows from \myautoref{thm:reasonable-semisimple}{itm:rs:partial-converse}. \qedhere
	\end{enumerate}
\end{proof}

\noindent
The gap between the necessary and sufficient conditions in \autoref{thm:completed-semisimple} is even larger than it was in \autoref{thm:reasonable-semisimple}.
We show that this gap is related to the approximation property. For simplicity, we restrict our attention to Banach spaces.

We recall some generalities. Let $\alpha$ be a finitely generated tensor norm, then we say (following \cite[\S 21.7]{Defant-Floret}) that a Banach space $E$ has the \index{alpha-approximation property@$\alpha$-approximation property}\emph{$\alpha$-approximation property} if for all Banach spaces $F$ the natural map $E \hattensor_\alpha F \to E \hattensor_\varepsilon F$ is injective. The $\pi$-approximation property (where $\pi$ denotes the projective tensor norm) is simply called the \index{approximation property}\emph{approximation property}. If a Banach space $E$ has the approximation property, then $E$ also has the $\alpha$-approximation property for every finitely generated tensor norm $\alpha$ (see \cite[Proposition 17.20]{Defant-Floret}).

Some tensor norms $\alpha$ have the property that every Banach space has the $\alpha$\-/approximation property (and therefore $E \hattensor_\alpha F \to E \hattensor_\varepsilon F$ is always injective). One of these is the injective tensor norm $\varepsilon$, for obvious reasons. More generally, this is true for every totally accessible tensor norm $\alpha$; see \cite[Proposition 21.7(2)]{Defant-Floret}. This includes all tensor norms which are (left and right) injective; see \cite[Proposition 21.1(3)]{Defant-Floret}.

\begin{corollary}
	\label{cor:projective-completed-semisimple}
	Let $E$ and $F$ be Banach spaces, let $E_+ \subseteq E$, $F_+ \subseteq F$ be convex cones, and let $\alpha$ be a finitely generated tensor norm. If $E$ or $F$ has the $\alpha$-approximation property, then the projective cone $\hatminwedge[\alpha]{E_+}{F_+} \subseteq E \hattensor_\alpha F$ is semisimple if and only if $E_+ = \{0\}$, or $F_+ = \{0\}$, or both $E_+$ and $F_+$ are semisimple.
\end{corollary}
\begin{proof}
	The $\alpha$-approximation property guarantees that $E \hattensor_\alpha F \to E \hattensor_\varepsilon F$ is injective. If $E_+ = \{0\}$ or $F_+ = \{0\}$, then $\hatminwedge[\pi]{E_+}{F_+} = \{0\}$. The other cases follow from \autoref{thm:completed-semisimple}.
\end{proof}

The proofs of \autoref{cor:projective-semisimple} and \autoref{cor:projective-completed-semisimple} rely on the injective cone to draw conclusions about the projective cone. However, in general these two can be far apart (see \autoref{rmk:min-not-dense-in-max}). If the map $E \hattensor_\alpha F \to E \hattensor_\varepsilon F$ is not injective, then the injective cone $\hatmaxwedge[\alpha]{E_+}{F_+}$ is not proper, but that does not mean that the projective cone $\hatminwedge[\alpha]{E_+}{F_+}$ cannot be semisimple. This leaves open the following interesting question, to which we do not know the answer:

\begin{question}
	\label{q:completed-semisimple}
	Let $E,F$ be real Banach spaces, and let $E_+ \subseteq E$, $F_+ \subseteq F$ be closed proper cones. Is the projective cone $\hatminwedge[\pi]{E_+}{F_+}$ in the completed projective tensor product $E \hattensor_\pi F$ necessarily contained in a closed proper cone?
\end{question}

Equivalently: if the positive continuous linear functionals separate points on $E$ and $F$, then do the positive continuous bilinear forms $E \times F \to \R$ separate points on $E \hattensor_\pi F$?

By \autoref{cor:projective-semisimple}, the positive continuous bilinear forms separate points on $E \tensor_\pi F$, but that is not enough. Furthermore, if $E$ or $F$ has the approximation property, then the positive bilinear forms of rank one already separate points on $E \hattensor_\pi F$, but this technique does not work in the absence of the approximation property.

\chapter{Basic additional properties in the finite-dimensional case}
\label{chp:finite-dimensional}

In the previous chapters, we studied tensor products of convex cones in general (possibly infinite-dimensional) real vector spaces.
In that setting, not many results had been known in the literature, and several basic questions had been unanswered.
In the finite-dimensional setting, the situation is very different.
In a different part of the literature, completely separate from the functional analysis literature, questions around tensor products of closed, proper and generating cones have been studied by many authors in a variety of different fields, such as linear algebra, operator theory, geometry, approximation theory, and theoretical physics.
(For a comprehensive overview of these connections, see \mysecref{sec:intro:background}.)

In this chapter, we give a brief overview of the most important additional properties in the finite-dimensional case.
We give new, streamlined proofs of several known results, and we extend them to general convex cones in finite-dimensional spaces (i.e.{} cones which are not necessarily closed, proper, or generating).
In \mysecref{sec:separable}, we show that the projective and injective cones can be interpreted as certain cones of positive operators, at least when $E_+$ and $F_+$ are closed.
In \mysecref{sec:closed-projective-cone}, we show that the closure $\overline{\minwedge{E_+}{F_+}}$ of the projective cone $\minwedge{E_+}{F_+}$ is equal to the projective cone $\minwedge{\overline{E_+}}{\overline{F_+}}$.
Finally, in \mysecref{sec:retracts}, we look more closely at the concept of retracts, which was already covered briefly in \mysecref{sec:convex-cones}.
Here we prove some basic properties of retracts, and we provide many examples of retracts in standard cones, which we will use in the next chapter.

The results from this chapter will be used extensively in the next chapter, where we give many examples where the projective cone is closed and properly contained in the injective cone.

\section{Additional notation}
\label{sec:notation-ii}

We follow notation from \mychpref{chp:preliminaries}; see also the glossary of notation on \autopageref{symbols}.

In this chapter (and the next), all vector spaces will be finite\-/dimensional.
Recall that the \index{lineality space}\emph{lineality space} of a convex cone $\mywedge \subseteq E$ is the subspace \glsadd{lineality space}$\lineal(\mywedge) = \mywedge \cap -\mywedge$.
We say that $\mywedge$ is \index{convex cone!proper}\emph{proper} if $\lineal(\mywedge) = \{0\}$, and \index{convex cone!generating}\emph{generating} if $\mywedge - \mywedge = E$.

The \index{dual cone}\index{dual cone!algebraic}\emph{\textup(algebraic\textup) dual cone} \glsadd{algebraic dual cone}$\mywedge\algdual \subseteq E\algdual$ is the set of positive linear functionals:
\[ \mywedge\algdual := \big\{\varphi \in E\algdual \, : \, \langle x,\varphi\rangle \geq 0 \ \text{for all $x \in \mywedge$}\big\}. \]
This is a closed cone in the (finite\-/dimensional) space $E\algdual$, and the natural isomorphism $E\algdualdual \cong E$ identifies the double dual cone $\mywedge\algdualdual$ with the closure $\overline{\mywedge}$.

If $\mywedge$ is a convex cone, then a \index{base of a convex cone}\emph{base of $\mywedge$} is a convex subset $\mathcal B \subseteq \mywedge \setminus \{0\}$ such that each $x \in \mywedge \setminus \{0\}$ can be written uniquely as $x = \lambda b$ with $\lambda > 0$ and $b \in \mathcal B$.
If $\mywedge$ is generating, then the bases of $\mywedge$ are in bijective correspondence with the strictly $\mywedge$-positive linear functionals on $E$ (see \cite[Theorem 1.47]{Aliprantis-Tourky}).
Furthermore, every closed proper cone in a finite\-/dimensional space admits a compact base (e.g.~\cite[Corollary 3.8]{Aliprantis-Tourky}).

We say that a convex cone $E_+ \subseteq E$ is a \index{simplex cone}\emph{simplex cone} (or \index{Yudin cone|see {simplex cone}}\emph{Yudin cone}) if it is generated by a basis of $E$, or equivalently, if every base of $E_+$ is a simplex.
A simplex cone turns $E$ into a Dedekind complete Riesz space (see \cite[Theorem 3.17]{Aliprantis-Tourky}).
Furthermore, a cone in a finite\-/dimensional space is a simplex cone if and only if it is a closed lattice cone (see \cite[Theorem 3.21]{Aliprantis-Tourky}).

We fix notation for a number of standard cones.
For $n \geq 1$, we let \glsadd{Lorentz}$\Lorentz{n} \subseteq \R^n$ denote the $n$\nobreakdash-dimensional \index{second-order cone}\emph{second-order cone} (or \index{Lorentz cone|see {second-order cone}}\emph{Lorentz cone}, or \index{ice cream cone|see {second-order cone}}\emph{ice cream cone}),
\[ \Lorentz{n} := \big\{ (x_1,\ldots,x_n) \in \R^n \, : \, \sqrt{x_1^2 + \cdots + x_{n-1}^2} \leq x_n \big\}. \]
(By convention, $\Lorentz{1}$ is just the standard cone $\R_+ \subseteq \R$.)
Furthermore, let \glsadd{symm}$\symm{n} \subseteq \R^{n\times n}$ and \glsadd{herm}$\herm{n} \subseteq \C^{n \times n}$ denote the spaces of real symmetric and complex hermitian $n \times n$ matrices, respectively.
We denote the respective \index{positive semidefinite cone}positive semidefinite cones by \glsadd{symm+}$\symm{n}_+$ and \glsadd{herm+}$\herm{n}_+$:
\begin{align*}
	\symm{n}_+ &:= \{\text{$n\times n$ real positive semidefinite matrices}\};\\\noalign{\smallskip}
	\herm{n}_+ &:= \{\text{$n \times n$ complex positive semidefinite matrices}\}.
\end{align*}
Recall that $\symm{2}_+$ is isomorphic with the Lorentz cone $\Lorentz{3}$, for instance via the isomorphism
\[ \symm{2} \to \R^3,\qquad \begin{pmatrix} a & b \\ b & c \end{pmatrix} \mapsto (a - c , 2b , a + c). \]
(Use that $A \in \symm{2}$ is positive semidefinite if and only if $\tr(A) \geq 0$ and $\det(A) \geq 0$.)

\section{Simplex-factorable positive linear maps}
\label{sec:separable}
For the remainder of this chapter (and the next), it will be convenient to reformulate questions regarding the projective and injective cones in terms of positive linear operators.

If $E$ and $F$ are preordered by convex cones $E_+ \subseteq E$, $F_+ \subseteq F$, then a positive linear operator $T : E \to F$ is called \index{positive linear map!separable}\emph{separable}%
	\hair\footnote{The terminology was introduced by other authors in connection with quantum theory (e.g.~\cite{Hildebrand-descriptions, Aubrun-et-al-i}).}
if it can be written as $T = \sum_{i=1}^k \varphi_i \tensor y_i$, where $\varphi_1,\ldots,\varphi_k \in E_+\algdual$ are positive linear functionals and $y_1,\ldots,y_k \in F_+$ are positive elements.

If $G$ is another finite\-/dimensional real vector space, preordered by a convex cone $G_+ \subseteq G$, then we say that a positive linear map $T : E \to F$ \emph{factors through} $G$ is there exist positive linear maps $R : E \to G$ and $S : G \to F$ such that $T = S \circ R$.
We say that $T$ \emph{factors through a simplex cone}%
	\hair\footnote{More accurately, we should say that $T$ \emph{factors through a finite\-/dimensional Archimedean Riesz space}, but this is too wordy.}
if it factors through some $\R^n$, ordered by the standard cone $\R_{\geq 0}^n$.

\begin{proposition}
	\label{prop:separable-maps-factor}
	A positive linear map $T : E \to F$ is separable if and only if it factors through a simplex cone $\R_{\geq 0}^n$. If this is the case, one may take $n \leq \dim(E) \times \dim(F)$.
\end{proposition}
\begin{proof}
	``$\Longrightarrow$''. Write $T = \sum_{i=1}^k \varphi_i \tensor y_i$ with $\varphi_1,\ldots,\varphi_k \in E_+\algdual$, $y_1,\ldots,y_k \in F_+$.
	By Carath\'eodory's theorem for cones (see e.g.{} \cite[Corollary 17.1.2]{Rockafellar}), we may assume without loss of generality that $k \leq \dim(E\algdual \tensor F) = \dim(E) \times \dim(F)$.
	Now define $R : E \to \R^k$ and $S : \R^k \to F$ by
	\begin{align*}
		R(x) &:= (\varphi_1(x),\ldots,\varphi_k(x));\\\noalign{\smallskip}
		S(\lambda_1,\ldots,\lambda_k) &:= \lambda_1 y_1 + \cdots + \lambda_k y_k.
	\end{align*}
	Then $R$ and $S$ are positive ($\R^k$ equipped with the standard cone $\R_{\geq 0}^k$), and $T = S \circ R$.
	
	``$\Longleftarrow$''. Suppose that $T$ factors as
	\begin{center}
		\begin{tikzcd}
			E \arrow[r, "R"] & \R^n \arrow[r, "S"] & F,
		\end{tikzcd}
	\end{center}
	with $R$ and $S$ positive ($\R^n$ equipped with the standard cone $\R_{\geq 0}^n$). Let $e_1,\ldots,e_n \in \R^n$ denote the standard basis of $\R^n$, and $e_1\algdual,\ldots,e_n\algdual$ the corresponding dual basis. Define $\varphi_1,\ldots,\varphi_n \in E\algdual$ and $y_1,\ldots,y_n \in F$ by setting $\varphi_i := e_i\algdual \circ R$ and $y_i := S(e_i)$. Then $\varphi_1,\ldots,\varphi_n \in E_+\algdual$, $y_1,\ldots,y_n \in F_+$, and $T = \sum_{i=1}^n \varphi_i \tensor y_i$.
\end{proof}

\begin{corollary}
	\label{cor:separable-maps-factor}
	If $E_+$ and $F_+$ are closed, then:
	\begin{enumerate}[label=(\alph*)]
		\item $\maxwedgeFIN{E_+\algdual}{F_+}$ is the set of all positive linear maps $E \to F$;
		\item $\minwedgeFIN{E_+\algdual}{F_+}$ is the set of all positive linear maps $E \to F$ that factor through a simplex cone.
	\end{enumerate}
\end{corollary}
\begin{proof}
	It is well known that $\maxwedgeFIN{E_+\algdual}{F_+}$ can be identified with the cone of linear maps $T : E \to F$ satisfying $T[\overline{E_+}] \subseteq \overline{F_+}$ (see \autoref{rmk:injective-property}). Since $E_+$ and $F_+$ are closed, these are simply the positive linear maps $E \to F$.
	
	Under this identification, it is clear from the definition that $\minwedgeFIN{E_+\algdual}{F_+}$ is the subset of separable positive linear maps $E \to F$. By \autoref{prop:separable-maps-factor} these are precisely the positive linear maps $E \to F$ that factor through a simplex cone.
\end{proof}

\section{The closure of the projective cone}
\label{sec:closed-projective-cone}
In this section, we prove that the closure of the projective cone $\minwedgeFIN{E_+}{F_+}$ is equal to the projective tensor product of $\overline{E_+}$ and $\overline{F_+}$. In particular, the projective tensor product of closed convex cones is closed. In the case where $E_+$ and $F_+$ are closed, proper, and generating, this was already established by Tam \cite{Tam-projective-closed}.

The proof is carried out in three steps. First we prove the result for closed proper cones, thereby giving another proof of the aforementioned result by Tam. Secondly, we extend this to all closed convex cones by decomposing a closed convex cone as the sum of a closed proper cone and a subspace. After this, it will be relatively simple to deduce the general formula for the closure of the projective cone.

\subsection{The projective tensor product of closed proper cones}
Recall that a \index{base of a convex cone}\emph{base} (of $E_+$) is a convex subset $\mathcal B \subseteq E_+$ with $0 \notin \mathcal B$ such that every $x \in E_+ \setminus \{0\}$ can be written uniquely as $x = \lambda b$ with $\lambda > 0$ and $b \in \mathcal B$. A key property is that a subset $\mathcal B \subseteq E_+$ is a base if and only if there is a strictly positive linear functional $f : E \to \R$ such that $\mathcal B = \{x \in E_+ \, : \, f(x) = 1\}$; see e.g.~\cite[Theorem 1.47]{Aliprantis-Tourky}.

Not every proper cone has a base. However, every closed proper cone in a finite\-/dimensional space has a compact base (e.g.~\cite[Corollary 3.8]{Aliprantis-Tourky}), and conversely the convex cone generated by a compact convex set $S \subseteq \R^n \setminus \{0\}$ is a closed proper cone (e.g.~\cite[Lemma 3.12]{Aliprantis-Tourky}).

\begin{proposition}
	\label{prop:bases}
	Let $E,F$ be real vector spaces, and let $E_+ \subseteq E$, $F_+ \subseteq F$ be convex cones having bases $\mathcal B_{E_+} \subseteq E_+$, $\mathcal B_{F_+} \subseteq F_+$. Then $\conv(\mathcal B_{E_+} \tensor \mathcal B_{F_+})$ is a base of $\minwedgeFIN{E_+}{F_+}$.
\end{proposition}
\begin{proof}
	Let $f : E \to \R$ and $g : F \to \R$ be strictly positive linear functionals such that $\mathcal B_{E_+} = \{x \in E_+ \, : \, f(x) = 1\}$ and $\mathcal B_{F_+} = \{y \in F_+ \, : \, g(y) = 1\}$. Then $f \tensor g$ is a strictly positive linear functional on $E \tensor F$ (with respect to the projective cone), and we have
	\[ \mathcal B_{E_+} \tensor \mathcal B_{F_+} \subseteq \{z \in \minwedgeFIN{E_+}{F_+} \, : \, (f \tensor g)(z) = 1\}. \]
	Since the right-hand side is convex, it follows that
	\begin{equation}
		\conv(\mathcal B_{E_+} \tensor \mathcal B_{F_+}) \subseteq \{z \in \minwedgeFIN{E_+}{F_+} \, : \, (f \tensor g)(z) = 1\}. \label{eqn:bases}
	\end{equation}
	On the other hand, every non-zero element of $\minwedgeFIN{E_+}{F_+}$ can be written as a positive multiple of an element in $\conv(\mathcal B_{E_+} \tensor \mathcal B_{F_+})$, so we must have equality in \eqref{eqn:bases}.
\end{proof}

\begin{corollary}[{\cite{Tam-projective-closed}}]
	\label{cor:projective-proper-closed}
	If $E_+\subseteq E$, $F_+ \subseteq F$ are closed proper cones, then $\minwedgeFIN{E_+}{F_+}$ is also a closed proper cone.
\end{corollary}
\begin{proof}
	Choose compact bases $\mathcal B_{E_+} \subseteq E_+$, $\mathcal B_{F_+} \subseteq F_+$. By \autoref{prop:bases}, $\conv(\mathcal B_{E_+} \tensor \mathcal B_{F_+})$ is a base for $\minwedgeFIN{E_+}{F_+}$. In particular, $0 \notin \conv(\mathcal B_{E_+} \tensor \mathcal B_{F_+})$.
	
	The natural map $E \times F \to E \tensor F$ is continuous, so $\mathcal B_{E_+} \tensor \mathcal B_{F_+}$ is compact in $E \tensor F$. Since the convex hull of a compact set in $\R^n$ is also compact (e.g.~\cite[Theorem 3.20(d)]{Rudin}), it follows that $\conv(\mathcal B_{E_+} \tensor \mathcal B_{F_+})$ is compact. Thus, $\minwedgeFIN{E_+}{F_+}$ is generated by a compact convex set not containing $0$, so it follows that $\minwedgeFIN{E_+}{F_+}$ is a closed proper cone.
\end{proof}

\begin{remark}
	\label{rmk:another-way}
	The proof of \autoref{cor:projective-proper-closed} shows directly that $\minwedgeFIN{E_+}{F_+}$ is proper whenever $E_+$ and $F_+$ are closed proper cones. As such, this provides yet another way to prove that the projective tensor product of proper cones is always proper, in addition to the different ways discussed in \autoref{rmk:alternative-proofs}.
\end{remark}

\subsection{The projective tensor product of closed convex cones}
In order to extend \autoref{cor:projective-proper-closed} to all closed convex cones, we decompose each of $E_+$ and $F_+$ as the sum of a closed proper cone and a subspace.
The (straightforward) proof of the following classical result is omitted.
\begin{proposition}
	\label{prop:wedge-complement}
	Let $E$ be a finite\-/dimensional, and let $E_+ \subseteq E$ be a closed convex cone. Let $\lineal(E_+) := E_+ \cap -E_+$ be the lineality space of $E_+$, and let $\lineal(E_+)^\perp$ be any complementary subspace of $\lineal(E_+)$. Then $\lineal(E_+)^\perp_+ := \lineal(E_+)^\perp \cap E_+$ is a closed proper cone, and one has
	\[ E_+ = \lineal(E_+) + \lineal(E_+)^\perp_+. \]
\end{proposition}
\noindent
Conversely, the sum of a closed proper cone and a subspace need not be closed. (Example: let $E_+$ be the cone generated by $\{(x,y,1) \in \R^3 \, : \, (x - 1)^2 + y^2 \leq 1\}$, and $X := \spn\{(0,0,1)\} \subseteq \R^3$.) However, we have the following partial converse of \autoref{prop:wedge-complement}.

\begin{proposition}
	\label{prop:sum-complementary-subspace}
	Let $E$ be a finite\-/dimensional, let $E_+ \subseteq E$ be a closed convex cone, and let $X \subseteq E$ be a subspace. If $\spn(E_+) \cap X = \{0\}$, then $E_+ + X$ is a closed convex cone.
\end{proposition}
\begin{proof}
	Extend $\spn(E_+)$ to a complementary subspace $X^\perp$ of $X$. Let $P : E \to X^\perp$ be the projection $x + x^\perp \mapsto x^\perp$. Then $P$ is continuous, and $E_+ + X = P^{-1}[E_+]$, so $E_+ + X$ is closed.
\end{proof}

The preceding propositions give us a way to decompose the cones and later put them back together. To see what happens when we lift the pieces separately, we use the following observation.

\begin{proposition}
	\label{prop:wedge-subspace}
	Let $E,F$ be real vector spaces, and let $E_+ \subseteq E$, $F_+ \subseteq F$ be convex cones. If at least one of $E_+$ and $F_+$ is a subspace, then $\minwedgeFIN{E_+}{F_+}$ is a subspace as well.
\end{proposition}
\begin{proof}
	A convex cone $G_+$ is a subspace precisely when one has $s \in G_+$ if and only if $-s \in G_+$. This property is preserved by the projective tensor product.
\end{proof}

We can now extend \autoref{cor:projective-proper-closed} to all closed convex cones.

\begin{theorem}
	\label{thm:projective-cone-closed}
	Let $E,F$ be finite\-/dimensional, and let $E_+ \subseteq E$, $F_+ \subseteq F$ be closed convex cones. Then $\minwedgeFIN{E_+}{F_+}$ is closed as well.
\end{theorem}
\begin{proof}
	Choose complementary subspaces $\lineal(E_+)^\perp \subseteq E$ and $\lineal(F_+)^\perp \subseteq F$ of $\lineal(E_+)$ and $\lineal(F_+)$. Then, by \autoref{prop:wedge-complement}, we have $E_+ = \lineal(E_+) + \lineal(E_+)^\perp_+$ and $F_+ = \lineal(F_+) + \lineal(F_+)^\perp_+$, with $\lineal(E_+)^\perp_+$ and $\lineal(F_+)^\perp_+$ closed proper cones. It follows that
	\begin{align*}
		\minwedgeFIN{E_+}{F_+} \, =& \, \overbrace{\big(\minwedgeFIN{\lineal(E_+)}{\lineal(F_+)}\big)}^{\text{subspace}} + \overbrace{\big(\minwedgeFIN{\lineal(E_+)}{\lineal(F_+)^\perp_+}\big)}^{\text{subspace}} + \overbrace{\big(\minwedgeFIN{\lineal(E_+)^\perp_+}{\lineal(F_+)}\big)}^{\text{subspace}} \\\noalign{\smallskip}
		&+ \underbrace{\big(\minwedgeFIN{\lineal(E_+)^\perp_+}{\lineal(F_+)^\perp_+}\big)}_{\text{closed proper cone}}.
	\end{align*}
	(The first three terms are subspaces by \autoref{prop:wedge-subspace}; the fourth is a closed proper cone by \autoref{cor:projective-proper-closed}.)
	The three subspaces in the preceding formula are contained in the subspace $(\lineal(E_+) \tensor \lineal(F_+)) + (\lineal(E_+) \tensor \lineal(F_+)^\perp) + (\lineal(E_+)^\perp \tensor \lineal(F_+))$, whereas $\minwedgeFIN{\lineal(E_+)^\perp_+}{\lineal(F_+)^\perp_+}$ is a closed proper cone contained within $\lineal(E_+)^\perp \tensor \lineal(F_+)^\perp$. These containing subspaces are complementary, so it follows from \autoref{prop:sum-complementary-subspace} that $\minwedgeFIN{E_+}{F_+}$ is closed.
\end{proof}

\begin{remark}
	We should point out that nothing like \autoref{thm:projective-cone-closed} is true in the infinite\-/dimensional setting.
	In fact, the projective tensor product of closed proper cones in Banach spaces might not even be Archimedean (see e.g.~\cite[Remark 3.12]{Paulsen-Todorov-Tomforde}).
\end{remark}

\subsection{The closure of the projective cone; duality}
Using \autoref{thm:projective-cone-closed}, it is now relatively easy to prove the following.

\begin{theorem}
	\label{thm:projective-closure}
	Let $E$ and $F$ be finite\-/dimensional real vector spaces, and let $E_+ \subseteq E$, $F_+ \subseteq F$ be convex cones. Then the closure of the projective cone $\minwedgeFIN{E_+}{F_+}$ is the projective cone $\minwedgeFIN{\overline{E_+}}{\overline{F_+}}$.
\end{theorem}
\begin{proof}
	``$\supseteq$''. Given $x \in \overline{E_+}$, $y \in \overline{F_+}$, choose sequences $\{x_n\}_{n=1}^\infty$ and $\{y_n\}_{n=1}^\infty$ in $E_+$ and $F_+$ converging to $x$ and $y$, respectively. Then $x \tensor y = \lim_{n\to\infty} x_n \tensor y_n \in \overline{\minwedgeFIN{E_+}{F_+}}$.
	(Alternatively, use \autoref{prop:minwedge-weak-closure}.)
	
	``$\subseteq$''. Evidently, $\minwedgeFIN{E_+}{F_+} \subseteq \minwedgeFIN{\overline{E_+}}{\overline{F_+}}$.
	By \autoref{thm:projective-cone-closed}, $\minwedgeFIN{\overline{E_+}}{\overline{F_+}}$ is closed, so we also have $\overline{\minwedgeFIN{E_+}{F_+}} \subseteq \minwedgeFIN{\overline{E_+}}{\overline{F_+}}$.
\end{proof}

Other consequences of \autoref{thm:projective-cone-closed} include the following.

\begin{corollary}
	\label{cor:closures}
	Let $E,F$ be finite\-/dimensional, and let $E_+ \subseteq E$, $F_+ \subseteq F$ be convex cones. Then:
	\begin{enumerate}[label=(\alph*)]
		\item $\minwedgeFIN{E_+}{F_+}$ is dense in $\maxwedgeFIN{E_+}{F_+}$ if and only if $\minwedgeFIN{\overline{E_+}}{\overline{F_+}} = \maxwedgeFIN{\overline{E_+}}{\overline{F_+}}$;
		\item\label{itm:closed-duality} $(\maxwedgeFIN{E_+}{F_+})\algdual = \minwedgeFIN{E_+\algdual}{F_+\algdual}$.
	\end{enumerate}
\end{corollary}
\begin{proof}\ \par
	\begin{enumerate}[label=(\alph*)]
		\item It is clear from the definition that $\maxwedgeFIN{E_+}{F_+} = \maxwedgeFIN{\overline{E_+}}{\overline{F_+}}$, and that this cone is always closed.
		Thus, the conclusion follows immediately from \autoref{thm:projective-closure}.
		
		\item We have $(\minwedgeFIN{E_+\algdual}{F_+\algdual})\algdual = \maxwedgeFIN{E_+\algdualdual}{F_+\algdualdual} = \maxwedgeFIN{\overline{E_+}}{\overline{F_+}}$. Using again that $\maxwedgeFIN{E_+}{F_+} = \maxwedgeFIN{\overline{E_+}}{\overline{F_+}}$, we find that $(\minwedgeFIN{E_+\algdual}{F_+\algdual})\algdual = \maxwedgeFIN{E_+}{F_+}$. By \autoref{thm:projective-cone-closed}, $\minwedgeFIN{E_+\algdual}{F_+\algdual}$ is closed, so the result follows by duality.
		\qedhere
	\end{enumerate}
\end{proof}

In other words, if the spaces are finite\-/dimensional and the cones are closed, then we have full duality between the projective and injective cones.

\section{Retracts}
\label{sec:retracts}
In \mychpref{chp:many-examples}, we will look at the problem of determining whether or not $\minwedgeFIN{E_+}{F_+}$ and $\maxwedgeFIN{E_+}{F_+}$ coincide.
This problem can sometimes be reduced to lower dimensional spaces by using retracts.

Let $(F,F_+)$ be a finite\-/dimensional preordered vector space. Then a subspace $E \subseteq F$ is called an \index{order retract}\emph{order retract} if there exists a positive projection $F \to E$.
More generally, another preordered space $(G,G_+)$ is \emph{isomorphically an order retract} if there exist positive linear maps $T : G \to F$ and $S : F \to G$ such that $S \circ T = \id_G$.
Note that in this case $T$ is automatically bipositive (i.e.~a pullback) and $S$ is automatically a pushforward, and $\ran(T) \subseteq F$ is a retract of $F$ which is order isomorphic to $G$.

For simplicity, we shall omit the word \emph{order} when talking about retracts, for there is minimal chance of confusion with other types of retracts (e.g.~from topology).

Although retracts do not appear to be a very common notion in the theory of ordered vector spaces, some of the results from this section were discovered independently by Aubrun, Lami and Palazuelos \cite{Aubrun-et-al-i}.
	
\begin{remark}
	\label{rmk:retracts}
	Some basic properties of retracts:
	\begin{enumerate}[label=(\alph*)]
		\item if $E$ is a retract of $(F,F_+)$ and $F_+$ is a proper cone, then $E_+ := E \cap F_+$ is a proper cone (after all, $E_+$ is a subcone);
		\item\label{itm:retract:generating} if $E$ is a retract of $(F,F_+)$ and $F_+$ is generating, then $E_+$ is generating in $E$ (after all, there exists a surjective positive operator $F \to E$); 
		\item if $(E,E_+)$ is isomorphically a retract of $(F,F_+)$, and if $(F,F_+)$ is isomorphically a retract of $(G,G_+)$, then $(E,E_+)$ is isomorphically a retract of $(G,G_+)$.
		\item\label{itm:retract:dual} If $(E,E_+)$ is isomorphically a retract of $(F.F_+)$, then $(E\algdual,E_+\algdual)$ is isomorphically a retract of $(F\algdual,F_+\algdual)$. After all, if $T : E \to F$ and $S : F \to E$ are positive linear maps with $\id_E = S \circ T$, then $S\algdual : E\algdual \to F\algdual$ and $T\algdual : F\algdual \to E\algdual$ are positive linear maps with
		\[ \id_{E\algdual} = (\id_E)\algdual = (S \circ T)\algdual = T\algdual \circ S\algdual. \]
	\end{enumerate}
\end{remark}

By \myautoref{rmk:retracts}{itm:retract:generating}, if $F_+ \subseteq F$ is generating, then a retract $E \subseteq F$ is uniquely determined by its positive part $E_+ := E \cap F_+$, so instead of saying that \emph{$E$ is a retract of $(F,F_+)$} we will simply say that \emph{$E_+$ is a retract of $F_+$}.

\begin{example}
	\label{xmpl:retracts}
	We present some examples of retracts.
	\begin{enumerate}[label=(\alph*)]
		\item\label{itm:retract:ray} If $E$ is finite\-/dimensional and $E_+$ is a closed and proper convex cone, then every (not necessarily extremal) ray in $E_+$ is a retract. Indeed, let $x_0 \in E_+ \setminus \{0\}$ be arbitrary, and let $\varphi_0 \in E_+\algdual$ be a strictly positive linear functional. Then $\frac{1}{\varphi_0(x_0)}\cdot \varphi_0 \tensor x_0$ defines a positive projection onto $\spn(x_0)$.
		
		\item If $n \leq m$, then $\R_{\geq 0}^n$ is a retract of $\R_{\geq 0}^m$, for instance via the maps $\R^n \to \R^m$ and $\R^m \to \R^n$ given respectively by padding with zeroes and projecting onto the first $n$ coordinates. Although \ref{itm:retract:ray} shows that these are not the only retracts, we will show in \autoref{lem:simplex} that every retract of a simplex cone is once again a simplex cone.
		
		\item\label{itm:retracts:PSD} In the same manner, if $n \leq m$, then $\symm{n}_+$ is a retract of $\symm{m}_+$, and $\herm{n}_+$ is a retract of $\herm{m}_+$.
		
		\item\label{itm:retracts:Lorentz} If $n \leq m$, then $\Lorentz{n}$ is a retract of $\Lorentz{m}$ via the maps $T : \R^n \to \R^m$, $S : \R^m \to \R^n$ given by
		\begin{align*}
			T(x_1,\ldots,x_n) &= (x_1,\ldots,x_{n-1},0,\ldots,0,x_n);\\\noalign{\smallskip}
			S(y_1,\ldots,y_m) &= (y_1,\ldots,y_{n-1},y_m).
		\end{align*}

		\item $\R_{\geq 0}^n$ is a retract of $\symm{n}_+$ via the map $T : \R^n \to \symm{n}$ that maps $x$ to the diagonal matrix whose entries are specified by $x$, and the map $S : \symm{n} \to \R^n$ that maps $A$ to the diagonal of $A$.
		
		\item\label{itm:retracts:symm-herm} $\symm{n}_+$ is a retract of $\herm{n}_+$, via the maps $T : \symm{n} \to \herm{n}$, $A \mapsto A$ and $S : \herm{n} \to \symm{n}$, $A \mapsto \tfrac{1}{2}(A + \overline{A})$.
		
		\item $\herm{n}_+$ is a retract of $\symm{2n}_+$, via the maps $T : \herm{n} \to \symm{2n}$ and $S : \symm{2n} \to \herm{n}$ given by
		\[ T(A + iB) = \begin{pmatrix}
			A & -B \\
			B & A
		\end{pmatrix},\qquad S\begin{pmatrix}
			A_1 & A_2 \\
			A_3 & A_4
		\end{pmatrix} = \frac{1}{2}(A_1 + A_4) + \frac{i}{2}(A_3 - A_2). \qedhere \]
	\end{enumerate}
\end{example}

A more advanced example occurs in polyhedral cones. If $E_+$ is a proper and generating polyhedral cone with extremal directions $\{x_0,\ldots,x_k\}$, then a \index{vertex figure}\emph{vertex figure at $x_0$} is a subcone of the form $E_+ \cap \ker(\varphi_0)$, where $\varphi_0 \in E\algdual$ is a linear form such that $\varphi_0(x_0) < 0$ and $\varphi_0(x_i) > 0$ for all $i > 0$. Vertex figures are combinatorially dual to facets (e.g.~\cite[Theorem 11.5]{Brondsted}).

\begin{proposition}
	\label{prop:polyhedral-retracts}
	Let $E$ be finite\-/dimensional and let $E_+ \subseteq E$ be a proper and generating polyhedral cone. Then every vertex figure and every facet of $E_+$ is a retract.
\end{proposition}
\begin{proof}
	Let $x_0,\ldots,x_k$ be the extremal directions of $E_+$, and suppose that $\varphi_0 \in E\algdual$ defines a vertex figure at $x_0$ (i.e.~$\varphi_0(x_0) < 0$ and $\varphi_0(x_i) > 0$ for all $i > 0$).
	We show that $E_+ \cap \ker(\varphi_0)$ is a retract.
	
	By scaling, we may assume without loss of generality that $\varphi_0(x_0) = -1$.
	Define $P_{\varphi_0} : E \to E$ by $y \mapsto y + \varphi_0(y)x_0$. We show that $P_{\varphi_0}$ is a positive projection onto $\ker(\varphi_0)$. For all $y \in E$ we have
	\[ \varphi_0(P_{\varphi_0}(y)) = \varphi_0(y) + \varphi_0(y)\varphi_0(x_0) = \varphi_0(y) - \varphi_0(y) = 0, \]
	which shows that $\ran(P_{\varphi_0}) \subseteq \ker(\varphi_0)$. Furthermore, if $y \in \ker(\varphi_0)$, then $P_{\varphi_0}(y) = y + 0 = y$, so $P_{\varphi_0}$ is a projection onto $\ker(\varphi_0)$.
	To prove positivity, it suffices to show that $P_{\varphi_0}(x_i) \in E_+$ for all $i$. We distinguish two cases:
	\begin{itemize}
		\item For $i = 0$, we have $P_{\varphi_0}(x_0) = x_0 + \varphi_0(x_0)x_0 = x_0 - x_0 = 0 \in E_+$.
		\item For $i > 0$, we have $\varphi_0(x_i) > 0$, hence $P_{\varphi_0}(x_i) = x_i + \varphi_0(x_i)x_0 \in E_+$.
	\end{itemize}
	This shows that every vertex figure is a retract. Additionally, note that $\ker(P_{\varphi_0}) = \spn(x_0)$; this will be used in the second part of the proof.
	
	Now let $M \subseteq E_+$ be a facet. Then $M$ corresponds with an extremal direction $\psi_0 \in E_+\algdual \setminus \{0\}$ of the dual cone, in such a way that $M^\perp = \spn(\psi_0)$. Choose a vertex figure $N \subseteq E_+\algdual$ at $\psi_0$. The preceding argument shows that there are positive linear maps $T : \spn(N) \hookrightarrow E\algdual$ and $S : E\algdual \twoheadrightarrow \spn(N)$ such that $\id_{\spn(N)} = ST$. Furthermore, the construction gives us the additional property that $\ker(S) = \spn(\psi_0)$. Dualizing the retract (see \myautoref{rmk:retracts}{itm:retract:dual}) shows that $\spn(N)\algdual$ is isomorphically a retract of $E$, by means of the maps $S\algdual : \spn(N)\algdual \hookrightarrow E$ and $T\algdual : E \twoheadrightarrow \spn(N)\algdual$. Since $\ran(S\algdual) = {}^\perp \ker(S) = {}^\perp \{\psi_0\} = \spn(M)$, this shows that $M$ is a retract of $E_+$.
\end{proof}

\begin{remark}
	In fact, retractions give a \emph{geometric} duality between facets and vertex figures (in addition to the well-known \emph{combinatorial} duality). If $M \subseteq E_+$ is a facet corresponding to the extremal direction $\psi_0 \in E_+\algdual \setminus \{0\}$ of the dual cone, then one can show that:
	\begin{itemize}
		\item every vertex figure at $\psi_0$ admits a \emph{unique} positive projection (namely, the one from the proof of \autoref{prop:polyhedral-retracts});
		\item there is a bijective correspondence between vertex figures at $\psi_0$ and positive projections $E \to \spn(M)$ that map every element of $E_+ \setminus M$ in the relative interior of $M$.
	\end{itemize}
	As we have no use for this, the proof is omitted.
\end{remark}

Retracts can be useful in the theory of ordered tensor products. For instance, in \autoref{xmpl:projective-bipositive-fail-2} and \autoref{xmpl:injective-pushforward-fail-2} we proved that the projective cone does not preserve subspaces and the injective cone does not preserve quotients, but retracts are sufficiently rigid to be preserved by both (a similar role is played by complemented subspaces in the theory of normed tensor products).
The following result shows that retracts can also be useful for comparing the projective and injective cones.

\begin{proposition}[{cf.~\cite[Proposition 8]{Aubrun-et-al-i}}]
	\label{prop:retract}
	Let $G$ and $H$ be finite\-/dimensional real vector spaces, let $G_+ \subseteq G$, $H_+ \subseteq H$ be closed convex cones, and let $E \subseteq G$ and $F \subseteq H$ be retracts.
	If $\minwedgeFIN{E_+}{F_+} \neq \maxwedgeFIN{E_+}{F_+}$, then $\minwedgeFIN{G_+}{H_+} \neq \maxwedgeFIN{G_+}{H_+}$.
\end{proposition}
\begin{proof}
	We prove the contrapositive: assuming that $\minwedgeFIN{G_+}{H_+} = \maxwedgeFIN{G_+}{H_+}$, we prove that $\minwedgeFIN{E_+}{F_+} = \maxwedgeFIN{E_+}{F_+}$.
	By \myautoref{rmk:retracts}{itm:retract:dual}, we may identify $E\algdual$ with a retract of $G\algdual$.
	Choose positive projections $\pi_{E\algdual} : G\algdual \twoheadrightarrow E\algdual$ and $\pi_F : H \twoheadrightarrow F$.
	
	Let $T : E\algdual \to F$ be a positive linear map.
	Since every positive operator $G\algdual \to H$ factors through a simplex cone, we may choose positive operators $S_1 : G\algdual \to \R^m$ and $S_2 : \R^m \to H$ (where $\R^m$ carries the standard cone $\R_{\geq 0}^m$) so that the following diagram commutes:
	\begin{center}
		\begin{tikzcd}
			& G\algdual \arrow[dr, two heads, "\pi_{E\algdual}"] \arrow[rr, dashed, "S_1"] & & \R^m \arrow[rr, dashed, "S_2"] & &  H \arrow[dr, two heads, "\pi_F"] \\
			E\algdual \arrow[ur, hook] \arrow[rr, "\id_{E\algdual}"] \arrow[rrrrrr, bend right, "T"', rounded corners, to path={ (\tikztostart.south) |- ++(2,-.45) -- ++(4.8,0) \tikztonodes -| (\tikztotarget.south)}] & & E\algdual \arrow[rr, "T"] & & F \arrow[ur, hook] \arrow[rr, "\id_F"] & & F.
		\end{tikzcd}
	\end{center}
	Now $T : E\algdual \to F$ factors through a simplex cone, so $T \in \minwedgeFIN{E_+}{F_+}$.
\end{proof}

In \mychpref{chp:many-examples}, we will show that $\minwedgeFIN{G_+}{H_+} \neq \maxwedgeFIN{G_+}{H_+}$ for a large class of cones.
For combinations of standard cones, we will use \autoref{prop:retract} to reduce the problem to 3-dimensional cones (see \autoref{thm:many-examples-ii}).
However, it was shown in \cite[Lemma S14]{Aubrun-et-al-i} that most convex cones do not have retracts%
		\hair\footnote{More precisely, for $n \geq 4$, the set of $(n-1)$-dimensional convex bodies whose homogenizations have an $(n-1)$-dimensional retract is meagre with respect to the Hausdorff measure.}%
, so having retracts is an exceptional property.
Hence, for non-standard cones $G_+$ and $H_+$, different strategies are needed to show that $\minwedgeFIN{G_+}{H_+} \neq \maxwedgeFIN{G_+}{H_+}$.
Several such techniques will be discussed in the next chapter.

\chapter{Many examples where the projective and injective cone differ}
\label{chp:many-examples}

A question which has attracted considerable attention is under which circumstances the projective cone $\minwedge{E_+}{F_+}$ is dense in the injective cone $\maxwedge{E_+}{F_+}$.
Birnbaum \cite[Prop.{} 3]{Birnbaum} showed that this is always the case if $E$ and $F$ are locally convex lattices, and gave an example which shows that it is not true in general.
However, the infinite-dimensional version of this problem does not appear to be well understood.

A lot more is known in the finite-dimensional case (with closed, proper and generating cones).
In this setting, the projective cone is automatically closed, so the question is now whether or not the two cones coincide.
In \cite[p.{} 197]{Barker-monotone}, Barker asked to find precise necessary and sufficient conditions for the two cones to coincide, and later formulated the following conjecture:

\begin{conjecture}[{Barker, \cite[p.{} 277]{Barker-survey}}]
	\label{conj:Barker}
	Let $E,F$ be finite\-/dimensional, and let $E_+ \subseteq E$, $F_+ \subseteq F$ be closed, proper and generating convex cones.
	Then $\minwedge{E_+}{F_+} \neq \maxwedge{E_+}{F_+}$ unless at least one of $E_+$ and $F_+$ is a simplex cone.
\end{conjecture}

Various partial results in this direction have been obtained over the years.
Before Barker formulated his conjecture, it had already been proved in the case that $F_+ = E_+\algdual$ by Barker and Loewy \cite{Barker-Loewy}, and in the case that $E_+$ and $F_+$ are polyhedral by Poole \cite[Thm.{} 5.15]{Poole-dissertation}.
More recently, Huber and Netzer \cite{Huber-Netzer} proved it in the case that $E_+$ is polyhedral and $F_+$ is a positive semidefinite cone (or vice versa).

In this chapter, we prove \autoref{conj:Barker} for nearly all%
	\hair\footnote{The term `nearly all' has a precise meaning; namely, up to a $\sigma$-porous set. Since the set of closed, proper and generating convex cones which are not smooth or not strictly convex form a $\sigma$-porous set \cite{Zamfirescu-smooth-strictly-convex}, the results in this chapter prove \autoref{conj:Barker} for nearly all $E_+$ and $F_+$.}
convex cones, thereby also providing new proofs for each of the aforementioned known cases.
First, in \mysecref{sec:primal-tensor-dual}, we give another proof for the case $F_+ = E_+\algdual$, and we extend the aforementioned result of Barker and Loewy \cite{Barker-Loewy} to non-proper closed cones.
Then, in \mysecref{sec:polyhedral}, we give another proof for the polyhedral case, originally due to Poole \cite[Thm.{} 5.15]{Poole-dissertation}.
Our main contribution comes in \mysecref{sec:smooth-strictly-convex}, where we prove the conjecture when $\dim(E) \geq \dim(F)$ and $E_+$ is smooth or strictly convex.
Finally, in \mysecref{sec:standard-cones}, we prove it for all possible combinations of standard cones (polyhedral cones, second-order cones, and positive semidefinite cones), thereby also providing a new proof of the mixed polyhedral/positive semidefinite case settled by Huber and Netzer \cite{Huber-Netzer}.

Although some of the cases we prove had been known before, all proofs in this chapter are original.
However, as this manuscript was being written, the results from this chapter were superseded by independent work of Aubrun, Lami, Palazuelos and Pl\'avala \cite{Aubrun-et-al-ii}, who were able to prove \autoref{conj:Barker} in full generality.
Our results were obtained independently around the same time, and the proofs are completely different.

\section{The tensor product of a closed convex cone with its dual}
\label{sec:primal-tensor-dual}

The simplest case where the projective and injective cone are different occurs when considering the tensor product of a closed convex cone with its dual.
The main result of this section, \autoref{thm:min-equals-max}, is a slight extension of a well-known result of Barker and Loewy \cite[Prop.~3.1]{Barker-Loewy} (see also \cite[Thm.~4]{Tam-projective-closed}), who proved it for convex cones that are closed, proper and generating.
The proof below is much simpler than the original proof in \cite{Barker-Loewy}, since it was not realized at the time that the projective tensor product of closed convex cones is automatically closed, but comparable in size with Tam's alternative proof \cite[Thm.{} 4]{Tam-projective-closed}.

If $E_+$ and $F_+$ are closed, then by \autoref{cor:separable-maps-factor} one has $\minwedgeFIN{E_+\algdual}{F_+} = \maxwedgeFIN{E_+\algdual}{F_+}$ if and only if every positive linear map $E \to F$ factors through a simplex cone. This language makes it much easier to think about the difference between projective and injective cones.

If $T$ or $S$ factors through a simplex cone, then so does the composition $S \circ T$.
This shows that the separable positive operators form an \emph{ideal} in the \emph{semiring} of positive operators.
(Ideals of operators also play an important role in the theory of normed tensor products; e.g.~\cite{Defant-Floret}. We won't make much use of this terminology.)

\begin{theorem}[{cf.~\cite[Prop.{} 3.1]{Barker-Loewy}, \cite[Thm.{} 4]{Tam-projective-closed}}]
	\label{thm:min-equals-max}
	Let $E$ be finite\-/dimensional and let $E_+ \subseteq E$ be a closed convex cone. Then the following are equivalent:
	\begin{enumerate}[label=(\roman*)]
		\item\label{itm:minmax:simplex} $E_+$ is a simplex cone;
		\item\label{itm:minmax:id} $\id_E : E \to E$ is separable \textup(i.e.~factors through a simplex cone\textup);
		\item\label{itm:minmax:trace} for every positive linear map $T : E \to E$, one has $\tr(T) \geq 0$;
		\item\label{itm:minmax:EdxF} for every finite\-/dimensional real vector space $F$ and every closed convex cone $F_+ \subseteq F$, one has $\minwedgeFIN{E_+\algdual}{F_+} = \maxwedgeFIN{E_+\algdual}{F_+}$;
		\item\label{itm:minmax:FxE} for every finite\-/dimensional real vector space $F$ and every closed convex cone $F_+ \subseteq F$, one has $\minwedgeFIN{F_+}{E_+} = \maxwedgeFIN{F_+}{E_+}$;
		\item\label{itm:minmax:EdxE} $\minwedgeFIN{E_+\algdual}{E_+} = \maxwedgeFIN{E_+\algdual}{E_+}$.
	\end{enumerate}
\end{theorem}
\begin{proof}
	$\myref{itm:minmax:simplex} \implies \myref{itm:minmax:id}$. Let $x_1,\ldots,x_d$ be a basis of $E$ which generates $E_+$, and let $x_1\algdual,\ldots,x_d\algdual$ be the corresponding dual basis. Then $x_1\algdual,\ldots,x_d\algdual \in E_+\algdual$, and $\id_E = \sum_{i=1}^d x_i\algdual \tensor x_i$.
	
	$\myref{itm:minmax:id} \Longleftrightarrow \myref{itm:minmax:trace}$. The trace $\tr \in L(E,E)\algdual = (E\algdual \tensor E)\algdual = E \tensor E\algdual$ is the transpose of the identity $\id_E \in L(E,E) = E\algdual \tensor E$. Since $(\maxwedgeFIN{E_+\algdual}{E_+})\algdual = \minwedgeFIN{E_+}{E_+\algdual}$, we see that the trace defines a positive linear functional (i.e.~$\tr \in (\maxwedgeFIN{E_+\algdual}{E_+})\algdual$) if and only if $\id_E$ is separable (i.e.~$\id_E \in \minwedgeFIN{E_+\algdual}{E_+}$).
	
	$\myref{itm:minmax:id} \implies \myref{itm:minmax:EdxF}$ and $\myref{itm:minmax:id} \implies \myref{itm:minmax:FxE}$. Since $\id_E : E \to E$ factors through a simplex cone, so do all positive linear maps to or from $E$:
	\begin{center}
		\begin{tikzcd}[column sep=small, row sep=tiny]
			& & & \R^m \arrow[ddr] & & & & \R^m \arrow[ddr] & & \\
			& & & & & \quad\text{or}\quad & & & & & \\
			F\algdual \arrow[rr] & & E \arrow[uur] \arrow[rr, "\id_E"] & & E & & E \arrow[uur] \arrow[rr, "\id_E"] & & E \arrow[rr] & & F.
		\end{tikzcd}
	\end{center}
	Therefore $\minwedgeFIN{F_+}{E_+} = \maxwedgeFIN{F_+}{E_+}$ and $\minwedgeFIN{E_+\algdual}{F_+} = \maxwedgeFIN{E_+\algdual}{F_+}$.
	
	$\myref{itm:minmax:EdxF} \implies \myref{itm:minmax:EdxE}$ and $\myref{itm:minmax:FxE} \implies \myref{itm:minmax:EdxE}$. Clear.
	
	$\myref{itm:minmax:EdxE} \implies \myref{itm:minmax:id}$. Note that $\id_E : E \to E$ is positive.
	
	$\myref{itm:minmax:id} \implies \myref{itm:minmax:simplex}$. Write $\id_E = \sum_{i=1}^k \varphi_i \tensor x_i$ with $\varphi_1,\ldots,\varphi_k \in E_+\algdual$, $x_1,\ldots,x_k \in E_+$. For arbitrary $x \in E_+$ we have $x = \id_E(x) = \sum_{i=1}^k \varphi_i(x)x_i$ with $\varphi_1(x),\ldots,\varphi_k(x) \geq 0$, so we see that $E_+$ is generated by $x_1,\ldots,x_k$. In particular, it follows that $E_+$ is a polyhedral cone.
	Furthermore, since we have $E = \ran(\id_E) \subseteq \spn(x_1,\ldots,x_k)$, it follows that $x_1,\ldots,x_k$ must span $E$, so $E_+$ is generating. Dually, if $x,-x \in E_+$, then $\varphi_1(x) = \ldots = \varphi_k(x) = 0$, hence $x = \id_E(x) = \sum_{i=1}^k \varphi_i(x)x_i = 0$, which shows that $E_+$ is a proper cone.
	
	Since both $E_+$ and $E_+\algdual$ are proper polyhedral cones, each has a finite number of extremal rays generating the cone. Let $\{\psi_i\}_{i=1}^n$ and $\{y_j\}_{j=1}^m$ be (representatives of) the extremal directions of $E_+\algdual$ and $E_+$, respectively. Writing every $\varphi_i$ and every $x_j$ as a positive combination of the extremal rays, we can expand our expression of $\id_E$ to
	\[ \id_E = \sum_{i=1}^n \sum_{j=1}^m \lambda_{ij} \psi_i \tensor y_j,\quad\text{with $\lambda_{ij} \geq 0$ for all $i$ and all $j$}. \]
	For every $j$ we have $y_j = \id_E(y_j) = \sum_{i=1}^n \sum_{k=1}^m \lambda_{ik} \psi_i(y_j) y_k = \sum_{i=1}^n \lambda_{ij} \psi_i(y_j) y_j$, for by extremality of $y_j$ the terms $\lambda_{ik} \psi_i(y_j) y_k$ with $k \neq j$ must be zero. It follows that $\sum_{i=1}^n \lambda_{ij} \psi_i(y_j) = 1$ for all $j$. Therefore:
	\[ \dim(E) = \tr(\id_E) = \tr\left(\sum_{i=1}^n \sum_{j=1}^m \lambda_{ij} \psi_i \tensor y_j \right) = \sum_{j=1}^m \sum_{i=1}^n \lambda_{ij} \psi_i(y_j) = \sum_{j=1}^m 1 = m. \]
	Since $\spn(y_1,\ldots,y_m) = \spn(E_+) = E$, it follows that $y_1,\ldots,y_m$ is a basis of $E$. This proves that $E_+$ is a simplex cone.
\end{proof}

\begin{remark}
	Taking the tensor product of a space with its dual is also a common technique in the theory of normed tensor products.
	For instance, \autoref{thm:min-equals-max} is very similar to a result about the approximation property; see \cite[Theorem 5.6]{Defant-Floret}.
\end{remark}

The following corollary is immediate.

\begin{corollary}
	\label{cor:min-equals-max}
	Let $E_+ \subseteq \R^n$ be a self-dual cone. Then $\minwedgeFIN{E_+}{E_+} = \maxwedgeFIN{E_+}{E_+}$ if and only if $E_+$ is a simplex cone.
\end{corollary}

In particular, it follows that $\minwedgeFIN{\symm{n}_+}{\symm{n}_+} \neq \maxwedgeFIN{\symm{n}_+}{\symm{n}_+}$ and $\minwedgeFIN{\herm{n}_+}{\herm{n}_+} \neq \maxwedgeFIN{\herm{n}_+}{\herm{n}_+}$ whenever $n \geq 2$.
This has been known for a long time in relation to quantum theory and $C^*$-algebras, and is related to the difference between positive and \emph{completely positive} operators.
The interested reader is encouraged to refer to the expository article by Ando \cite[\S 2]{Ando}.

\section{Tensor products of polyhedral cones}
\label{sec:polyhedral}

In this section, we prove \autoref{conj:Barker} for polyhedral cones (see \autoref{thm:polyhedral} below).
As this manuscript was being written, the same result was proved independently in simultaneous work by Aubrun, Lami, and Palazuelos \cite[Result 2]{Aubrun-et-al-i}, and we later discovered that it had already been proved by M. Poole in the 1970s \cite[Thm.{} 5.15]{Poole-dissertation}.
Our proof is different from the proofs of Poole and of Aubrun, Lami, and Palazuelos, and uses a simple combinatorial argument in terms of the face lattice.

First, we use retracts to reduce the problem to the 3-dimensional case.
Using the results from \mysecref{sec:retracts} and \mysecref{sec:primal-tensor-dual}, we can prove the following lemmas.

\begin{lemma}
	\label{lem:simplex}
	Every retract of a finite\-/dimensional simplex cone is a simplex cone.
\end{lemma}
\begin{proof}
	Let $F_+$ be a finite\-/dimensional simplex cone and let $E_+$ be a retract of $F_+$. By \autoref{thm:min-equals-max}, we have $\minwedgeFIN{E_+\algdual}{F_+} = \maxwedgeFIN{E_+\algdual}{F_+}$, so it follows from \autoref{prop:retract} that $\minwedgeFIN{E_+\algdual}{E_+} = \maxwedgeFIN{E_+\algdual}{E_+}$. Another application of \autoref{thm:min-equals-max} shows that $E_+$ is a simplex cone.
\end{proof}

\begin{lemma}
	\label{lem:polyhedral-retracts}
	Let $E$ be finite\-/dimensional and let $E_+ \subseteq E$ be a proper and generating polyhedral cone. Then $E_+$ is a simplex cone if and only if every $3$-dimensional retract of $E_+$ is a simplex cone.
\end{lemma}
\begin{proof}
	If $\dim(E) \leq 2$, then $E_+$ is automatically a simplex cone, and there are no $3$-dimensional retracts, so the statement is vacuously true.
	
	Assume $\dim(E) \geq 3$. If $E_+$ is a simplex cone, then every retract of $E_+$ is a simplex cone, by \autoref{lem:simplex}. Conversely, if $E_+$ is not a simplex cone, then one of the following must be true (use \cite[Theorem 12.19]{Brondsted} and homogenization):
	\begin{itemize}
		\item $\dim(E_+) = 3$;
		\item $E_+$ has a facet that is not a simplex cone;
		\item $E_+$ has a vertex figure that is not a simplex cone.
	\end{itemize}
	Since facets and vertex figures are retracts (\autoref{prop:polyhedral-retracts}), it follows by induction that $E_+$ has a $3$-dimensional retract that is not a simplex cone.
\end{proof}

\newcommand{\facet}[2]{\mathcal F_{#1,#2}}
\newcommand{\myfacetformula}[6]{%
	(#5 == #3) || (mod(#5, #1) == mod(#3 + 1, #1)) || (#6 == #4) || (mod(#6, #2) == mod(#4 + 1, #2)) %
}
\newcommand{\myintersectionformulaA}[6]{%
	(((#5 == 1) || (#5 == 2)) && ((#6 == 3) || (#6 == 4))) || (((#5 == 3) || (#5 == 4)) && ((#6 == 1) || (#6 == 2))) %
}
\newcommand{\myintersectionformulaB}[6]{%
	(((#5 == 1) || (#5 == 2)) && ((#6 == 4) || (#6 == 5))) || (((#5 == 3) || (#5 == 4)) && ((#6 == 2) || (#6 == 3))) %
}
\newcommand{\myintersectionformulaAB}[6]{%
	( \myintersectionformulaA{#1}{#2}{#3}{#4}{#5}{#6} ) && ( \myintersectionformulaB{#1}{#2}{#3}{#4}{#5}{#6} ) %
}
\newcommand{\drawface}[7]{
	\let\myformula=#6
	\begin{scope}[scale=.25,yscale=-1]
		\draw (0,0) grid (#2,#1);
		\begin{scope}[xshift=-5mm,yshift=-5mm]
			\foreach \x in {1,...,#1} {
				\foreach \y in {1,...,#2} {
					\pgfmathtruncatemacro{\res}{ \myformula{#1}{#2}{#3}{#4}{\x}{\y} }
					\ifnum\res=1
						\draw (\y,\x) ++(-.5,-.5) -- ++(1,1) ++(-1,0) -- ++(1,-1);
					\fi
				}
			}
		\end{scope}
		\pgfmathsetmacro{\xhalf}{#2 / 2}
		\pgfmathsetmacro{\yhalf}{#1 / 2}
		\coordinate (P#3#4-L) at (0,\yhalf);
		\coordinate (P#3#4-R) at (#2,\yhalf);
		\def\tmpA{#7}
		\def\tmpB{axes}
		\ifx\tmpA\tmpB
			\draw[->,>=stealth'] (-.75,0) -- (-.75,#1) node[left] {$m$};
			\draw[->,>=stealth'] (0,-.75) -- (#2,-.75) node[above] {$n$};
		\fi
		\def\tmpB{nolabel}
		\ifx\tmpA\tmpB
		\else
			\node[below,inner sep=3mm] at (\xhalf,#1) {$#5$};
		\fi
	\end{scope}
}

All that remains is to prove that the projective and injective tensor products of two $3$-dimensional polyhedral cones are different, unless one of the two is a simplex cone.
For this we use a combinatorial argument, based on the results from \mychpref{chp:projective}.
(For a different proof, see \cite{Aubrun-et-al-i}.)

The proof essentially boils down to finding a combinatorial obstruction.
The high-level idea behind the proof is that, if $\minwedgeFIN{E_+}{F_+} = \maxwedgeFIN{E_+}{F_+}$, then $(\minwedgeFIN{E_+}{F_+})\algdual = \minwedgeFIN{E_+\algdual}{F_+\algdual}$, so by \autoref{thm:projective-extremal-rays} we known exactly what the extremal rays of $\minwedgeFIN{E_+}{F_+}$ and $(\minwedgeFIN{E_+}{F_+})\algdual$ are.
This gives us enough information to determine the face lattice of $\minwedgeFIN{E_+}{F_+}$.
However, if both $E_+$ and $F_+$ have at least $4$ extremal rays, then it turns out that the lattice thus obtained is not \emph{graded} (in other words, it contains maximal chains of different lengths), which contradicts a well-known property of polyhedral cones.

The proof below uses slightly different terminology than the preceding high-level idea, and does not proceed by contradiction.

\begin{lemma}
	\label{lem:polyhedral-3x3}
	Let $E_+$ and $F_+$ be proper and generating polyhedral cones in $\R^3$. Then $\minwedgeFIN{E_+}{F_+} = \maxwedgeFIN{E_+}{F_+}$ if and only if at least one of $E_+$ and $F_+$ is a simplex cone.
\end{lemma}
\begin{proof}
	A proper and generating polyhedral cone in $\R^3$ is the homogenization of a polygon. Let $v_1,\ldots,v_m \in \R^3$ be (representatives of) the extremal directions of $E_+$ in such a way that the neighbours of $v_i$ are $v_{i-1}$ and $v_{i+1}$ (modulo $m$). In the same way, let  $w_1,\ldots,w_n \in F_+$ be the extremal directions of $F_+$ (in cyclic order). Furthermore, let $\varphi_1,\ldots,\varphi_m \in E_+\algdual$ and $\psi_1,\ldots,\psi_n \in F_+\algdual$ be the extremal directions of $E_+\algdual$ and $F_+\algdual$, in such a way that $\varphi_i$ (resp.~$\psi_j$) represents the facet of $E_+$ (resp.~$F_+$) that contains $v_i$ and $v_{i+1}$ (resp.~$w_j$ and $w_{j+1}$).
	
	If $E_+$ or $F_+$ is a simplex cone, then $\minwedgeFIN{E_+}{F_+} = \maxwedgeFIN{E_+}{F_+}$ by \autoref{thm:min-equals-max}. So assume that neither $E_+$ nor $F_+$ is a simplex cone, i.e.~$m,n \geq 4$. We show by a combinatorial argument that $(\minwedgeFIN{E_+}{F_+})\algdual$ must be larger than $\minwedgeFIN{E_+\algdual}{F_+\algdual} = (\maxwedgeFIN{E_+}{F_+})\algdual$.
	
	By \autoref{thm:projective-extremal-rays}, the extremal directions of the projective cone $\minwedgeFIN{E_+}{F_+}$ are given by $\{v_i \tensor w_j \, : \, i\in [m], j \in [n]\}$, and the extremal directions of $\minwedgeFIN{E_+\algdual}{F_+\algdual}$ are given by $\{\varphi_i \tensor \psi_j \, : \, i\in [m], j \in [n]\}$. Furthermore, by \myautoref{cor:rank-one}{itm:injective-rank-one}, the extremal directions of $\minwedgeFIN{E_+\algdual}{F_+\algdual}$ are also extremal for the (larger) cone $(\minwedgeFIN{E_+}{F_+})\algdual = \maxwedgeFIN{E_+\algdual}{F_+\algdual}$. To complete the proof, we show that this larger cone must have extremal directions which are not of the form $\varphi_i \tensor \psi_j$. (By \myautoref{cor:rank-one}{itm:injective-rank-one}, these must have rank $\geq 2$.) Equivalently, the projective cone $\minwedgeFIN{E_+}{F_+}$ must have facets which cannot be written as the tensor product of a facet in $E_+$ and a facet in $F_+$.
	
	Given $k \in [m]$ and $\ell \in [n]$, let $\facet{k}{\ell}$ denote the facet of $\minwedgeFIN{E_+}{F_+}$ corresponding to the extremal direction $\varphi_k \tensor \psi_\ell \in (\minwedgeFIN{E_+}{F_+)}\algdual = \maxwedgeFIN{E_+\algdual}{F_+\algdual}$.
	The extremal directions in $\facet{k}{\ell}$ are those $v_i \tensor w_j$ with $i \in \{k,k+1\}$ or $j \in \{\ell,\ell+1\}$, for one has $v_i \tensor w_j \in \facet{k}{\ell}$ if and only if $0 = \langle v_i \tensor w_j , \varphi_k \tensor \psi_\ell \rangle = \langle v_i , \varphi_k \rangle \langle w_j , \psi_\ell \rangle$.
	In particular, $\facet{k}{\ell}$ contains $2m + 2n - 4$ extremal rays.
	
	We proceed to consider a specific intersection of the $\facet{k}{\ell}$, namely
	\[ C := \facet{1}{1} \cap \facet{1}{2} \cap \facet{3}{3} \cap \facet{3}{4}. \]
	Clearly $C$ is a face of $\minwedgeFIN{E_+}{F_+}$. We claim that $C$ contains only $4$ extremal rays. To that end, note that $\facet{1}{1} \cap \facet{3}{3}$ contains exactly $8$ extremal rays, namely $v_i \tensor w_j$ with $(i,j) \in {(\{1,2\} \times \{3,4\})} \cup {(\{3,4\} \times \{1,2\})}$. This is illustrated in the figure below (with $m = 6$ and $n = 8$).
	\begin{center}
		\begin{tikzpicture}
			\begin{scope}[xshift=0cm]
				\drawface{6}{8}{1}{1}{\facet{1}{1}}{\myfacetformula}{axes}
			\end{scope}
			\begin{scope}[xshift=3cm]
				\drawface{6}{8}{3}{3}{\facet{3}{3}}{\myfacetformula}{}
			\end{scope}
			\begin{scope}[xshift=6cm]
				\drawface{6}{8}{0}{0}{\facet{1}{1} \cap \facet{3}{3}}{\myintersectionformulaA}{}
			\end{scope}
			\node at ($(P11-R)!.5!(P33-L)$) {$\cap$};
			\node at ($(P33-R)!.5!(P00-L)$) {$=$};
		\end{tikzpicture}
	\end{center}
	Similarly, $\facet{1}{2} \cap \facet{3}{4}$ has the same pattern, but shifted one step in the second coordinate, so we see that the intersection of $\facet{1}{1} \cap \facet{3}{3}$ and $\facet{1}{2} \cap \facet{3}{4}$ contains $4$ extremal rays:
	\begin{center}
		\begin{tikzpicture}
			\begin{scope}[xshift=0cm]
				\drawface{6}{8}{1}{1}{\facet{1}{1} \cap \facet{3}{3}}{\myintersectionformulaA}{axes}
			\end{scope}
			\begin{scope}[xshift=3cm]
				\drawface{6}{8}{3}{3}{\facet{1}{2} \cap \facet{3}{4}}{\myintersectionformulaB}{}
			\end{scope}
			\begin{scope}[xshift=6cm]
				\drawface{6}{8}{0}{0}{C}{\myintersectionformulaAB}{}
			\end{scope}
			\node at ($(P11-R)!.5!(P33-L)$) {$\cap$};
			\node at ($(P33-R)!.5!(P00-L)$) {$=$};
		\end{tikzpicture}
	\end{center}
	(We have to be aware of a subtlety here: if $n = 4$, then the pattern of $\facet{1}{2} \cap \facet{3}{4}$ is ``wrapped around'' from right to left, but this does not affect the conclusion.)
	
	Since $C$ contains $4$ extremal rays, we have $\dim(C) \leq 4$. However, note that the only $\facet{k}{\ell}$ containing $C$ are the four facets defining $C$. Since we know from classical polyhedral geometry that $C$ must be contained in at least $9 - \dim(C) \geq 5$ facets, this shows that $\minwedgeFIN{E_+}{F_+}$ has facets which are not of the form $\facet{k}{\ell}$. Equivalently, the dual cone $(\minwedgeFIN{E_+}{F_+})\algdual = \maxwedgeFIN{E_+\algdual}{F_+\algdual}$ has extremal directions which are not of the form $\varphi_i \tensor \psi_j$, so we have $\maxwedgeFIN{E_+\algdual}{F_+\algdual} \neq \minwedgeFIN{E_+\algdual}{F_+\algdual}$. By duality, it follows that $\minwedgeFIN{E_+}{F_+} \neq \maxwedgeFIN{E_+}{F_+}$.
\end{proof}

\begin{theorem}[{\cite[Thm.{} 5.15]{Poole-dissertation}}]
	\label{thm:polyhedral}
	Let $G,H$ be finite\-/dimensional and let $G_+ \subseteq G$, $H_+ \subseteq H$ be proper and generating polyhedral cones.
	Then $\minwedgeFIN{G_+}{H_+} = \maxwedgeFIN{G_+}{H_+}$ if and only if at least one of $G_+$ and $H_+$ is a simplex cone.
\end{theorem}
\begin{proof}
	If $G_+$ or $H_+$ is a simplex cone, then it follows from \autoref{thm:min-equals-max} that $\minwedgeFIN{G_+}{H_+} = \maxwedgeFIN{G_+}{H_+}$.
	So assume that neither $G_+$ nor $H_+$ is a simplex cone.
	By \autoref{lem:polyhedral-retracts}, we may choose $3$-dimensional retracts $E_+$ (resp.{} $F_+$) of $G_+$ (resp.{} $H_+$) such that neither $E_+$ nor $F_+$ is a simplex cone.
	It follows from \autoref{lem:polyhedral-3x3} that $\minwedgeFIN{E_+}{F_+} \neq \maxwedgeFIN{E_+}{F_+}$, hence it follows from \autoref{prop:retract} that $\minwedgeFIN{G_+}{H_+} \neq \maxwedgeFIN{G_+}{H_+}$.
\end{proof}

\section{Tensor product with a smooth or strictly convex cone}
\label{sec:smooth-strictly-convex}

In this section, we prove \autoref{conj:Barker} in the case that $\dim(E) \geq \dim(F)$ and $E_+$ is smooth or strictly convex.
(This is \autoref{thm:intro:many-examples-i} from \mychpref{chp:introduction}.)
The argument is based on a generalized John's decomposition of the identity.

We recall some common terminology.
A \index{convex body}\emph{convex body} is a compact convex set $C \subseteq \R^n$ with non-empty interior.
The \glsadd{polar}\emph{\textup(one-sided\textup) polar} of $C$ is the set $C\polar = \{y \in (\R^n)\algdual \, : \, \langle x,y\rangle \leq 1\ \text{for all $x\in C$}\}$.
An \index{affine transformation}\emph{affine transformation} is an invertible affine map $\R^n \to \R^n$; that is, a map of the form $x \mapsto T_0x + y_0$ with $T_0 \in \GL{n}$ and $y_0 \in \R^n$ fixed.

If $C_1,C_2 \subseteq \R^n$ are convex bodies, then a compactness argument shows that there is an affine transformation $T$ such that $T[C_1] \subseteq C_2$ and $\vol(T[C_1])$ is maximal among all affine transformations $T'$ for which $T'[C_1] \subseteq C_2$.
If the maximum is attained for $T = I_n$ (the identity transformation), then we say that $C_1$ is \emph{in a maximum volume position inside $C_2$}.
Furthermore, we say that $C_1$ is \index{John decomposition}\emph{in John's position inside $C_2$} if $C_1 \subseteq C_2$ and there exist $m \in \N$, $x_1,\ldots,x_m \in \boundary C_1 \cap \boundary C_2$, $y_1,\ldots,y_m \in \boundary C_1\polar \cap \boundary C_2\polar$ and $\lambda_1,\ldots,\lambda_m > 0$ such that $\langle x_i,y_i \rangle = 1$ for all $i$, and
\begin{align*}
	I_n \, &= \, \sum_{i=1}^m \lambda_i x_i \tensor y_i \qquad \text{and} \qquad 0 \, = \, \sum_{i=1}^m \lambda_i x_i \, = \, \sum_{i=1}^m \lambda_i y_i.
\end{align*}
Gordon, Litvak, Meyer and Pajor \cite{Gordon-et-al} proved the following result, building on earlier extensions (\cite{Giannopoulos-et-al, Bastero-Romance}) of Fritz John's classical theorem (\cite{John}).

\begin{theorem}[{\cite[Theorem 3.8]{Gordon-et-al}}]
	\label{thm:generalized-John}
	Let $C_1,C_2 \subseteq \R^n$ be convex bodies such that $C_1$ is in a maximum volume position inside $C_2$. Then there exists $z \in \interior(C_1)$ such that $C_1 - z$ is in John's position inside $C_2 - z$.
\end{theorem}

For our purposes, we will only need the following (much weaker) corollary.

\begin{corollary}
	\label{cor:John-affine}
	Let $C_1,C_2 \subseteq \R^n$ be convex bodies. Then there is an affine transformation $T : \R^n \to \R^n$ such that $T[C_1] \subseteq C_2$ and $\boundary T[C_1] \cap \boundary C_2$ contains an affine basis of $\R^n$.
\end{corollary}
(Equivalently, there is a $T$ such that $T[C_1] \subseteq C_2$ and the set of points where $T[C_1]$ and $C_2$ touch is not contained in an affine hyperplane.)
\begin{proof}[Proof of \autoref{cor:John-affine}]
	Let $T : \R^n \to \R^n$ be an affine transformation such that $T[C_1]$ is in a maximum volume position inside $C_2$. By \autoref{thm:generalized-John}, we may choose $z \in \interior(T[C_1])$, $m \in \N$, $x_1,\ldots,x_m \in \boundary (T[C_1] - z) \cap \boundary (C_2 - z)$, $y_1,\ldots,y_m \in \boundary (T[C_1] - z)\polar \cap \boundary (C_2 - z)\polar$ and $\lambda_1,\ldots,\lambda_m > 0$ such that $\langle x_i,y_i \rangle = 1$ for all $i$, and
	\[ I_n \, = \, \sum_{i=1}^m \lambda_i x_i \tensor y_i \qquad \text{and} \qquad 0 \, = \, \sum_{i=1}^m \lambda_i x_i \, = \, \sum_{i=1}^m \lambda_i y_i. \tag*{$(*)$} \]
	After an appropriate rescaling of the $\lambda_i$, the second formula in $(*)$ shows that $0 \in \aff(x_1,\ldots,x_m)$, so it follows that $\aff(x_1,\ldots,x_m) = \spn(x_1,\ldots,x_m)$.
	Moreover, it follows immediately from the first formula in $(*)$ that $\spn(x_1,\ldots,x_m) = \R^n$, so we conclude that $\{x_1,\ldots,x_m\}$ contains an affine basis, say $x_1,\ldots,x_{n+1}$.
	Consequently, $x_1 + z , \ldots, x_{n+1} + z$ is an affine basis in $\boundary T[C_1] \cap \boundary C_2$.
\end{proof}

Since every closed and proper convex cone has a compact base, the following homogenization of \autoref{cor:John-affine} follows immediately.

\begin{corollary}
	\label{cor:John-linear}
	Let $E$ and $F$ be finite\-/dimensional real vector spaces with $\dim(E) = \dim(F)$, and let $E_+ \subseteq E$, $F_+ \subseteq F$ be closed, proper, and generating convex cones. Then there exists a positive linear transformation $T : E \to F$ such that $\boundary T[E_+] \cap \boundary F_+$ contains a \textup(linear\textup) basis of $F$.
\end{corollary}

Using the preceding results, we can prove the main result of this section.

\begin{theorem}
	\label{thm:many-examples-i}
	Let $E$, $F$ be finite\-/dimensional real vector spaces, and let $E_+ \subseteq E$, $F_+ \subseteq F$ be closed, proper, and generating convex cones. If $\dim(E) \geq \dim(F)$, and if $E_+$ is strictly convex or smooth, then one has $\minwedgeFIN{E_+}{F_+} = \maxwedgeFIN{E_+}{F_+}$ if and only if $F_+$ is a simplex cone.
\end{theorem}
\begin{proof}
	First assume that $\minwedgeFIN{E_+}{F_+} = \maxwedgeFIN{E_+}{F_+}$, with $E_+$ strictly convex and $\dim(E) \geq \dim(F)$.
	Choose an interior point $x_0 \in E_+$ and a linear subspace $G \subseteq E$ of dimension $\dim(F)$ through $x_0$.
	Then $G_+ := G \cap E_+$ is closed, proper, and generating, so by \autoref{cor:John-linear} we may choose a positive linear isomorphism $T : F\algdual \to G$ such that $\boundary T[F_+\algdual] \cap \boundary G_+$ contains a basis $\{b_1,\ldots,b_m\}$ of $G$.
	Note that every $b_i$ is also a boundary point of $E_+$ (this a basic property of topological boundaries), and therefore an extremal direction of $E_+$ (since $E_+$ is strictly convex).
	For all $i$, write $a_i := T^{-1}(b_i) \in \boundary F_+\algdual$; then $\{a_1,\ldots,a_m\}$ is a basis of $F\algdual$.

	If $\iota : G \hookrightarrow E$ denotes the inclusion, then $\iota \circ T : F\algdual \to E$ is positive, so by assumption we may write $\iota \circ T = \sum_{i=1}^k y_i \tensor x_i$ with $x_1,\ldots,x_k \in E_+$ and $y_1,\ldots,y_k \in F_+$ (where the $y_i$ act as linear functionals on $F\algdual$). Since $b_i$ is extremal and
	\[ b_i = T(a_i) = \sum_{j=1}^k \langle y_j,a_i \rangle x_j, \]
	it follows that at least one of the $x_j$ is a positive multiple of $b_i$, and $\langle y_j , a_i\rangle = 0$ whenever $x_j$ is not a positive multiple of $b_i$.
	In particular, if $x_j$ is not a positive multiple of any one of the $b_i$, then $\langle y_j , a_i \rangle = 0$ for all $i$, so $y_j = 0$ (since $\{a_1,\ldots,a_m\}$ is a basis of $F\algdual$).
	Thus, after removing the zero terms, every $x_j$ is a positive multiple of some $b_i$, and so in particular belongs to $T[F_+\algdual]$.
	This shows that not only $\iota \circ T$, but also $T$ is separable, and not only with respect to the cones $F_+\algdual$ and $G_+$, but even with respect to the cones $F_+\algdual$ and $T[F_+\algdual]$.
	Since $\id_{F\algdual} = T^{-1} \circ T$, it follows from the ideal property of separable operators that $\id_{F\algdual}$ is also separable.
	Hence, by \autoref{thm:min-equals-max}, $F_+\algdual$ is a simplex cone. Since $F_+$ is closed, it follows that $F_+ = F_+\algdualdual$ is also a simplex cone.

	Now assume that $\minwedgeFIN{E_+}{F_+} = \maxwedgeFIN{E_+}{F_+}$ with $E_+$ smooth and $\dim(E) \geq \dim(F)$.
	By duality (see \myautoref{cor:closures}{itm:closed-duality}), it follows that
	\[ \minwedgeFIN{E_+\algdual}{F_+\algdual} = (\maxwedgeFIN{E_+}{F_+})\algdual = (\minwedgeFIN{E_+}{F_+})\algdual = \maxwedgeFIN{E_+\algdual}{F_+\algdual}. \]
	Since $E_+$ is smooth, the dual cone $E_+\algdual$ is strictly convex, so it follows from the first part of the proof that $F_+\algdual$ must be a simplex cone. Since $F_+$ is closed, it follows that $F_+ = F_+\algdualdual$ is also a simplex cone.
\end{proof}

\section{Tensor products of standard cones; applications to operator systems}
\label{sec:standard-cones}

In this section, we prove \autoref{conj:Barker} for all combinations of standard cones, thereby proving \autoref{thm:intro:many-examples-ii} and \autoref{cor:intro:operator-systems} from \mychpref{chp:introduction}.
Just as in \mysecref{sec:polyhedral}, we use retracts (see \mysecref{sec:retracts}) to reduce the problem to the three-dimensional case.

\subsection*{Standard cones}

By combining the results obtained thus far, we can easily prove \autoref{conj:Barker} for all combinations of standard cones.

\begin{theorem}[{cf.~\cite{Poole-dissertation,Huber-Netzer}}]
	\label{thm:many-examples-ii}
	Let $G$, $H$ be finite\-/dimensional real vector spaces, and let $G_+ \subseteq G$, $H_+ \subseteq H$ be closed, proper, and generating convex cones. Assume that each of $G_+$ and $H_+$ is one of the following \textup(all combinations allowed\textup):
	\begin{enumerate}[label=(\roman*)]
		\item a polyhedral cone;
		\item a second-order cone $\Lorentz{n}$;
		\item a \textup(real or complex\textup) positive semidefinite cone $\symm{n}_+$ or $\herm{n}_+$.
	\end{enumerate}
	Then one has $\minwedgeFIN{G_+}{H_+} = \maxwedgeFIN{G_+}{H_+}$ if and only if at least one of $G_+$ and $H_+$ is a simplex cone.
\end{theorem}
\begin{proof}
	Suppose that neither $G_+$ nor $H_+$ is a simplex cone.
	We claim that $G_+$ (resp.~$H_+$) has a three-dimensional retract $E_+$ (resp.~$F_+$) which is isomorphic with one of the following:
	\begin{itemize}
		\item the three-dimensional Lorentz cone $\Lorentz{3}$ (which is isomorphic to $\symm{2}_+$);
		\item a proper and generating polyhedral cone $P \subseteq \R^3$ that is not a simplex cone.
	\end{itemize}
	To prove the claim, we distinguish three cases:
	\begin{itemize}
		\item If $G_+$ is polyhedral, then this follows from \autoref{lem:polyhedral-retracts}.
		
		\item If $G_+ = \Lorentz{n}$, then the assumption that $G_+$ is not a simplex cone forces $n \geq 3$. Hence it follows from \myautoref{xmpl:retracts}{itm:retracts:Lorentz} that $\Lorentz{3}$ is a retract of $G_+$.
		
		\item If $G_+ = \symm{n}_+$ or $\herm{n}_+$, then the assumption that $G_+$ is not a simplex cone forces $n \geq 2$, so it follows from \myautoref{xmpl:retracts}{itm:retracts:PSD} and \myautoref{xmpl:retracts}{itm:retracts:symm-herm} that $\symm{2}_+$ ($\cong \Lorentz{3}$) is a retract of $G_+$.
	\end{itemize}
	Next, we show that $\minwedgeFIN{E_+}{F_+} \neq \maxwedgeFIN{E_+}{F_+}$, again distinguishing three cases.
	\begin{itemize}
		\item If $E_+ = F_+ = \Lorentz{3}$, then this follows from \autoref{cor:min-equals-max}, since the Lorentz cone is self-dual.
		
		\item If $E_+ = \Lorentz{3}$ and $F_+$ is polyhedral (or vice versa), then this follows from \autoref{thm:many-examples-i}, since $\Lorentz{3}$ is strictly convex and $\dim(E) = \dim(F)$.
		
		\item The case where both $E_+$ and $F_+$ are polyhedral follows from \autoref{lem:polyhedral-3x3}.
	\end{itemize}
	Since $E_+$ and $F_+$ are retracts of $G_+$ and $H_+$ satisfying $\minwedgeFIN{E_+}{F_+} \neq \maxwedgeFIN{E_+}{F_+}$, it follows from \autoref{prop:retract} that $\minwedgeFIN{G_+}{H_+} \neq \maxwedgeFIN{G_+}{H_+}$.
\end{proof}

\subsection*{Operator systems}

Some of our results can be reformulated in terms of operator systems.
Let $C \subseteq \R^d$ be a closed, proper, and generating convex cone, and let $n \in \N_1$ be a positive integer.
Following notation from \cite{Fritz-Netzer-Thom}, we denote the projective and injective tensor products $\minwedgeFIN{\herm{n}_+}{C}$ and $\maxwedgeFIN{\herm{n}_+}{C}$ by $\minopsys[n]{C}$ and $\maxopsys[n]{C}$, respectively, and we write $\minopsys{C} = \{\minopsys[n]{C}\}_{n=1}^\infty$ and  $\maxopsys{C} = \{\maxopsys[n]{C}\}_{n=1}^\infty$.

\begin{corollary}[{Special case of \cite[Corollary 2]{Aubrun-et-al-ii}}]
	\label{cor:operator-systems}
	Let $C \subseteq \R^d$ be a closed, proper, and generating convex cone. If $d \leq 4$, or if $C$ is strictly convex, or smooth, or polyhedral, or \textup(real or complex\textup) positive semidefinite, then the following are equivalent:
	\begin{enumerate}[label=(\roman*)]
		\item\label{itm:opsys:simplex} $C$ is a simplex cone;
		\item\label{itm:opsys:systems-equal} the minimal and maximal operator systems $\minopsys{C}$ and $\maxopsys{C}$ are equal;
		\item\label{itm:opsys:levels-equal} there exists $n \geq 2$ for which $\minopsys[n]{C} = \maxopsys[n]{C}$;
		\item\label{itm:opsys:level-two-equal} one has $\minopsys[2]{C} = \maxopsys[2]{C}$.
	\end{enumerate}
\end{corollary}
\begin{proof}
	$\myref{itm:opsys:simplex} \Longrightarrow \myref{itm:opsys:systems-equal}$. This follows from \autoref{thm:min-equals-max}.
	
	$\myref{itm:opsys:systems-equal} \Longrightarrow \myref{itm:opsys:levels-equal}$. Trivial.
	
	$\myref{itm:opsys:levels-equal} \Longrightarrow \myref{itm:opsys:level-two-equal}$. If $\minwedgeFIN{\herm{n}_+}{C} = \maxwedgeFIN{\herm{n}_+}{C}$ for some $n \geq 2$, then it follows from \autoref{prop:retract} that $\minwedgeFIN{\herm{2}_+}{C} = \maxwedgeFIN{\herm{2}_+}{C}$, since $\herm{2}_+$ is a retract of $\herm{n}_+$, by \myautoref{xmpl:retracts}{itm:retracts:PSD}.
	
	$\myref{itm:opsys:level-two-equal} \Longrightarrow \myref{itm:opsys:simplex}$. First we prove that $\herm{2}_+$ is strictly convex. Indeed, the interior points of $\herm{2}_+$ are the positive definite matrices, so the boundary points are the singular $2 \times 2$ positive semidefinite matrices. Consequently, a non-zero boundary point of $\herm{2}_+$ must be a rank one positive semidefinite matrix, which is known to be extremal (it is a positive multiple of a rank one orthogonal projection).
	
	Now suppose that $C$ is of one of the forms described in the theorem, but not a simplex cone.
	If $d \leq \dim(\herm{2}) = 4$, or if $C$ is smooth or strictly convex, then it follows from \autoref{thm:many-examples-i} that $\minwedgeFIN{\herm{2}_+}{C} \neq \maxwedgeFIN{\herm{2}_+}{C}$.
	If $C$ is positive semidefinite or polyhedral (but not a simplex cone), then it follows from \autoref{thm:many-examples-ii} that $\minwedgeFIN{\herm{2}_+}{C} \neq \maxwedgeFIN{\herm{2}_+}{C}$.
	Either way, we have $\minopsys[2]{C} \neq \maxopsys[2]{C}$.
\end{proof}

\section{Closing remarks}
\label{sec:Aubrun-et-al}

As mentioned before, Aubrun, Lami, Palazuelos and Pl\'avala \cite{Aubrun-et-al-ii} independently proved \autoref{conj:Barker} in full generality.

\begin{theorem}[{\cite[Theorem A]{Aubrun-et-al-ii}}]
	\label{thm:Aubrun-et-al}
	Let $E$, $F$ be finite\-/dimensional real vector spaces, and let $E_+ \subseteq E$, $F_+ \subseteq F$ be closed, proper, and generating convex cones.
	Then one has $\minwedgeFIN{E_+}{F_+} = \maxwedgeFIN{E_+}{F_+}$ if and only if at least one of $E_+$ and $F_+$ is a simplex cone.
\end{theorem}

The following example shows that this is no longer true if we omit the requirement that $E_+$ or $F_+$ is proper or generating.

\begin{example}
	\label{xmpl:Aubrun-et-al}
	Let $F_+ \subseteq F$ be a ``partial simplex cone''; that is, a cone generated by $m < \dim(F)$ linearly independent vectors $x_1,\ldots,x_m \in F$.
	Furthermore, let $E$ be another finite\-/dimensional space, and let $E_+ \subseteq E$ be an arbitrary closed, proper, and generating cone.
	
	Since $E_+$ is generating, every positive linear map $T : E \to F$ has its range contained in $\spn(F_+)$.
	Since $\spn(F_+)$ is ordered by a simplex cone, this shows that every positive linear map $E \to F$ is simplex-factorable, hence $\minwedgeFIN{E_+\algdual}{F_+} = \maxwedgeFIN{E_+\algdual}{F_+}$.
\end{example}

In the preceding example, $E_+\algdual$ and $F_+$ are closed and proper and $E_+\algdual$ is generating, so the requirement that $F_+$ is generating cannot be omitted from \autoref{thm:Aubrun-et-al}.
Furthermore, by duality, we also have $\minwedgeFIN{E_+}{F_+\algdual} = \maxwedgeFIN{E_+}{F_+\algdual}$, which shows that the requirement that $F_+$ is proper cannot be omitted either.

In a sense, a partial simplex cone (or its dual) is almost a simplex cone. In fact, we can extend \autoref{thm:Aubrun-et-al} to show that all examples must be of this form.
If $E_+$ is a closed convex cone, then we define the \emph{proper reduction} $\pred(E_+)$ of $E_+$ as the positive cone of $\spn(E_+) / \lineal(E_+)$.
Equivalently, choose subspaces $E_1,E_2,E_3 \subseteq E$ such that $E_1 = \lineal(E_+)$, $E_1 \oplus E_2 = \spn(E_+)$, and $E_1 \oplus E_2 \oplus E_3 = E$; then the proper reduction of $E_+$ is the positive cone $(E_2)_+ := E_2 \cap E_+$ of $E_2$, viewed as a closed, proper, and generating cone in $E_2$.
It is readily verified that the projection $E \to E_2$, $(e_1,e_2,e_3) \mapsto e_2$ is positive (every projection onto $\spn(E_+)$ is positive, and adding or subtracting elements of the lineality space does not affect positivity), so $\pred(E_+)$ is a retract of $E_+$.
Therefore \autoref{thm:Aubrun-et-al} has the following extension.

\begin{corollary}
	Let $E$, $F$ be finite\-/dimensional real vector spaces, and let $E_+ \subseteq E$, $F_+ \subseteq F$ be closed convex cones. If $\minwedgeFIN{E_+}{F_+} = \maxwedgeFIN{E_+}{F_+}$, then at least one of $\pred(E_+)$ and $\pred(F_+)$ is a simplex cone.
\end{corollary}
\begin{proof}
	Since $\pred(E_+)$ and $\pred(F_+)$ are retracts of $E_+$ and $F_+$, it follows from \autoref{prop:retract} that $\minwedgeFIN{\pred(E_+)}{\pred(F_+)} = \maxwedgeFIN{\pred(E_+)}{\pred(F_+)}$.
	But $\pred(E_+)$ and $\pred(F_+)$ are closed, proper, and generating, so it follows from \autoref{thm:Aubrun-et-al} that at least one of $\pred(E_+)$ and $\pred(F_+)$ must be a simplex cone.
\end{proof}

The converse is not true; it can happen that $\pred(E_+)$ and $\pred(F_+)$ are simplex cones but $\minwedgeFIN{E_+}{F_+} \neq \maxwedgeFIN{E_+}{F_+}$.
This is because $\minwedgeFIN{\pred(E_+)}{\pred(F_+)} = \maxwedgeFIN{\pred(E_+)}{\pred(F_+)}$ does not necessarily imply $\minwedgeFIN{E_+}{F_+} = \maxwedgeFIN{E_+}{F_+}$; the implication of \autoref{prop:retract} only runs in the other direction.
As an extreme example, consider the case where $E_+ = \{0\}$ and $F_+ = F$; then one has $\minwedgeFIN{E_+}{F_+} = \{0\}$ but $\maxwedgeFIN{E_+}{F_+} = E \tensor F$.
More generally, \autoref{thm:min-equals-max} shows that $\minwedgeFIN{E_+\algdual}{E_+} \neq \maxwedgeFIN{E_+\algdual}{E_+}$ whenever $E_+$ is not proper or not generating, regardless of whether or not $\pred(E_+)$ is a simplex cone.

\chapter{Open problems}
\label{chp:open-problems}

We conclude this memoir with a few open problems.

\paragraph{Does the closure of the projective cone preserve higher faces?}
We showed in \mysecref{sec:projective-faces} that the projective cone preserves faces.
Furthermore, it follows from \autoref{prop:reasonable-faces-ideals} that the closure of the projective cone preserves extremal rays (provided that $E_+$ and $F_+$ are weakly closed).
However, we suspect that this result is of limited use in practice, because infinite\-/dimensional cones often do not have sufficiently many extremal rays.
Does the closure of the projective cone also preserve higher faces, in a sense similar to \autoref{thm:projective-faces}?
In particular, if $E_+$ and $F_+$ are weakly closed proper cones and $M \subseteq E_+$ and $M \subseteq F_+$ are faces, then is $\overline{\minwedge{M}{N}}^{\,\weak}$ a face of $\overline{\minwedge{E_+}{F_+}}^{\,\weak}$?

As a partial result, it follows from \myautoref{prop:reasonable-faces-ideals}{itm:reasonable-ideals} that $(\spn(M) \tensor \spn(N)) \cap \overline{\minwedge{E_+}{F_+}}^{\,\weak}$ is a face of $\overline{\minwedge{E_+}{F_+}}^{\,\weak}$, but this can in principle be larger than $\overline{\minwedge{M}{N}}^{\,\weak}$.

\paragraph{Does the projective norm preserve extreme points?}
We showed in \mysecref{sec:application-to-convex-sets} that the algebraic tensor product $\conv(C \settensor D)$ of symmetric convex sets $C$ and $D$ preserves proper faces.
Is this still true if we pass to the closure of $\conv(C \settensor D)$?
In particular, we do not known whether the projective norm preserves extreme points of the closed unit ball.
See also \autoref{rmk:projective-norm-extreme-points}.

\paragraph{Is the projective cone still semisimple in the completed projective tensor product?}
We showed in \mysecref{sec:semisimple-algebraic} that every reasonable crosscone in the algebraic tensor product $E \tensor F$ is semisimple whenever the base cones $E_+$ and $F_+$ are semisimple.
This is no longer true if we pass to the completed locally convex tensor product $E \hattensor_\alpha F$, because the injective cone $\hatmaxwedge[\alpha]{E_+}{F_+}$ is not proper if the natural map $E \hattensor_\alpha F \to E \hattensor_\varepsilon F$ fails to be injective (see \autoref{cor:completed-injective-proper-cone}).
However, we don't know the answer if we match the projective cone with the projective norm; see \autoref{q:completed-semisimple}.

\paragraph{Is there a way to determine all extremal rays of the injective cone?}
In \mysecref{sec:injective-extremal-rays}, we showed that the injective cone preserves extremal rays.
\myautoref{cor:rank-one}{itm:injective-rank-one} shows that all extremal rays of tensor rank one are of this form, but \autoref{xmpl:injective-extremal-rays} shows that there may be extremal rays of higher rank.
Is there a way to determine all extremal rays of the injective cone?

We note that it was already pointed out earlier by Tam \cite[p.~75]{Tam-faces} that this appears to be a difficult problem.
In fact, it is already difficult for proper and generating polyhedral cones in finite\-/dimensional spaces; see \cite{Bogart-Contois-Gubeladze}.
Sufficient conditions for some $z \in E \tensor F$ to be an extremal direction of the injective cone were studied by various authors in the context of positive operators; see for instance \cite{Tam-survey} and the references contained therein.

\paragraph{Are there other interesting tensor cones?}
In this paper, we have drawn parallels between normed and ordered tensor products.
Taking this a step further, we may define a \emph{tensor cone} as a way of choosing for each pair of preordered vector spaces $(E,E_+)$ and $(F,F_+)$ a reasonable crosscone $\tensorwedge{E_+}{F_+}$ in the tensor product $E \tensor F$, in such a way that positive linear maps are preserved; that is, $(T \tensor S)[\tensorwedge{E_+}{F_+}] \subseteq \tensorwedge{G_+}{H_+}$ whenever $T[E_+] \subseteq G_+$ and $S[F_+] \subseteq H_+$.
The projective/injective cone defines a tensor cone which behaves similarly to its normed counterpart.
Many more tensor \emph{norms} are known, but to our knowledge no other tensor \emph{cones} have been studied in the literature.%
	\hair\footnote{A few other cones have been defined in the tensor product of specific types of ordered vector spaces (e.g.{} for Archimedean Riesz spaces \cite{Fremlin}, or for Archimdean ordered vector spaces \cite{Grobler-Labuschagne,vanGaans-Kalauch}), but not for arbitrary ordered vector spaces, and rarely in connection with a positive mapping property.}
Are there other interesting and/or natural tensor cones?

Every tensor cone defines a tensor norm via a construction similar to \autoref{prop:absolutely-convex-tensor-product}.
Conversely, can every tensor norm be extended to a tensor cone in a natural way?
Are there cone-theoretic analogues of Grothendieck's 14 natural tensor norms?

An important difference between cones and norms is that there is no notion of two cones being equivalent.
So even if this programme would succeed, the resulting theory might not be as nice as the normed theory.
It is unclear if a cone-theoretic analogue of, say, Grothendieck's inequality, can exist.
Therefore it is also conceivable that there are more than 14 natural tensor cones.

Our findings about faces and extremal rays of the injective cone already show that ordered tensor products are not completely analogous to normed tensor products (see \autoref{rmk:injective-norm-extreme-points}), so perhaps there are limits to the analogy.

\paragraph{Is there a full proof of \autoref*{thm:Aubrun-et-al} without relying on computer algebra?}
In \cite{Aubrun-et-al-ii}, Aubrun, Lami, Palazuelos and Pl\'avala proved a very general result about the difference between the projective and injective cone (stated as \autoref{thm:Aubrun-et-al} above), which contains all our results from \mychpref{chp:many-examples} as a special case.
Their proof uses an ingenious geometric argument, but one step in their proof relies on a couple of large computations, which the authors verified using computer algebra.
On the other hand, the special cases that we proved in this paper only rely on relatively simple geometric arguments.
Can we find an alternative proof of \autoref{thm:Aubrun-et-al} which does not rely on computer algebra?

\paragraph{Is there an infinite-dimensional version of \autoref*{thm:Aubrun-et-al}?}
The theorem of Aubrun, Lami, Palazuelos and Pl\'avala (stated as \autoref{thm:Aubrun-et-al} above) provides precise necessary and sufficient conditions for the projective cone to be equal to the injective cone, when the base cones are closed, proper and generating in finite-dimensional spaces.
Can this be extended to infinite-dimensional spaces?
Are there examples where the projective cone $\minwedge{E_+}{F_+}$ is dense in the injective cone $\maxwedge{E_+}{F_+}$ (with respect to some compatible topology on $E \tensor F$) where $E_+$ and $F_+$ are weakly closed but neither is a lattice cone?
How about in the \emph{completed} locally convex tensor product?

Apart from \autoref{thm:Aubrun-et-al}, very few results in this direction are known.
A small positive result is \cite[Prop.{} 3]{Birnbaum}, which shows that $\minwedge{E_+}{F_+}$ is dense in $\maxwedge{E_+}{F_+}$ with respect to the projective topology whenever $E$ and $F$ are locally convex lattices.

\appendix
\chapter{Ideals, faces, and duality}
\label{app:faces}

Certain special subsets of a convex set, the so-called \emph{faces}, play an important role in convex geometry.
For instance, convex polytopes are commonly studied in terms of their face lattice, and the extreme points and extremal rays of convex sets play an important role in convex analysis and optimization.
The face structure of a convex cone has also been studied extensively; see for instance \cite[\S{}2.13]{Stoer-Witzgall} and the works of Barker and Tam \cite{Barker-faces,Barker-duality,Tam-duality,Tam-faces}.

Similarly, certain special subspaces, the so-called \emph{order ideals}, play a special role in the theory of ordered vector spaces.
In \cite{Kadison-representation}, Kadison used maximal order ideals in the proof of his celebrated representation theorem, and Bonsall continued the study of order ideals in \cite{Bonsall}.
However, as attention shifted from general ordered vector spaces to lattice-ordered (i.e.~Riesz) spaces, order ideals appear to have been forgotten in favour of lattice ideals (sometimes also called order ideals).
As a result, the theory of order ideals is not so well-known.

In this appendix, we develop/recall the basics of order ideals in a general preordered vector space.
In \mysecref{sec:faces-ideals}, we give several different equivalent definitions of an order ideal, and we show that the order ideals of a preordered vector space $(E,E_+)$ are closely related to the faces of the positive cone $E_+$.
This is very useful, as it allows us to quotient out a face, which is one of the main tools in the construction of faces of the projective cone in \mysecref{sec:projective-faces}.

In \mysecref{sec:isomorphism-theorems}, we outline the homomorphism and isomorphism theorems for ideals in ordered vector spaces.
As an application, we show that the maximal order ideals are precisely the supporting hyperplanes of the positive cone.
This shows that general (non-maximal) order ideals can be thought of as being the ``supporting subspaces'' of the positive cone.

Finally, in \mysecref{sec:dual-and-exposed}, we extend the theory of dual faces (see \cite{Barker-duality,Tam-duality}) to cones in infinite-dimensional spaces.
We show that it is now necessary to make a distinction between \emph{dual} and \emph{exposed} faces, although the two notions coincide if the ambient space is a separable normed space.

\section{Faces and ideals}
\label{app:faces-ideals} 
\label{sec:faces-ideals} 
Let $E$ be a preordered vector space with positive cone $E_+ \subseteq E$, and let $\leq$ be the vector preorder corresponding with $E_+$.
A subset $M \subseteq E$ is \index{full|see {order-convex}}\emph{full} (or \index{order-convex}\emph{order-convex}) if $x \leq y \leq z$ with $x,z \in M$ implies $y \in M$.
A non-empty subset $M \subseteq E_+$ that is a convex cone in its own right is called a \index{subcone}\emph{subcone}.
Recall that a \index{face}\emph{face} (or \emph{extremal set}) of $E_+$ is a (possibly empty) convex subset $M \subseteq E_+$ such that, if $M$ intersects the relative interior of a line segment in $E_+$, then $M$ contains both endpoints of that segment.

\begin{proposition}
	\label{prop:face-full-subcone}
	A non-empty subset $M \subseteq E_+$ is a face if and only if it is a full subcone.
\end{proposition}
\begin{proof}
	``$\Longrightarrow$''. Suppose that $M$ is a face. First we show that $M$ is a convex cone. If $x\in M$ and $\lambda > 1$, then the line segment from $0$ to $\lambda x$ contains $x$ in its relative interior, so the endpoints $0$ and $\lambda x$ must also belong to $M$. Then, since $M$ is convex, for all $\lambda \in [0,1]$ we also have $\lambda x \in M$. Since a face is convex by assumption, we conclude that $M$ is a convex cone.
	
	To see that $M$ is full, suppose that $x \leq y \leq z$ with $x,z \in M$. Then $x,z \in E_+$, so in particular we have $y \geq x \geq 0$, or in other words, $y \in E_+$. Furthermore, we have $z - y \in E_+$ (since $y \leq z$), so it follows that $y + 2(z - y) \in E_+$. Since $z = y + (z - y)$ is in the relative interior of the line segment from $y$ to $y + 2(z - y)$, we must have $y \in M$, which proves that $M$ is full.
	
	``$\Longleftarrow$''. Suppose that $M \subseteq E_+$ is a full subcone, and suppose that $x,z \in E_+$ and $\lambda \in (0,1)$ are such that $y := \lambda x + (1 - \lambda)z$ belongs to $M$. If $x = z$, then evidently $x = z = y \in M$, so assume $x \neq z$. Then $y$ lies in the relative interior of the line segment between $x$ and $z$, so for small enough $\mu < 0$ the point $\mu x + (1 - \mu) y$ also lies on this line segment. In particular, $\mu x + (1 - \mu) y \geq 0$, or equivalently, $y \geq \frac{-\mu}{1 - \mu} x$. But we have $\frac{-\mu}{1 - \mu} > 0$ and $x \geq 0$, so we find $0 \leq \frac{-\mu}{1 - \mu} x \leq y$. Since $M$ is full, it follows that $\frac{-\mu}{1 - \mu} x \in M$, and therefore $x \in M$. Analogously, $z \in M$.
\end{proof}

Note that the lineality space $\lineal(E_+) = E_+ \cap -E_+ = \{x \in E_+ \, : \, 0 \leq x \leq 0\}$ and the cone $E_+$ itself are full subcones, and therefore faces of $E_+$. Furthermore, clearly every face contains $\lineal(E_+)$ and is contained in $E_+$, so these are the unique \index{face!minimal}minimal and \index{face!maximal}maximal faces of $E_+$.\hair\footnote{In order-theoretic terms, these are the \emph{least} and the \emph{greatest} element in the set of faces (ordered by inclusion).}

Next we come to the subject of ideals. If $I \subseteq E$ is a subspace, then we define the \index{quotient cone}\emph{quotient cone} $(E/I)_+$ to be the image of $E_+$ under the canonical map $E \to E/I$.

\begin{proposition}[Equivalent definitions of an order ideal]
	\label{prop:ideal-definitions}
	Let $(E,E_+)$ be a preordered vector space.
	For a linear subspace $I \subseteq E$, the following are equivalent:
	\begin{enumerate}[label=(\roman*)]
		\item\label{itm:ideal:full} $I$ is full;
		\item\label{itm:ideal:-yxy} if $-y \leq x \leq y$ and $y \in I$, then $x \in I$;
		\item\label{itm:ideal:0xy} if $0 \leq x \leq y$ and $y \in I$, then $x \in I$;
		\item\label{itm:ideal:face} $I_+ := I \cap E_+$ is a face of $E_+$;
		\item\label{itm:ideal:quotient} the quotient cone $(E/I)_+$ is proper.
	\end{enumerate}
\end{proposition}
\begin{proof}
	$\myref{itm:ideal:full} \Longrightarrow \myref{itm:ideal:-yxy}$. Clear.
	
	$\myref{itm:ideal:-yxy} \Longrightarrow \myref{itm:ideal:0xy}$. Suppose that $0 \leq x \leq y$ and $y \in I$. Then we also have $-y \leq 0 \leq x$, so we find $-y \leq x \leq y$. It follows that $x \in I$.
	
	$\myref{itm:ideal:0xy} \Longrightarrow \myref{itm:ideal:face}$. Every linear subspace is a convex cone, and the intersection of two convex cones is a convex cone, so $I_+ \subseteq E_+$ is a subcone. If $x \leq y \leq z$ with $x,z \in I_+$, then in particular $0 \leq y \leq z$ with $z \in I$, so we have $y \in I$. Furthermore, we have $y \geq x \geq 0$, so $y \in I_+$, which shows that $I_+$ is full. By \autoref{prop:face-full-subcone}, $I_+$ is a face of $E_+$.
	
	$\myref{itm:ideal:face} \Longrightarrow \myref{itm:ideal:quotient}$. Let $z \in (E/I)_+ \cap -(E/I)_+$ be given, then we may choose $x,y \in E_+$ such that $z = \pi(x) = \pi(-y)$. It follows that $\pi(x + y) = 0$, so $x + y \in I$. As such, we have $0 \leq x \leq x + y$ and $0 \leq y \leq x + y$ with $0,x + y \in I_+$, so we find $x,y \in I_+$ (since $I_+$ is full). It follows that $z = 0$, which shows that $(E/I)_+$ is a proper cone.
	
	$\myref{itm:ideal:quotient} \Longrightarrow \myref{itm:ideal:full}$. Clearly the natural map $\pi : E \to E/I$ is positive. Suppose that $x \leq y \leq z$ with $x,z \in I$, then $0 = \pi(x) \leq \pi(y) \leq \pi(z) = 0$, so it follows that $\pi(y) = 0$ (since $(E/I)_+$ is a proper cone). Therefore: $y \in I$.
\end{proof}

A subspace $I$ satisfying any one (and therefore all) of the conditions of \autoref{prop:ideal-definitions} is called an \index{order ideal}\emph{order ideal}, or simply \emph{ideal} if no ambiguity can arise (i.e.~if the space does not have additional algebraic structure). Order ideals have been studied since the 1950s (e.g.~\cite{Kadison-representation}, \cite{Bonsall}), but the link between ideals and faces does not appear to be well-known.

We give a few useful ways to obtain ideals or faces:
\begin{proposition}
	\label{prop:ideals-obtained}
	Let $E,F$ be vector spaces and let $E_+ \subseteq E$, $F_+ \subseteq F$ be convex cones.
	\begin{enumerate}[label=(\alph*)]
		\item\label{itm:ideal-face} If $M \subseteq E_+$ is a non-empty face, then $\spn(M)$ is an ideal satisfying $M = \spn(M) \cap E_+$.
		
		\item\label{itm:ideal-pullback} If $T : E \to F$ is a positive linear map and if $J \subseteq F$ is an ideal, then $T^{-1}[J] \subseteq E$ is an ideal.
	\end{enumerate}
\end{proposition}
\begin{proof}
	\leavevmode
	\begin{enumerate}[label=(\alph*)]
		\item Clearly $M \subseteq \spn(M) \cap E_+$. Moreover, since $M$ is a convex cone, every $x \in \spn(M)$ can be written as $x = m_1 - m_2$ with $m_1,m_2 \in M$. If furthermore $x \in E_+$, then we find $0 \leq x \leq m_1$ (because $m_1 - x = m_2 \geq 0$), and therefore $x \in M$ (because $M$ is full). This shows that $M = \spn(M) \cap E_+$. It follows from \myautoref{prop:ideal-definitions}{itm:ideal:face} that $\spn(M)$ is an ideal.
		
		\item If $x \leq y \leq z$ and $x,z\in T^{-1}[J]$, then $T(x) \leq T(y) \leq T(z)$ with $T(x),T(z) \in J$. Since $J$ is full, it follows at once that $T(y) \in J$, which shows that $T^{-1}[J]$ is also full. \qedhere
	\end{enumerate}
\end{proof}

It follows from \myautoref{prop:ideal-definitions}{itm:ideal:face} and \myautoref{prop:ideals-obtained}{itm:ideal-face} that the map $I \mapsto I_+$ defines a surjective many-to-one correspondence between the ideals and the non-empty faces.

A first (and rather important) application of this correspondence is given in \myautoref{prop:ideal-face-kernel}{itm:face-kernel} below. If $\varphi : E \to \R$ is a positive linear functional, then $\ker(\varphi) \cap E_+$ is easily seen to be a face, and faces of this type are called \index{face!exposed}\emph{exposed}. This can be generalized in the following way: if $F$ is any vector space with a proper cone $F_+ \subseteq F$, and if $T : E \to F$ is a positive linear map, then it is still relatively easy to see that $\ker(T) \cap E_+$ is a face. (It is crucial that $F_+$ is proper!) Although not every face is exposed, the following result shows that this slight extension already captures all faces.

\begin{proposition}
	\label{prop:ideal-face-kernel}
	Let $E$ be a vector space and let $E_+ \subseteq E$ be a convex cone.
	\begin{enumerate}[label=(\alph*)]
		\item\label{itm:ideal-kernel} (cf.~\cite[\S 2, p.~403]{Bonsall}) A subspace $I \subseteq E$ is an ideal if and only if it occurs as the kernel of a positive linear map $T : E \to F$ with $F_+$ proper.
		
		\item\label{itm:face-kernel} A non-empty subset $M \subseteq E_+$ is a face if and only if it can be written as $M = \ker(T) \cap E_+$ with $T : E \to F$ positive and $F_+$ proper.
	\end{enumerate}
\end{proposition}
\begin{proof}\leavevmode
	\begin{enumerate}[label=(\alph*)]	
		\item If $I \subseteq E$ is an ideal, then $(E/I)_+$ is a proper cone (by \myautoref{prop:ideal-definitions}{itm:ideal:quotient}), the map $T : E \to E/I$ is positive, and $I = \ker(T)$.
		
		Conversely, if $T : E \to F$ is a positive linear map with $F_+ \subseteq F$ a proper cone, then $\{0\} \subseteq F_+$ is an ideal (because $F_+$ is proper), so it follows from \myautoref{prop:ideals-obtained}{itm:ideal-pullback} that $\ker(T)$ is an ideal in $E$.
		
		\item If $M \subseteq E_+$ is a face, then $I := \spn(M)$ is an ideal with $M = I \cap E_+$ (by \myautoref{prop:ideals-obtained}{itm:ideal-face}), so $(E/I)_+$ is a proper cone, the map $T : E \to E/I$ is positive, and $M = \ker(T) \cap E_+$.
		
		Conversely, if $T : E \to F$ is a positive linear map with $F_+ \subseteq F$ a proper cone, then it follows from \ref{itm:ideal-kernel} that $\ker(T)$ is an ideal, so $\ker(T) \cap E_+$ is a face.
		\qedhere
	\end{enumerate}
\end{proof}

\begin{remark}
	Just as $\lineal(E_+)$ and $E_+$ are the smallest and the largest face of $E_+$, the smallest and the largest ideals of $E$ are $\lineal(E_+)$ and $E$. Apart from this, the maximal ideals $\neq E$ are of some interest; see \autoref{cor:maximal-ideals} below.
	
	For now, we show that the smallest ideal has the following special property.
\end{remark}

\begin{proposition}
	\label{prop:quotient-bipositive}
	Let $E$ be a vector space, $E_+ \subseteq E$ a convex cone, and $I \subseteq E$ a subspace. Then the quotient $\pi_I : E \to E/I$ is bipositive if and only if $I \subseteq \lineal(E_+)$.
	
	In particular, the only ideal $I \subseteq E$ for which the quotient $\pi_I : E \to E/I$ is bipositive is the minimal ideal $I = \lineal(E_+)$.
\end{proposition}
\begin{proof}
	Bipositivity of the quotient $E \to E/I$ means that, if $x \in E_+$ and $x + I = y + I$, then $y \in E_+$. Equivalently: if $x \in E_+$ and $z \in I$, then $x + z \in E_+$. Evidently this is the case if and only if $I \subseteq E_+$ (use that $0 \in E_+$). But $I$ is a subspace, so we have $I \subseteq E_+$ if and only if $I \subseteq \lineal(E_+)$.
	
	If $I$ is an ideal, then we have $\lineal(E_+) \subseteq I$ (every ideal contains the minimal ideal), so the second conclusion follows immediately.
\end{proof}

\begin{remark}
	\label{rmk:bipositive-injective}
	If $E_+$ is proper and if $F_+$ is arbitrary, then every bipositive map $T : E \to F$ is automatically injective, since $\ker(T) = T^{-1}[\{0\}] \subseteq T^{-1}[F_+] = E_+$ is a subspace contained in $E_+$, which must therefore be $\{0\}$. The preceding proposition shows that this is no longer true if $E_+$ is not proper.
\end{remark}

\section{The homomorphism and isomorphism theorems}
\label{app:isomorphism-theorems} 
\label{sec:isomorphism-theorems} 
In connection with the ideal theory, we investigate to which extent the homomorphism and isomorphism theorems hold for ordered vector spaces.

The homomorphism theorem and the third isomorphism theorem hold true for ordered vector spaces.

\begin{proposition}[Homomorphism theorem]
	\label{prop:homomorphism-theorem}
	Let $E$, $F$ be vector spaces, $E_+ \subseteq E$, $F_+ \subseteq F$ convex cones, $T : E \to F$ a positive linear map, and $I \subseteq E$ a subspace with $I \subseteq \ker(T)$. Then there is a unique positive linear map $\tilde T : E/I \to F$ for which the following diagram commutes:
	\begin{center}
		\begin{tikzcd}[column sep=tiny]
			E \arrow[rr, "T"] \arrow[dr, swap, "\pi_I"] & & F.\\
			& E/I \arrow[ur, dashed, swap, "\tilde T"] & 
		\end{tikzcd}
	\end{center}
\end{proposition}
\begin{proof}
	Since $I \subseteq \ker(T)$, there is a unique linear map $\tilde T : E/I \to F$ for which the diagram commutes. This map is automatically positive: if $y \in (E/I)_+$, then there is some $x \in E_+$ such that $y = \pi_I(x)$, and it follows that $\tilde T(y) = T(x) \in T[E_+] \subseteq F_+$.
\end{proof}
\begin{proposition}[Third isomorphism theorem]
	\label{prop:third-isomorphism-theorem}
	Let $E$ be a vector space, $E_+ \subseteq E$ a convex cone, and $I \subseteq J \subseteq E$ subspaces. Then the natural isomorphism $(E/I)/(J/I) \cong E/J$ is bipositive for the respective quotient cones. Furthermore, the bijective correspondence $J \mapsto J/I$ between the subspaces $I \subseteq J \subseteq E$ and the subspaces of $E/I$ restricts to a bijective correspondence of order ideals \textup(in other words, $J$ is an ideal in $E$ if and only if $J/I$ is an ideal in $E/I$\textup).
\end{proposition}
\begin{proof}
	We have the following commutative diagram of linear maps:
	\begin{center}
		\begin{tikzcd}[column sep=tiny]
			E \arrow[rr, "\pi_J"] \arrow[dr, swap, "\pi_I"] & & E/J.\\
			& E/I \arrow[ur, swap, "\pi_{J/I}"] & 
		\end{tikzcd}
	\end{center}
	To see that the natural isomorphism $(E/I)/(J/I) \cong E/J$ is bipositive, note that pushforwards commute: an element of $E/J$ belongs to either one of the pushforward cones $(E/J)_+$ and $((E/I)/(J/I))_+$ if and only if it has a positive element of $E$ in its preimage.
	
	Since a subspace is an ideal if and only if the quotient cone is proper, it follows immediately that $J$ is an ideal in $E$ if and only if $J/I$ is an ideal in $E/I$.
\end{proof}

Analogous results hold for \emph{closed} ideals in ordered \emph{topological} vector spaces. (We assume no compatibility between the positive cone and the topology, so questions of continuity and positivity are completely separate from one another.)

Contrary to the preceding results, the first and second isomorphism theorems fail for ordered vector spaces. We only have the following weaker statements, of which the (simple) proofs are omitted.

\begin{proposition}[Partial first isomorphism theorem]
	Let $E$, $F$ be vector spaces, $E_+ \subseteq E$, $F_+ \subseteq F$ convex cones, and $T : E \to F$ a positive linear map. Then the natural linear isomorphism $E/\ker(T) \stackrel{\sim}{\longrightarrow} \ran(T)$ is positive, but not necessarily bipositive.
\end{proposition}
Counterexample against bipositivity: $E = F$ and $T = \id_E$, but $E_+$ strictly contained in $F_+$.
\begin{proposition}[Partial second isomorphism theorem]
	Let $F$ be a vector space, $F_+ \subseteq F$ a convex cone, $E \subseteq F$ a subspace, and $I \subseteq F$ an order ideal. Then $E + I$ is a subspace of $F$, $E \cap I$ is an order ideal of $E$, and the natural linear isomorphism $E/(E \cap I) \stackrel{\sim}{\longrightarrow} (E + I)/I$, $x + (E \cap I) \mapsto x + I$ is positive, but not necessarily bipositive.
\end{proposition}
Counterexample against bipositivity: $F = \R^2$ with standard cone, and $E,I \subseteq F$ two different one-dimensional subspaces, each of which meets $F_+$ only in $0$. Then $E_+ := E \cap F_+ = \{0\}$, so the cone of $E/(E \cap I)$ is $\{0\}$, whereas the cone of $(E + I)/I = \R^2/I$ is generating.

\subsection{Classification of maximal order ideals}
As an application of the preceding results, we show that the third isomorphism theorem gives a geometric characterization of the maximal order ideals.

Following common terminology from algebra, we say that an order ideal $I \subseteq E$ is \index{order ideal!proper}\emph{proper} if $I \neq E$, and \index{order ideal!maximal}\emph{maximal} if it is proper and not contained in another proper ideal.
Furthermore, we say that a preordered vector space $E$ is \index{ordered (topological) vector space!simple}\emph{simple} if $E_+$ is proper and if $E$ has exactly two order ideals (namely, the trivial ideals $\{0\}$ and $E$).
Bonsall \cite[Theorem 2]{Bonsall} proved that an ordered vector space is simple if and only if it is one-dimensional (with either the standard cone or the zero cone).%
	\hair\footnote{Bonsall also includes $\{0\}$ among the simple ordered spaces, but we require \emph{exactly} two ideals. (Similarly, we believe that $1$ is not prime, the empty topological space is not connected, etc.) This is just a matter of convention.}
Combining this with \autoref{prop:third-isomorphism-theorem}, we find:

\begin{corollary}
	\label{cor:maximal-ideals}
	The maximal order ideals of $E$ are precisely the supporting hyperplanes of $E_+$.
\end{corollary}
\begin{proof}
	It is easy to see that the supporting hyperplanes of $E_+$ are precisely the kernels of the non-zero positive linear functionals. (For a proof, see e.g.~\cite[Proposition 4.1]{Dobben-semisimplicity}.) Furthermore, it follows from \autoref{prop:third-isomorphism-theorem} that an ideal $I \subseteq E$ is maximal if and only if $E/I$ is simple.
	
	If $\varphi : E \to \R$ is a non-zero positive linear functional, then $\ker(\varphi)$ is an ideal, which is maximal since $E/\ker(\varphi)$ is one-dimensional and therefore simple.
	
	Conversely, if $I \subseteq E$ is a maximal ideal, then $E/I$ is simple, so $\dim(E/I) = 1$ and the quotient cone $(E/I)_+$ is either $\{0\}$ or isomorphic to the standard cone $\R_{\geq 0}$. Either way, we can choose a linear isomorphism $E/I \stackrel{\sim}{\longrightarrow} \R$ which is positive (but not necessarily bipositive), so that the composition $\varphi : E \to E/I \to \R$ is a non-zero positive linear functional with $I = \ker(\varphi)$.
\end{proof}

For more on maximal ideals, see \cite[\S 4]{Bonsall}.

\section{Dual and exposed faces}
\label{app:dual-and-exposed} 
\label{sec:dual-and-exposed} 
In the finite\-/dimensional setting (with closed cones), dual faces are well studied in the literature (see e.g.~\cite{Barker-duality}, \cite{Weis}). We outline a theory of face duality in dual pairs.

A \index{positive pairing}\emph{positive pairing} is a dual pair $\langle E , F \rangle$ of (real) preordered vector spaces such that $\langle x , y \rangle \geq 0$ whenever $x \in E_+$, $y \in F_+$. In this case we say $\langle E_+ , F_+ \rangle$ is a \emph{positive pair}.\hair\footnote{There is a slight abuse of notation here, for if $E_+$ and $F_+$ are not generating, then the positive pair depends not only on $E_+$ and $F_+$, but also on $E$ and $F$ (but this will cause no confusion).}

Given a positive pair $\langle E_+ , F_+ \rangle$ and a non-empty subset $N \subseteq F_+$, we define the \index{face!predual}\emph{\textup(pre\textup)dual face}
\[ \glsadd{predualface}{}^\predualface N := E_+ \cap {}^\perp N = \big\{x \in E_+ \, : \, \langle x , y \rangle = 0\ \text{for all $y \in N$}\big\}. \]
Analogously, for a non-empty subset $M \subseteq E_+$ we define the \index{face!dual}\emph{dual face}
\[ \glsadd{dualface}M^\dualface := F_+ \cap M^\perp = \big\{y \in F_+ \, : \, \langle x , y\rangle = 0\ \text{for all $x \in M$}\big\}. \]
Note that ${}^\predualface N$ (resp.~$M^\dualface$) depends not only on $N$ (resp.~$M$), but also implicitly on $E_+$ (resp.~$F_+$).

\begin{proposition}
	\label{prop:df}
	Let $\langle E_+ , F_+ \rangle$ be a positive pair.
	\begin{enumerate}[label=(\alph*)]
		\item\label{itm:df:face} If $N \subseteq F_+$ is non-empty, then ${}^\predualface N$ is a face of $E_+$.
		
		\item\label{itm:df:reverse-inclusion} If $N_1 \subseteq N_2 \subseteq F_+$ are non-empty, then ${}^\predualface N_1 \supseteq {}^\predualface N_2$.
		
		\item\label{itm:df:second-dual} If $N \subseteq F_+$ is non-empty, then $N \subseteq ({}^\predualface N)^\dualface$ and ${}^\predualface N = {}^\predualface (({}^\predualface N)^\dualface)$.
	\end{enumerate}
	Similar statements hold with $N$ and ${}^\predualface N$ replaced by $M$ and $M^\dualface$.
\end{proposition}
\begin{proof}
	\ \par
	\begin{enumerate}[label=(\alph*)]
		\item Every $y \in F_+$ defines a positive linear functional $\varphi_y : E \to \R$, $x \mapsto \langle x , y\rangle$. As such, the set $E_+ \cap \ker(\varphi_y)$ is a face, by \myautoref{prop:ideals-obtained}{itm:ideal-pullback}. Since we can write
		\[ {}^\predualface N = \bigcap_{y \in N} (E_+ \cap \ker(\varphi_y)), \]
		it follows that ${}^\predualface N$ is also a face of $E_+$.
		
		\item This follows from the definition, since ${}^\perp N_1 \supseteq {}^\perp N_2$.
		
		\item If $y \in N$, then by definition one has $\langle x , y \rangle = 0$ for all $x \in {}^\predualface N$, so it follows that $y \in ({}^\predualface N)^\dualface$. This proves the inclusion $N \subseteq ({}^\predualface N)^\dualface$.
		
		Write $M := {}^\predualface N$. It follows from the preceding argument that $M \subseteq {}^\predualface (M^\dualface) = {}^\predualface (({}^\predualface N)^\dualface)$. On the other hand, combining the inclusion $N \subseteq ({}^\predualface N)^\dualface$ with \ref{itm:df:reverse-inclusion}, we find $M = {}^\predualface N \supseteq {}^\predualface (({}^\predualface N)^\dualface)$, so we conclude that equality holds.
		\qedhere
	\end{enumerate}
\end{proof}

A face $M \subseteq E_+$ is said to be an \index{face!dual!with respect to a positive pairing}\emph{$\langle E_+ , F_+ \rangle$-dual face} if $M = {}^\predualface N$ for some non-empty subset $N \subseteq F_+$, and an \index{face!exposed!with respect to a positive pairing}\emph{$\langle E_+ , F_+ \rangle$-exposed face} if there is some $y_0 \in F_+$ such that $M = E_+ \cap \ker(\varphi_{y_0})$, or equivalently, if $M$ is the $\langle E_+ , F_+ \rangle$-dual face of a singleton. Likewise, the faces $N \subseteq F_+$ of the form $N = M^\dualface$ (resp.~$N = \{x_0\}^\dualface$) are the $\langle F_+ , E_+ \rangle$-dual (resp.~$\langle F_+ , E_+ \rangle$-exposed) faces of $F_+$.

The operations $M \mapsto M^\dualface$ and $N \mapsto {}^\predualface N$ define a so-called \emph{Galois connection} (see e.g.~\cite[\S 6.5]{Bergman} for the definition). It follows that the set of $\langle E_+ , F_+ \rangle$-dual faces, ordered by inclusion, forms a complete lattice, which we denote as $\dualfacelattice{E_+}{F_+}$.

If $\langle E_+ , G_+ \rangle$ is a positive pair and if $F \subseteq G$ and $F_+ \subseteq F \cap G_+$, then evidently one has $\dualfacelattice{E_+}{F_+} \subseteq \dualfacelattice{E_+}{G_+}$, but the inclusion $\dualfacelattice{E_+}{F_+} \hookrightarrow \dualfacelattice{E_+}{G_+}$ should not be expected to be a lattice homomorphism.

Given a dual pair $\langle E , E' \rangle$ and a convex cone $E_+ \subseteq E$, the most natural lattice of dual faces in $E_+$ is the lattice $\dualfacelattice{E_+}{E_+'}$, where $E_+' \subseteq E'$ is the dual cone of $E_+$. (This is the largest of all lattices $\dualfacelattice{E_+}{F_+}$ with $F_+ \subseteq E'$.)
The $\langle E_+ , E_+' \rangle$-dual (resp.~$\langle E_+ , E_+' \rangle$-exposed) faces will simply be called the \emph{dual} (resp.~\emph{exposed}) faces of $E_+$.

\subsection{The difference between dual and exposed faces}
If $E$ is finite\-/dimensional and if $E_+$ is closed, then every dual face is exposed, so $\dualfacelattice{E_+}{E_+\algdual}$ is simply the lattice of exposed faces (see e.g.~\cite{Barker-duality}). We intend to show that things become more complicated in the infinite\-/dimensional case. We illustrate these subtleties by establishing various equivalent definitions of dual and exposed faces.

For notational simplicity, we formulate the results in the remainder of this appendix not for dual pairs but for locally convex spaces. We recall some basic theory. If $\lintop$ is a locally convex topology on $E$ that is compatible with the dual pair $\langle E,E'\rangle$, then a subspace $I \subseteq E$ is $\lintop$-closed if and only if it is weakly closed. If this is the case, then the quotient $E/I$ is once again a (Hausdorff) locally convex space, and $(E/I)\topdual \cong I^\perp$ as vector spaces. Furthermore, if $\lintop$ is the weak topology $\sigma(E,E')$, then $E/I$ carries the weak topology $\sigma(E/I,I^\perp)$ (see e.g.~\cite[\S V.2]{Conway} or \cite[\S 22]{Kothe-I}).

We shall say that a convex cone $E_+ \subseteq E$ is \index{convex cone!quasi-semisimple}\emph{quasi-semisimple} if $E_+'$ separates points on $E_+$; that is, if for every $x \in E_+$ there is some $\varphi_x \in E_+'$ such that $\langle x,\varphi_x\rangle > 0$. This is equivalent to the (geometric) requirement that $E_+ \cap \lineal(\,\overline{E_+}\,) = \{0\}$, since $\lineal(\,\overline{E_+}\,) = {}^\perp (E_+')$. It follows that a quasi-semisimple cone is automatically proper. Clearly every semisimple cone (in particular, every closed proper cone in a locally convex space) is quasi-semisimple.

\begin{proposition}
	\label{prop:exposed-face}
	Let $E$ be locally convex. For a face $M \subseteq E_+$ the following are equivalent:
	\begin{enumerate}[label=(\roman*)]
		\item\label{itm:exposed-face} $M$ is exposed.
		\item\label{itm:is-exposed} There exists some $\varphi_0 \in M^\dualface$ such that for all $x \in E_+ \setminus M$ one has $\langle x,\varphi_0\rangle > 0$.
		\item\label{itm:quotient-admits-strictly-positive-functional} $M = E_+ \cap \overline\spn(M)$, and the quotient $(E/\,\overline\spn(M))_+$ admits a strictly positive continuous linear functional.
	\end{enumerate}
\end{proposition}
\begin{proof}
	$\myref{itm:exposed-face} \Longrightarrow \myref{itm:quotient-admits-strictly-positive-functional}$. Choose $\varphi_0 \in E_+'$ such that $M = E_+ \cap \ker(\varphi_0)$. Then $\overline\spn(M) \subseteq \ker(\varphi_0)$, so it follows that $M = E_+ \cap \overline\spn(M)$ and that $\varphi_0$ factors through $E/\overline\spn(M)$:
	\begin{center}
		\begin{tikzcd}[column sep=tiny]
			E \arrow[rr, "\varphi_0"] \arrow[dr, two heads, "\pi"] & & \R \\
			& E/\overline\spn(M) \arrow[ur, swap, "\psi_0"] & 
		\end{tikzcd}
	\end{center}
	If $E/\overline\spn(M)$ is equipped with the quotient cone, then $\psi_0 : E/\overline\spn(M) \to \R$ is strictly positive.
	
	$\myref{itm:quotient-admits-strictly-positive-functional} \Longrightarrow \myref{itm:is-exposed}$. We have $(E/\overline\spn(M))\topdual \cong \overline\spn(M)^\perp = M^\perp$. Consequently, if $\psi_0 : E/\overline\spn(M) \to \R$ is continuous and strictly positive, then the composition $\varphi_0 : E \stackrel{\pi}{\longrightarrow} E/\overline\spn(M) \stackrel{\psi_0}{\longrightarrow} \R$ is continuous and positive, and belongs to $M^\perp$. It follows that $\varphi_0 \in E_+' \cap M^\perp = M^\dualface$. Furthermore, every $x \in E_+ \setminus M$ satisfies $\langle x , \varphi_0 \rangle = \langle \pi x , \psi_0 \rangle > 0$, since $\psi_0$ is strictly positive.
	
	$\myref{itm:is-exposed} \Longrightarrow \myref{itm:exposed-face}$. The requirement $\{\varphi_0\} \subseteq M^\dualface$ ensures that $M \subseteq {}^\predualface (M^\dualface) \subseteq {}^\predualface \{\varphi_0\}$, and the assumption that $\langle x , \varphi_0 \rangle > 0$ for all $x \in E_+ \setminus M$ guarantees that ${}^\predualface \{\varphi_0\} \subseteq M$.
\end{proof}

\begin{proposition}
	\label{prop:dual-face}
	Let $E$ be locally convex. For a face $M \subseteq E_+$ the following are equivalent:
	\begin{enumerate}[label=(\roman*)]
		\item\label{itm:dual-face} $M$ is a dual face.
		\item\label{itm:separates-points} For every $x \in E_+ \setminus M$ there is some $\varphi_x \in M^\dualface$ such that $\langle x,\varphi_x\rangle > 0$.
		\item\label{itm:quotient-is-quasi-semisimple} $M = E_+ \cap \overline\spn(M)$, and the quotient $(E/\,\overline\spn(M))_+$ is quasi-semisimple.
	\end{enumerate}
\end{proposition}
\begin{proof}
	$\myref{itm:dual-face} \Longrightarrow \myref{itm:quotient-is-quasi-semisimple}$. Choose some non-empty $N \subseteq E_+'$ such that $M = {}^\predualface N = E_+ \cap {}^\perp N$. Then $\overline\spn(M) \subseteq {}^\perp N$, so it follows that $M = E_+ \cap \overline\spn(M)$ and that every $\varphi \in N$ factors through $E / \overline\spn(M)$:
	\begin{center}
		\begin{tikzcd}[column sep=tiny]
			E \arrow[rr, "\varphi"] \arrow[dr, "\pi"] & & \R \\
			& E / \overline\spn(M) \arrow[ur, swap, "\psi"] &
		\end{tikzcd}
	\end{center}
	Write $\mywedge$ for the positive cone of $E / \overline\spn(M)$, and let $y \in \mywedge$ be such that $\langle y , \psi \rangle = 0$ for all $\psi \in \mywedge\topdual$. Choose $x \in E_+$ such that $y = \pi(x)$, then for every $\varphi \in N$ we may choose some $\psi \in \mywedge\topdual$ such that $\varphi = \psi \circ \pi$, and therefore $\langle x , \varphi \rangle = \langle y , \psi \rangle = 0$. It follows that $x \in M$, and therefore $y = 0$, showing that $\mywedge$ is quasi-semisimple.
	
	$\myref{itm:quotient-is-quasi-semisimple} \Longrightarrow \myref{itm:separates-points}$. Let $x \in E_+ \setminus M$ be given. Then $x$ is mapped to a non-zero positive vector in $E / \overline\spn(M)$, so there is a positive continuous linear functional $\psi_x : E/\overline\spn(M) \to \R$ such that $\langle \pi x , \psi_x \rangle > 0$ (by quasi-semisimplicity). Now the composition $\varphi_x : E \stackrel{\pi}{\longrightarrow} E/\overline\spn(M) \stackrel{\psi_x}{\longrightarrow} \R$ is continuous and positive, and $\langle x , \varphi_x \rangle = \langle \pi x , \psi_x \rangle > 0$.
	
	$\myref{itm:separates-points} \Longrightarrow \myref{itm:dual-face}$. We have $M \subseteq {}^\predualface (M^\dualface)$, and the assumption that every $x \in E_+ \setminus M$ admits some $\varphi_x \in M^\dualface$ such that $\langle x , \varphi_x \rangle > 0$ guarantees that ${}^\predualface (M^\dualface) \subseteq M$.
\end{proof}

The subtle difference between exposed and dual faces becomes apparent by comparing \myautoref{prop:exposed-face}{itm:is-exposed} with \myautoref{prop:dual-face}{itm:separates-points}: the only difference is the order of the quantifiers!

In the finite\-/dimensional setting, it is well-known that the dual faces are precisely the exposed faces (see e.g.~\cite{Barker-duality}). We extend this to separable normed spaces. Counterexamples in other settings will be given below.

\begin{theorem}[{Compare \cite[Proposition 15.2]{Schaefer-III}}]
	\label{thm:dual-face-is-exposed}
	Let $E$ be locally convex, and let $M \subseteq E_+$ be a face of the form $E_+ \cap I$, where $I \subseteq E$ is a closed subspace. If $E/I$ admits a separable norm compatible with the dual pair $\langle E/I , (E/I)\topdual \rangle$ \textup($\, =\langle E/I,I^\perp\rangle\,$\textup), then $M$ is a dual face if and only if $M$ is exposed.
\end{theorem}
\begin{proof}
	Every exposed face is dual. For the converse, suppose that $M$ is a dual face, and let $\lVert\:\cdot\:\rVert$ be a separable norm compatible with the dual pair $\langle E/I,I^\perp\rangle$. We shall understand $E/I$ and $(E/I)\topdual$ to be equipped with the respective norm topologies. Since $E/I$ is separable, its dual $(E/I)\topdual$ and every subset thereof is weak\nobreakdash-$*$ separable (this is because it can be written as the union of a countable family of separable metrizable spaces; see \cite[\S 21.3.(5)]{Kothe-I}). As such, we may choose a weak\nobreakdash-$*$ dense countable subset $N = \{\varphi_k\}_{k=1}^\infty$ in the dual cone $(E/I)_+\topdual$.
	Define $\psi := \sum_{k=1}^\infty \frac{\varphi_k}{2^k \lVert \varphi_k\rVert}$; this is well-defined because $(E/I)\topdual$ is a Banach space.\footnote{Technically this is not entirely well-defined; if $\varphi_k = 0$, then we must replace $\frac{\varphi_k}{2^k \lVert \varphi_k\rVert}$ by $0$.} Since $(E/I)_+\topdual$ is a closed convex cone, we have $\psi \in (E/I)_+\topdual$.
	
	We claim that $\psi$ is a strictly positive functional. To that end, let $x \in (E/I)_+$ be such that $\psi(x) = 0$. For all $k$ we have $\varphi_k(x) \geq 0$, but $\psi(x) = \sum_{k=1}^\infty \frac{\varphi_k(x)}{2^k \lVert \varphi_k\rVert} = 0$, so we must have $\varphi_k(x) = 0$. It follows that $x \in {}^\perp N = {}^\perp ((E/I)_+\topdual) = \lineal(\,\overline{(E/I)_+}\,)$. Since $M$ is a dual face, the quotient face $(E/I)_+$ is quasi-semisimple, so $(E/I)_+ \cap \lineal(\,\overline{(E/I)_+}\,) = \{0\}$. It follows that $x = 0$, which shows that $\psi$ is a strictly positive functional. We conclude that $M$ is exposed.
\end{proof}

\begin{corollary}
	A face of finite codimension is dual if and only if it is exposed.
\end{corollary}

\begin{corollary}
	In a separable normed space, the dual faces are precisely the exposed faces.
\end{corollary}

\begin{corollary}
	\label{cor:closed-lineality-space-exposed}
	If $E$ is a separable normed space and if $E_+$ is closed, then $\lineal(E_+)$ is exposed.
\end{corollary}
After all, $\lineal(E_+) = \lineal(\,\overline{E_+}\,) = {}^\perp (E_+\topdual) = {}^\predualface (E_+\topdual)$ is a dual face.

\begin{corollary}
	\label{cor:quasi-semisimple-cone-in-separable-normed-space}
	Every quasi-semisimple cone \textup(in particular, every closed proper cone\textup) in a separable normed space admits a strictly positive continuous linear functional.
\end{corollary}

In general, not every dual face is exposed. As a generic example, let $E_+$ be a cone that is semisimple but does not admit a strictly positive functional. Then $\{0\}$ is a dual face, for by semisimplicity, $E_+'$ separates points, so ${}^\predualface (E_+') = {}^\perp (E_+') = \{0\}$. However, $\{0\}$ is not exposed, since there is no strictly positive functional.

We give two concrete realizations of this generic example: one in an inseparable Hilbert space, and one in a separable Fr\'echet space. These examples show that the preceding corollaries cannot easily be extended beyond the setting of separable normed spaces.

\begin{example}
	\label{xmpl:non-exposed-i}
	Let $\Omega$ be an uncountable set, and consider the Hilbert space $E = \ell_\R^2(\Omega)$ with the non-negative cone $E_+ = \{x \in \ell_\R^2(\Omega) \, : \, x_\omega \geq 0\ \text{for all $\omega \in \Omega$}\}$. Then $E_+$ is semisimple, so $\{0\}$ is a dual face. However, $E_+$ does not admit a strictly positive functional, since every vector in $E\topdual = \ell_\R^2(\Omega)$ is zero in all but at most countably many coordinates.
\end{example}

\begin{example}
	\label{xmpl:non-exposed-ii}
	Let $s$ be the space of all (real) sequences with the topology of pointwise convergence. Then $s$ is a separable Fr\'echet space with topological dual $s\topdual = c_{00}$, the space of sequences of finite support. (The last statement is a special case of duality between products and locally convex direct sums; see \cite[\S 22.5.(2)]{Kothe-I}.) The non-negative cone in $s$ is closed and proper, so $\{0\}$ is a dual face. However, there is no strictly positive functional.
\end{example}

As a final remark, we point out that certain faces stand no chance of being either exposed or a dual face.
We know from \myautoref{prop:ideals-obtained}{itm:ideal-face} that every face is the positive part of an ideal, but a dual or exposed face must be the positive part of a \emph{closed} ideal.
It may happen that $E_+ \cap \overline\spn(M)$ is larger than $M$, in which case \myautoref{prop:dual-face}{itm:quotient-is-quasi-semisimple} (resp.~\myautoref{prop:exposed-face}{itm:quotient-admits-strictly-positive-functional}) shows that $M$ cannot be a dual (resp.~exposed) face.

As a concrete example, let $E = \ell_\R^\infty$ with its usual cone and norm, and let $M = E_+ \cap c_{00}$ be the set of all non-negative sequences with finite support. Then $M$ is a face, but $E_+ \cap \,\overline\spn(M) = E_+ \cap \, c_0$ is the (larger) set of all non-negative sequences converging to $0$.
Therefore $M$ is not exposed or a dual face.

\glsadd[format=glsignore]{Mperp}
\glsadd[format=glsignore]{perpN}
\glsadd[format=glsignore]{tensor product}
\nocite{*} 

\backmatter
\small

\cleardoublepage
\setglossarystyle{long3col}%
\myglsFindWidestNameAndPagenum%
\renewenvironment{theglossary}{%
	\renewcommand{\arraystretch}{1.33}%
	\begin{longtable}{p{\myglsWidestName}p{\myglsdescwidth}p{\myglsWidestPagenum}}%
}{%
	\end{longtable}%
}%
\printnoidxglossary[type=symbols,title=\myglossaryname]

\cleardoublepage
\printindex

\bibliographystyle{my-alphaurl}
\bibliography{convex_cones}

\providecommand{\noopsort}[1]{} \providecommand{\DobbendeBruyn}{van Dobben de
  Bruyn} \providecommand{\Do}{Dob} \providecommand{\Gaans}{van Gaans}
  \providecommand{\Ga}{Gaa} \providecommand{\Jonge}{de Jonge}
  \providecommand{\Jo}{Jon} \providecommand{\MeyerNieberg}{Meyer-Nieberg}
  \providecommand{\Me}{Mey} \providecommand{\Rooij}{van Rooij}
  \providecommand{\Ro}{Roo} \providecommand{\SaintRaymond}{Saint--Raymond}
  \providecommand{\Sa}{Sai} \providecommand{\BenAmor}{Ben Amor}
  \providecommand{\Be}{Ben}
\providecommand{\bibhairspace}{\ifmmode\mskip1mu\else\kern0.08em\fi}
\begin{thebibliography}{GLMP04}

\bibitem[ABJ18]{Azouzi-BenAmor-Jaber}
Y.~Azouzi, M.\bibhairspace A.~\BenAmor{}, and J.~Jaber.
\newblock The tensor product of $f$-algebras.
\newblock {\em Quaestiones Mathematicae}, \textbf{41}(3):359--369, 2018.
\newblock \href {https://doi.org/10.2989/16073606.2017.1382018}
  {\path{doi:10.2989/16073606.2017.1382018}}.

\bibitem[ALP19]{Aubrun-et-al-i}
G.~Aubrun, L.~Lami, and C.~Palazuelos.
\newblock Universal entangleability of non-classical theories, 2019.
\newblock Preprint.
\newblock URL: \url{https://arxiv.org/abs/1910.04745v1}.

\bibitem[ALPP21]{Aubrun-et-al-ii}
G.~Aubrun, L.~Lami, C.~Palazuelos, and M.~Pl{\'a}vala.
\newblock Entangleability of cones.
\newblock {\em Geometric and Functional Analysis}, \textbf{31}(2):181--205,
  2021.
\newblock \href {https://doi.org/10.1007/s00039-021-00565-5}
  {\path{doi:10.1007/s00039-021-00565-5}}.

\bibitem[And04]{Ando}
T.~Ando.
\newblock Cones and norms in the tensor product of matrix spaces.
\newblock {\em Linear Algebra and its Applications}, \textbf{379}:3--41, 2004.

\bibitem[AS17]{Aubrun-Szarek}
G.~Aubrun and S.\bibhairspace J.~Szarek.
\newblock {\em {A}lice and {B}ob Meet {B}anach}, volume 223 of {\em
  Mathematical Surveys and Monographs}.
\newblock American Mathematical Society, 2017.

\bibitem[AT07]{Aliprantis-Tourky}
C.\bibhairspace D.~Aliprantis and R.~Tourky.
\newblock {\em Cones and Duality}, volume~84 of {\em Graduate Studies in
  Mathematics}.
\newblock American Mathematical Society, 2007.

\bibitem[Bar73]{Barker-faces}
G.\bibhairspace P.~Barker.
\newblock The lattice of faces of a finite dimensional cone.
\newblock {\em Linear Algebra and its Applications}, \textbf{7}(1):71--82,
  1973.

\bibitem[Bar76]{Barker-monotone}
G.\bibhairspace P.~Barker.
\newblock Monotone norms and tensor products.
\newblock {\em Linear and Multilinear Algebra}, \textbf{4}(3):191--199, 1976.

\bibitem[Bar78a]{Barker-duality}
G.\bibhairspace P.~Barker.
\newblock Faces and duality in convex cones.
\newblock {\em Linear and Multilinear Algebra}, \textbf{6}(3):161--169, 1978.

\bibitem[Bar78b]{Barker-perfect}
G.\bibhairspace P.~Barker.
\newblock Perfect cones.
\newblock {\em Linear Algebra and its Applications}, \textbf{22}:211--221,
  1978.

\bibitem[Bar81]{Barker-survey}
G.\bibhairspace P.~Barker.
\newblock Theory of cones.
\newblock {\em Linear Algebra and its Applications}, \textbf{39}:263--291,
  1981.

\bibitem[BCG13]{Bogart-Contois-Gubeladze}
T.~Bogart, M.~Contois, and J.~Gubeladze.
\newblock Hom-polytopes.
\newblock {\em Mathematische Zeitschrift}, \textbf{273}(3--4):1267--1296, 2013.

\bibitem[Ber15]{Bergman}
G.\bibhairspace M.~Bergman.
\newblock {\em An Invitation to General Algebra and Universal Constructions}.
\newblock Universitext. Springer, second edition, 2015.

\bibitem[BGY22]{BenAmor-Gok-Yaman}
M.\bibhairspace A.~\BenAmor{}, {\"O}.~Gok, and D.~Yaman.
\newblock Tensor products of ideals and projection bands, 2022.
\newblock Preprint.
\newblock URL: \url{http://arxiv.org/abs/2207.13796}.

\bibitem[Bir76]{Birnbaum}
D.\bibhairspace A.~Birnbaum.
\newblock Cones in the tensor product of locally convex lattices.
\newblock {\em American Journal of Mathematics}, \textbf{98}(4):1049--1058,
  1976.

\bibitem[BL75]{Barker-Loewy}
G.\bibhairspace P.~Barker and R.~Loewy.
\newblock The structure of cones of matrices.
\newblock {\em Linear Algebra and its Applications}, \textbf{12}(1):87--94,
  1975.

\bibitem[Bla16]{Blanco}
A.~Blanco.
\newblock On the positive approximation property.
\newblock {\em Positivity}, \textbf{20}(3):719--742, 2016.

\bibitem[BLP87]{projectionally-exposed-faces}
G.\bibhairspace P.~Barker, M.~Laidacker, and G.~Poole.
\newblock Projectionally exposed cones.
\newblock {\em SIAM Journal on Algebraic Discrete Methods},
  \textbf{8}(1):100--105, 1987.

\bibitem[Bon54]{Bonsall}
F.\bibhairspace F.~Bonsall.
\newblock Sublinear functionals and ideals in partially ordered vector spaces.
\newblock {\em Proceedings of the London Mathematical Society, 3rd series},
  \textbf{4}(1):402--418, 1954.

\bibitem[BR02]{Bastero-Romance}
J.~Bastero and M.~Romance.
\newblock John's decomposition of the identity in the non-convex case.
\newblock {\em Positivity}, \textbf{6}(1):1--16, 2002.

\bibitem[Br{\o}83]{Brondsted}
A.~Br{\o}ndsted.
\newblock {\em An Introduction to Convex Polytopes}, volume~90 of {\em Graduate
  Texts in Mathematics}.
\newblock Springer, 1983.

\bibitem[BT22]{Buskes-Thorn}
G.~Buskes and P.~Thorn.
\newblock Two results on {F}remlin's {A}rchimedean {R}iesz space tensor
  product, 2022.
\newblock Preprint.
\newblock URL: \url{http://arxiv.org/abs/2206.06283}.

\bibitem[Con07]{Conway}
J.\bibhairspace B.~Conway.
\newblock {\em A Course in Functional Analysis}, volume~96 of {\em Graduate
  Texts in Mathematics}.
\newblock Springer, second edition, 2007.

\bibitem[Day62]{Day}
M.\bibhairspace M.~Day.
\newblock {\em Normed Linear Spaces}.
\newblock Ergebnisse der Mathematik und ihrer Grenzgebiete. Springer, 1962.
\newblock Second printing, corrected.

\bibitem[DF93]{Defant-Floret}
A.~Defant and K.~Floret.
\newblock {\em Tensor Norms and Operator Ideals}, volume 176 of {\em
  Mathematics Studies}.
\newblock North-Holland, 1993.

\bibitem[DFS08]{Diestel-Fourie-Swart}
J.~Diestel, J.\bibhairspace H.~Fourie, and J.~Swart.
\newblock {\em The metric theory of tensor products}.
\newblock American Mathematical Society, 2008.

\bibitem[\Do20]{Dobben-semisimplicity}
J.~\DobbendeBruyn{}.
\newblock Representations and semisimplicity of ordered topological vector
  spaces, 2020.
\newblock Preprint.
\newblock URL: \url{https://arxiv.org/abs/2009.11777}.

\bibitem[\Do22]{Dobben-extensions}
J.~\DobbendeBruyn{}.
\newblock Almost all positive continuous linear functionals can be extended.
\newblock {\em Positivity}, \textbf{26}(1):\#15, 2022.

\bibitem[DS70]{Dermenjian-Saint-Raymond}
Y.~Dermenjian and J.~\SaintRaymond{}.
\newblock Produit tensoriel de deux cones convexes saillants.
\newblock {\em S{\'e}minaire {C}hoquet (initation {\`a} l'analyse)},
  \textbf{9}(20):1--6, 1969/70.
\newblock URL: \url{http://www.numdam.org/item/SC_1969-1970__9_2_A10_0/}.

\bibitem[FNT17]{Fritz-Netzer-Thom}
T.~Fritz, T.~Netzer, and A.~Thom.
\newblock Spectrahedral containment and operator systems with
  finite-dimensional realization.
\newblock {\em SIAM Journal on Applied Algebra and Geometry},
  \textbf{1}(1):556--574, 2017.

\bibitem[Fre72]{Fremlin}
D.\bibhairspace H.~Fremlin.
\newblock Tensor products of {A}rchimedean vector lattices.
\newblock {\em American Journal of Mathematics}, \textbf{94}(3):777--798, 1972.

\bibitem[Fre74]{Fremlin-Banach}
D.\bibhairspace H.~Fremlin.
\newblock Tensor products of {B}anach lattices.
\newblock {\em Mathematische Annalen}, \textbf{211}(2):87--106, 1974.

\bibitem[FT79]{Fremlin-decomposition}
D.\bibhairspace H.~Fremlin and M.~Talagrand.
\newblock A decomposition theorem in tensor products of {A}rchimedean vector
  lattices.
\newblock {\em Mathematika}, \textbf{26}(2):302--305, 1979.

\bibitem[GK10]{vanGaans-Kalauch}
O.~\Gaans{} and A.~Kalauch.
\newblock Tensor products of {A}rchimedean partially ordered vector spaces.
\newblock {\em Positivity}, \textbf{14}(4):705--714, 2010.

\bibitem[GL88]{Grobler-Labuschagne}
J.\bibhairspace J.~Grobler and C.\bibhairspace C.\bibhairspace A.~Labuschagne.
\newblock The tensor product of archimedean ordered vector spaces.
\newblock {\em Mathematical Proceedings of the Cambridge Philosophical
  Society}, \textbf{104}(2):331--345, 1988.

\bibitem[GL89]{Grobler-Labuschagne-f-algebra}
J.\bibhairspace J.~Grobler and C.\bibhairspace C.\bibhairspace A.~Labuschagne.
\newblock An $f$-algebra approach to the {R}iesz tensor product of
  {A}rchimedean {R}iesz spaces.
\newblock {\em Quaestiones Mathematicae}, \textbf{12}(4):425--438, 1989.
\newblock \href {https://doi.org/10.1080/16073606.1989.9632194}
  {\path{doi:10.1080/16073606.1989.9632194}}.

\bibitem[GLMP04]{Gordon-et-al}
Y.~Gordon, A.\bibhairspace E.~Litvak, M.~Meyer, and A.~Pajor.
\newblock John's decomposition in the general case and applications.
\newblock {\em Journal of Differential Geometry}, \textbf{68}(1):99--119, 2004.

\bibitem[GPT01]{Giannopoulos-et-al}
A.~Giannopoulos, I.~Perissinaki, and A.~Tsolomitis.
\newblock John's theorem for an arbitrary pair of convex bodies.
\newblock {\em Geometriae Dedicata}, \textbf{84}:63--79, 2001.

\bibitem[Gro55]{Grothendieck-top}
A.~Grothendieck.
\newblock Produits tensoriels topologiques et espaces nucl{\'e}aires.
\newblock {\em Memoirs of the American Mathematical Society}, \textbf{16},
  1955.

\bibitem[Gro56]{Grothendieck-resume}
A.~Grothendieck.
\newblock R{\'e}sum{\'e} de la th{\'e}orie m{\'e}trique des produits tensoriels
  topologiques.
\newblock {\em Boletim da Sociedade de Matem{\'a}tica de S{\~a}o Paulo},
  \textbf{8}:1--79, 1956.

\bibitem[HFP76]{Haynsworth-Fiedler-Ptak}
E.~Haynsworth, M.~Fiedler, and V.~Pt{\'a}k.
\newblock Extreme operators on polyhedral cones.
\newblock {\em Linear Algebra and its Applications},
  \textbf{13}(1--2):163--172, 1976.

\bibitem[Hil08]{Hildebrand-descriptions}
R.~Hildebrand.
\newblock Semidefinite descriptions of low-dimensional separable matrix cones.
\newblock {\em Linear Algebra and its Applications}, \textbf{429}(4):901--932,
  2008.

\bibitem[HN21]{Huber-Netzer}
B.~Huber and T.~Netzer.
\newblock A note on non-commutative polytopes and polyhedra.
\newblock {\em Advances in Geometry}, \textbf{21}(1):119--124, 2021.
\newblock \href {https://doi.org/10.1515/advgeom-2020-0029}
  {\path{doi:10.1515/advgeom-2020-0029}}.

\bibitem[HP68]{Hulanicki-Phelps}
A.~Hulanicki and R.\bibhairspace R.~Phelps.
\newblock Some applications of tensor products of partially-ordered linear
  spaces.
\newblock {\em Journal of Functional Analysis}, \textbf{2}(2):177--201, 1968.

\bibitem[Joh48]{John}
F.~John.
\newblock Extremum problems with inequalities as subsidiary conditions.
\newblock In {\em Studies and Essays presented to {R}.{} {C}ourant on his 60th
  Birthday}, pages 187--204. Interscience, 1948.

\bibitem[JR77]{deJonge-vanRooij}
E.~\Jonge{} and A.\bibhairspace C.\bibhairspace M.~\Rooij{}.
\newblock {\em Introduction to {R}iesz spaces}, volume~78 of {\em Mathematical
  Centre Tracts}.
\newblock Mathematisch Centrum, Amsterdam, 1977.

\bibitem[Kad51]{Kadison-representation}
R.\bibhairspace V.~Kadison.
\newblock A representation theory for commutative topological algebra.
\newblock {\em Memoirs of the {A}merican {M}athematical {S}ociety}, \textbf{7},
  1951.

\bibitem[Kal78]{Kalton-quotients}
N.\bibhairspace J.~Kalton.
\newblock Quotients of {F}-spaces.
\newblock {\em Glasgow Mathematical Journal}, \textbf{19}(2):103--108, 1978.

\bibitem[Kle59]{Klee-smoothness-rotundity}
V.~Klee.
\newblock Some new results on smoothness and rotundity in normed linear space.
\newblock {\em Mathematische Annalen}, \textbf{139}(1):51--63, 1959.

\bibitem[K{\"o}t79]{Kothe-II}
G.~K{\"o}the.
\newblock {\em Topological Vector Spaces II}, volume 237 of {\em Grundlehren
  der mathematischen Wissenschaften}.
\newblock Springer, New York, NY, 1979.

\bibitem[K{\"o}t83]{Kothe-I}
G.~K{\"o}the.
\newblock {\em Topological Vector Spaces I}, volume 159 of {\em Grundlehren der
  mathematischen Wissenschaften}.
\newblock Springer, Berlin, Heidelberg, second revised printing edition, 1983.

\bibitem[Lam17]{Lami-dissertation}
L.~Lami.
\newblock Non-classical correlations in quantum mechanics and beyond, 2017.
\newblock Ph.D.{} thesis, Universitat Aut{\`o}noma de Barcelona.

\bibitem[\Me91]{Meyer-Nieberg}
P.~\MeyerNieberg{}.
\newblock {\em Banach Lattices}.
\newblock Universitext. Springer, 1991.

\bibitem[Mer64]{Merklen}
H.~Merklen.
\newblock Producto tensorial de espacios vectoriales ordenados.
\newblock {\em Notas de Matem{\'a}ticas (Universidad Nacional de Ingenier{\'i}a
  de Per{\'u})}, \textbf{2}(2):41--57, 1964.

\bibitem[Mes15]{Messerschmidt}
M.~Messerschmidt.
\newblock Normality of spaces of operators and quasi-lattices.
\newblock {\em Positivity}, \textbf{19}(4):695--724, 2015.

\bibitem[M{\"u}l21]{Muller-lecture-notes}
M.\bibhairspace P.~M{\"u}ller.
\newblock Probabilistic theories and reconstructions of quantum theory.
\newblock {\em SciPost Physics Lecture Notes}, \textbf{28}:41pp., 2021.
\newblock \href {https://doi.org/10.21468/SciPostPhysLectNotes.28}
  {\path{doi:10.21468/SciPostPhysLectNotes.28}}.

\bibitem[Mul97]{Mulansky}
B.~Mulansky.
\newblock Tensor products of convex cones.
\newblock In G.~N{\"u}rnberger, J.\bibhairspace W.~Schmidt, and G.~Walz,
  editors, {\em Multivariate Approximation and Splines}, volume 125 of {\em
  International Series on Numerical Mathematics}. Birkh{\"a}user, Basel, 1997.
\newblock \href {https://doi.org/10.1007/978-3-0348-8871-4_14}
  {\path{doi:10.1007/978-3-0348-8871-4_14}}.

\bibitem[Nie82]{Nielsen-operators}
N.\bibhairspace J.~Nielsen.
\newblock The ideal property of tensor products of {B}anach lattices with
  applications to the local structure of spaces of absolutely summing
  operators.
\newblock {\em Studia Mathematica}, \textbf{74}(3):247--272, 1982.

\bibitem[Nie88]{Nielsen-approximation}
N.\bibhairspace J.~Nielsen.
\newblock The positive approximation property of {B}anach lattices.
\newblock {\em Israel Journal of Mathematics}, \textbf{62}(1):99--112, 1988.

\bibitem[Per67]{Peressini}
A.\bibhairspace L.~Peressini.
\newblock {\em Ordered Topological Vector Spaces}.
\newblock Harper's Series in Modern Mathematics. Harper {\&} Row, 1967.

\bibitem[Pis20]{Pisier-Connes}
G.~Pisier.
\newblock {\em Tensor products of $C^*$-algebras and operator spaces}.
\newblock Number~96 in London Mathematical Society Student Texts. Cambridge
  University Press, 2020.

\bibitem[Pl{\'a}21]{Plavala-lecture-notes}
M.~Pl{\'a}vala.
\newblock General probabilistic theories: An introduction, 2021.
\newblock Preprint.
\newblock URL: \url{https://arxiv.org/abs/2103.07469}.

\bibitem[Poo75]{Poole-dissertation}
M.\bibhairspace W.~Poole.
\newblock Structure properties of polyhedral cones, 1975.
\newblock PhD thesis, Auburn University.

\bibitem[Pop68]{Popa-I}
N.~Popa.
\newblock Produits tensoriels ordonn{\'e}s.
\newblock {\em Revue Roumaine de Math{\'e}matiques Pures et Appliqu{\'e}es},
  \textbf{13}(2):235--246, 1968.

\bibitem[Pop69]{Popa-II}
N.~Popa.
\newblock Sur les produits tensoriels ordonn{\'e}s.
\newblock {\em Mathematische Zeitschrift}, \textbf{110}(3):206--210, 1969.

\bibitem[PS69]{Peressini-Sherbert}
A.\bibhairspace L.~Peressini and D.\bibhairspace R.~Sherbert.
\newblock Ordered topological tensor products.
\newblock {\em Proceedings of the London Mathematical Society, 3rd series},
  \textbf{19}(1):177--190, 1969.

\bibitem[PSS18]{Passer-Shalit-Solel}
B.~Passer, O.\bibhairspace M.~Shalit, and B.~Solel.
\newblock Minimal and maximal matrix convex sets.
\newblock {\em Journal of Functional Analysis}, \textbf{274}(11):3197--3253,
  2018.

\bibitem[PTT11]{Paulsen-Todorov-Tomforde}
V.\bibhairspace I.~Paulsen, I.\bibhairspace G.~Todorov, and M.~Tomforde.
\newblock Operator system structures on ordered spaces.
\newblock {\em Proceedings of the London Mathematical Society, 3rd series},
  \textbf{102}(1):25--49, 2011.

\bibitem[Roc70]{Rockafellar}
R.\bibhairspace T.~Rockafellar.
\newblock {\em Convex Analysis}.
\newblock Princeton University Press, 1970.

\bibitem[RS82]{Ruess-Stegall-extreme}
W.\bibhairspace M.~Ruess and C.\bibhairspace P.~Stegall.
\newblock Extreme points in duals of operator spaces.
\newblock {\em Mathematische Annalen}, \textbf{261}(4):535--546, 1982.

\bibitem[RS86]{Ruess-Stegall-weak*-denting}
W.\bibhairspace M.~Ruess and C.\bibhairspace P.~Stegall.
\newblock Weak*-denting points in duals of operator spaces.
\newblock In N.\bibhairspace J.~Kalton and E.~Saab, editors, {\em Banach
  Spaces, Proceedings of the {M}issouri Conference held in {C}olumbia, {USA},
  {J}une 24--29, 1984}, volume 1166 of {\em Lecture Notes in Mathematics},
  pages 158--169. Springer, 1986.

\bibitem[Rud91]{Rudin}
W.~Rudin.
\newblock {\em Functional Analysis}.
\newblock McGraw--Hill, second edition, 1991.

\bibitem[Rya02]{Ryan}
R.\bibhairspace A.~Ryan.
\newblock {\em Introduction to Tensor Products of {B}anach Spaces}.
\newblock Springer Monographs in Mathematics. Springer, London, 2002.

\bibitem[Sch58]{Schaefer-I}
H.~Schaefer.
\newblock Halbgeordnete lokalkonvexe {V}ektorr{\"a}ume.
\newblock {\em Mathematische Annalen}, \textbf{135}(2):115--141, 1958.

\bibitem[Sch60]{Schaefer-III}
H.~Schaefer.
\newblock Halbgeordnete lokalkonvexe {V}ektorr{\"a}ume {III}.
\newblock {\em Mathematische Annalen}, \textbf{141}(2):113--142, 1960.

\bibitem[Sch72]{Schaefer-tensor}
H.\bibhairspace H.~Schaefer.
\newblock Normed tensor products of {B}anach lattices.
\newblock {\em Israel Journal of Mathematics}, \textbf{13}(4):400--415, 1972.

\bibitem[Sch74]{Schaefer-pos-ops-book}
H.\bibhairspace H.~Schaefer.
\newblock {\em Banach Lattices and Positive Operators}.
\newblock Grundlehren der mathematischen Wissenschaften. Springer, 1974.

\bibitem[Sch99]{Schaefer}
H.\bibhairspace H.~Schaefer.
\newblock {\em Topological Vector Spaces}, volume~3 of {\em Graduate Texts in
  Mathematics}.
\newblock Springer, second edition, 1999.

\bibitem[Sha21]{Shalit-dilation}
O.\bibhairspace M.~Shalit.
\newblock Dilation theory: A guided tour.
\newblock In M.\bibhairspace A.~Bastos, L.~Castro, and A.\bibhairspace
  Y.~Karlovich, editors, {\em Operator Theory, Functional Analysis and
  Applications}, pages 551--623. Birkh{\"a}user, Cham, 2021.

\bibitem[ST90]{projectionally-exposed-faces-ii}
C.\bibhairspace H.~Sung and B.\bibhairspace S.~Tam.
\newblock A study of projectionally exposed cones.
\newblock {\em Linear Algebra and its Applications}, \textbf{139}:225--252,
  1990.

\bibitem[SW70]{Stoer-Witzgall}
J.~Stoer and C.~Witzgall.
\newblock {\em Convexity and Optimization in Finite Dimensions I}, volume 163
  of {\em Die Grundlehren der mathematischen Wissenschaften}.
\newblock Springer, 1970.

\bibitem[Tam77a]{Tam-dissertation}
B.\bibhairspace S.~Tam.
\newblock Some aspects of finite dimensional cones, 1977.
\newblock PhD thesis, University of Hong Kong.

\bibitem[Tam77b]{Tam-projective-closed}
B.\bibhairspace S.~Tam.
\newblock Some results of polyhedral cones and simplicial cones.
\newblock {\em Linear and Multilinear Algebra}, \textbf{4}(4):281--284, 1977.

\bibitem[Tam85]{Tam-duality}
B.\bibhairspace S.~Tam.
\newblock On the duality operator of a convex cone.
\newblock {\em Linear Algebra and its Applications}, \textbf{64}:33--56, 1985.

\bibitem[Tam92]{Tam-faces}
B.\bibhairspace S.~Tam.
\newblock On the structure of the cone of positive operators.
\newblock {\em Linear Algebra and its Applications}, \textbf{167}:65--85, 1992.

\bibitem[Tam95]{Tam-survey}
B.\bibhairspace S.~Tam.
\newblock Extreme positive operators on convex cones.
\newblock In K.\bibhairspace Y.~Chan and M.\bibhairspace C.~Liu, editors, {\em
  Five Decades as a Mathematician and Educator: On the 80th Birthday of
  Professor {Y}ung-{C}how {W}ong}, pages 515--558. World Scientific, Singapore,
  1995.

\bibitem[Tse76]{Tseitlin}
I.\bibhairspace I.~Tseitlin.
\newblock The extreme points of the unit ball of certain spaces of operators.
\newblock {\em Matematicheskie Zametki}, \textbf{20}(4):521--527, 1976.

\bibitem[Wei12]{Weis}
S.~Weis.
\newblock Duality of non-exposed faces.
\newblock {\em Journal of Convex Analysis}, \textbf{19}(3):815--835, 2012.

\bibitem[Wer87]{Werner}
D.~Werner.
\newblock Denting points in tensor products of {B}anach spaces.
\newblock {\em Proceedings of the American Mathematical Society},
  \textbf{101}(1):122--126, 1987.

\bibitem[Wit74]{Wittstock}
G.~Wittstock.
\newblock Eine {B}emerkung {\"u}ber {T}ensorprodukte von {B}anachverb{\"a}nden.
\newblock {\em Archiv der Mathematik}, \textbf{25}(1):627--634, 1974.

\bibitem[Wor19]{Wortel}
M.~Wortel.
\newblock Lexicographic cones and the ordered projective tensor product.
\newblock In G.~Buskes, M.~de~Jeu, P.~Dodds, A.~Schep, F.~Sukochev, J.~van
  Neerven, and A.~Wickstead, editors, {\em Positivity and Noncommutative
  Analysis: Festschrift in Honour of Ben de Pagter on the Occasion of his 65th
  Birthday}, pages 601--609. Birkh{\"a}user, Cham, 2019.
\newblock \href {https://doi.org/10.1007/978-3-030-10850-2_30}
  {\path{doi:10.1007/978-3-030-10850-2_30}}.

\bibitem[Zaa97]{Zaanen}
A.\bibhairspace C.~Zaanen.
\newblock {\em Introduction to Operator Theory in {R}iesz Spaces}.
\newblock Springer, 1997.

\bibitem[Zam87]{Zamfirescu-smooth-strictly-convex}
T.~Zamfirescu.
\newblock Nearly all convex bodies are smooth and strictly convex.
\newblock {\em Monatshefte f{\"u}r Mathematik}, \textbf{103}(1):57--62, 1987.

\end{thebibliography}

\end{document}